\documentclass[3p,times]{elsarticle}
\usepackage{natbib}
\biboptions{sort&compress}
\usepackage{fleqn}
\usepackage{amsmath}
\usepackage{amsfonts}
\usepackage{amssymb}
\usepackage{amsthm}
\usepackage{epsfig}
\usepackage{makeidx}
\usepackage{relsize}
\usepackage[export]{adjustbox}
\usepackage{graphics}
\usepackage{epstopdf}
\usepackage{graphicx}
\usepackage{color}
\usepackage[normalem]{ulem}
\usepackage{algorithm}
\usepackage{algpseudocode}
\def\A{{\bf A}}
\def\C{{\bf C}}
\def\L{{\bf L}}

\def\b{{\bf b}}
\def\d{{\bf d}}
\def\e{{\bf e}}
\def\n{{\bf n}}
\def\p{{\bf p}}
\def\r{{\bf r}}
\def\s{{\bf s}}
\def\t{{\bf t}}
\def\z{{\bf z}}
\def\q{{\bf q}}
\def\x{{\bf x}}
\def\y{{\bf y}}

\newcommand{\ZZ}{{\mathbb Z}}
\newcommand{\RR}{\mathbb{R}}
\newcommand{\NN}{\mathbb{N}}

\newcommand{\field}[1]{\mathbb{#1}}
\newcommand{\R}{\field{R}}

\newcommand\ti{\mathrm{i}}

\def\bnu{\mbox{\boldmath $\nu$}}
\def\bmu{\mbox{\boldmath $\mu$}}
\def\brho{\mbox{\boldmath $\rho$}}
\def\btau{\mbox{\boldmath $\tau$}}
\def\bchi{\mbox{\boldmath $\chi$}}
\def\bzeta{\mbox{\boldmath $\zeta$}}

\newcommand{\be}{\begin{equation}}
\newcommand{\ee}{\end{equation}}
\newcommand{\ba}{\begin{eqnarray}}
\newcommand{\ea}{\end{eqnarray}}
\newcommand{\bi}{\begin{itemize}}
\newcommand{\ei}{\end{itemize}}

\newtheorem{dfn}{Definition}
\newtheorem{rmk}{Remark}

\newtheorem{crl}{Corollary}
\newtheorem{prn}{Proposition}

\newtheorem{prb}{Problem}
\newtheorem{exm}{Example}

\begin{document}

\begin{frontmatter}

\title{Pythagorean-Hodograph B-Spline Curves}

\author[label1]{Gudrun Albrecht\corref{cor1}}
\ead{gudrun.albrecht@univ-valenciennes.fr}

\author[label2]{Carolina Vittoria Beccari}
\ead{carolina.beccari2@unibo.it}

\author[label1]{Jean-Charles Canonne}
\ead{jean-charles.canonne@univ-valenciennes.fr}

\author[label3]{Lucia Romani}
\ead{lucia.romani@unimib.it}

\cortext[cor1]{Corresponding author.}

\address[label1]{Univ Lille Nord de France, UVHC, LAMAV, FR CNRS 2956, F-59313
Valenciennes, France.}

\address[label2]{Department of Mathematics, University of Bologna, P.zza Porta San Donato 5, 40127 Bologna, Italy}

\address[label3]{Department of Mathematics and Applications, University of Milano-Bicocca, Via R. Cozzi 55, 20125 Milano, Italy}

\begin{abstract}
We introduce the new class of planar Pythagorean-Hodograph (PH)
B-Spline curves. They can be seen as a generalization of the
well-known class of planar Pythagorean-Hodograph (PH) B\'ezier
curves, presented by R. Farouki and T. Sakkalis in 1990, including
the latter ones as special cases. Pythagorean-Hodograph B-Spline
curves are non-uniform parametric B-Spline curves whose arc-length
is a B-Spline function as well. An important consequence of this
special property is that the offsets of Pythagorean-Hodograph
B-Spline curves are non-uniform rational B-Spline (NURBS) curves.
Thus, although Pythagorean-Hodograph B-Spline curves have fewer
degrees of freedom than general B-Spline curves of the same degree,
they offer unique advantages for computer-aided design and
manufacturing, robotics, motion control, path planning, computer
graphics, animation, and related fields. After providing a general
definition for this new class of planar parametric curves, we
present useful formulae for their construction, discuss their
remarkable attractive properties and give some examples of their
practical use.
\end{abstract}


\end{frontmatter}

\section{Introduction}

The purpose of the present article is to introduce the general
concept of Pythagorean--Hodograph (PH) B--Spline curves.
On the one hand, B--Spline curves, since their introduction by
Schoenberg \cite{schoenberg} in 1946 have become the standard for
curve representation in all areas where curve design is an issue,
see, e.g., \cite{deboor, gallier, hoschek}.
On the other hand, the concept of polynomial PH curves has widely
been studied since its introduction by Farouki and Sakkalis in
\cite{faroukisakkalis}. The essential characteristic of these curves
is that the Euclidean norm of their hodograph is also polynomial,
thus yielding the useful properties of admitting a closed--form
polynomial representation of their arc--length as well as exact
rational parameterizations of their offset curves. These polynomial
curves are defined over the space of polynomials using its Bernstein
basis thus yielding a control point or so--called B\'ezier
representation for them. Rational and spatial counterparts of
polynomial PH curves have as well been proposed, and most recently
an algebraic--trigonometric counterpart, so--called
Algebraic--Trigonometric Pythagorean--Hodograph (ATPH) curves have
been introduced in \cite{phtrig13}.

So far a general theory for {\it B--Spline curves} having the {\it
PH property} is missing. To the best of the authors' knowledge the
only partial attempt in this direction has been made in
\cite{faroukiphbsplines}, where the problem of determining a
B--Spline form of a $C^2$ PH quintic spline curve interpolating
given points is addressed. Prior to this, based on \cite{ph,
faroukial2001} in \cite{pelosial2007} a relation between a planar
$C^2$ PH quintic spline curve and the control polygon of a related
$C^2$ cubic B-Spline curve is presented.

The present article shows how to construct a general PH B--Spline
curve of arbitrary degree, over an arbitrary knot sequence. To this
end, we start by defining the complex variable model of a B--Spline
curve $\z(t)$ of degree $n$, defined over a knot partition $\bmu$.
We then square $\z(t)$ by using results for the product of
normalized B--Spline basis functions from \cite{cinesi, morken}.
Here, the determination of the required coefficients involves the
solution of linear systems of equations. Finally, the result is
integrated in order to obtain the general expression of the PH
B--Spline curve, i.e., its B--Spline control points and its knot
partition $\brho$. General formulae are derived also for the
parametric speed, the arc length and the offsets of the resulting
curves. The interesting subclasses of \emph{clamped} and
\emph{closed} PH B--spline curves are discussed in great detail.
When the degree is 3 and 5, explicit expressions of their control
points are given together with the B-Spline representation of the
associated arc-length and the rational B-Spline representation of
their offsets. Finally, clamped quintic PH B--Spline curves are used
to solve a second order Hermite interpolation problem.

The remainder of the paper is organized as follows. In section
\ref{sec2} we recall the basic definition of B--Spline curves as
well as the Pythagorean Hodograph (PH) concept, thus defining the
notion of a PH B--Spline curve. In section \ref{sec3}, the
general construction of PH B--Spline curves is developed (section
\ref{sec3.1}), and then adapted to the important particular
cases of clamped and closed PH B--Spline curves (section
\ref{sec3.2}). In section \ref{sec:properties}, general formulae for their parametric speed, arc length and offsets are
given.
 Section \ref{sec5}
is devoted to presenting the explicit expressions regarding clamped
and closed PH B--Spline curves of degree $3$ and $5$. Finally, in
section \ref{sec6} we solve a second order Hermite interpolation
problem by clamped PH B--Spline curves of degree $5$. Conclusions
are drawn in section \ref{sec7}.

\section{Preliminary notions and notation}
\label{sec2}

While a B\'ezier curve is univocally identified by its degree, a
B-Spline curve involves more information, namely an arbitrary number
of control points, a knot vector and a degree, which are related by
the formula {\it number of knots - number of control points = degree
+1}. For readers not familiar with B-Spline curves, we first recall
the definition of normalized B-Spline basis functions and
successively the one of planar B-Spline curve (see, e.g.,
\cite{hoschek}).

\begin{dfn}\label{defbspline}
Let $\bmu = \{t_i \in \RR \; | \; t_i \le t_{i+1} \}_{i \in \ZZ}$ be a sequence of non-decreasing real numbers called \emph{knots}, and let $n \in \NN$.
The $i$-th \emph{normalized B-spline basis function} of degree $n$ defined over the \emph{knot partition} $\bmu$ is
the function $N_{i,\bmu}^n(t)$ having support $[t_i, t_{i+n+1}]$ and defined recursively as
\begin{equation*}
N_{i,\bmu}^n(t) = \frac{t-t_i}{t_{i+n} - t_i} N_{i,\bmu}^{n-1}(t) +
\frac{t_{i+n+1}-t}{t_{i+n+1} - t_{i+1}} N_{i+1,\bmu}^{n-1}(t),
\end{equation*}
where
\begin{equation*}
N_{i,\bmu}^0(t) = \left\{ \begin{array}{l} 1\, , \quad \hbox{if} \ t \in [t_i, t_{i+1}) \\
 0\, , \quad \hbox{otherwise}
 \end{array}\right.
\end{equation*}
and $``\frac{0}{0} = 0"$.
\end{dfn}

\begin{dfn}\label{defbspline_curve}
Let $m,n \in \NN$ with $m \ge n$, $\bmu =\{t_{i}\}_{i=0,...,m+n+1}$
be a finite knot partition, and $\s_0, \ldots, \s_m \in \RR^2 $.
Then, the planar parametric curve
\begin{equation*}
\s(t) = \sum_{i=0}^{m} \s_i \, N_{i,\bmu}^n(t)\, , \quad t \in [t_n,
t_{m+1}],
\end{equation*}
is called a planar \emph{B--Spline curve} (of degree $n$ associated
with the knot partition $\bmu$) with de Boor points or control
points $\s_0, \ldots, \s_m$.
\end{dfn}

\begin{rmk}
If $M_i$ denotes the multiplicity of the knot $t_i$, then $\s(t)$ is
of continuity class $C^{n-max_i(M_i)}(t_n,t_{m+1})$.
\end{rmk}

\smallskip
If the knot vector $\bmu$ does not have any particular structure,
the B-Spline curve $\s(t)$ will not pass through the first and last
control points neither will be tangent to the first and last legs of
the control polygon. In this case $\s(t)$ is simply called
\emph{open} B-spline curve. In order to clamp $\s(t)$ so that it is
tangent to the first and the last legs at the first and last control
points, respectively (as a B\'ezier curve does), the multiplicity of
the first and the last knot must be adapted. For later use, the
precise conditions we use for identifying a clamped B-Spline curve
are the following.

\begin{rmk}\label{rembspline_curve_a}
If $t_{0} = t_{1} = \ldots = t_n$ and $t_{m+1} = \ldots = t_{m+n+1}$,
then the B-Spline curve $\s(t)$ given in Definition \ref{defbspline_curve} verifies
$$\s(t_n) = \s_0, \quad \s(t_{m+1})=\s_m$$
as well as
$$
\s'(t_n)=\frac{n}{t_{n+1}-t_1} \, (\s_1-\s_0), \quad
\s'(t_{m+1})=\frac{n}{t_{m+n}-t_m} \, (\s_m-\s_{m-1}),
$$
i.e., $\s(t)$ is a \emph{clamped} B--Spline curve. Moreover, if all
the knots $t_{n+1}, \, \ldots, \, t_{m}$ are simple, then $\s(t) \in
C^{n-1}(t_n, t_{m+1})$.
%Thus, the clamped B-spline curve is an interesting cross between a general open B-Spline curve and a B\'ezier curve.
\end{rmk}

\smallskip
On the other hand, to make the B-Spline curve $\s(t)$ closed, some knot intervals and control points must be repeated
such that the start and the end of the generated curve join together forming a closed loop. The precise conditions to be satisfied by knot intervals and control points in order to get a closed B-Spline curve are recalled in the following.

\begin{rmk}\label{rembspline_curve_b}
If, in Definition \ref{defbspline_curve} we replace $m$ by $m+n$,
consider the knot partition $\bmu = \{t_{i}\}_{i=0,...,m+2n+1}$ with
$t_{m+1+k}-t_{m+k} = t_k - t_{k-1}$ for $k=2, \ldots, 2n-1$,  assume
$\s_0, \ldots, \s_m \in \RR^2 $ to be distinct control points and
$\s_{m+1}=\s_0, \ldots, \s_{m+n}=\s_{n-1}$, then the B-Spline curve
\begin{equation*}
\s(t) = \sum_{i=0}^{m+n} \s_i N_{i,\bmu}^n(t)\, , \quad  t \in [t_n,
t_{m+n+1}],
\end{equation*}
has the additional property
$$\s(t_n) = \s(t_{m+n+1}),$$
i.e., $\s(t)$ is a \emph{closed} B--Spline curve. Moreover, if all
the knots $t_{n}, \, \ldots, \, t_{m+n+1}$ are simple, then $\s(t)
\in C^{n-1}[t_{n}, t_{m+n+1}]$.
\end{rmk}

At this point, we have all the required preliminary notions to generalize the definition of Pythagorean-Hodograph B\'ezier curves
(see \cite{faroukisakkalis}) to \emph{Pythagorean-Hodograph B-Spline curves}.

\begin{dfn}\label{ph}
For $p,n \in \NN, p \ge n$, let $u(t)$, $v(t)$ and $w(t)$ be non-zero degree-$n$ spline functions
over the knot partition $\bmu =\{t_{i}\}_{i=0,...,p+n+1}$, i.e., let
$$
u(t)=\sum_{i=0}^p u_i N_{i,\bmu}^n(t), \quad v(t)=\sum_{i=0}^p v_i N_{i,\bmu}^n(t), \quad w(t)=\sum_{i=0}^p w_i N_{i,\bmu}^n(t), \qquad t \in [t_n, t_{p+1}],
$$
with $u_i, v_i, w_i \in \RR$ for all $i=0, \ldots, p$, such that
$u(t)$ and $v(t)$ are non-constant and do not have a non--constant
spline function over the partition $\bmu$ as common factor. Then,
the planar parametric curve $(x(t), \, y(t))$ whose coordinate
components have first derivatives of the form \be x'(t)=w(t) \Big(
u^2(t)-v^2(t) \Big) \qquad \hbox{and} \qquad y'(t)=2 w(t)u(t)v(t),
\label{xyprime} \ee is called a \emph{planar Pythagorean-Hodograph
B-Spline curve} or \emph{a planar PH B-Spline curve} of degree $2n +
1$.
\end{dfn}

\smallskip
Indeed, as in the case of PH polynomial B\'ezier curves
\cite{faroukisakkalis,farouki95}, the parametric speed of the plane curve $(x(t), \, y(t))$ is
given by \be \label{sigmareal} \sigma(t) :=
\sqrt{(x'(t))^2+(y'(t))^2} = w(t)\, (u^2(t) + v^2(t)), \ee and
its unit tangent, unit normal and (signed) curvature are given
respectively by \be \label{tnk} \t=\frac{(u^2-v^2,2uv)}{u^2+v^2},
\quad \n=\frac{(2uv,v^2-u^2)}{u^2+v^2}, \quad
\kappa=\frac{2(uv'-u'v)}{w(u^2+v^2)^2}, \ee where, for
conciseness, in (\ref{tnk}) the parameter $t$ is omitted.

\smallskip
In the following we will restrict our attention to the so-called \emph{primitive} case
$w(t)=1$. Since in this case equation \eqref{sigmareal} simplifies as $\sigma(t)=u^2(t)+v^2(t)$, the primitive case coincides with the \emph{regular} case.
In this case, the representation (\ref{xyprime}) may be
obtained by squaring the complex function $\z(t) = u(t) + \ti v(t)$
yielding $\z^2(t) = u^2(t)-v^2(t) + \ti 2 u(t) v(t)$. The coordinate components $x'(t), y'(t)$ of the hodograph $\r'(t)$ of
the parametric curve $\r(t)=(x(t), y(t))$ are thus given by the real and imaginary part of
$\z^2(t)$, respectively. In the remainder of the paper we will exclusively use
this complex notation, and we will thus write
\be \label{xyprimecomplex}
\r'(t) = x'(t) + \ti y'(t) = u^2(t)-v^2(t) + \ti 2 u(t) v(t) =
\z^2(t)\,,
\ee
as also previously done for planar PH quintics \cite{faroukisakkalis,farouki95}.
Since, by construction, $\r'(t)$ is a degree-$2n$ B-spline curve, then the PH B-Spline
curve  $\r(t) = \int \r'(t) dt$ has degree $2n+1$. In the next section we will construct the corresponding knot vector and thus know the continuity class of the
resulting PH B-Spline curve $\r(t)$.

\section{Construction of Pythagorean--Hodograph B--Spline curves}
\label{sec3}

\subsection{The general approach}
\label{sec3.1}

We start with a knot partition of the form
\be \label{mu} \bmu =
\{t_{i}\}_{i=0,...,p+n+1} \ee
over which a degree-$n$ B-Spline curve
\[\z(t) = u(t) + \ti
v(t)
\]
is defined for $t \in [t_n, t_{p+1}]$. Thus, according to Definition
\ref{defbspline_curve}, the planar parametric curve $\z(t)$ can be
written as \be \label{zbs} \z(t) = \sum_{i=0}^p \z_i
N_{i,\bmu}^n(t)\,, \quad  t \in [t_n, t_{p+1}]\,, \ee where
$\z_i=u_i + \ti v_i$, $i=0, \ldots, p$. To express the product
$\z^2(t)$ as a B-Spline curve, according to \cite{morken, cinesi},
we have to augment the multiplicity of each single knot $t_i$ to $n
+ 1$. We thus obtain the knot partition \be\label{nu} \bnu =
\{s_i\}_{i=0,...,(p+n+2)(n+1)-1} = \{<t_i>^{n+1}\}_{i=0,...,p+n+1}\,
, \ee where $<t_i>^{k}$ denotes a knot $t_i$ of multiplicity $k$.
The product $\mathbf{z} ^{2}(t)$ is thus a degree-$2n$ B-Spline
curve over the knot partition $\bnu$, which can be written in the
form
\begin{equation}
\label{zsquare}\mathbf{p}(t)=\mathbf{z}^{2}(t)=\sum_{i=0}^{p}\sum_{j=0}^{p}\mathbf{z}_{i}\mathbf{z}_{j}N^{n}_{i,\bmu}(t)N^{n}_{j,\bmu}(t)
= \sum_{k=0}^q \p_k N_{k,\bnu}^{2n}(t)\, ,
\end{equation}
with $q=(n+1)(p+n)$, according to Definition \ref{defbspline_curve}.
Our goal is thus to obtain the explicit expressions of the
coefficients $\p_k$, for $k=0,\ldots,q$. To this end we set
$f_{i,j}(t) := N^{n}_{i,\bmu}(t) N^{n}_{j,\bmu}(t)$ and look for the
unknown coefficients $\bchi^{i,j}:=(\chi_{0}^{i,j}, \chi_{1}^{i,j},
..., \chi_{q}^{i,j})^T$, $i,j=0,...,p$ such that
\begin{equation}\label{prod}
f_{i,j}(t) = \sum_{k=0}^q \chi_k^{i,j} N_{k,\bnu}^{2n}(t)\,.
\end{equation}
For accomplishing this we apply the method from \cite{cinesi} as follows. Let
$<\cdot,\cdot>$ be an inner product of the linear space of B-splines
of degree $2n$ with knot vector $\bnu$. According to \cite{cinesi},
for any pair of functions $a(t)$, $b(t)$ defined over the interval $[t_0,t_{p+n+1}]$,
we use $<a(t),b(t)> = \int_{t_0}^{t_{p+n+1}} a(t) b(t) \, dt$ to construct the $(q+1) \times (q+1)$
%equations
%\begin{equation*}\label{affine}
%<f_{i,j}(t),N^{2n}_{l,\bnu}(t)>=\sum_{k=0}^{q}\chi_{k}^{i,j} <N^{2n}_{k,\bnu}(t),N^{2n}_{l,\bnu}(t)>, \qquad l=0,\ldots, q, \end{equation*}
%which yield the
linear equation system
\begin{equation}\label{lgschi}
\A \bchi^{i,j}=\b^{i,j},
\end{equation}
with
$$
\A=(a_{k,l})_{k,l=0,...,q}, \quad
a_{k,l}:=\langle N_{k,\bnu}^{2n}, N_{l,\bnu}^{2n} \rangle=\int_{t_0}^{t_{p+n+1}} N_{k,\bnu}^{2n}(t) \, N_{l,\bnu}^{2n}(t) \, dt
$$
and
$$\b^{i,j}=(b^{i,j}_l)_{l=0,...,q}, \quad b^{i,j}_l:=\langle f_{i,j}, N_{l,\bnu}^{2n} \rangle=\int_{t_0}^{t_{p+n+1}}  f_{i,j}(t) \, N_{l,\bnu}^{2n}(t) \, dt.
$$
Since $\{ N^{2n}_{k,\bnu}(t)\}_{k=0,...,q}$ are linearly independent, the matrix $\A$ is a Gramian and therefore nonsingular.
This allows us to work out the unknown coefficients $\{ \chi_{k}^{i,j} \}_{k=0,...,q}$
by solving the linear system in \eqref{lgschi}.

\begin{rmk}
Since $\b^{i,j}=\b^{j,i}$ for all $i,j=0,...,p$, then $\bchi^{i,j}=\bchi^{j,i}$ for all $i,j=0,...,p$.
Therefore, the unknown vectors to be obtained from \eqref{lgschi} are indeed $\bchi^{i,j}$, $i=0,...,p$, $j=0,...,i$.
Being $\A$ a non-singular Gramian matrix, all the corresponding linear systems always have a unique solution.
Moreover, since $a_{h,k}=0$ if $|h-k|>2n$, $\A$ is not only symmetric and positive definite, but also of band form.
Thus, by applying the Cholesky decomposition algorithm one can compute the factorization $\A=\L \L^T$, where
$\L$ is a lower triangular matrix of the same band form of $\A$ (i.e. such that $l_{h,k}=0$ if $h-k>2n$).
The non-zero elements of $\L$ may be determined row by row by the formulas
$$
\begin{array}{lll}
l_{h,k}&=&\left(a_{h,k}-\sum_{s=h-2n}^{h-1} l_{h,s} \, l_{k,s}\right)/l_{k,k}, \quad k=h-2n,...,h-1 \smallskip \\
l_{h,h}&=&\left(a_{h,h}-\sum_{s=h-2n}^{h-1} l_{h,s}^2 \right)^{\frac{1}{2}},
\end{array}
$$
with the convention that $l_{r,c}=0$ if $c \leq 0$ or $c>r$.
Hence, the solution of each linear system in \eqref{lgschi} can be easily obtained by solving the two triangular linear systems
$\L \y^{i,j} =\b^{i,j}$ and $\L^T \bchi^{i,j} = \y^{i,j}$
via the formulas
$$
\begin{array}{lll}
y^{i,j}_h&=&\left(b_h-\sum_{k=h-2n}^{h-1} l_{h,k} \, y^{i,j}_k \right)/l_{h,h}, \qquad h=0,...,q \smallskip \\
{\chi}^{i,j}_h&=&\left(y^{i,j}_h- \sum_{k=h+1}^{h+2n} l_{k,h} \, {\chi}^{i,j}_k \right)/l_{h,h}, \qquad h=0,...,q
\end{array}
$$
where a similar convention as above is adopted with respect to suffices outside the permitted ranges (see \cite{martin}).
\end{rmk}

\smallskip
From the computed expressions of $\chi_{k}^{i,j}$, $k=0,...,q$, $0 \le i,j \le p$,  we thus get
\begin{equation}
\label{pks}
\mathbf{p}(t)=\mathbf{z}^{2}(t)=\sum_{k=0}^{q}\sum_{i=0}^{p}\sum_{j=0}^{p}\chi_{k}^{i,j}\mathbf{z}_{i}\mathbf{z}_{j}N^{2n}_{k,\bnu}(t)
\qquad \hbox{and} \qquad \mathbf{p}_{k}=\sum_{i=0}^{p}\sum_{j=0}^{p}\chi_{k}^{i,j}\mathbf{z}_{i}\mathbf{z}_{j}\,, \quad k=0, \ldots, q.
\end{equation}
The resulting PH B-Spline curve $\r(t)$ is now obtained by
integrating $\p(t)$ as: \be \label{rt} \r(t) = \int \p(t) dt =
\sum_{i=0}^{q+1} \r_i N_{i,\brho}^{2n+1}(t)\, , \; \; t \in [t_n,
t_{p+1}]\, ,\ee where
$\brho = \{s'_i\}_{i=0,...,(p+n+2)(n+1)+1}$ with
$s'_i = s_{i-1}$ for $i=1, \ldots, (p+n+2)(n+1)$, $t_{-1}=s'_0\le
s'_1$ and $s'_{(p+n+2)(n+1)} \le s'_{(p+n+2)(n+1)+1} = t_{p+n+2}$,
i.e., \be \label{rho} \brho =\{t_{-1}, \
\{<t_k>^{n+1}\}_{k=0,...,p+n+1}, \ t_{p+n+2}\} \ee with the additional
knots $t_{-1}, t_{p+n+2}$, as well as \be \label{rs} \r_{i+1} = \r_i
+ \frac{s'_{i+2n+2} - s'_{i+1}}{2n+1}\, \p_i = \r_i +
\frac{s_{i+2n+1} - s_{i}}{2n+1}\, \p_i\,, \ee for $i=0,\ldots,q$ and
arbitrary $\r_0$.

\begin{rmk}
Note that, by construction, $s_{2n}=s'_{2n+1}=t_n$ as well as $s_{q+1}=s'_{q+2}=t_{p+1}$, namely the B-spline curves $\z(t)$, $\p(t)$ and $\r(t)$ are defined on the same domain.
If the knot partition $\bmu$ contains simple inner
knots $t_{n+1}, \ldots, t_p$, then the degree-$n$ spline $\z(t) \in C^{n-1}(t_n, t_{p+1})$. As a consequence, the degree-$2n$ spline $\r'(t) \in C^{n-1}(t_n, t_{p+1})$
and the degree-$(2n+1)$ spline
$\r(t) \in C^{n}(t_n, t_{p+1})$.
\end{rmk}

\subsection{Construction of clamped and closed PH B-Spline curves}
\label{sec3.2}

We now consider the conditions for obtaining a clamped, respectively
closed, PH B-Spline curve $\r(t)$.

\begin{prn} \label{prop1}
Let $\r(t)$ be the PH B-Spline curve in (\ref{rt}) defined over the
knot partition $\brho$ in (\ref{rho}), where for a clamped,
respectively, closed PH B-Spline curve $\r(t)$ we assume $p=m$,
respectively, $p=m+n$.
\begin{enumerate}
\item[a)]
For $\r(t)$ to be clamped, i.e., satisfying
\be
\label{condclamped}
\begin{array}{c}
\r(t_n)=\r_0\;, \quad \r(t_{m+1})=\r_{q+1}, \smallskip\\
\r'(t_n)=\frac{2n+1}{s'_{2n+2}-s'_1} \, (\r_1-\r_0), \quad
\r'(t_{m+1})=\frac{2n+1}{s'_{q+2n+2}-s'_{q+1} \, (\r_{q+1}-\r_{q})},
\end{array}
\ee the following conditions have to be fulfilled \be \label{cond1}
\sum_{k=0}^n \r_{(n-1)(n+1)+k+1}\, B_k^n(\alpha) = \r_0\; \;
\mbox{and} \;\; \sum_{k=0}^{n-1} \p_{(n-1)(n+1)+k+1}\,
B_k^{n-1}(\alpha) = \p_0\; \;\mbox{with} \;\;\; \alpha = \frac{t_n -
t_{n-1}}{t_{n+1} - t_{n-1}}\,, \ee \be \label{cond2} \sum_{k=0}^n
\r_{m(n+1)+k+1}\, B_k^n(\beta) = \r_{q+1}\; \; \mbox{and} \;\;
\sum_{k=0}^{n-1} \p_{m(n+1)+k+1}\, B_k^{n-1}(\beta) = \p_{q}\;
\;\mbox{with} \;\;\; \beta = \frac{t_{m+1} - t_{m}}{t_{m+2} -
t_{m}}\,, \ee where $B_k^n(t)= {n \choose i} t^i (1-t)^{n-i}$, $k=0,...,n$ denote
the Bernstein polynomials of degree $n$.
\item[b)]
For $\r(t)$ to be closed and of continuity class $C^n$ at the
junction point $\r(t_n)=\r(t_{m+n+1})$, we require the fulfillment
of the conditions \be \label{condclosed2}
\sum_{j=n(n+1)-k}^{(m+n+1)(n+1)-k-1} (s_{j+2n+1} - s_j)\; \p_j =
{\bf 0} \,, \quad \hbox{for} \quad k=0, \ldots, n, \ee and \be
\label{condclosed3} t_{m+1+k}-t_{m+k}  = t_k - t_{k-1}\,, \quad
\hbox{for} \quad k= n, n+1. \ee
\end{enumerate}
\end{prn}
\begin{proof}
According to \cite{gallier} for every degree-$n$ B-Spline curve $\x(u) =
\sum_{i=0}^{q}\mathbf{x}_{i}N^{n}_{i,\bmu}(u)\in \RR^d$
over the knot partition $\bmu = \{t_i\}_{i=0,...,n+q+1}$ there exists a
unique multi-affine, symmetric application or {\it blossom} $X: \RR^n
\longrightarrow \RR^d, (u_1, \ldots, u_n) \mapsto X(u_1,\ldots, u_n)$
such that $ \x(u) = X(u,\ldots,u) = X(<u>^n)$.
Its control points are $\x_i = X(t_{i+1}, \ldots, t_{n+i})$, $i=0,\ldots,q$.
Thus, denoting $P(u_1, \ldots, u_{2n})$ the blossom of the curve
$\p(u)$ from (\ref{zsquare}) over the knot partition $\bnu$ from
(\ref{nu}), the control points $\p_i$ may be written as:
$$
\begin{array}{ll}
\p_{k-1} = P(<t_0>^{n+1-k}, <t_{1}>^{\min\{n-1+k,n+1\}},
<t_{2}>^{\max\{k-2,0\}}), & \ \hbox{for} \ \ k=1, \ldots, n,\\
\p_{j(n+1)+k-1} = P(<t_j>^{n+1-k}, <t_{j+1}>^{\min\{n-1+k,n+1\}},
<t_{j+2}>^{\max\{k-2,0\}}), & \ \hbox{for} \ \ j= 1, \ldots, p+n-1 \ \hbox{and} \ k=0, \ldots, n,\\
\p_{(p+n)(n+1)+k-1} = P(<t_{p+n}>^{n+1-k},<t_{p+n+1}>^{\min\{n-1+k,n+1\}}), & \ \hbox{for} \ \ k=0,1.
\end{array}
$$
Analogously, denoting $R(u_1, \ldots, u_{2n+1})$ the blossom of the
curve $\r(u)$ from (\ref{rt}) over the knot partition $\brho$ from
(\ref{rho}), the control points $\r_i$ may be written as
$$
\begin{array}{ll}
\r_{j(n+1)+k} = R(<t_j>^{n+1-k}, <t_{j+1}>^{\min\{n+k,n+1\}}, <t_{j+2}>^{\max\{k-1,0\}}), \ & \ \hbox{for} \ \ j=0, \ldots, p+n-1 \ \hbox{and} \ k=0, \ldots, n,\\
\r_{(p+n)(n+1)+k} = R(<t_{p+n}>^{n+1-k},
<t_{p+n+1}>^{\min\{n+k,n+1\}}), \ & \ \hbox{for} \ \ k=0,1.
\end{array}
$$
Recalling de Boor's algorithm and the properties of blossoms, the
control points involved for calculating a point $\r(t_j) =
R(<t_j>^{2n+1})$ are the following: \be \label{cpinv} R(<t_{j-1}>^n,
<t_j>^{n+1}) = \r_{j(n+1)-n}, \ldots, R(<t_{j}>^{n+1},
<t_{j+1}>^{n}) = \r_{j(n+1)} \ee
\begin{enumerate}
\item[a)]
We wish to obtain a clamped curve satisfying conditions
(\ref{condclamped}). In order to satisfy the positional constraints
we thus apply de Boor's algorithm for calculating $\r(t_n)$,
respectively, $\r(t_{m+1})$ to the control points (\ref{cpinv}) for
$j=n$, respectively, $j=m+1$. For $j=n$ we obtain \be
\begin{array}{l}
R(<t_{n-1}>^{n-l}, <t_n>^{n+1+k}, <t_{n+1}>^{l-k}) =  (1-\alpha)\, R(<t_{n-1}>^{n+1-l}, <t_n>^{n+k}, <t_{n+1}>^{l-k}) \\
+ \alpha\, R(<t_{n-1}>^{n-l}, <t_n>^{n+k}, <t_{n+1}>^{l+1-k}) \quad  \hbox{for} \quad k,l=1,\ldots, n.
\end{array}
\ee

This yields the following condition for $\alpha$, which results thus
to be independent of the indices $k,l$:
\[
(1-\alpha)\, t_{n-1} + \alpha \,t_{n+1} = t_n.
\]
De Boor's algorithm thus degenerates to de Casteljau's algorithm
yielding the first equation of condition (\ref{cond1}). The first
equation of condition (\ref{cond2}) is obtained analogously.

In order to satisfy the tangential constraints of
(\ref{condclamped}), we first note that they are equivalent to the
following positional constraints for $\p(t)$: $\p(t_n) = \p_0$,  $\p(t_{m+1})=\p_q$.
We thus apply the same reasoning as above to $\p(t)$, and obtain the
second equations in (\ref{cond1}) and (\ref{cond2}).

\item[b)]
We wish to obtain a closed curve $\r(t)$ with \be
\label{condclosed0} \r(t_n) = \r(t_{m+n+1})\,.\ee Recalling de
Boor's algorithm and the properties of blossoms the control points
involved for calculating a point $\r(t_j) = R(<t_j>^{2n+1})$ are the
following:
$$
R(<t_{j-1}>^n, <t_j>^{n+1}) = \r_{j(n+1)-n}, \ldots,
R(<t_{j}>^{n+1}, <t_{j+1}>^{n}) = \r_{j(n+1)}
$$
In order for condition (\ref{condclosed0}) to hold the following
points and their corresponding knot intervals thus have to coincide:
\be \label{condclosed1}
\r_{n(n+1)-k} = \r_{(m+n+1)(n+1)-k}\,, \quad \hbox{for} \quad k=0,\ldots,n,
\ee
as well as
\be \label{condclosedknots}
t_{m+1+k}-t_{m+k} = t_k -t_{k-1}\,, \quad \hbox{for} \quad k=n,n+1.
\ee
Setting $ f_j
= \frac{s_{j+2n+1} - s_j}{2n+1}$ for $j=0,\ldots,q$ and considering
condition (\ref{rs}), condition (\ref{condclosed1}) is equivalent to
\[
\sum_{j=0}^{n(n+1)-k-1} f_j \p_j = \sum_{j=0}^{(m+n+1)(n+1)-k-1}
f_j \p_j, \quad \hbox{for} \quad k=0, \ldots, n,
\]
or equivalently
\[ \sum_{j=n(n+1)-k}^{(m+n+1)(n+1)-k-1} f_j \p_j ={\bf 0}, \quad \hbox{for} \quad k=0, \ldots, n.
\]
These conditions also guarantee the maximum possible continuity
class at the junction point.
\end{enumerate}
\end{proof}

\smallskip
In the clamped case of the above proposition we notice that if
$t_{n-1}=t_n$ then $\alpha =0$ which yields $ \r(t_n) =
\r_{n(n+1)-n}$ and $\p(t_n) = \p_{n(n+1)-n}$, and if
$t_{m+1}=t_{m+2}$ then $\beta =1$ which yields $ \r(t_{m+1}) =
\r_{(m+1)(n+1)}$ and $\p(t_{m+1}) = \p_{(m+1)(n+1)-1}$. If in the
knot partition $\bmu$ from (\ref{mu}) we have $t_0=\ldots=t_n$ and
$t_{m+1}=\ldots=t_{m+n+1}$, i.e., if $\z(t)$ from (\ref{zbs}) is a
clamped curve itself, the first respectively last $(n-1)(n+1)+2$
control points of $\r(t)$ and $\p(t)$ coincide, i.e., \be
\label{cpmult} \r_0 = \ldots = \r_{n(n+1)-n} \; \mbox{and} \;\;
\r_{(m+1)(n+1)} = \ldots = \r_{(m+n)(n+1)+1}\,, \ee as well as \be
\label{cpmult2} \p_0 = \ldots = \p_{n(n+1)-n} \; \mbox{and} \;\;
\p_{(m+1)(n+1)-1} = \ldots = \p_{(m+n)(n+1)}\,. \ee

In this case condition (\ref{condclamped}) from Proposition
\ref{prop1} a) is automatically satisfied yielding a more intuitive
way of obtaining a clamped PH B-Spline curve. In order to simplify
the notation we remove redundant knots in the knot partition $\bnu$
from (\ref{nu}) together with the control point multiplicities from
(\ref{cpmult}) and summarize the result in the following Corollary.
\begin{crl} \label{clampedcase}
Let $\z(t) = \sum_{i=0}^{m} \z_i N_{i,\bmu}^n(t)\, , \; t \in [t_n,
t_{m+1}]$ be a clamped B-Spline curve over the knot partition \be
\label{muclamped} \bmu =\{<t_n>^{n+1}, \ \{t_{i}\}_{i=n+1,...,m}, \ <t_{m+1}>^{n+1}\} \ee
as in Remark \ref{rembspline_curve_a}. Then,
\begin{equation}
\label{zsquareclamped}\mathbf{p}(t)=\mathbf{z}^{2}(t) = \sum_{k=0}^q
\p_k N_{k,\bnu}^{2n}(t)\, ,
\end{equation}
where $q=2n +(n+1) (m-n)$ and $$\bnu =
\{s_i\}_{i=0,...,4n+1+(n+1)(m-n)} = \{<t_n>^{2n+1}, \
\{<t_{i}>^{n+1}\}_{i=n+1,...,m}, \ <t_{m+1}>^{2n+1}\}\,,$$ as well as
$$\r(t) = \int \p(t) dt = \sum_{i=0}^{q+1} \r_i
N_{i,\brho}^{2n+1}(t)\, , \; \; t \in [t_n, t_{m+1}]$$ where $\brho
= \{s'_i\}_{i=0,...,4n+3+(n+1)(m-n)}$ with $s'_i = s_{i-1}$ for $i=1,
\ldots, 4n+2+(n+1)(m-n)$, $s'_0=s'_1$ and $s'_{4n+3+(n+1)(m-n)} =
s'_{4n+2+(n+1)(m-n)}$, i.e., \be \label{rhoclamped} \brho =
\{<t_n>^{2n+2}, \ \{<t_{i}>^{n+1}\}_{i=n+1,...,m}, \ <t_{m+1}>^{2n+2}\}\,,\ee and the control points $\r_i$ satisfy
(\ref{rs}).
\end{crl}
\begin{proof}
The result is obtained by removing $(n-1)(n+1)+1$ of the multiple
control points from (\ref{cpmult}) together with $n^2$ of the
multiple knots of $\bnu$ from (\ref{nu}) at the beginning and at the
end. The same result is obtained by proceeding with the general
construction of the PH B-Spline curve starting with a clamped
B-Spline curve over the knot partition $\bmu$ from (\ref{muclamped}).
\end{proof}

\smallskip
In this way, in both cases (the clamped and the closed one), the
resulting PH B-Spline curve is of degree $2n+1$ and of continuity
class $C^n$. For example, for $n=1$ we obtain PH B-Spline curves of
degree $3$ and continuity class $C^1$, while for $n=2$ we have PH
B-Spline curves of degree $5$ and continuity class $C^2$.

\section{Parametric speed, arc length and offsets}\label{sec:properties}

According to (\ref{sigmareal}), the parametric speed of the regular
PH curve $\mathbf{r}(t)=x(t) + \ti y(t)$ is given by
\[\sigma(t) = |\r'(t)| = |\z^2(t)| = \z(t)\,
\bar{\z}(t)\, .
\]
Exploiting (\ref{zbs}) we thus obtain \be \label{sigma2} \sigma(t) =
\sum_{i=0}^p \sum_{j=0}^p \z_i \bar{\z}_j N_{i,\bmu}^n(t)
N_{j,\bmu}^n(t) = \sum_{k=0}^q \sigma_k N_{k,\bnu}^{2n}(t)\, ,\ee
where \be \label{sigmak} \sigma_k = \sum_{i=0}^p \sum_{j=0}^p
\chi_{k}^{i,j} \z_i \bar{\z}_j \ee in analogy to (\ref{pks}).

\subsection{Arc-length}
The arc length of the PH B-Spline curve is thus obtained as \be
\label{arclength} \int \sigma(t) dt = \sum_{i=0}^{q+1} l_i
N_{i,\brho}^{2n+1}(t)\, , \; \; t \in [t_n, t_{p+1}]\, ,\ee where
\be \label{ls} l_{i+1} = l_i + \frac{s'_{i+2n+2} - s'_{i+1}}{2n+1}\,
\sigma_i = l_i + \frac{s_{i+2n+1} - s_{i}}{2n+1}\, \sigma_i\,, \ee
with $l_0=0$. The cumulative arc length is given by
$$
\mathlarger{\mathlarger{\ell}}(\xi)=\int_{t_{n}}^{\xi} \sigma(t) dt = \sum_{i=0}^{q+1} l_i
\left(N_{i,\brho}^{2n+1}(\xi)-N_{i,\brho}^{2n+1}(t_{n})\right)
$$
and the curve's total arc length thus is \be
\label{totallength}
L=\mathlarger{\mathlarger{\ell}}(t_{p+1})=\int_{t_{n}}^{t_{p+1}} \sigma(t) dt =
\sum_{i=0}^{q+1} l_i \left(N_{i,\brho}^{2n+1}(t_{p+1}) -
N_{i,\brho}^{2n+1}(t_{n})\right) \, . \ee This general formula simplifies
in the clamped and closed cases as follows. In the clamped case for
the knot partition $\brho$ from (\ref{rhoclamped}) we notice that
\[
N_{i,\brho}^{2n+1}(t_n) = \left\{ \begin{array}{cc} 1\, , &
\mbox{if}\;\; \; i=0 \\
0\, , & \mbox{else}
\end{array} \right.
\]
and
\[
N_{i,\brho}^{2n+1}(t_{m+1}) = \left\{ \begin{array}{cc} 1\, , &
\mbox{if}\;\; \; i=2n+1+(n+1)(m-n) \\
0\, , & \mbox{else.}
\end{array} \right.
\]
Considering $l_0=0$ and recalling Corollary \ref{clampedcase}, in
the clamped case the total arc length $L$ in (\ref{totallength})
thus becomes \be \label{totallengthclamped} L =
l_{2n+1+(n+1)(m-n)}\,. \ee Due to the structure of the knot
partition $\brho$ in (\ref{rho}), we notice that
\[
N_{i,\brho}^{2n+1}(t_n) \left\{ \begin{array}{cc} \ne 0\, , &
\mbox{if}\;\; \; (n-1)(n+1) < i < n(n+1)+1\\
= 0\, , & \mbox{else}
\end{array} \right.
\]
and
\[
N_{i,\brho}^{2n+1}(t_{p+1}) \left\{ \begin{array}{cc} \ne 0\, , &
\mbox{if}\;\; \; p(n+1) < i < (p+1)(n+1)+1 \\
= 0\, , & \mbox{else.}
\end{array} \right.
\]
In this case the total arc length $L$ from (\ref{totallength})  thus
becomes
\ba \label{totallengthgen} L &=&
\sum_{i=p(n+1)+1}^{(p+1)(n+1)} l_i N_{i,\brho}^{2n+1}(t_{p+1}) -
\sum_{i=(n-1)(n+1)+1}^{n(n+1)} l_i N_{i,\brho}^{2n+1}(t_{n})
\nonumber \\
&=& \sum_{k=0}^{n} \Big(l_{p(n+1)+1+k}\,
N_{p(n+1)+1+k,\brho}^{2n+1}(t_{p+1}) - l_{(n-1)(n+1)+1+k}\,
N_{(n-1)(n+1)+1+k,\brho}^{2n+1}(t_{n}) \Big)\,. \ea In the case of a
closed curve from Proposition \ref{prop1} b) we have $p=m+n$ and
conditions (\ref{condclosed3}). On the knot partition $\brho$ from
(\ref{rho}) the normalized B--Spline basis functions having as
support $[t_{n-1}, t_{n+1}]$ are $N_{i,\brho}^{2n+1}(t)$ for
$i=(n-1)(n+1)+1, \ldots, n(n+1)$, and those having as support
$[t_{m+n}, t_{m+n+2}]$ are $N_{i,\brho}^{2n+1}(t)$ for
$i=(m+n)(n+1)+1, \ldots, (m+n+1)(n+1)$. With conditions
(\ref{condclosed3}) this means %\be \label{basisfctsper}
\[
N_{(n-1)(n+1)+1+k,\brho}^{2n+1}(t_n) =
N_{(m+n)(n+1)+1+k,\brho}^{2n+1}(t_{m+n+1})\,, \;\; \mbox{for}\;\; \;
k=0,\ldots,n\,.
\]
%\ee
In the case of a closed curve its total arc
length $L$ from (\ref{totallengthgen}) thus reads \be
\label{Lclosed1} L = \sum_{k=0}^{n} \Big(l_{p(n+1)+1+k} -
l_{(n-1)(n+1)+1+k}\Big) \, N_{(n-1)(n+1)+1+k,\brho}^{2n+1}(t_{n})\,,
\ee which, by taking into account (\ref{ls}), becomes \be
\label{totallengthclosed} L = \sum_{k=0}^{n}
\left(\sum_{j=(n-1)(n+1)+k+1}^{(m+n)(n+1)+k}
\frac{s_{j+2n+1}-s_j}{2n+1} \, \sigma_j\right) \,
N_{(n-1)(n+1)+1+k,\brho}^{2n+1}(t_{n})\,. \ee

\subsection{Offsets}
The offset curve $\r_h(t)$ at (signed) distance $h$ of a PH B-Spline curve
$\r(t)$ is the locus defined by
\[\r_h(t) = \r(t) + h\, \n(t)
\]
where
\[\n(t) = \frac{(y'(t),-x'(t))}{\sqrt{(x'(t))^2+(y'(t))^2}} =
\frac{-\ti \r'(t)}{\sigma(t)}
= \frac{-\ti\z^2(t)}{\sigma(t)}\,.
\]
(Note that, since we are dealing with the regular case, $\sigma(t)\neq0$ and the offset curve is always well defined.)
Thus
\[\r_h(t) = \frac{\sigma(t) \r(t) - \ti \, h \,  \z^2(t)}{\sigma(t)}\,.
\]
Herein the product $\sigma(t) \r(t)$ reads as
\[\sigma(t) \r(t) = \sum_{i=0}^{q+1} \sum_{j=0}^q \sigma_j \r_i
N_{i,\brho}^{2n+1}(t) N_{j,\bnu}^{2n}(t)\,.
\]
Again, according to \cite{morken, cinesi}, we can write
\begin{equation}\label{function_gij}
N_{i,\brho}^{2n+1}(t) N_{j,\bnu}^{2n}(t) = \sum_{k=0}^w
\zeta_{k}^{i,j} N_{k,\btau}^{4n+1}(t)
\end{equation}
with the knot partition \be\label{tau} \btau = \{<t_{-1}>^{2n+1}, \ \{
<t_k>^{3n+2}\}_{k=0,...,p+n+1}, \ <t_{p+n+2}>^{2n+1}\} \ee and \be
\label{w} w=(3n+2)(p+n+2)-1\,. \ee

\begin{rmk}
In the clamped case, from Corollary
\ref{clampedcase} we obtain \be\label{tauclamped} \btau =
\{<t_n>^{4n+2}, \ \{<t_{k}>^{3n+2}\}_{k=n+1,...,m}, \ <t_{m+1}>^{4n+2}\}
\ee and \be \label{wclamped} w=4n+1+(m-n)(3n+2)\,. \ee
Differently, in the closed case, we have
\be\label{tauclosed}
\btau =\{
<t_{-1}>^{2n+1}, \ \{<t_{k}>^{3n+2}\}_{k=0,...,m+2n+1}, \ <t_{m+2n+2}>^{2n+1} \}
\ee
and
\be \label{wclosed} w=(3n+2)(m+2n+2)-1 \,. \ee
\end{rmk}

\smallskip
To work out the unknown coefficients $\bzeta^{i,j}:=(\zeta_0^{i,j}, \zeta_1^{i,j}, ..., \zeta_w^{i,j})^T$ in \eqref{function_gij} we solve the linear system
\begin{equation}\label{zetas}
\C \bzeta^{i,j}=\e^{i,j},
\end{equation}
with
$$
\C=(c_{k,h})_{k,h=0,...,w}, \quad \ c_{k,h}:=\langle
N_{k,\btau}^{4n+1}, N_{h,\btau}^{4n+1}
\rangle=\int_{t_0}^{t_{p+n+1}} \hspace{-0.2cm} N_{k,\btau}^{4n+1}(t)
\, N_{h,\btau}^{4n+1}(t) \, dt
$$
and
$$
\e^{i,j}=(e^{i,j}_h)_{h=0,...,w}, \quad \  e^{i,j}_h:=\langle
g_{i,j}, N_{h,\btau}^{4n+1}
\rangle=\int_{t_0}^{t_{p+n+1}}  g_{i,j}(t) \, N_{h,\btau}^{4n+1}(t) \, dt \quad \hbox{where} \quad
g_{i,j}(t):=N_{i,\brho}^{2n+1}(t) \, N_{j,\bnu}^{2n}(t).
$$
Like in the previous case, $\C$ is a banded Gramian, and thus nonsingular.
This guarantees that each of the linear systems in \eqref{zetas}  has a unique solution that can be efficiently computed by means of the
Cholesky decomposition algorithm for symmetric positive definite band matrices.\\
The computed expressions of $\zeta_k^{i,j}$, $k=0,...,w$, $i=0,...,q+1$, $j=0,...,q$ thus yield
\[\sigma(t) \r(t) =
\sum_{k=0}^{w}\sum_{i=0}^{q+1} \sum_{j=0}^q \zeta_k^{i,j} \sigma_j
\r_i N_{k,\btau}^{4n+1}(t)\,.
\]
By writing
\[\p(t) =
\p(t)\cdot 1 = \Big(\sum_{j=0}^q \p_j N_{j,\bnu}^{2n}(t) \Big)
\Big(\sum_{i=0}^{q+1} N_{i,\brho}^{2n+1}(t)\Big)
\]
and
\[\sigma(t) = \sigma(t)\cdot 1 = \Big(\sum_{j=0}^q \sigma_j
N_{j,\bnu}^{2n}(t)\Big) \Big(\sum_{i=0}^{q+1} N_{i,\brho}^{2n+1}(t)\Big), \]
we thus obtain
\[\p(t)  = \sum_{k=0}^{w}
\sum_{i=0}^{q+1} \sum_{j=0}^q \zeta_k^{i,j} \p_j
N_{k,\btau}^{4n+1}(t)  \]
\[\sigma(t)
= \sum_{k=0}^{w} \sum_{i=0}^{q+1} \sum_{j=0}^q \zeta_k^{i,j}
\sigma_j N_{k,\btau}^{4n+1}(t)\,.
\]
The offset curve $\r_h(t)$ finally has the form \be \label{offset3}
\r_h(t) = \frac{\sum_{k=0}^{w} \q_k
N_{k,\btau}^{4n+1}(t)}{\sum_{k=0}^{w} \gamma_k
N_{k,\btau}^{4n+1}(t)}\,, \ee
where, for $k=0,\ldots,w$,
\be
\label{qk} \q_k = \sum_{i=0}^{q+1} \sum_{j=0}^q (\sigma_j \r_i -
\ti\, h \,\p_j)\zeta_k^{i,j} \ee and \be \label{gammak} \gamma_k =
\sum_{i=0}^{q+1} \sum_{j=0}^q \sigma_j \zeta_k^{i,j}\,. \ee

\section{General explicit formulas for the cubic and quintic case}
\label{sec5}

\subsection{Clamped cubic PH B-Splines ($n=1$)}
\label{sec5n1clamped}

Let $m \in \NN$, $m \geq 1$. For a general knot vector
$\bmu =\{ \langle 0\rangle^2 < t_2 < \ldots < t_m  < \langle t_{m+1} \rangle^2 \}$
satisfying the constraints $t_0 =t_1 =0$ and $t_{m+1} =t_{m+2}$ (see Figure \ref{fig_knots_clamped_n1} first row),
by applying the above method we construct the knot partitions
$$
\begin{array}{l}
\bnu= \{\langle 0 \rangle^3 < \langle t_2 \rangle^2 < ... < \langle t_m \rangle^2 < \langle t_{m+1} \rangle^3 \}, \smallskip \\
\brho = \{\langle 0 \rangle^4 < \langle t_2 \rangle^2 < ... < \langle t_m \rangle^2 < \langle t_{m+1} \rangle^4 \}, \smallskip \\
\btau = \{\langle 0 \rangle^6 < \langle t_2 \rangle^5 < ... <
\langle t_m \rangle^5 < \langle t_{m+1} \rangle^6 \},
\end{array}
$$
illustrated in Figure \ref{fig_knots_clamped_n1}. Then, by solving the linear systems \eqref{lgschi} we calculate
the coefficients $\chi_{k}^{i,j}$, $0\le i,j \le m$, $0 \le {k} \le
2 m$. All of them turn out to be zero with the exception of \be
\label{chisn1clamped}
\begin{array}{l}
\chi_{2k}^{k,k} = 1\,, \quad k = 0, \ldots, m\,, \smallskip \\
\chi_{2k+1}^{k,k+1} = \chi_{2k+1}^{k+1,k} =\frac{1}{2}\,, \quad k = 0, \ldots, m-1.
\end{array}
\ee
In addition, we compute the coefficients $\zeta_k^{i,j}$, $0
\leq i \leq 2m+1$, $0 \leq j \leq 2m$, $0 \leq k \leq 5m$ as the
solutions to the linear systems \eqref{zetas}. All of them
turn out to be zero with the exception of \be \label{zetasn1clamped}
\begin{array}{l}
\zeta_{5k}^{2k,2k} = \frac{d_{k+1}}{D_{k}}\, , \;
\zeta_{5k}^{2k+1,2k} = \frac{d_{k}}{D_{k}}\, , \;
k=0,\ldots,m\, , \smallskip  \\
\zeta_{5k+1}^{2k,2k+1} = \frac{2d_{k+1}}{5 D_{k}}\, , \;
\zeta_{5k+1}^{2k+1,2k} = \frac{3}{5}\, , \; \zeta_{5k+1}^{2k+1,2k+1}
= \frac{2d_{k}}{5 D_{k}}\, , \;
k=0,\ldots,m-1\, , \smallskip  \\
\zeta_{5k+2}^{2k,2k+2} = \frac{d_{k+1}}{10 D_{k}}\, , \;
\zeta_{5k+2}^{2k+1,2k+1} = \frac{3}{5}\, , \;
\zeta_{5k+2}^{2k+1,2k+2} = \frac{d_{k}}{10 D_{k}}\, , \;
\zeta_{5k+2}^{2k+2,2k} = \frac{3}{10}\, , \;
k=0,\ldots,m-1\, , \smallskip  \\
\zeta_{5k+3}^{2k+1,2k+2} = \frac{3}{10}\, , \;
\zeta_{5k+3}^{2k+2,2k} = \frac{d_{k+2}}{10 D_{k+1}}\, , \;
\zeta_{5k+3}^{2k+2,2k+1} = \frac{3}{5}\, , \; \zeta_{5k+3}^{2k+3,2k}
= \frac{d_{k+1}}{10 D_{k+1}}\, , \;
\;k=0,\ldots,m-1\, , \smallskip  \\
\zeta_{5k+4}^{2k+2,2k+1} = \frac{2d_{k+2}}{5 D_{k+1}}\, , \;
\zeta_{5k+4}^{2k+2,2k+2} = \frac{3}{5}\, , \;
\zeta_{5k+4}^{2k+3,2k+1} = \frac{2d_{k+1}}{5 D_{k+1}}\, , \;
k=0,\ldots,m-1\, ,
\end{array}
\ee where $D_k:=d_k+d_{k+1}$, $k=0,...,m$ and $d_0=d_{m+1}:=0$.

\begin{figure}[h!]
\centering
%\vspace{-6cm}
\resizebox{8.0cm}{!}{\includegraphics{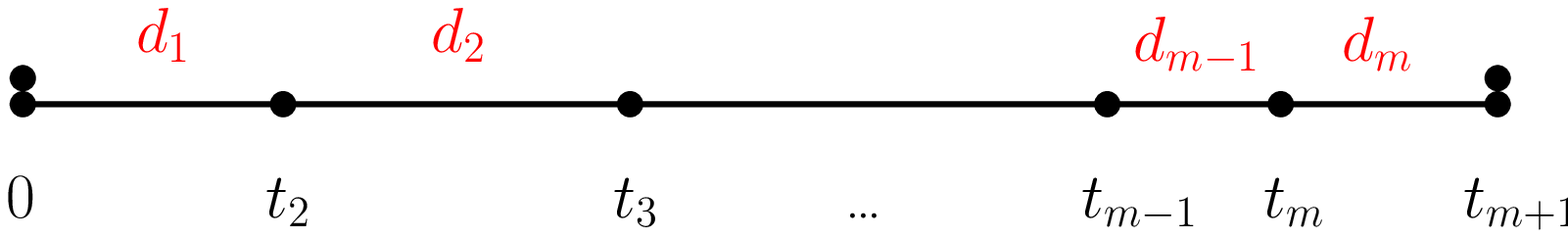}}\\
\resizebox{8.0cm}{!}{\includegraphics{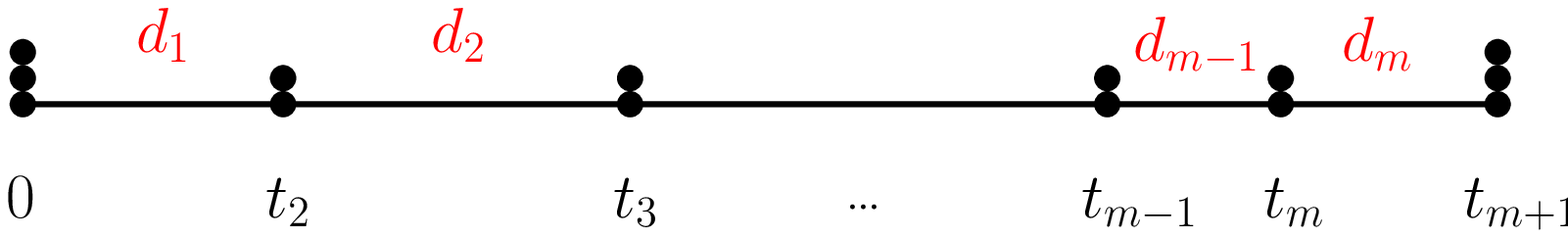}}\\
\resizebox{8.0cm}{!}{\includegraphics{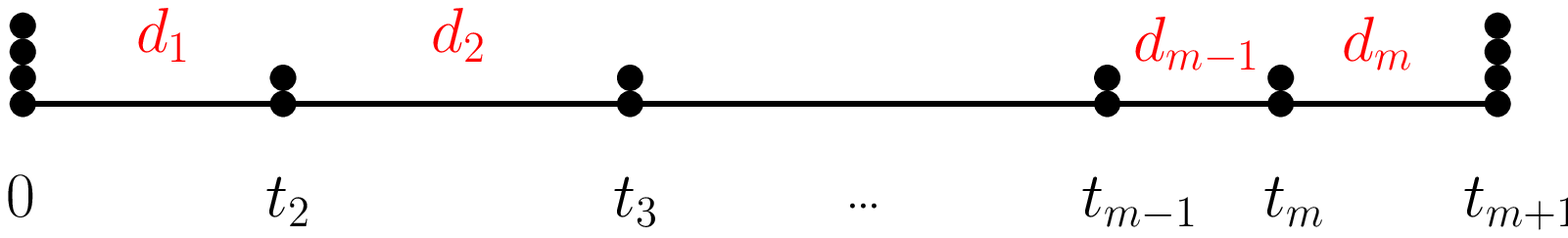}}\\
\resizebox{8.0cm}{!}{\includegraphics{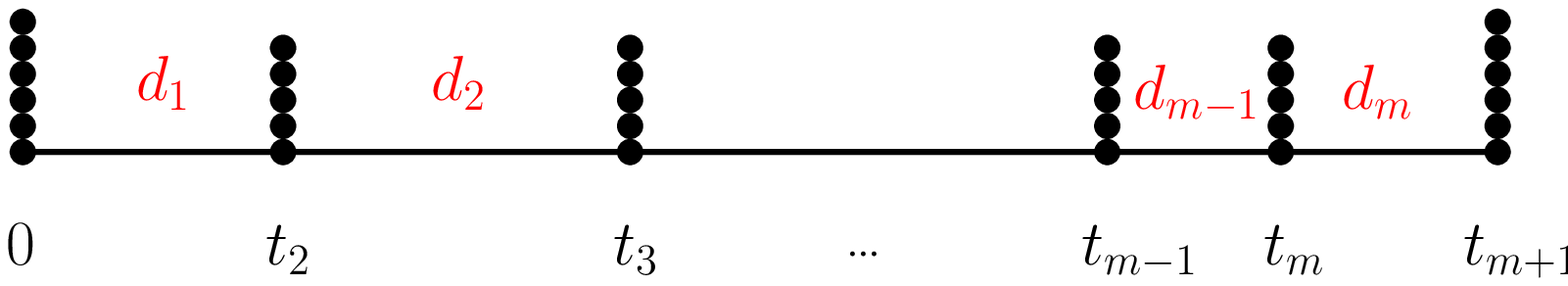}}
\caption{Knot partitions for the clamped case $n=1$. From top to
bottom: $\bmu$, $\bnu$, $\brho$, $\btau$.}
\label{fig_knots_clamped_n1}
\end{figure}

\smallskip
By means of the computed coefficients
$\{\chi_{k}^{i,j}\}^{0\le i,j \le m}_{0 \le {k} \le 2 m}$ we can thus
shortly write the control points of $\r'(t)$ as
$$
\begin{array}{l}
\p_{2k}=\z_k^2, \quad k=0,...,m\,,  \\
\p_{2k+1}=\z_k \z_{k+1}, \quad k=0,...,m-1\,,
\end{array}
$$
and the coefficients of the
parametric speed $\sigma(t)$ as
$$
\begin{array}{l}
\sigma_{2k}=\z_k \bar{\z}_k, \quad k=0,...,m\,,  \\
\sigma_{2k+1}=\frac12 \big( \z_k \bar{\z}_{k+1} + \z_{k+1} \bar{\z}_{k} \big), \quad k=0,...,m-1.
\end{array}
$$
Thus, according to (\ref{rt}), the clamped cubic PH B-Spline curve defined
over the knot partition $\brho$ is given by
$$
\r(t) = \sum_{i=0}^{2m+1} \r_i N_{i,\brho}^{3}(t)\,, \quad t \in [t_1, t_{m+1}] \qquad (t_1=0),
$$
with control points
$$
\begin{array}{l}
\r_{1}=\r_{0} + \frac{d_{1}}{3} \, \z_0^2, \\
\r_{2i+2}=\r_{2i+1} + \frac{d_{i+1}}{3} \, \z_i \z_{i+1}, \quad i=0,...,m-1\,, \\
\r_{2i+3}=\r_{2i+2} + \frac{d_{i+1}+d_{i+2}}{3} \, \z_{i+1}^2, \quad i=0,...,m-2\,,  \\
\r_{2m+1}=\r_{2m} + \frac{d_{m}}{3} \, \z_m^2,
\end{array}
$$
and arbitrary $\r_0$.

\begin{rmk}
Note that, when $m=1$ and $t_2=t_3=1$, the expressions of the
control points coincide with those of Farouki's PH B\'ezier cubic
from \cite{faroukisakkalis, farouki94}.
\end{rmk}

\smallskip
According to (\ref{totallengthclamped}) the total arc length of the
clamped PH B-Spline curve of degree $3$ is given by $L=l_{2m+1}$,
where
$$
\begin{array}{l}
l_0=0, \\
l_{1}= l_0 + \frac{d_{1}}{3} \, \z_0 \bar{\z}_0,  \\
l_{2i+2}=l_{2i+1} + \frac{d_{i+1}}{6} \, \big( \z_i \bar{\z}_{i+1} + \z_{i+1} \bar{\z}_{i} \big), \quad i=0,...,m-1\,, \\
l_{2i+3}=l_{2i+2} + \frac{d_{i+1}+d_{i+2}}{3} \,  \z_{i+1} \bar{\z}_{i+1}, \quad  i=0,...,m-2\,,  \\
l_{2m+1}=l_{2m} + \frac{d_{m}}{3} \, \z_m \bar{\z}_m.
\end{array}
$$
Over the knot partition $\btau$, the offset curve $\r_h(t)$ has the rational B-Spline form \be
\label{offsetn1clamped} \r_h(t) = \frac{\sum_{k=0}^{5m} \q_k
N_{k,\btau}^{5}(t)}{\sum_{k=0}^{5m} \gamma_k N_{k,\btau}^{5}(t)}\,,
\quad t \in [t_1, t_{m+1}], \ee where, by exploiting the explicit
expressions of the coefficients $\{\zeta_k^{i,j}\}^{0 \leq i \leq
2m+1, \, 0 \leq j \leq 2m}_{0 \leq k \leq 5m}$ from
(\ref{zetasn1clamped}), weights and control points can easily be
obtained; they are reported in the Appendix.

Some examples of clamped cubic PH B-Spline curves are shown in Figure
\ref{fig:clamped_n1}, and  their offsets are displayed in Figure
\ref{fig:offsets_clamped_n1}.

\begin{figure}[h!]
\centering
\includegraphics[height=0.25\textwidth,valign=t]{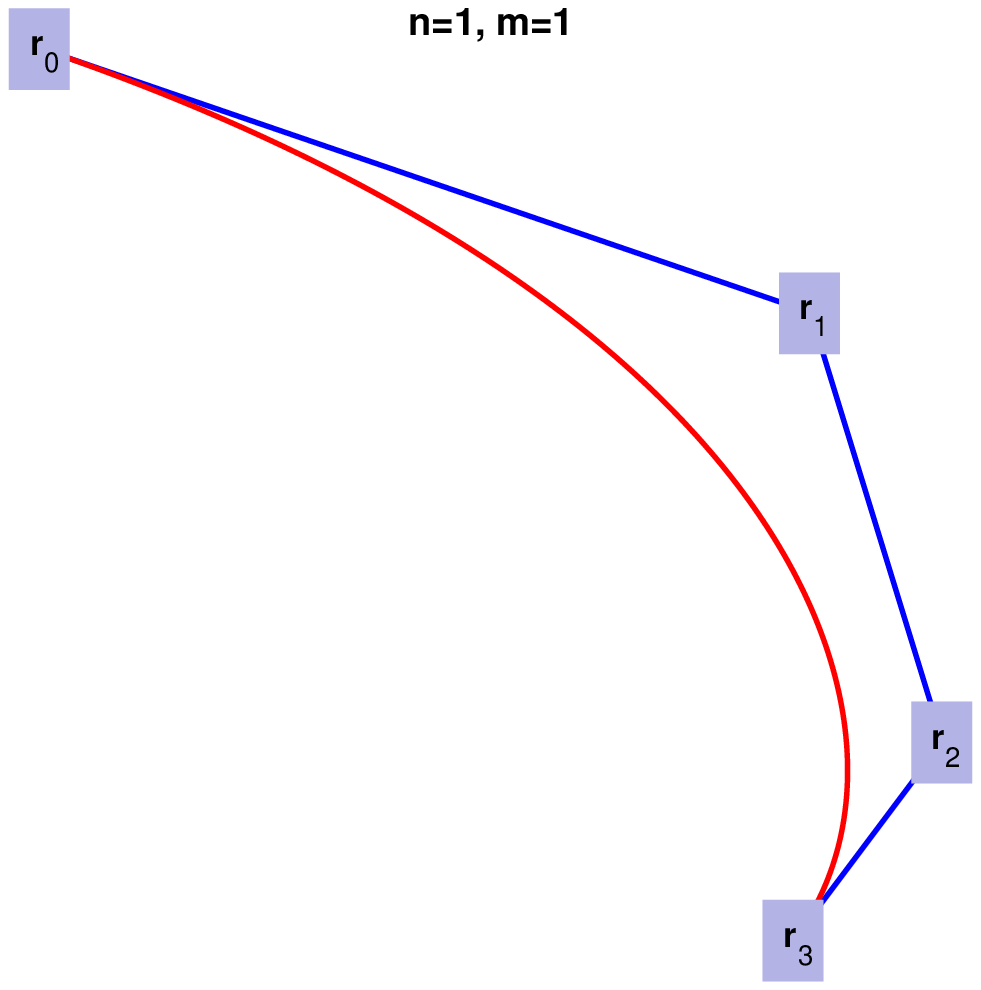}\hspace{-0.3cm}
\includegraphics[height=0.25\textwidth,valign=t]{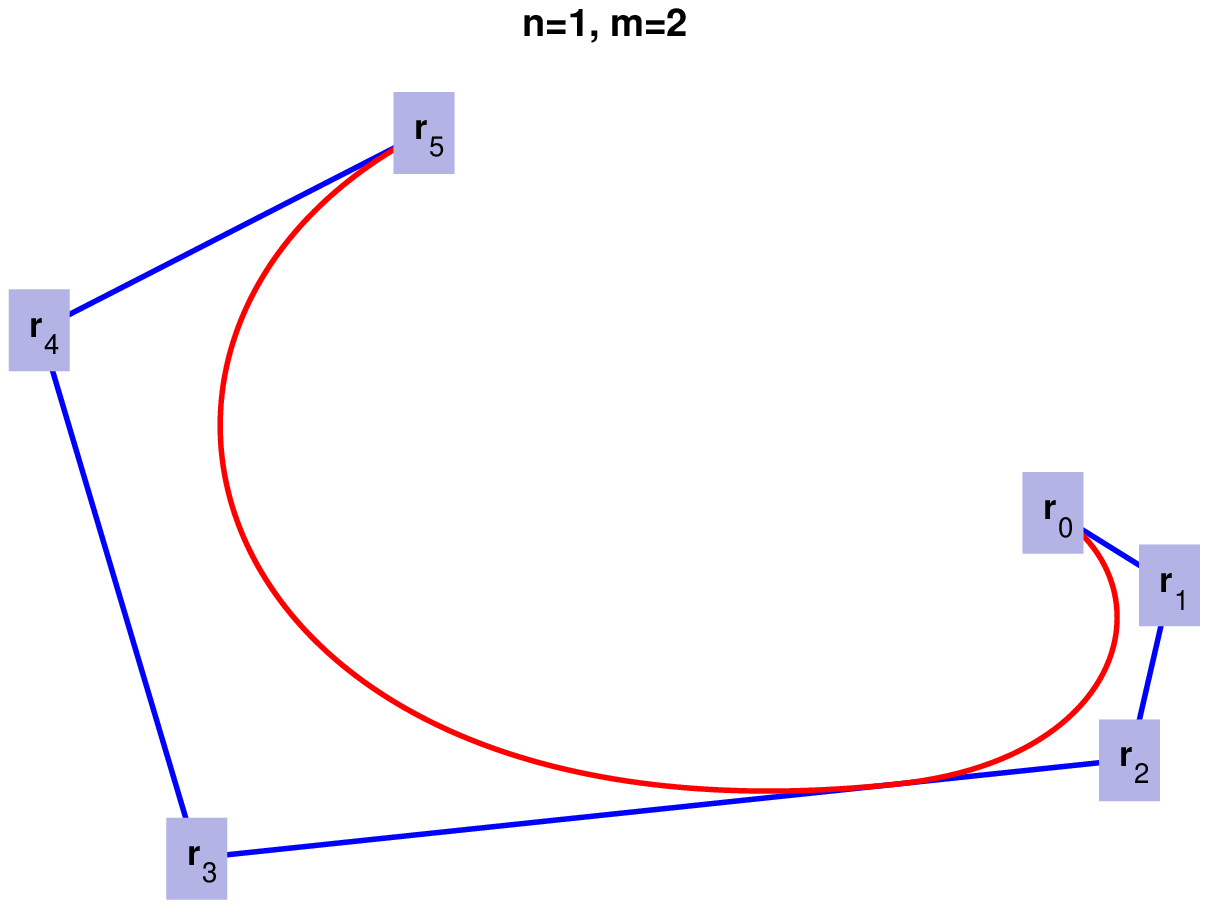}\hspace{-0.3cm}
\includegraphics[height=0.25\textwidth,valign=t]{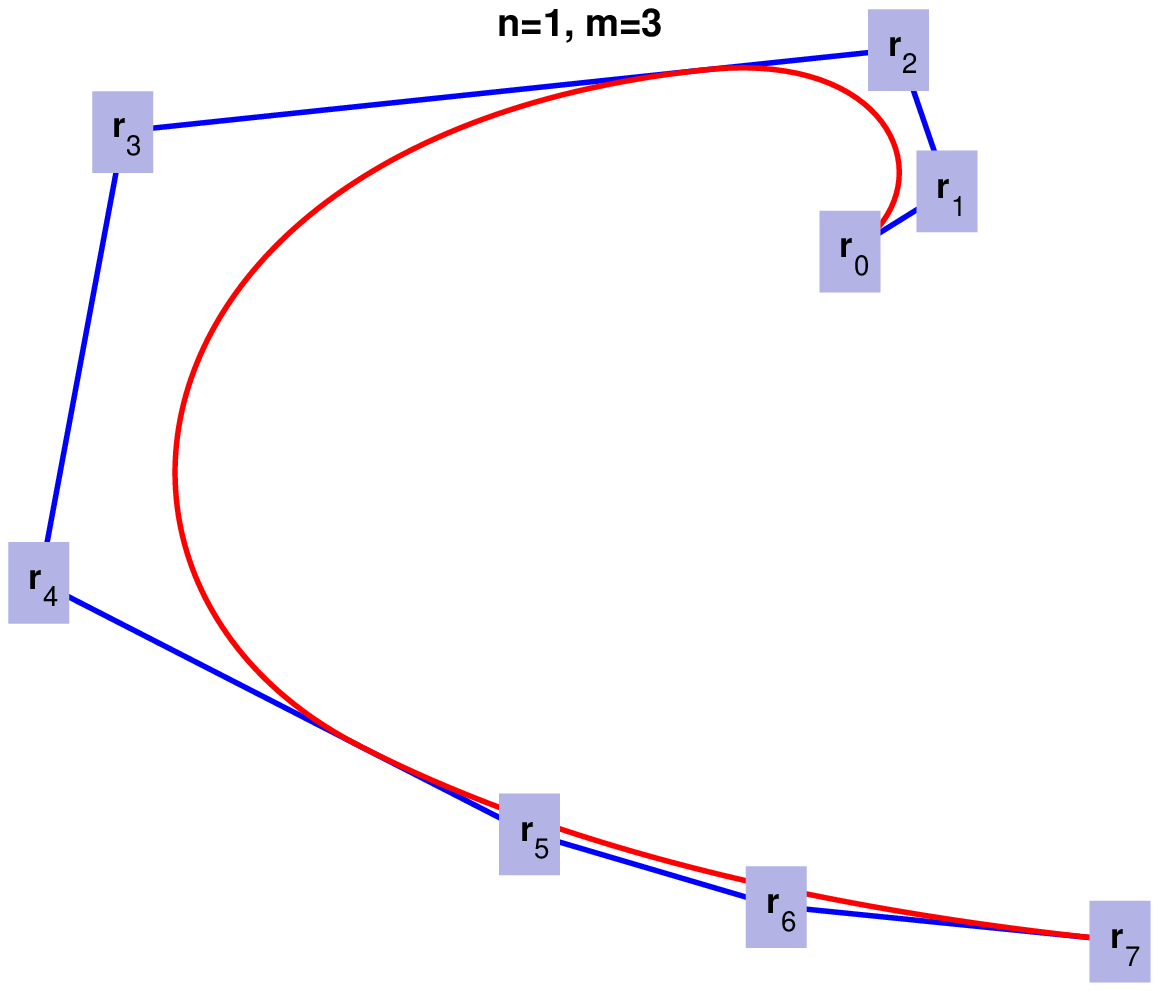}
\caption{Clamped cubic PH B-Spline curves with: $m=1$ (left), $m=2$ (center) and $m=3$ (right).}
\label{fig:clamped_n1}
\end{figure}

\begin{figure}[h!]
\centering
\includegraphics[height=0.25\textwidth,valign=t]{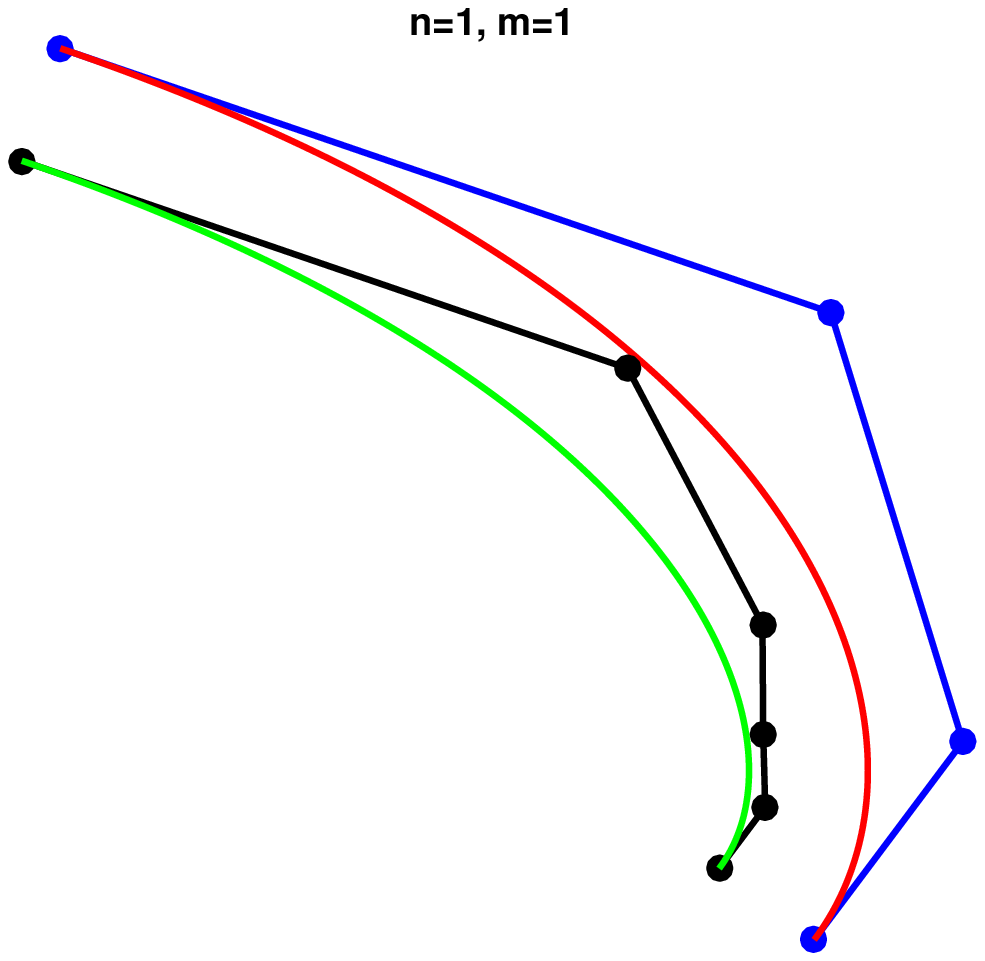}\hspace{-0.3cm}
\includegraphics[height=0.25\textwidth,valign=t]{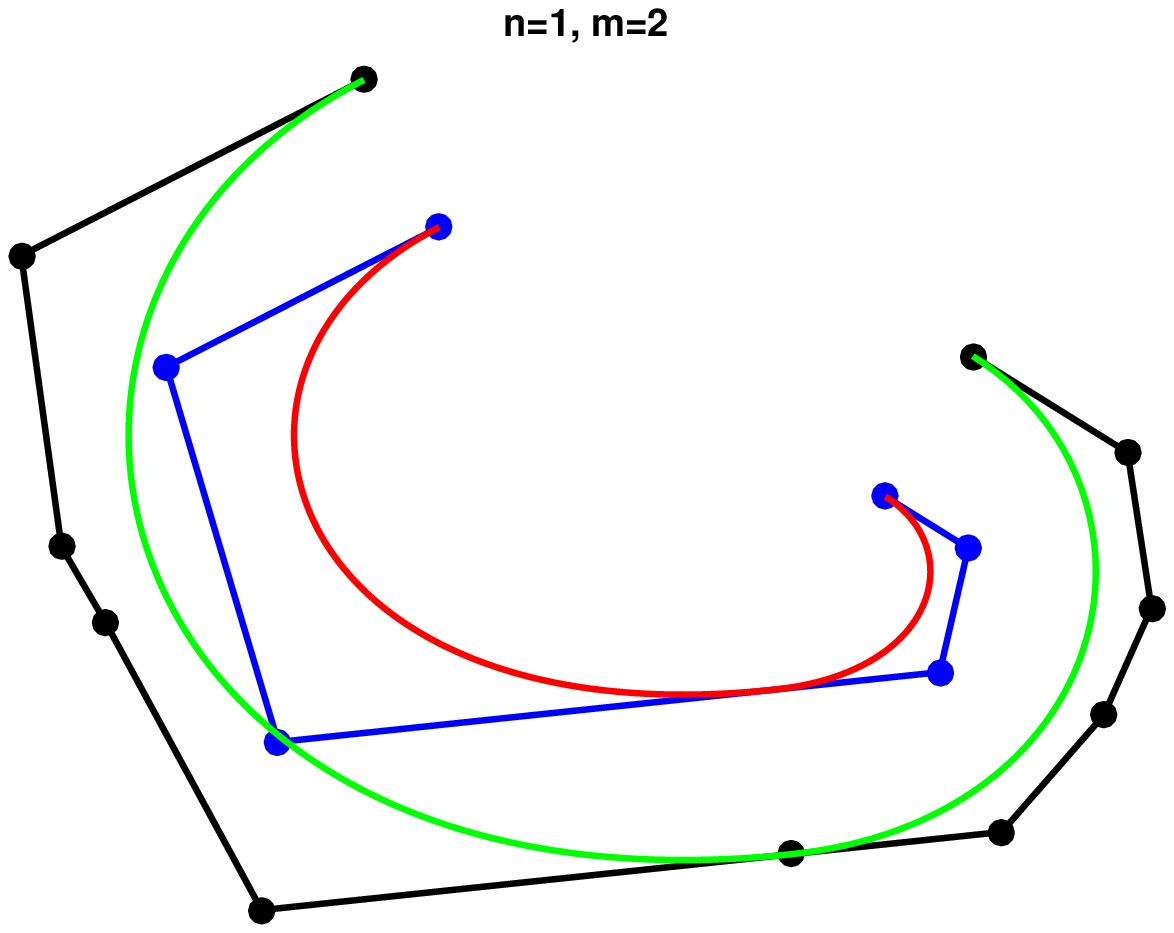}\hspace{-0.3cm}
\includegraphics[height=0.25\textwidth,valign=t]{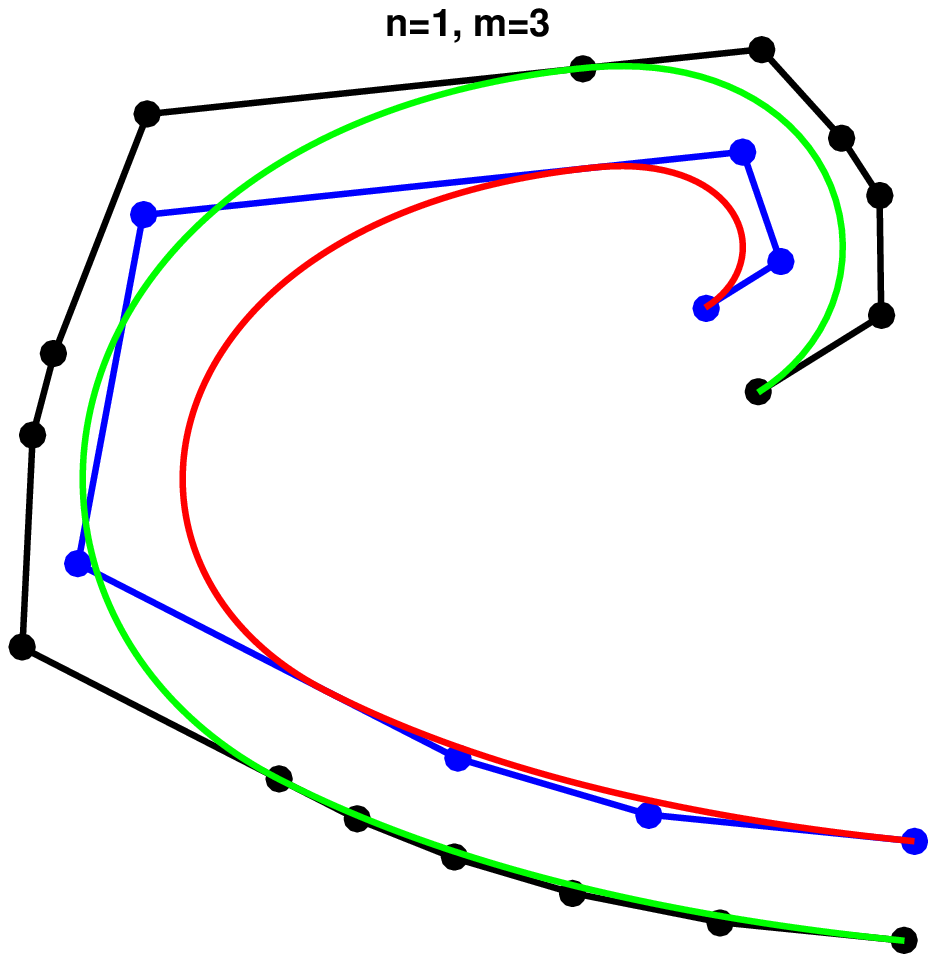}\\
\includegraphics[height=0.25\textwidth,valign=t]{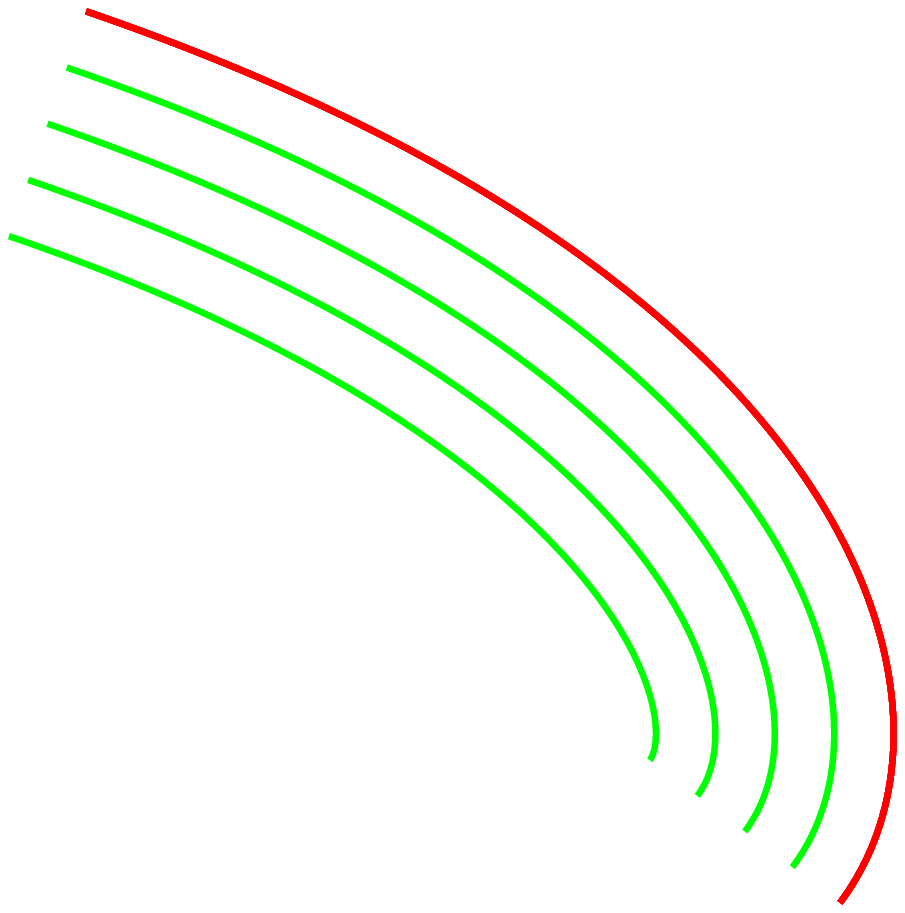}\hspace{-0.3cm}
\includegraphics[height=0.25\textwidth,valign=t]{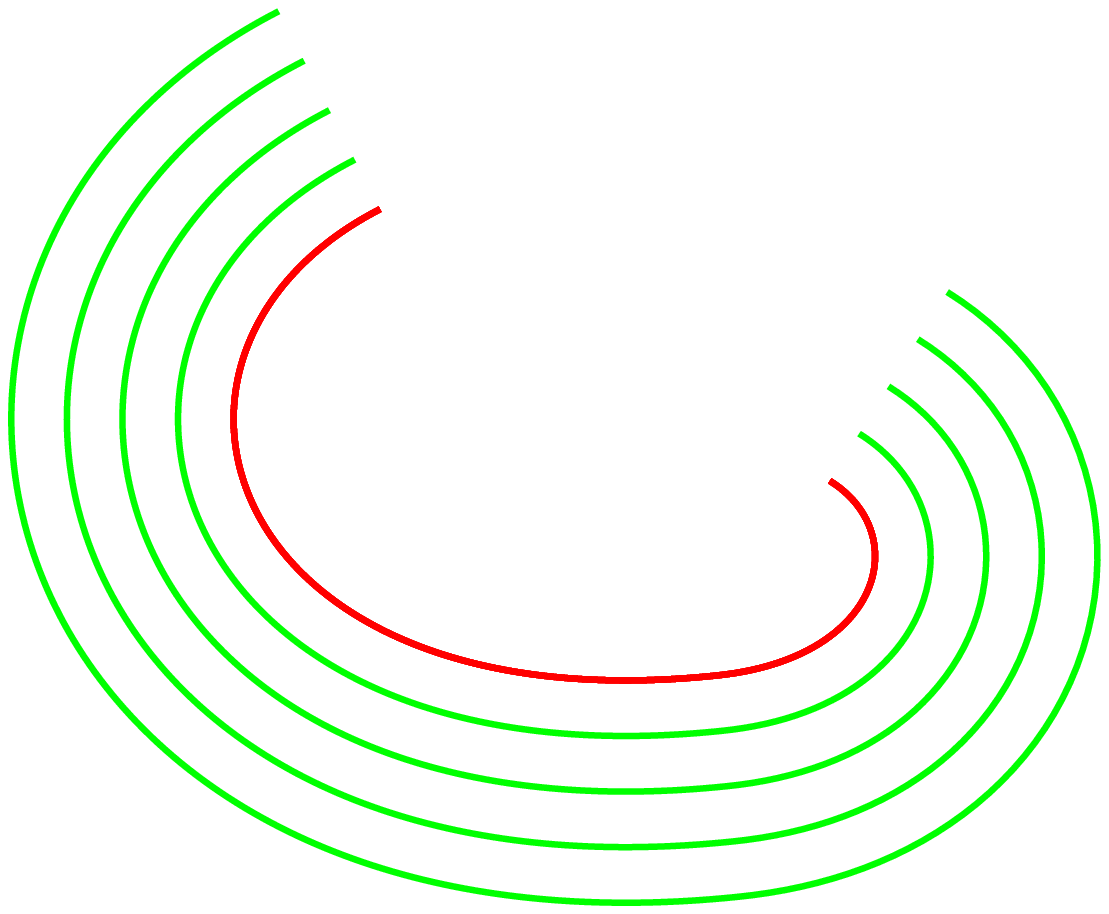}\hspace{-0.3cm}
\includegraphics[height=0.25\textwidth,valign=t]{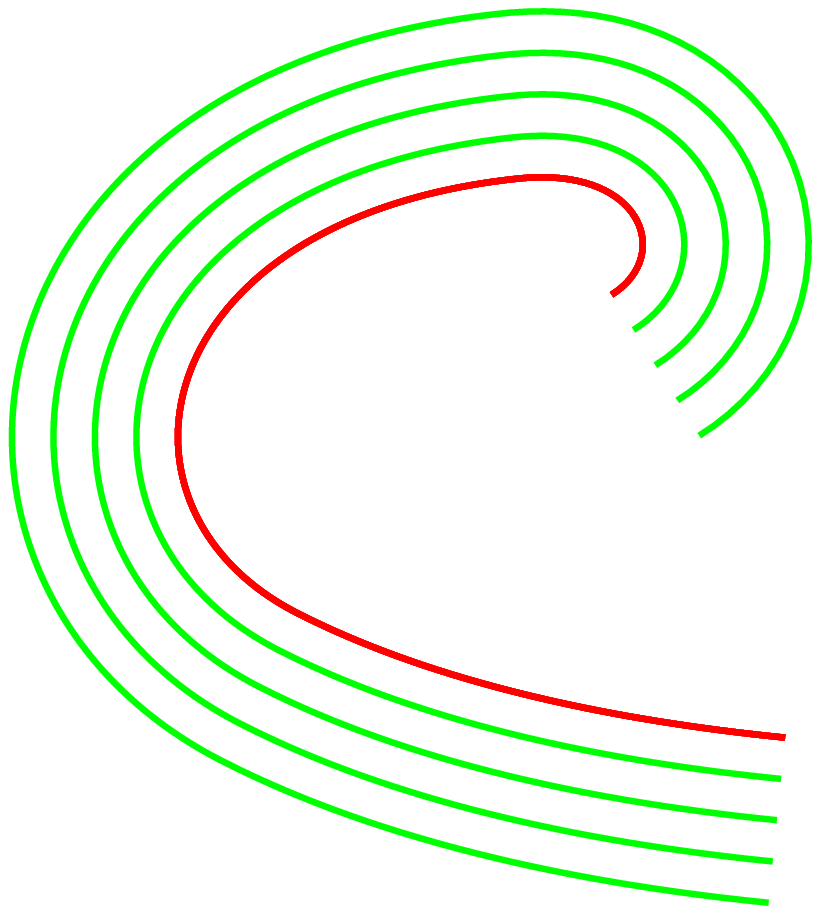}
\caption{Offsets of clamped cubic PH B-Spline curves from Figure
\ref{fig:clamped_n1} with (1. row) and without (2. row) control
polygon where: $m=1$ (left column), $m=2$ (center column) and $m=3$
(right column).} \label{fig:offsets_clamped_n1}
\end{figure}

\subsection{Clamped quintic PH B-Splines ($n=2$)}
\label{sec5n2clamped}

Let $m \in \NN$, $m \geq 2$. For a general knot vector
$\bmu =\{\langle 0\rangle^3  < t_3 < \ldots  < t_m < \langle t_{m+1}\rangle^3\}$
satisfying the constraints $t_0 = t_1 = t_2 = 0$ and $t_{m+1} =
t_{m+2} = t_{m+3}$ (see Figure \ref{fig_knots_clamped_n2} first
row), by applying the above method we construct the knot partitions
\be \label{partitionsn2clamped}
\begin{array}{l}
\bnu=\{\langle 0 \rangle^5 < \langle t_3 \rangle^3 < ... < \langle t_m \rangle^3 < \langle t_{m+1} \rangle^5 \}, \smallskip \\
\brho =\{\langle 0 \rangle^6 < \langle t_3 \rangle^3 < ... < \langle t_m \rangle^3 < \langle t_{m+1} \rangle^6 \}, \smallskip \\
\btau =\{\langle 0 \rangle^{10} < \langle t_3 \rangle^8 < ... <
\langle t_m \rangle^8 < \langle t_{m+1} \rangle^{10} \},
\end{array}
\ee
illustrated in Figure \ref{fig_knots_clamped_n2}. Then, by solving the linear systems \eqref{lgschi} we calculate
the coefficients $\chi_{k}^{i,j}$, $0\le i,j \le m$, $0 \le {k} \le
3 m - 2$. All of them turn out to be zero with the exception of
$$
\begin{array}{l}
\chi_0^{0,0}=1, \smallskip\\
\chi_1^{0,1}= \chi_1^{1,0} = \frac{1}{2}, \smallskip\\
\chi_{3k-1}^{k-1,k}= \chi_{3k-1}^{k,k-1} = \frac{1}{6} \, \frac{d_k \, d_{k+1}}{(d_{k-1}+d_k) \, (d_k + d_{k+1})}, \quad k=1,...,m-1\,,  \smallskip\\
\chi_{3k-1}^{k-1,k+1}= \chi_{3k-1}^{k+1,k-1} = \frac{1}{6} \, \frac{(d_k)^2}{(d_{k-1}+d_k)(d_k+d_{k+1})}, \quad k=1,...,m-1\,,  \smallskip\\
\chi_{3k-1}^{k,k}=\frac{2}{3} + \frac{1}{3} \, \frac{d_{k-1} \, d_{k+1}}{(d_{k-1}+d_k) \, (d_k + d_{k+1})}, \quad k=1,...,m-1\,,  \smallskip\\
\chi_{3k-1}^{k,k+1}= \chi_{3k-1}^{k+1,k} = \frac{1}{6} \, \frac{d_{k-1} \, d_k}{(d_{k-1}+d_k)\, (d_k + d_{k+1})}, \quad k=1,...,m-1\,,  \smallskip\\
\chi_{3k}^{k,k}= \frac{d_{k+1}}{d_k + d_{k+1}},   \quad k=1,...,m-2\,,  \smallskip\\
\chi_{3k}^{k,k+1}= \chi_{3k}^{k+1,k} = \frac{1}{2}\,\frac{d_k}{d_k + d_{k+1}},   \quad k=1,...,m-2\,,  \smallskip\\
\chi_{3k+1}^{k,k+1}=\chi_{3k+1}^{k+1,k} = \frac{1}{2}\, \frac{d_{k+1}}{d_k+d_{k+1}},   \quad k=1,...,m-2\,,  \smallskip\\
\chi_{3k+1}^{k+1,k+1}= \frac{d_k}{d_k +d_{k+1}},   \quad k=1,...,m-2\,,  \smallskip\\
\chi_{3m-3}^{m-1,m}=\chi_{3m-3}^{m,m-1}= \frac{1}{2}, \smallskip\\
\chi_{3m-2}^{m,m}= 1,
\end{array}
$$
where $d_{k}=t_{k+2}-t_{k+1}$, $k=1, ..., m-1$ and $d_0=d_{m}=0$.\\
In addition, we compute the coefficients $\zeta_k^{i,j}$, $0 \leq i
\leq 3m-1$, $0 \leq j \leq 3m-2$, $0 \leq k \leq 8m-7$ as the
solutions to the linear systems \eqref{zetas}. All of them
turn out to be zero with the exception of \be \label{zetasn2clamped}
\begin{array}{l}
\zeta_{8k}^{3k,3k} = \frac{d_{k+1}^2}{D_{k}^2}\, , \;
\zeta_{8k}^{3k+1,3k} = \frac{13 d_{k}d_{k+1}}{9 D_{k}^2}\, , \;
\zeta_{8k}^{3k+2,3k} = \frac{4 d_{k}^2}{9 D_{k}^2}\, , \;
\zeta_{8k}^{3k,3k+1} = \frac{5 d_{k}d_{k+1}}{9 D_{k}^2}\, , \;
\zeta_{8k}^{3k+1,3k+1} = \frac{5 d_{k}^2}{9 D_{k}^2}\, , \;
k=0,\ldots,m-1\, , \smallskip \\
\zeta_{8k+1}^{3k+1,3k} = \frac{5 d_{k+1}^2}{9 D_{k}^2}\, , \;
\zeta_{8k+1}^{3k+2,3k} = \frac{5 d_{k}d_{k+1}}{9 D_{k}^2}\, , \;
\zeta_{8k+1}^{3k+1,3k+1} = \frac{13 d_{k}d_{k+1}}{9 D_{k}^2}\, , \;
\zeta_{8k+1}^{3k+2,3k+1} = \frac{d_{k}^2}{D_{k}^2}\, , \;
\zeta_{8k+1}^{3k,3k+1} = \frac{4 d_{k+1}^2}{9 D_{k}^2}\, , \;
k=0,\ldots,m-1\, , \smallskip \\
\zeta_{8k+2}^{3k,3k+2} = \frac{d_{k+1}^2}{6 D_{k}^2}\, , \;
\zeta_{8k+2}^{3k+1,3k+2} = \frac{d_{k}d_{k+1}}{3 D_{k}^2}\, , \;
\zeta_{8k+2}^{3k+2,3k+2} = \frac{d_{k}^2}{6 D_{k}^2}\, , \;
\zeta_{8k+2}^{3k+1,3k+1} = \frac{5 d_{k+1}}{9 D_{k}}\, , \;
\smallskip \\
\zeta_{8k+2}^{3k+2,3k+1} = \frac{15 d_{k}}{18 D_{k}}\, , \;
\zeta_{8k+2}^{3k+2,3k} = \frac{5 d_{k+1}}{18 D_{k}}\, , \;
k=0,\ldots,m-2\, , \smallskip \\
\zeta_{8k+3}^{3k,3k+3} = \frac{d_{k+1}^2}{21 D_{k}^2}\, , \;
\zeta_{8k+3}^{3k+1,3k+3} = \frac{2 d_{k}d_{k+1}}{21 D_{k}^2}\, , \;
\zeta_{8k+3}^{3k+2,3k+3} = \frac{d_{k}^2}{21 D_{k}^2}\, , \;
\zeta_{8k+3}^{3k+1,3k+2} = \frac{5 d_{k+1}}{14 D_{k}}\, , \;
\zeta_{8k+3}^{3k+2,3k+2} = \frac{5 d_{k}}{14 D_{k}}\, , \;
\smallskip \\
\zeta_{8k+3}^{3k+2,3k+1} = \frac{10}{21}\, , \;
\zeta_{8k+3}^{3k+3,3k+1} = \frac{5 d_{k}}{42 D_{k}}\, , \;
\zeta_{8k+3}^{3k+3,3k} = \frac{5 d_{k+1}}{42 D_{k}}\, , \;
k=0,\ldots,m-2\, , \smallskip \\
\zeta_{8k+4}^{3k,3k+4} = \frac{d_{k+1}^3}{126 D_{k}^2 D_{k+1}}\, ,\;
\zeta_{8k+4}^{3k+1,3k+4} = \frac{d_{k} d_{k+1}^2}{63 D_{k}^2
D_{k+1}}\, ,\; \zeta_{8k+4}^{3k+2,3k+4} = \frac{d_{k}^2 d_{k+1}}{126
D_{k}^2 D_{k+1}}\, ,\; \zeta_{8k+4}^{3k,3k+3} = \frac{d_{k+1}^2
d_{k+2}}{126 D_{k}^2 D_{k+1}}\, ,\;
\smallskip \\
\zeta_{8k+4}^{3k+1,3k+3} = \frac{d_k d_{k+1} d_{k+2}}{63 D_{k}^2
D_{k+1}} + \frac{10 d_{k+1}}{63 D_{k}}\, ,\;
\zeta_{8k+4}^{3k+2,3k+3} = \frac{d_k^2
 d_{k+2}}{126 D_{k}^2 D_{k+1}} + \frac{10 d_{k}}{63 D_{k}}\,
,\; \zeta_{8k+4}^{3k+2,3k+2} = \frac{10}{21}\, , \;
\zeta_{8k+4}^{3k+3,3k+1} = \frac{5 d_k d_{k+2}}{126 D_{k} D_{k+1}} +
\frac{20}{63}\, ,\;
\smallskip \\
\zeta_{8k+4}^{3k+4,3k+1} = \frac{5 d_k d_{k+1}}{126 D_{k} D_{k+1}}\,
,\; \zeta_{8k+4}^{3k+3,3k} = \frac{5 d_{k+1} d_{k+2}}{126 D_{k}
D_{k+1}}\, ,\; \zeta_{8k+4}^{3k+4,3k} = \frac{5 d_{k+1}^2}{126 D_{k}
D_{k+1}}\, ,\;
k=0,\ldots,m-2\, , \smallskip \\
\zeta_{8k+5}^{3k+1,3k+4} = \frac{5 d_{k+1}^2}{126 D_{k} D_{k+1}}\,
,\; \zeta_{8k+5}^{3k+2,3k+4} = \frac{5 d_{k} d_{k+1}}{126 D_{k}
D_{k+1}}\, ,\; \zeta_{8k+5}^{3k+1,3k+3} = \frac{5 d_{k+1}
d_{k+2}}{126 D_{k} D_{k+1}}\, ,\; \zeta_{8k+5}^{3k+2,3k+3} = \frac{5
d_{k} d_{k+2}}{126 D_{k} D_{k+1}} + \frac{20}{63}\, ,\;
\smallskip \\
\zeta_{8k+5}^{3k+3,3k+2} = \frac{10}{21}\, , \;
\zeta_{8k+5}^{3k+3,3k+1} = \frac{d_k
 d_{k+2}^2}{126 D_{k} D_{k+1}^2} + \frac{10 d_{k+2}}{63 D_{k+1}}\,
,\; \zeta_{8k+5}^{3k+4,3k+1} = \frac{d_k d_{k+1}
 d_{k+2}}{63 D_{k} D_{k+1}^2} + \frac{10 d_{k+1}}{63 D_{k+1}}\, ,\;
\smallskip \\
\zeta_{8k+5}^{3k+5,3k+1} = \frac{d_k
 d_{k+1}^2}{126 D_{k} D_{k+1}^2} \,
,\; \zeta_{8k+5}^{3k+3,3k} = \frac{d_{k+1}
 d_{k+2}^2}{126 D_{k} D_{k+1}^2} \,
,\; \zeta_{8k+5}^{3k+4,3k} = \frac{d_{k+1}^2
 d_{k+2}}{63 D_{k} D_{k+1}^2} \,
,\; \zeta_{8k+5}^{3k+5,3k} = \frac{d_{k+1}^3}{126 D_{k} D_{k+1}^2}
\, ,\;
k=0,\ldots,m-2\, , \smallskip \\
\zeta_{8k+6}^{3k+2,3k+4} = \frac{5 d_{k+1}}{42 D_{k+1}}\, , \;
\zeta_{8k+6}^{3k+2,3k+3} = \frac{5 d_{k+2}}{42 D_{k+1}}\, , \;
\zeta_{8k+6}^{3k+3,3k+3} = \frac{10}{21}\, , \;
\zeta_{8k+6}^{3k+4,3k+2} = \frac{5 d_{k+1}}{14 D_{k+1}}\, , \;
\zeta_{8k+6}^{3k+3,3k+2} = \frac{5 d_{k+2}}{14 D_{k+1}}\, , \;
\smallskip \\
\zeta_{8k+6}^{3k+3,3k+1} = \frac{d_{k+2}^2}{21 D_{k+1}^2}\, , \;
\zeta_{8k+6}^{3k+4,3k+1} = \frac{2 d_{k+1}d_{k+2}}{21 D_{k+1}^2}\, ,
\; \zeta_{8k+6}^{3k+5,3k+1} = \frac{d_{k+1}^2}{21 D_{k+1}^2}\, , \;
k=0,\ldots,m-2\, , \smallskip \\
\zeta_{8k+7}^{3k+3,3k+4} = \frac{5 d_{k+1}}{18 D_{k+1}}\, , \;
\zeta_{8k+7}^{3k+3,3k+3} = \frac{15 d_{k+2}}{18 D_{k+1}}\, , \;
\zeta_{8k+7}^{3k+4,3k+3} = \frac{5 d_{k+1}}{9 D_{k+1}}\, , \;
\zeta_{8k+7}^{3k+3,3k+2} = \frac{d_{k+2}^2}{6 D_{k+1}^2}\, , \;
\smallskip \\
\zeta_{8k+7}^{3k+4,3k+2} = \frac{d_{k+1}d_{k+2}}{3 D_{k+1}^2}\, , \;
\zeta_{8k+7}^{3k+5,3k+2} = \frac{d_{k+1}^2}{6 D_{k+1}^2}\, , \;
k=0,\ldots,m-2\, ,
\end{array}
\ee where $D_k:=d_k+d_{k+1}$, $k=0,...,m-1$ with $d_0=d_{m}:=0$.

\begin{figure}[h!]
\centering
\resizebox{8.0cm}{!}{\includegraphics{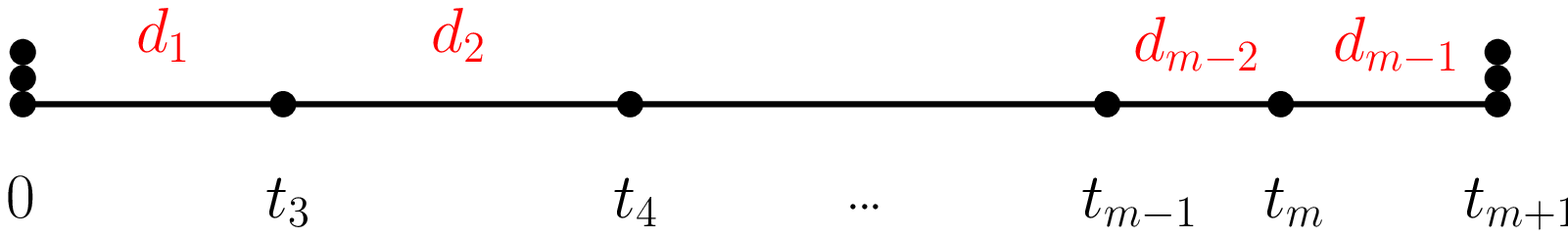}}\\
\resizebox{8.0cm}{!}{\includegraphics{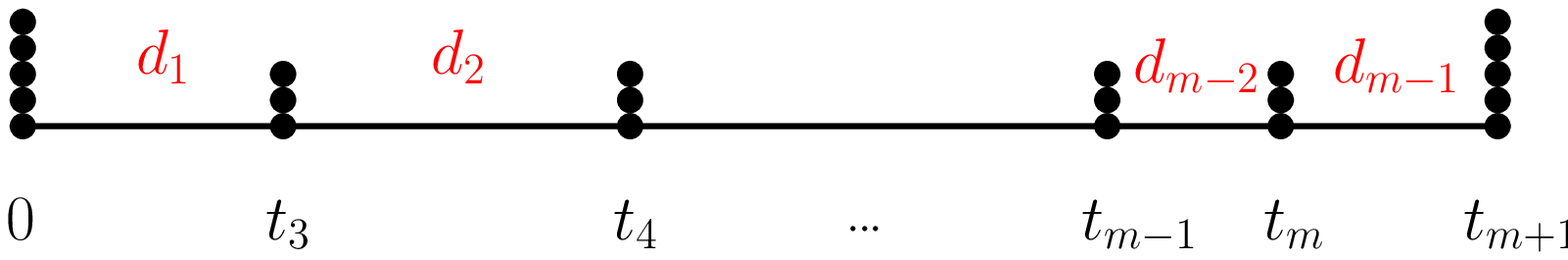}}\\
\resizebox{8.0cm}{!}{\includegraphics{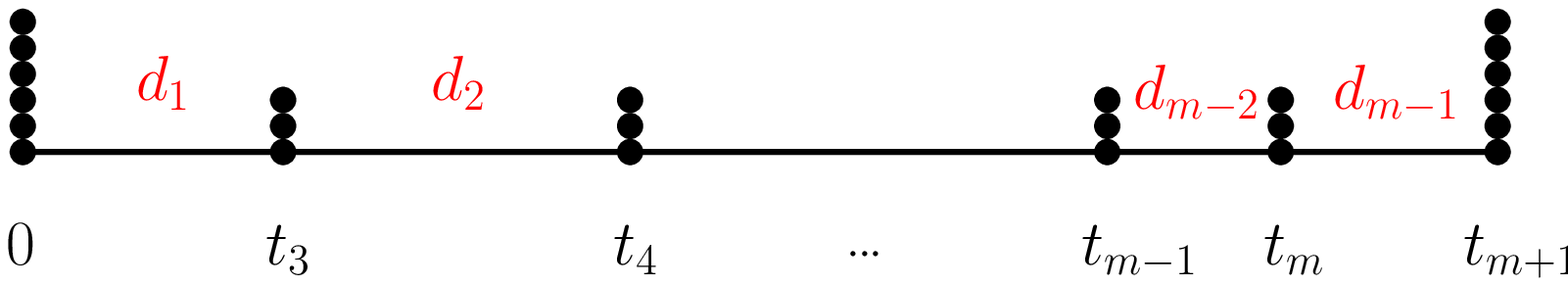}}\\
\resizebox{8.0cm}{!}{\includegraphics{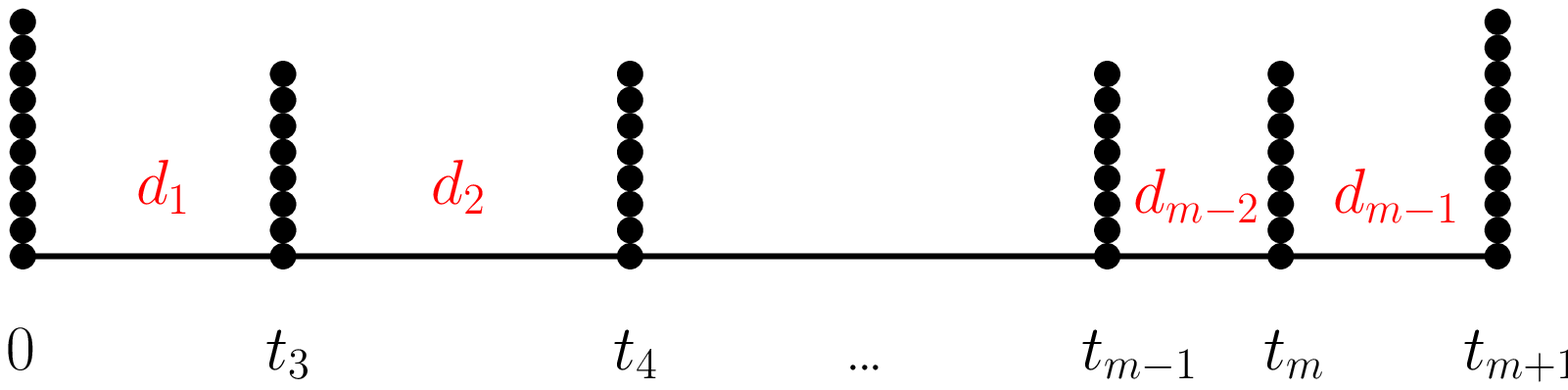}}
\caption{Knot partitions for the clamped case $n=2$. From top to bottom: $\bmu$, $\bnu$, $\brho$, $\btau$.}
\label{fig_knots_clamped_n2}
\end{figure}

By means of the computed
coefficients $\{\chi_{k}^{i,j}\}^{0\le i,j \le m}_{0 \le {k} \le 3 m
- 2}$ we can thus shortly write the control points of $\r'(t)$ as
$$
\begin{array}{l}
\p_0=\z_0^2,\smallskip\\
\p_1=\z_0 \z_1,\smallskip\\
\p_{3k-1}=\frac{2}{3} \z_k^2 + \frac{1}{3} \, \left(\frac{d_k \z_{k-1}+d_{k-1} \z_k}{d_{k-1}+d_{k}}\right) \, \left( \frac{d_{k+1} \z_k+d_k \z_{k+1}}{d_{k}+d_{k+1}} \right), \quad k=1,...,m-1\,,  \smallskip\\
\p_{3k}= \z_k \, \frac{d_{k+1} \z_k+d_k \z_{k+1}}{d_{k}+d_{k+1}},   \quad k=1,...,m-2\,,  \smallskip\\
\p_{3k+1}= \z_{k+1} \, \frac{d_{k+1} \z_k+d_k \z_{k+1}}{d_{k}+d_{k+1}},   \quad k=1,...,m-2\,,  \smallskip\\
\p_{3m-3}= \z_{m-1} \z_m, \smallskip\\
\p_{3m-2}= \z_m^2,
\end{array}
$$
and the coefficients of the
parametric speed $\sigma(t)$ as
$$
\begin{array}{l}
\sigma_{0}= \z_0 \bar{\z}_0,  \smallskip  \\
\sigma_{1}= \frac12 (\z_0 \bar{\z}_1 + \z_1 \bar{\z}_0), \smallskip  \\
\sigma_{3k-1}= \frac{1}{6} \, \frac{d_k \, d_{k+1}}{(d_{k-1}+d_{k})\, (d_k+d_{k+1})}
(\z_{k-1} \bar{\z}_{k} + \z_{k} \bar{\z}_{k-1})+
\frac{1}{6} \, \frac{(d_k)^2}{(d_{k-1}+d_{k})(d_k+d_{k+1})}
(\z_{k-1} \bar{\z}_{k+1} + \z_{k+1} \bar{\z}_{k-1}) \smallskip  \\
 \qquad \quad +\left( \frac{2}{3} + \frac{1}{3} \, \frac{d_{k-1} \, d_{k+1}}{(d_{k-1}+d_{k})\, (d_k+d_{k+1})} \right)
\z_k \bar{\z}_k +
\frac{1}{6} \, \frac{d_{k-1} \, d_{k}}{(d_{k-1}+d_{k})\, (d_k+d_{k+1})}
(\z_{k} \bar{\z}_{k+1} + \z_{k+1} \bar{\z}_{k}), \quad  k=1,...,m-1\,,  \smallskip \\
\sigma_{3k}= \frac{d_{k+1}}{d_{k}+d_{k+1}} \z_k \bar{\z}_k +
\frac{1}{2}\,\frac{d_{k}}{d_{k}+d_{k+1}} (\z_{k} \bar{\z}_{k+1} + \z_{k+1} \bar{\z}_{k}), \quad  k=1,...,m-2\,,  \smallskip\\
\sigma_{3k+1}= \frac{1}{2}\, \frac{d_{k+1}}{d_{k}+d_{k+1}} (\z_{k} \bar{\z}_{k+1} + \z_{k+1} \bar{\z}_{k}) +
\frac{d_{k}}{d_{k}+d_{k+1}} \z_{k+1} \bar{\z}_{k+1}, \quad  k=1,...,m-2\,,  \smallskip \\
\sigma_{3m-3}= \frac12 (\z_{m-1} \bar{\z}_m + \z_m \bar{\z}_{m-1}), \smallskip   \\
\sigma_{3m-2}= \z_m \bar{\z}_m,
\end{array}
$$
with $d_0:=0$ and $d_{m}:=0$.
Thus, according to (\ref{rt}), the clamped quintic PH B-Spline curve defined
over the knot partition $\brho$ is given by
$$
\r(t) = \sum_{i=0}^{3m-1} \r_i N_{i,\brho}^{5}(t)\,, \quad t \in [t_2, t_{m+1}] \qquad  (t_2=0),
$$
with control points \be \label{rs_clampedn2}
\begin{array}{l}
\r_{1} = \r_0 + \frac{d_{1}}{5}\, \z_0^2, \quad \vspace{0.04cm}\\
\r_{2} = \r_1 + \frac{d_{1}}{5}\, \z_0 \z_1, \quad \vspace{0.04cm}\\
\r_{3i} = \r_{3i-1} + \frac{d_{i}}{5}\, \left( \frac{2}{3} \z_i^2 + \frac{1}{3} \, \left(\frac{d_{i}\z_{i-1}+d_{i-1}\z_i}{d_{i-1}+d_{i}}\right)
\left( \frac{d_{i+1} \z_i+d_{i} \z_{i+1}}{d_{i}+d_{i+1}} \right)
\right), \quad i=1,...,m-1\,,  \vspace{0.04cm}\\
\r_{3i+1} = \r_{3i} + \frac{\z_i}{5} \, \left(d_{i+1} \z_i+d_{i}\z_{i+1}\right), \quad i=1,...,m-2\,,  \vspace{0.04cm}\\
\r_{3i+2} = \r_{3i+1} + \frac{\z_{i+1}}{5} \, \left(d_{i+1} \z_i+d_{i}\z_{i+1}\right), \quad i=1,...,m-2\,,  \vspace{0.04cm}\\
\r_{3m-2} = \r_{3m-3} + \frac{d_{m-1}}{5}\, \z_{m-1} \z_m,  \vspace{0.04cm}\\
\r_{3m-1} = \r_{3m-2} + \frac{d_{m-1}}{5}\, \z_m^2\,,
\end{array}
\ee and arbitrary $\r_0$.

\begin{rmk}
Note that, when $m=2$ and $t_3=t_4=t_5=1$, the expressions of the
control points coincide with those of Farouki's PH B\'ezier quintic
from \cite{faroukisakkalis, farouki94}.
\end{rmk}

\smallskip
According to (\ref{totallengthclamped}) the total arc length of the
clamped PH B-Spline curve of degree $5$ is given by $L=l_{3m-1}$,
where
$$
\begin{array}{l}
l_0=0, \vspace{0.04cm}\\
l_{1} = l_0 + \frac{d_{1}}{5}\, \z_0 \bar{\z}_0,   \vspace{0.04cm}\\
l_{2} = l_1 + \frac{d_{1}}{10} \, (\z_0 \bar{\z}_1 + \z_1 \bar{\z}_0),  \vspace{0.04cm}\\
l_{3i} = l_{3i-1} + \frac{2 d_i}{15} \z_i \bar{\z}_i + \frac{d_i}{15 (d_{i-1}+d_{i})(d_{i}+d_{i+1})} \Big( d_{i-1}d_{i+1} \z_i \bar{\z}_i + \frac{d_i^2}{2} (\z_{i-1} \bar{\z}_{i+1} + \z_{i+1} \bar{\z}_{i-1}) \vspace{0.04cm}\\
 \qquad + \frac{d_{i} d_{i+1}}{2} \, (\z_{i-1} \bar{\z}_{i} + \z_{i} \bar{\z}_{i-1}) + \frac{d_{i-1} d_{i}}{2} \, (\z_{i} \bar{\z}_{i+1} + \z_{i+1} \bar{\z}_{i}) \Big), \quad  i=1,...,m-1\,,  \vspace{0.04cm}\\
l_{3i+1} = l_{3i} + \frac{1}{5} \, \Big( d_{i+1} \z_i \bar{\z}_i +
\frac{d_{i}}{2} (\z_{i} \bar{\z}_{i+1} + \z_{i+1} \bar{\z}_{i}) \Big), \quad   i=1,...,m-2\,,  \vspace{0.04cm}\\
l_{3i+2} = l_{3i+1} + \frac{1}{5} \, \Big( d_{i}  \z_{i+1} \bar{\z}_{i+1} +
\frac{d_{i+1}}{2}  (\z_{i} \bar{\z}_{i+1} + \z_{i+1} \bar{\z}_{i}) \Big), \quad  i=1,...,m-2\,, \vspace{0.04cm}\\
l_{3m-2} = l_{3m-3} + \frac{d_{m-1}}{10}\, (\z_{m-1} \bar{\z}_m + \z_m \bar{\z}_{m-1}),  \vspace{0.04cm}\\
l_{3m-1} = l_{3m-2} + \frac{d_{m-1}}{5} \,  \z_m \bar{\z}_m.
\end{array}
$$
Over the knot partition $\btau$, the offset curve $\r_h(t)$ has the rational
B-Spline form \be \label{offsetn2clamped} \r_h(t) =
\frac{\displaystyle \sum_{k=0}^{8m-7} \q_k
N_{k,\btau}^{9}(t)}{\displaystyle \sum_{k=0}^{8m-7} \gamma_k
N_{k,\btau}^{9}(t)}\,, \quad t \in [t_2, t_{m+1}], \ee where, by
exploiting the explicit expressions of the coefficients
$\{\zeta_k^{i,j}\}^{0 \leq i \leq 3m-1, \, 0 \leq j \leq 3m-2}_{0
\leq k \leq 8m-7}$ from (\ref{zetasn2clamped}), weights and control
points are easily obtained; the explicit formulae are reported in the
Appendix.

\begin{figure}[h!]
\centering
\includegraphics[height=0.25\textwidth,valign=t]{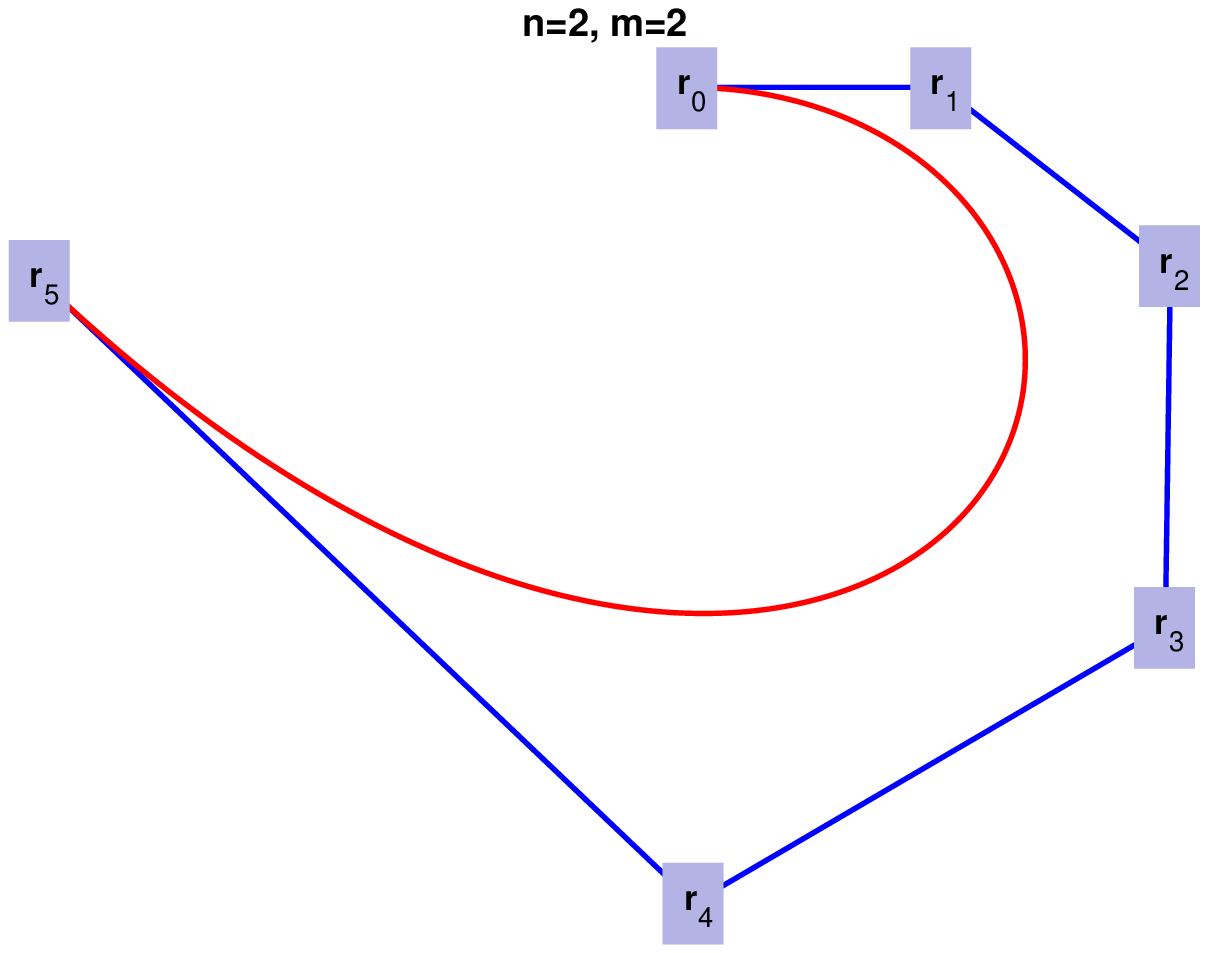}\hspace{-0.3cm}
\includegraphics[height=0.25\textwidth,valign=t]{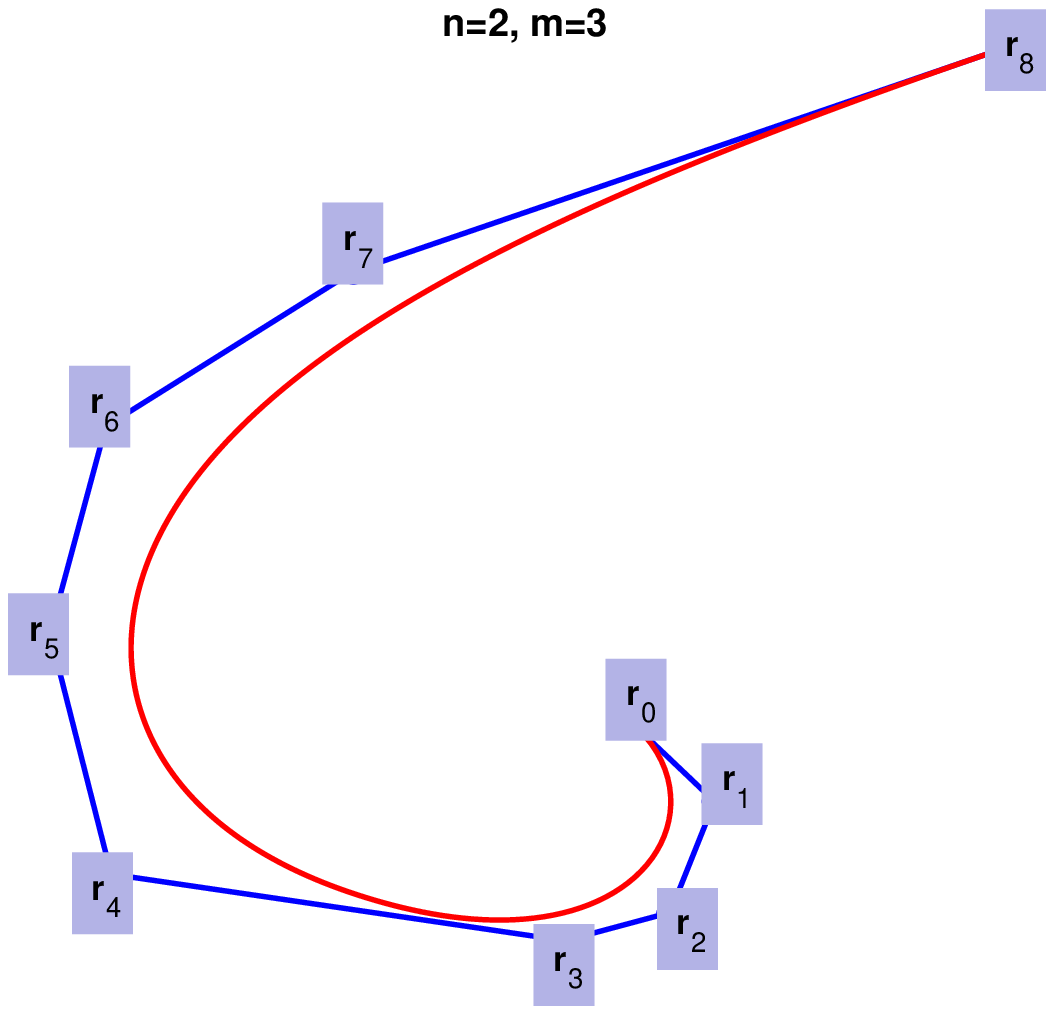}\hspace{-0.3cm}
\includegraphics[height=0.25\textwidth,valign=t]{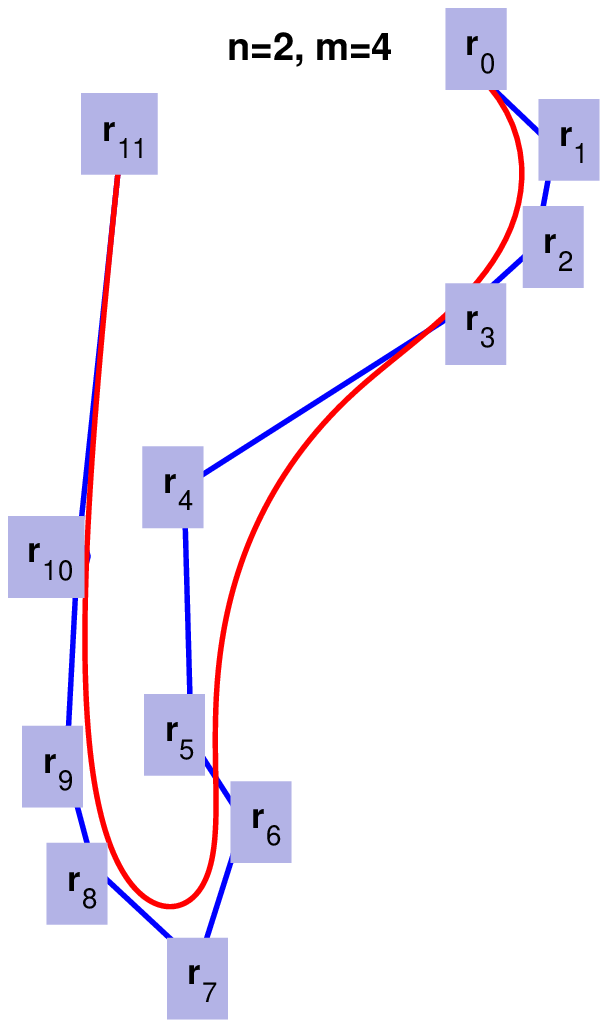}
\caption{Clamped quintic PH B-Spline curves with: $m=2$ (left), $m=3$ (center) and $m=4$ (right).}
\label{fig:clamped_n2m2}
\end{figure}

\begin{figure}[h!]
\centering
\includegraphics[height=0.25\textwidth,valign=t]{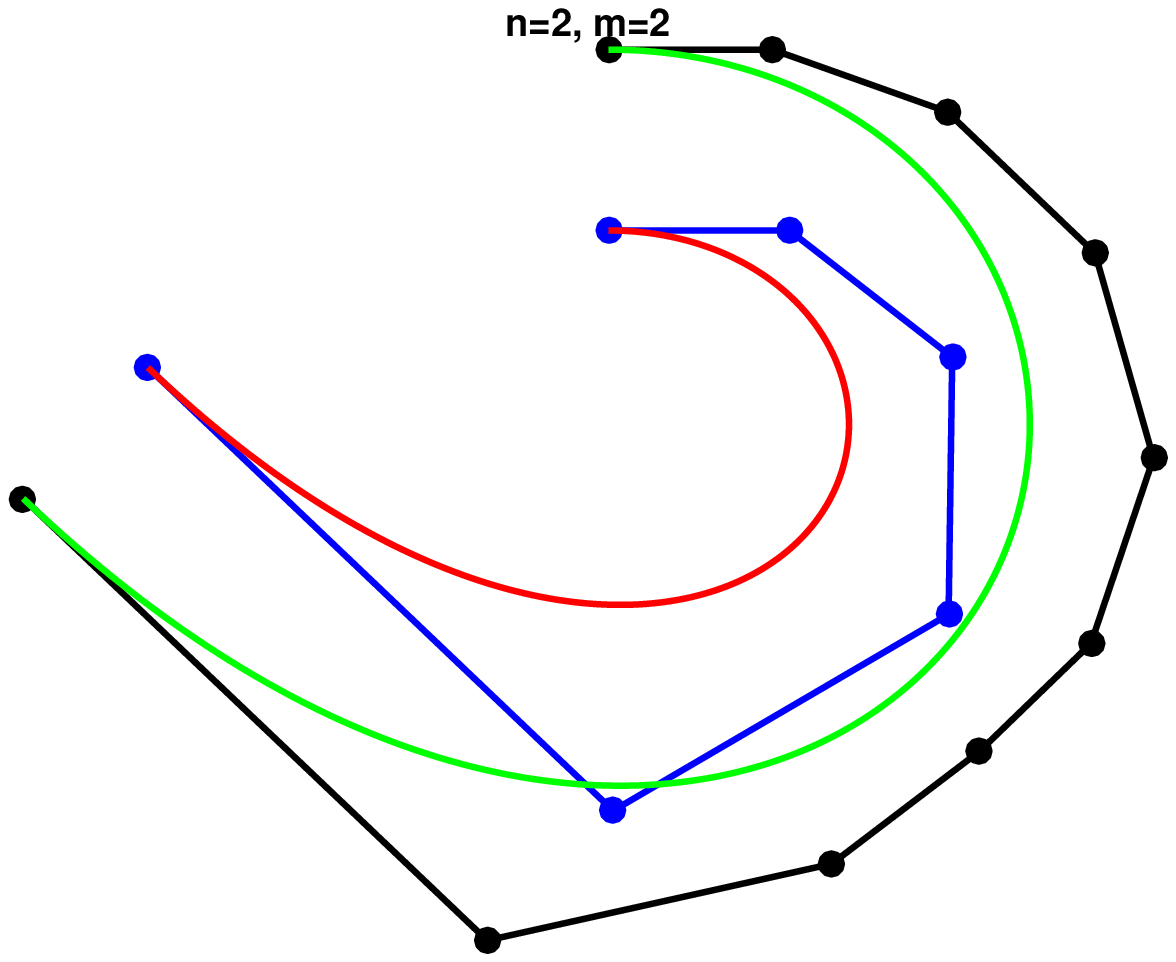}\hspace{-0.3cm}
\includegraphics[height=0.25\textwidth,valign=t]{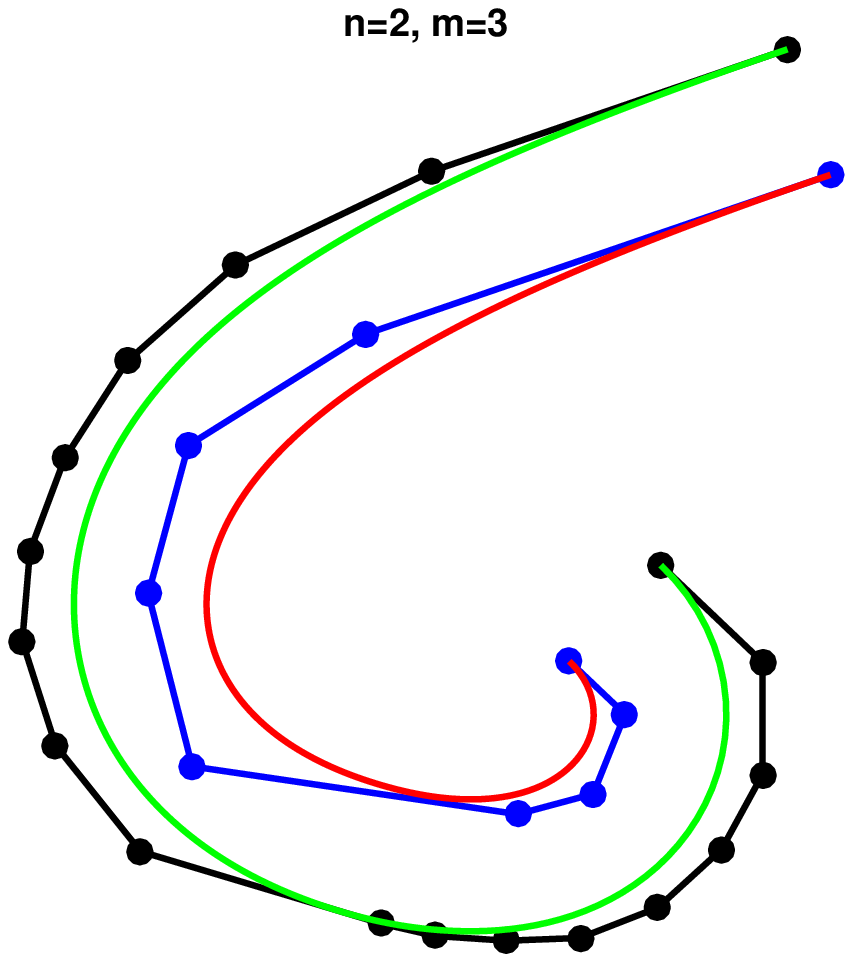}\hspace{-0.3cm}
\includegraphics[height=0.25\textwidth,valign=t]{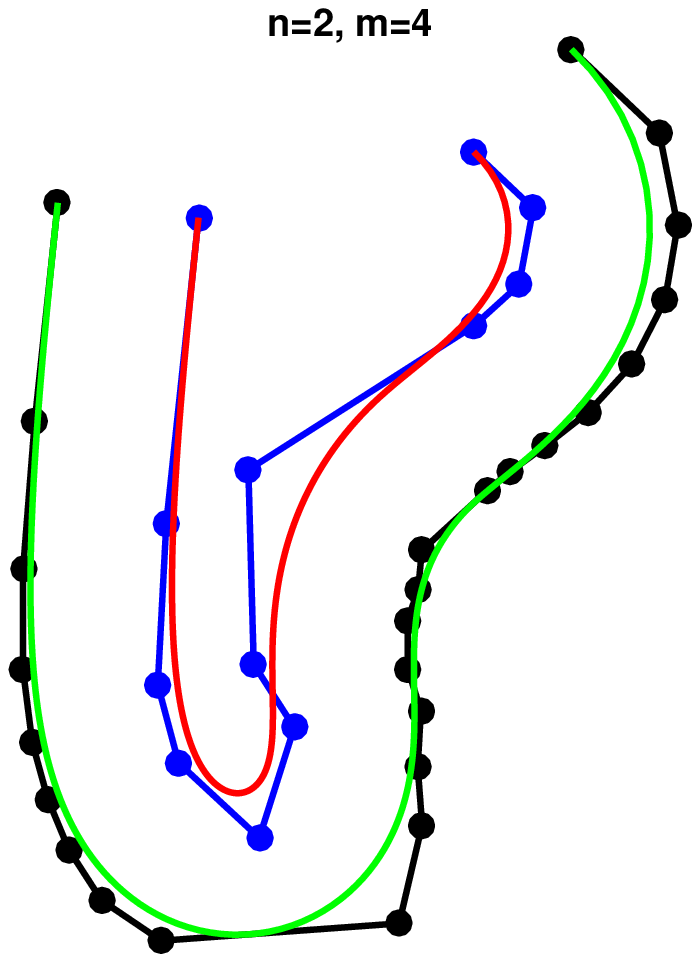}\\
\includegraphics[height=0.25\textwidth,valign=t]{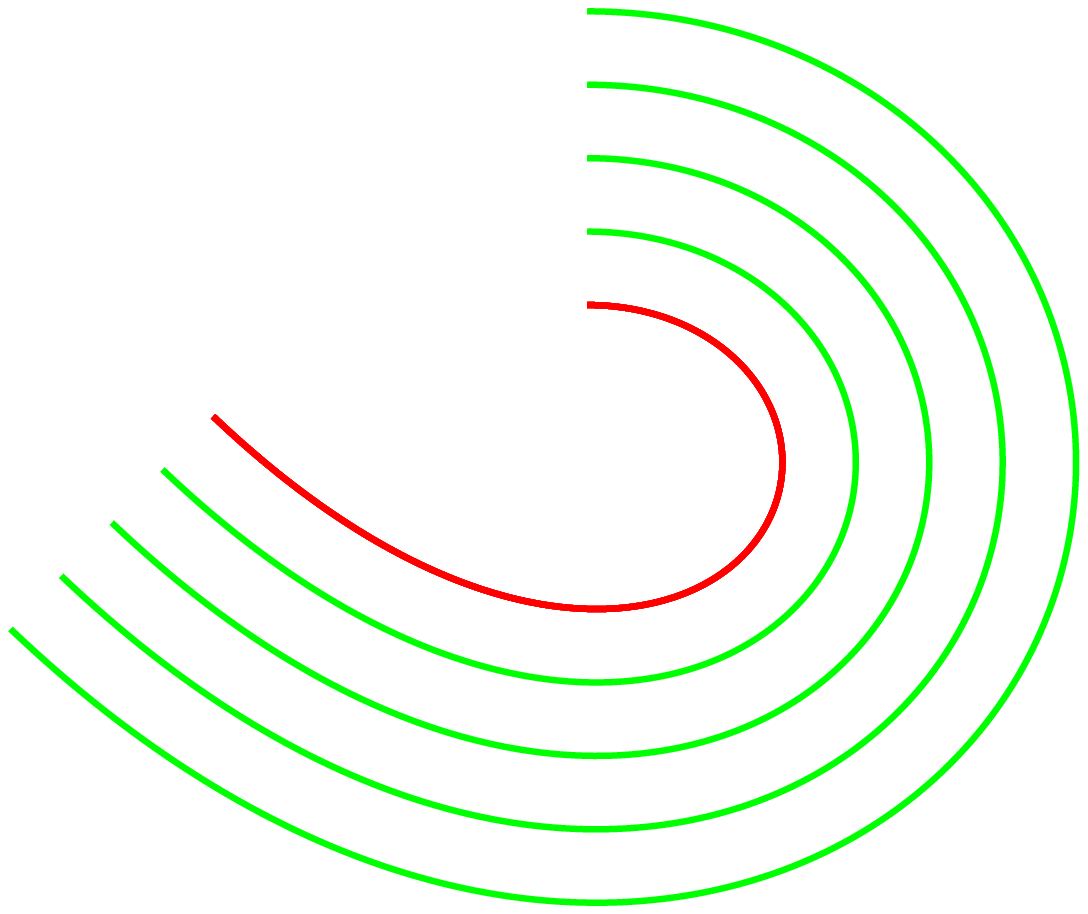}\hspace{-0.3cm}
\includegraphics[height=0.25\textwidth,valign=t]{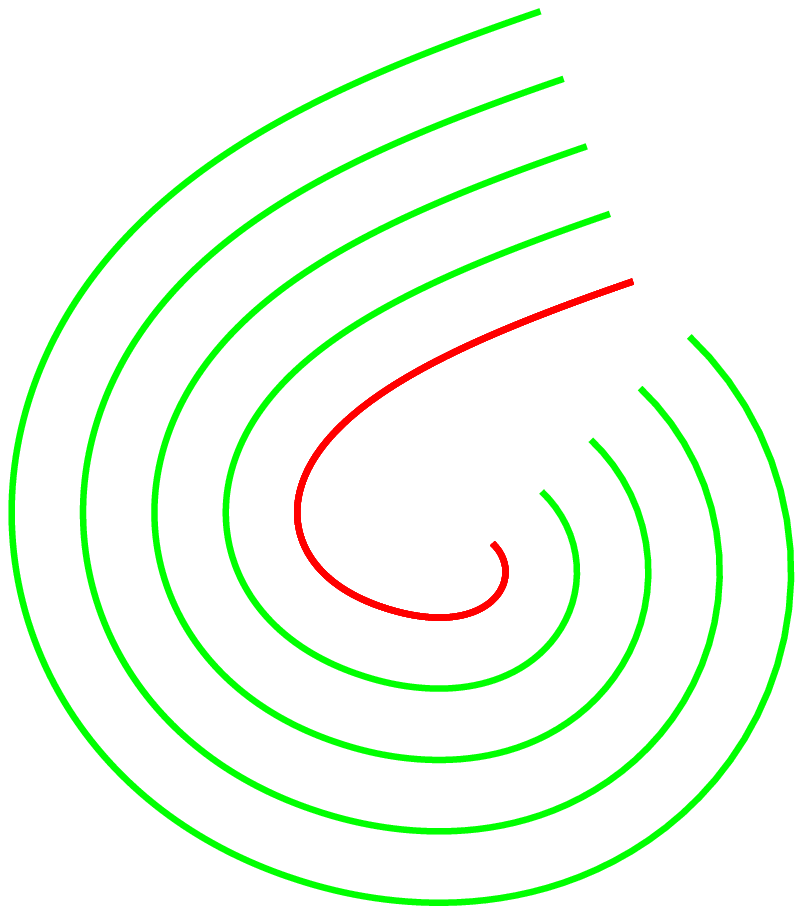}\hspace{-0.3cm}
\includegraphics[height=0.25\textwidth,valign=t]{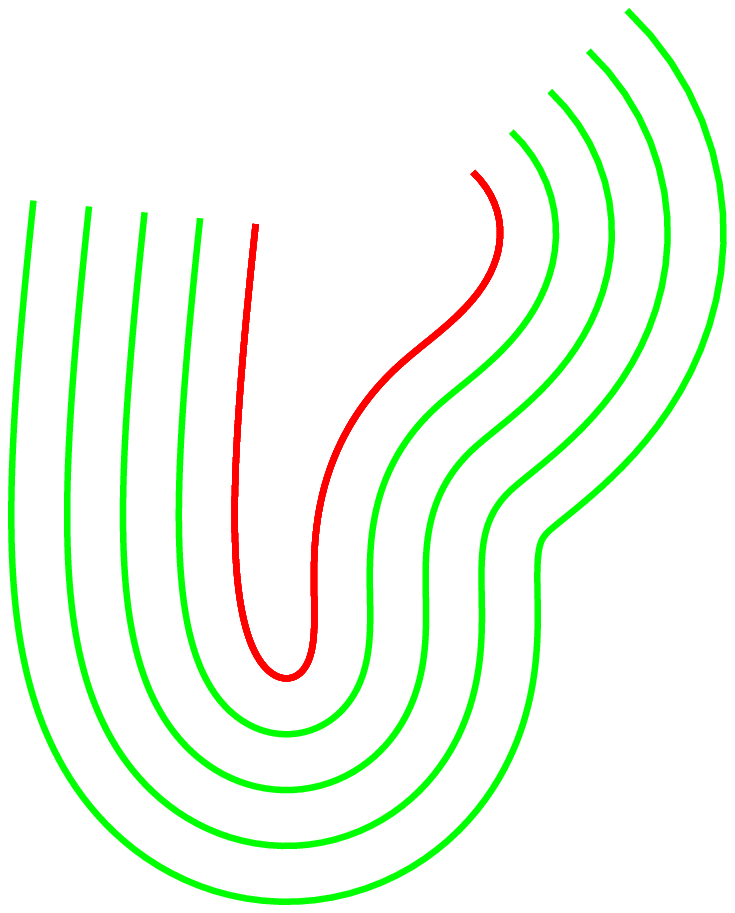}
\caption{Offsets of clamped quintic PH B-Spline curves from Figure
\ref{fig:clamped_n2m2} with (1. row) and without (2. row) control
polygon where: $m=2$ (left column), $m=3$ (center column) and $m=4$
(right column).} \label{fig:clamped_n2m3}
\end{figure}

Some examples of clamped quintic PH B-Spline curves are shown in Figure
\ref{fig:clamped_n2m2}, and their offsets are displayed in Figure
\ref{fig:clamped_n2m3}.

\subsection{Closed cubic PH B-Splines ($n=1$)}
\label{sec5n1closed}

Let $m \in \NN$, $m \geq 1$. For a general knot vector
$\bmu =\{ 0=t_0 < t_1 < ... < t_{m+3} \}$
(see Figure \ref{fig_knots_closed_n1} first row),
by applying the above method we construct the knot partitions
$$
\begin{array}{l}
\bnu= \{\langle 0 \rangle^2 < \langle t_1 \rangle^2 < ... < \langle t_{m+3} \rangle^2\}, \smallskip \\
\brho = \{t_{-1} < \langle 0 \rangle^2 < \langle t_1 \rangle^2 < ... < \langle t_{m+3} \rangle^2 < t_{m+4} \}, \smallskip \\
\btau = \{ \langle t_{-1} \rangle^3 < \langle 0 \rangle^5 < \langle
t_1 \rangle^5 < ... < \langle t_{m+3} \rangle^5 < \langle t_{m+4}
\rangle^3 \},
\end{array}
$$
illustrated in Figure \ref{fig_knots_closed_n1}. Then, by solving the linear systems \eqref{lgschi} we
calculate the coefficients $\chi_{k}^{i,j}$, $0\le
i,j \le m+1$, $0 \le {k} \le 2m+4$.
All of them turn out to be zero with the exception of
$$
\begin{array}{l}
\chi_{2k+1}^{k,k} = 1\,, \quad k = 0, \ldots, m+1\,, \smallskip \\
\chi_{2k+2}^{k,k+1} = \chi_{2k+2}^{k+1,k} =\frac{1}{2}\,, \quad k = 0, \ldots, m.
\end{array}
$$

\begin{figure}[h!]
\centering
\hspace{-0.1cm}\resizebox{9.0cm}{!}{\includegraphics{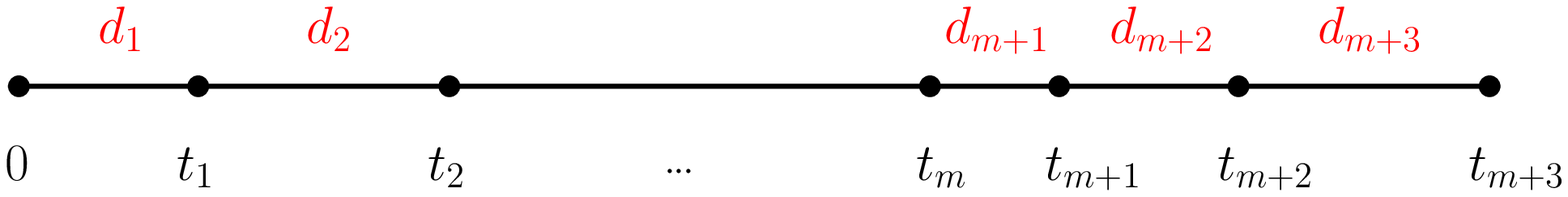}}\\
\hspace{-0.1cm}\resizebox{9.0cm}{!}{\includegraphics{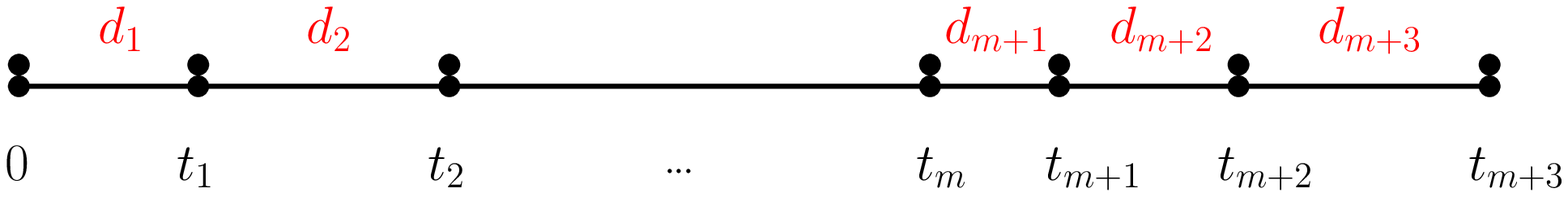}}\\
\resizebox{10.8cm}{!}{\includegraphics{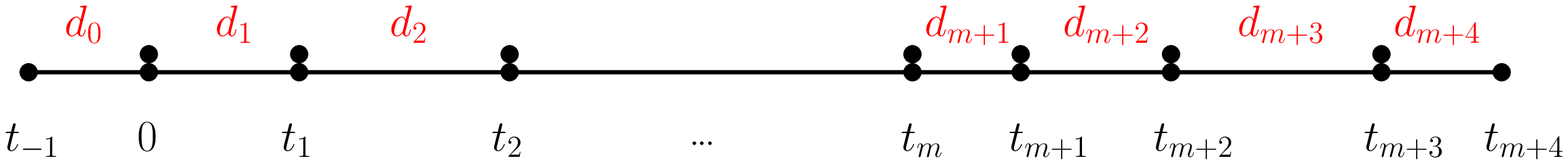}}\\
\resizebox{10.8cm}{!}{\includegraphics{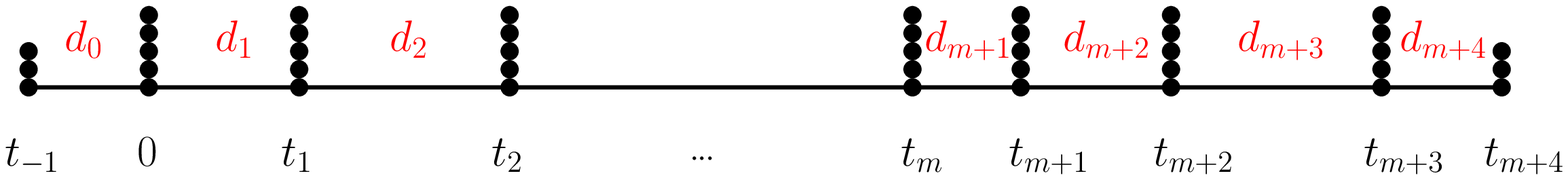}}
\caption{Knot partitions for the closed case $n=1$. From top to bottom: $\bmu$, $\bnu$, $\brho$, $\btau$.}
\label{fig_knots_closed_n1}
\end{figure}

In addition, we compute the coefficients $\zeta_k^{i,j}$, $0 \leq i
\leq 2m+5$, $0 \leq j \leq 2m+4$, $0 \leq k \leq 5m+19$ as the
solutions to the linear systems \eqref{zetas}. All of them
turn out to be zero with the exception of \be \label{zetasn1closed}
\begin{array}{l}
\zeta_{3}^{0,0} = \frac{2d_{0}}{5 D_{0}}\, , \; \zeta_{4}^{0,0} =
\frac{3}{5}\, , \; \zeta_{4}^{0,1} = \frac{d_{0}}{10 D_{0}}\, , \;
\smallskip \\
 \zeta_{5}^{0,1} = \frac{3}{10}\, , \; \zeta_{5}^{1,0}
= \frac{3}{5}\, , \; \zeta_{6}^{1,0} = \frac{2d_{2}}{5 D_{1}}\, , \;
\zeta_{6}^{1,1} = \frac{3}{5}\, , \; \zeta_{6}^{2,0} =
\frac{2d_{1}}{5 D_{1}}\, , \;
\smallskip \\
\zeta_{5k+7}^{2k+1,2k+1} = \frac{d_{k+2}}{D_{k+1}}\, , \;
\zeta_{5k+7}^{2k+2,2k+1} = \frac{d_{k+1}}{D_{k+1}}\, , \;
k=0,\ldots,m+1\, , \smallskip \\
\zeta_{5k+8}^{2k+1,2k+2} = \frac{2d_{k+2}}{5 D_{k+1}}\, , \;
\zeta_{5k+8}^{2k+2,2k+1} = \frac{3}{5}\, , \;
\zeta_{5k+8}^{2k+2,2k+2} = \frac{2d_{k+1}}{5 D_{k+1}}\, , \;
k=0,\ldots,m\, , \smallskip \\
\zeta_{5k+9}^{2k+1,2k+3} = \frac{d_{k+2}}{10 D_{k+1}}\, , \;
\zeta_{5k+9}^{2k+2,2k+2} = \frac{3}{5}\, , \;
\zeta_{5k+9}^{2k+2,2k+3} = \frac{d_{k+1}}{10 D_{k+1}}\, , \;
\zeta_{5k+9}^{2k+3,2k+1} = \frac{3}{10}\, , \;
k=0,\ldots,m\, , \smallskip \\
\zeta_{5k+10}^{2k+2,2k+3} = \frac{3}{10}\, , \;
\zeta_{5k+10}^{2k+3,2k+1}= \frac{d_{k+3}}{10 D_{k+2}}\, , \;
\zeta_{5k+10}^{2k+3,2k+2} = \frac{3}{5}\, , \;
\zeta_{5k+10}^{2k+4,2k+1} = \frac{d_{k+2}}{10 D_{k+2}}\, , \;
k=0,\ldots,m\, , \smallskip \\
\zeta_{5k+11}^{2k+3,2k+2} = \frac{2d_{k+3}}{5 D_{k+2}}\, , \;
\zeta_{5k+11}^{2k+3,2k+3} = \frac{3}{5}\, , \;
\zeta_{5k+11}^{2k+4,2k+2} = \frac{2d_{k+2}}{5 D_{k+2}}\, , \;
k=0,\ldots,m\, , \smallskip \\
 \zeta_{5m+13}^{2m+3,2m+4} = \frac{2d_{m+3}}{5
D_{m+2}}\, , \; \zeta_{5m+13}^{2m+4,2m+3} = \frac{3}{5}\, , \;
\zeta_{5m+13}^{2m+4,2m+4} = \frac{2d_{m+2}}{5 D_{m+2}}\, , \;
\zeta_{5m+14}^{2m+4,2m+4} = \frac{3}{5}\, , \;
\zeta_{5m+14}^{2m+5,2m+3} = \frac{3}{10}\, , \;
\smallskip \\
 \zeta_{5m+15}^{2m+5,2m+3} = \frac{d_{m+4}}{10
D_{m+3}}\, , \; \zeta_{5m+15}^{2m+5,2m+4} = \frac{3}{5}\, , \;
\zeta_{5m+16}^{2m+5,2m+4} = \frac{2d_{m+4}}{5 D_{m+3}}\, , \;
\end{array}
\ee where $D_k:=d_k+d_{k+1}$, $k=0,...,m+3$.

By means of the computed coefficients $\{\chi_{k}^{i,j}\}^{0\le i,j \le m+1}_{0 \le
{k} \le 2m+4}$ we can thus shortly write the control points of $\r'(t)$
as
$$
\begin{array}{l}
\p_0=0, \\
\p_{2k+1}= \z_k^2, \quad k=0,...,m+1\,,  \\
\p_{2k+2}= \z_k \z_{k+1}, \quad k=0,...,m\,,  \\
\p_{2m+4}=0,
\end{array}
$$
and the coefficients of the
parametric speed $\sigma(t)$ as
$$
\begin{array}{l}
\sigma_0=0, \\
\sigma_{2k+1}= \z_k \bar{\z}_k, \quad k=0,...,m+1\,,  \\
\sigma_{2k+2}= \frac12 \big( \z_k \bar{\z}_{k+1} + \z_{k+1} \bar{\z}_{k} \big), \quad k=0,...,m\,,  \\
\sigma_{2m+4}=0.
\end{array}
$$
Thus, according to (\ref{rt}), the closed cubic PH B-Spline curve defined
over the knot partition $\brho$ is given by
$$
\r(t) = \sum_{i=0}^{2m+5} \r_i N_{i,\brho}^{3}(t), \quad t \in  [t_1, t_{m+2}] \quad (t_0=0),
$$
with control points
$$
\begin{array}{l}
\r_1=\r_0, \\
\r_{2i+2}=\r_{2i+1} + \frac{d_{i+1}+d_{i+2}}{3}\, \z_i^2, \quad i=0,...,m+1\,,  \\
\r_{2i+3}=\r_{2i+2} + \frac{d_{i+2}}{3}\,  \z_i \z_{i+1}, \quad i=0,...,m\,,  \\
\r_{2m+5}=\r_{2m+4}\,,
\end{array}
$$
and arbitrary $\r_0$. Note that, due to condition
\eqref{condclosed3}, $d_{m+2}=d_1$, $d_{m+3}=d_{2}$. Moreover,
$\z_{m}$ and $\z_{m+1}$ must be suitably fixed in order to satisfy
condition \eqref{condclosed2}. According to (\ref{Lclosed1}) and
using the partition of unity $N_{1,\brho}^3(t_1) +
N_{2,\brho}^3(t_1)=1$, the total arc length of the closed PH B-Spline
curve of degree $3$ is given by
$$L= l_{2m+3} + (l_{2m+4}-l_{2m+3}-l_2) \, N_{2,\brho}^{3}(t_{1}),$$
where
%$N_{2,\brho}^{3}(t_{1})=N_{2m+4,\brho}^{3}(t_{m+2})$ and
$$
\begin{array}{l}
l_1=l_0=0, \\
l_{2i+2}=l_{2i+1} + \frac{d_{i+1}+d_{i+2}}{3}\, \z_{i} \bar{\z}_{i}, \quad  i=0,...,m+1\,,  \\
l_{2i+3}=l_{2i+2} + \frac{d_{i+2}}{6}\, \big( \z_i \bar{\z}_{i+1} + \z_{i+1} \bar{\z}_{i} \big), \quad  i=0,...,m\,,  \\
l_{2m+5}=l_{2m+4}.
\end{array}
$$
Over the knot partition $\btau$ the offset curve $\r_h(t)$
has the rational B-Spline form \be \label{offsetn1closed} \r_h(t) =
\frac{\sum_{k=0}^{5m+19} \q_k N_{k,\btau}^{5}(t)}{\sum_{k=0}^{5m+19}
\gamma_k N_{k,\btau}^{5}(t)}\,,\quad t \in [t_1, t_{m+2}],\ee where,
by exploiting the explicit expressions of the coefficients
$\{\zeta_k^{i,j}\}^{0 \leq i \leq 2m+5, \, 0 \leq j \leq 2m+4}_{0
\leq k \leq 5m+19}$ from (\ref{zetasn1closed}), weights and control
points can easily be obtained; their explicit formulae are reported
in the Appendix.

We now illustrate some closed degree $3$ PH B-Spline curves. To this
end we define
$$
r_\pm=-(4 d_1 d_2 +4 d_1 d_3 +4 d_2 d_3 +3 d_1^2) \z_0^2 -(4 d_1 d_2
\pm 2 d_1 d_3 +4 d_2 d_3) \z_0 \z_1 -(4 d_1 d_2 +4 d_1 d_3 +4 d_2
d_3 +3 d_3^2) \z_1^2\, ,
$$
$$
\begin{array}{lll}
R_\pm &=&-(4 d_1 d_2+4 d_1 d_4+4 d_2 d_4+3 d_1^2) \z_0^2 -(4 d_1 d_2+4 d_2 d_4) \z_0 \z_1 \pm 2 d_1 d_4 \z_0 \z_2 -(4 d_1 d_2+4 d_1 d_3+4 d_2 d_4+4 d_3 d_4) \z_1^2\\
&-&(4 d_1 d_3+4 d_3 d_4) \z_1 \z_2 -(4 d_1 d_3+4 d_1 d_4+4 d_3 d_4+3
d_4^2) \z_2^2\, .
\end{array}
$$
Figures \ref{fig:closed_n1_zclosed} and \ref{fig:closed_n1_zopen}
contain examples of cubic PH B-Spline curves obtained for the data
in Table \ref{table_datan1_closed}. Figure
\ref{fig:offsets_closed_n1} shows the offsets of the PH B-Spline
curves of Figure \ref{fig:closed_n1_zopen}.

\begin{table}
\centering
\begin{tabular}{|c|c|c|c|}
\hline
 $m$ & $\z(t)$ open/closed & Conditions on $\z(t)$ in order to
satisfy (\ref{condclosed2}) & Illustration \\
\hline $1$ & closed & $\z_1=\frac{-1-\sqrt{3} \ti}{2} \, \z_0,
\qquad \z_2=\z_0$ & Figure \ref{fig:closed_n1_zclosed}, first
column\\
\hline $2$ & closed & $ \z_2=-\frac{d_1 \z_0 + d_3
\z_1+\sqrt{r_-}}{2(d_1 + d_3)}, \qquad \z_3=\z_0, $ & Figure \ref{fig:closed_n1_zclosed}, second column\\
\hline $3$ & closed & $ \z_3=-\frac{d_1 \z_0 + d_4 \z_2 +
\sqrt{R_+}}{2(d_1 + d_4)}, \qquad \z_4=\z_0$ & Figure \ref{fig:closed_n1_zclosed}, third column\\
\hline $1$ & open &  $ \z_1=\frac{d_1-d_2+\sqrt{(d_1 + 3 d_2) (3 d_1
+ d_2)} \ti}{2(d_1 + d_2)} \, \z_0, \qquad \z_2=-\z_0$ & Figure \ref{fig:closed_n1_zopen}, first column\\
\hline $2$ & open &   $ \z_2=\frac{d_1 \z_0 - d_3 \z_1
+\sqrt{r_+}}{2(d_1 + d_3)}, \qquad \z_3=-\z_0,
$ & Figure \ref{fig:closed_n1_zopen}, second column\\
\hline $3$ & open &   $\z_3=\frac{d_1 \z_0 -d_4 \z_2
+\sqrt{R_-}}{2(d_1 + d_4)}, \qquad \z_4=-\z_0$ & Figure \ref{fig:closed_n1_zopen}, third column\\
\hline
\end{tabular}
 \caption{Data for the examples in the
closed case for $n=1$.} \label{table_datan1_closed}
\end{table}

\begin{figure}[h!]
\centering
\includegraphics[height=0.24\textwidth,valign=t]{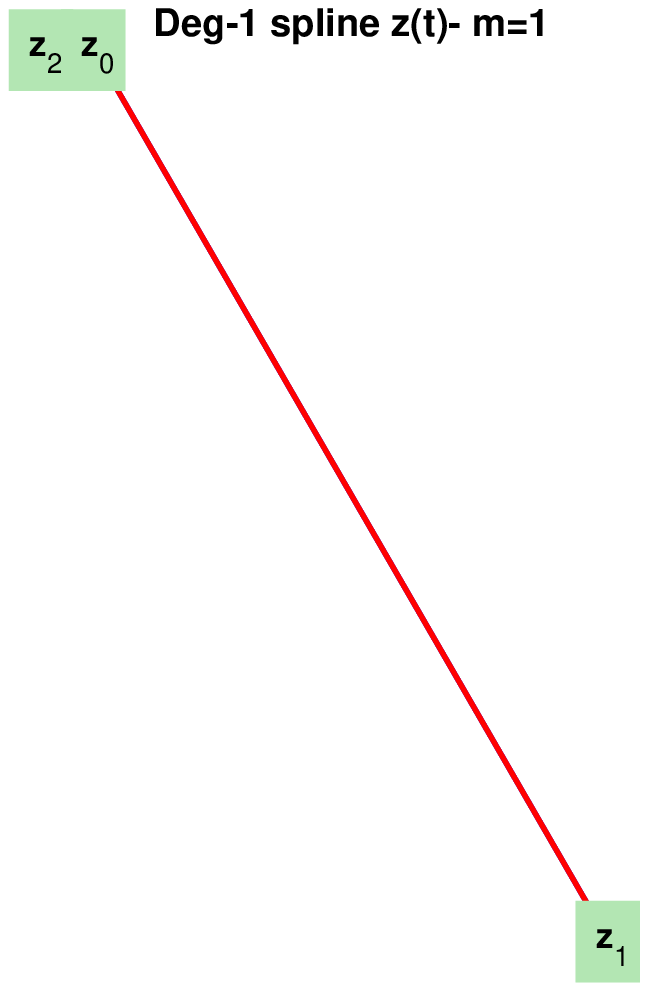}\hspace{-0.2cm}
\includegraphics[height=0.24\textwidth,valign=t]{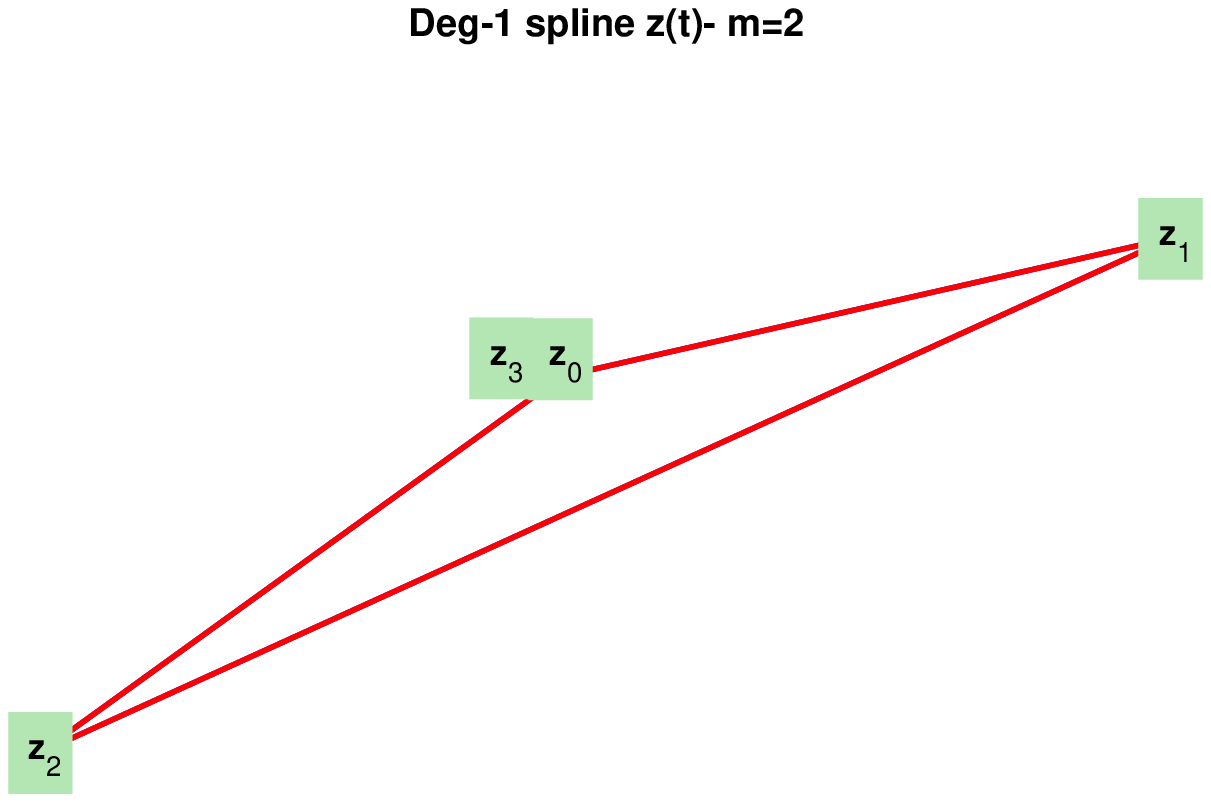}\hspace{-0.2cm}
\includegraphics[height=0.24\textwidth,valign=t]{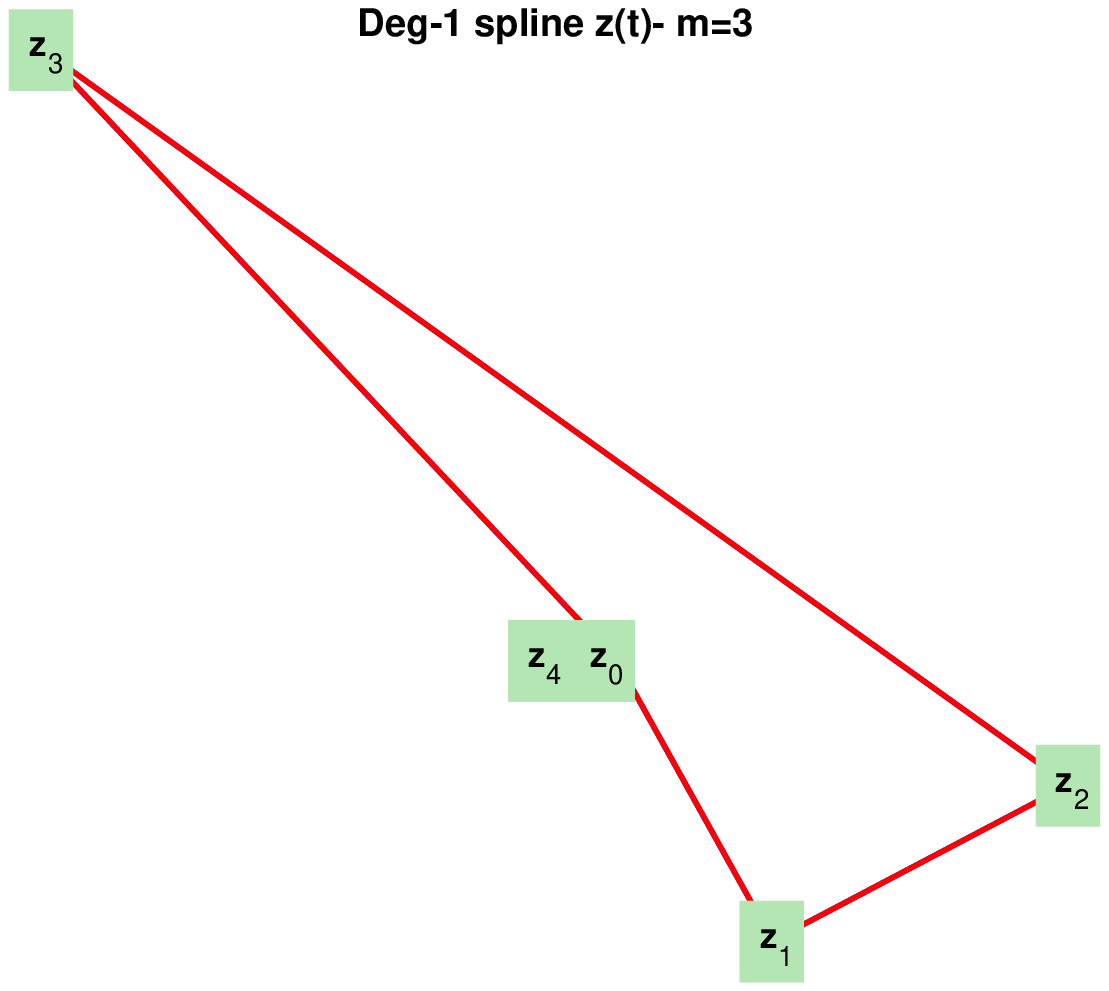}\\
\includegraphics[height=0.24\textwidth,valign=t]{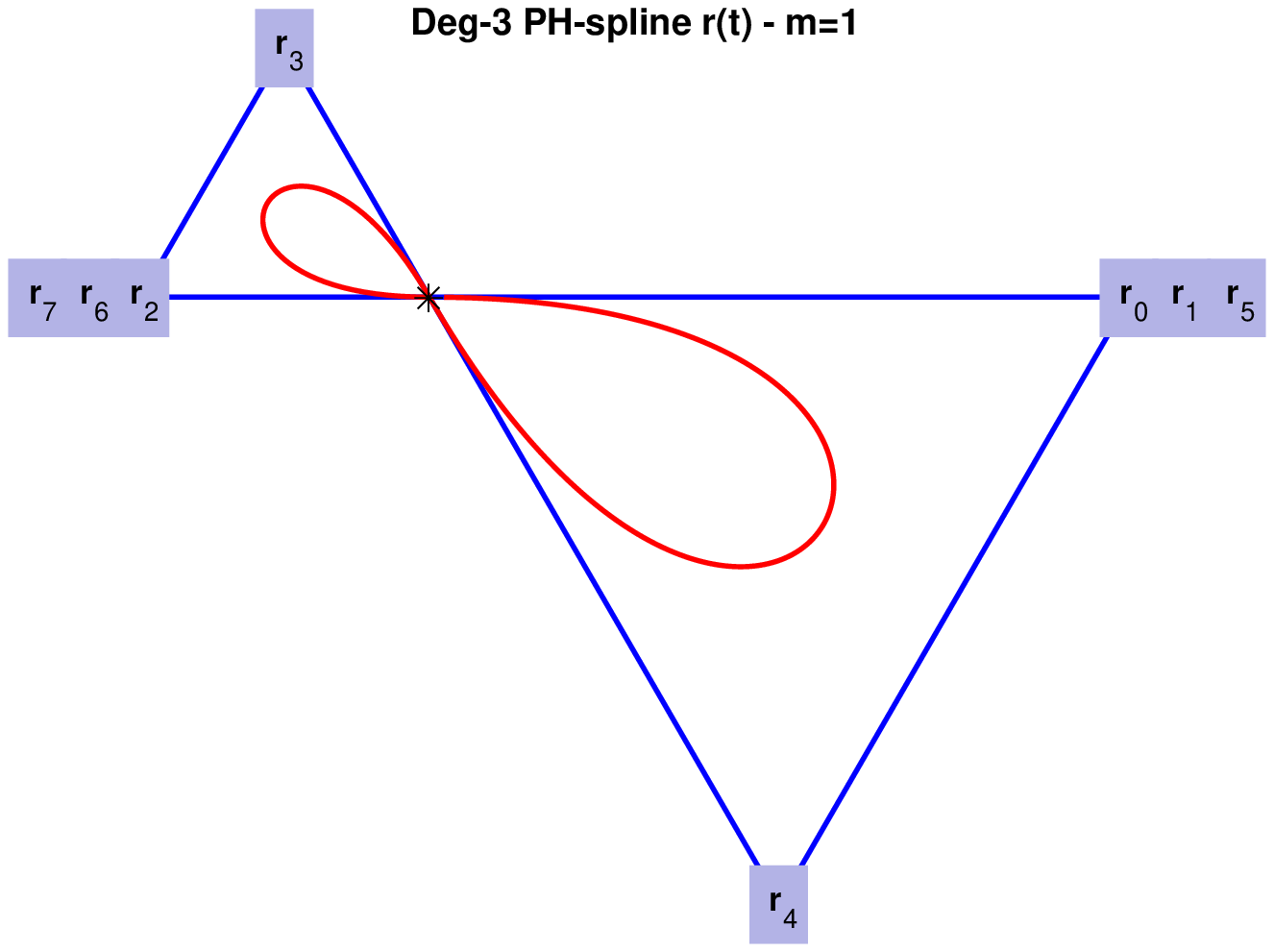}\hspace{-0.2cm}
\includegraphics[height=0.24\textwidth,valign=t]{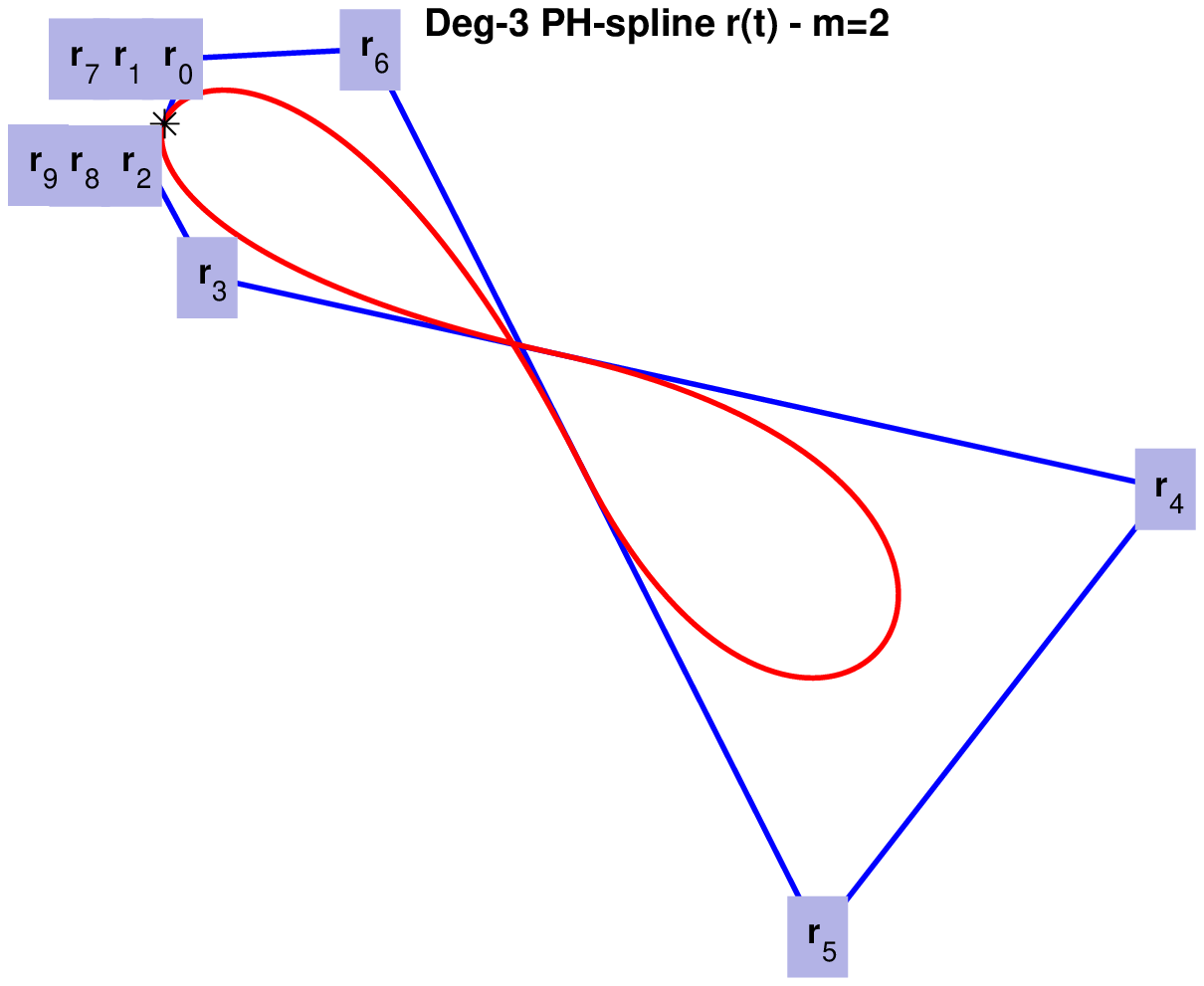}\hspace{-0.2cm}
\includegraphics[height=0.24\textwidth,valign=t]{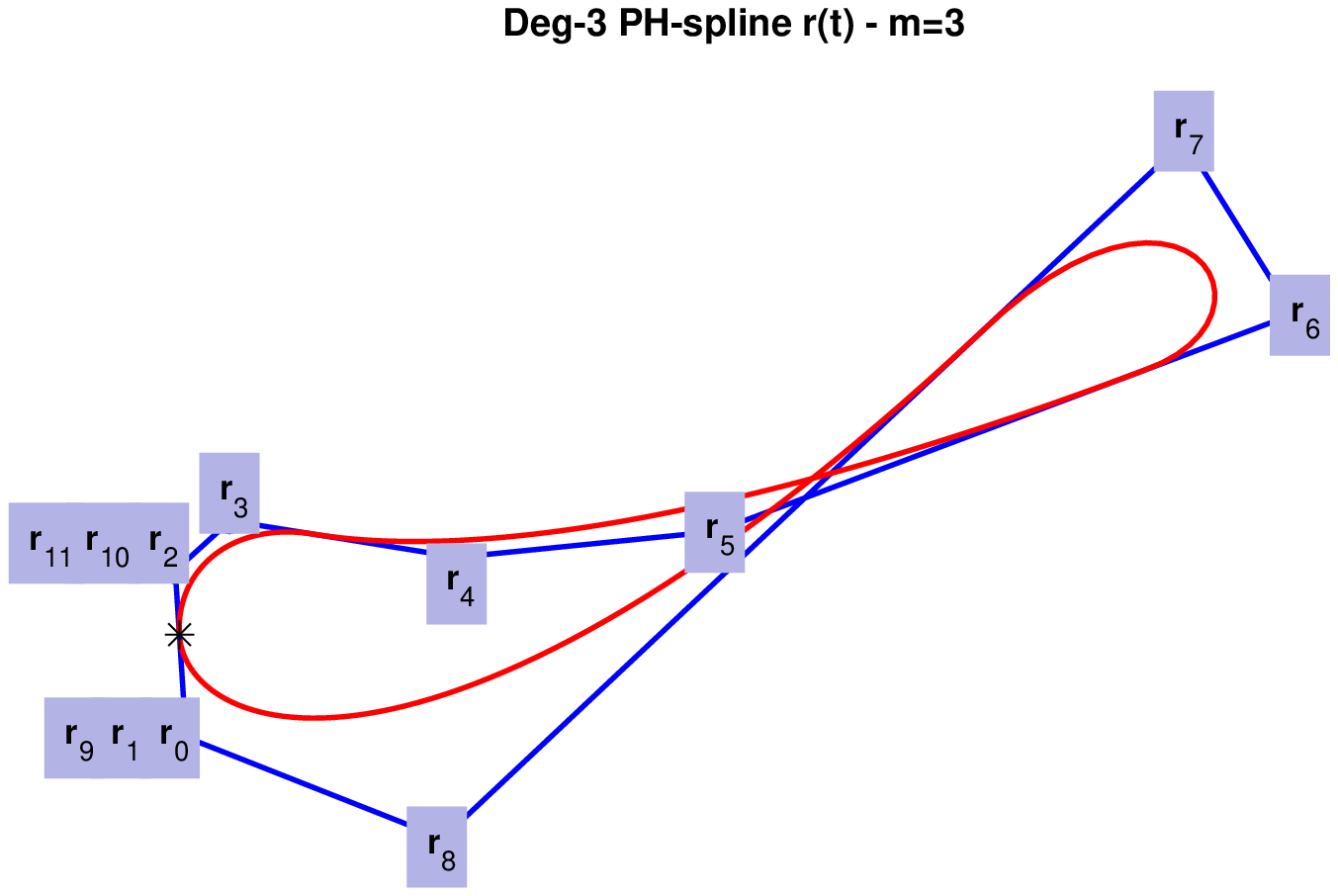}
\caption{Closed cubic PH B-Spline curves with $m=1$ (left), $m=2$ (center), $m=3$ (right) originated from closed degree-1 splines $\z(t)$.}
\label{fig:closed_n1_zclosed}
\end{figure}

\begin{figure}[h!]
\centering
\includegraphics[height=0.25\textwidth,valign=t]{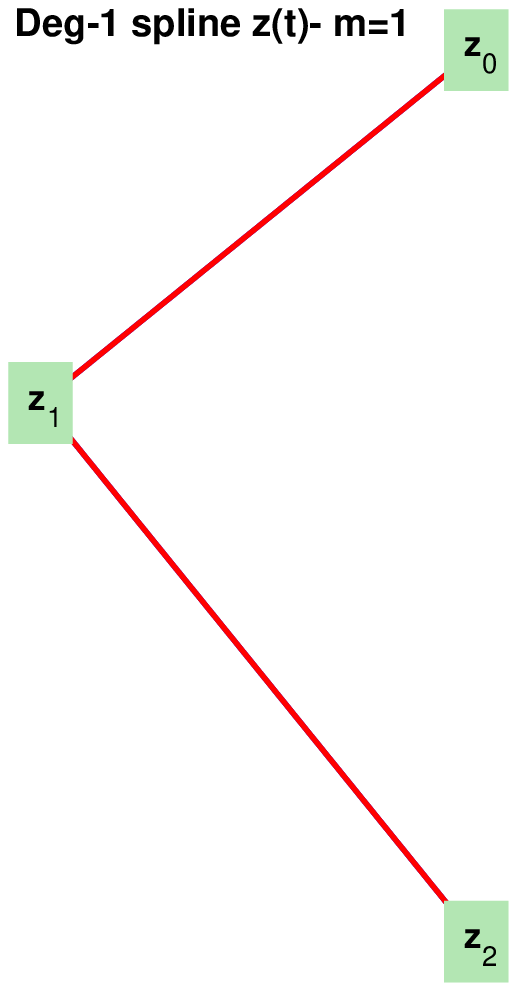}\hfill
\includegraphics[height=0.25\textwidth,valign=t]{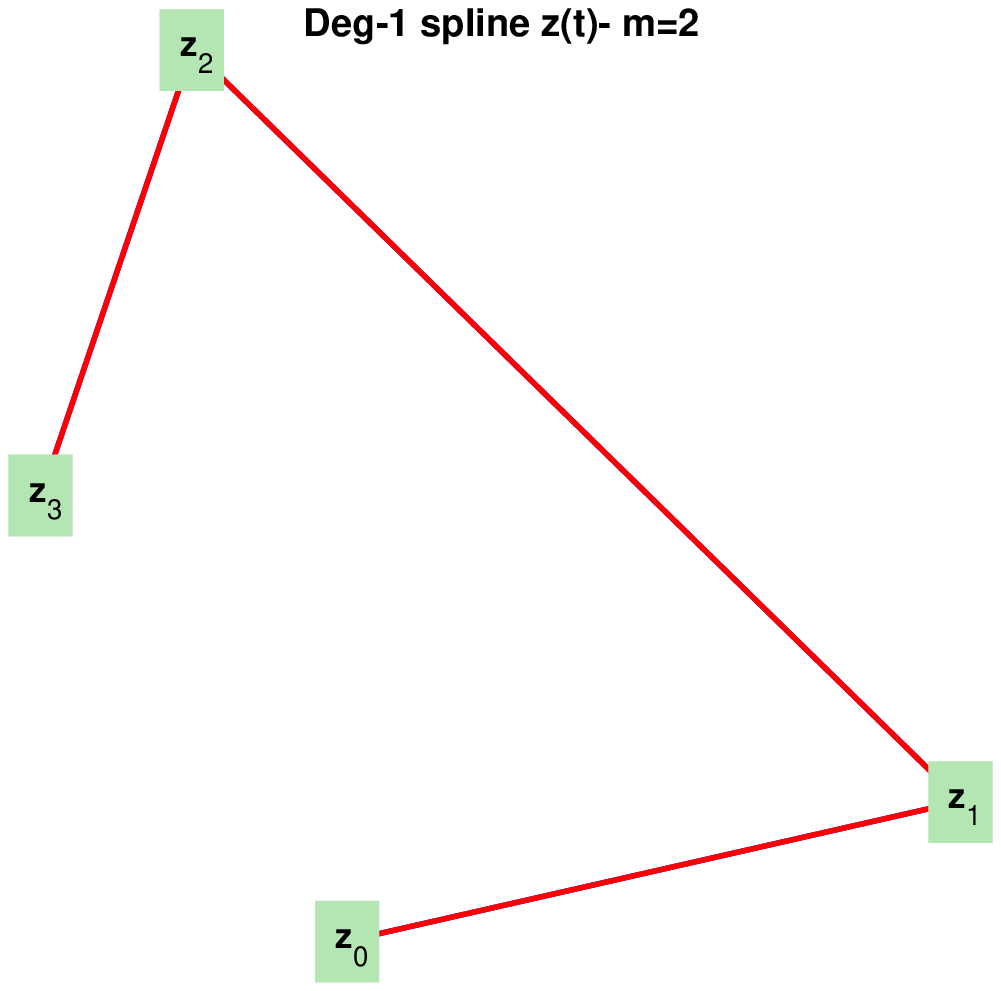}\hfill
\includegraphics[height=0.25\textwidth,valign=t]{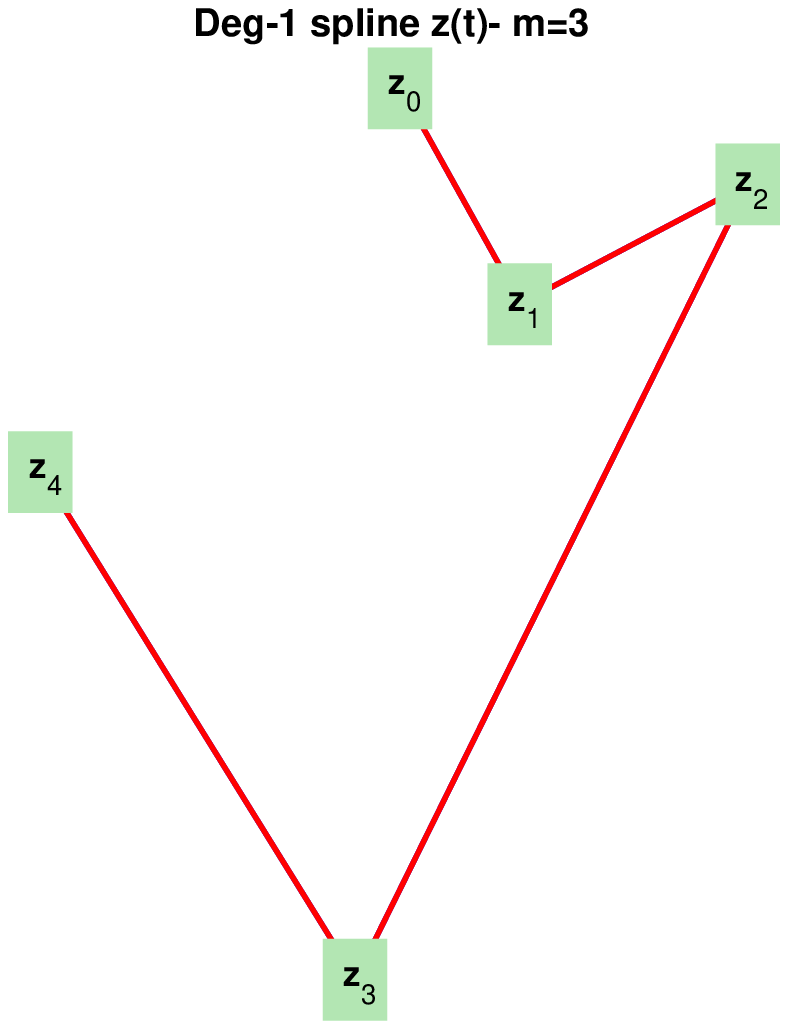}\\
\includegraphics[height=0.25\textwidth,valign=t]{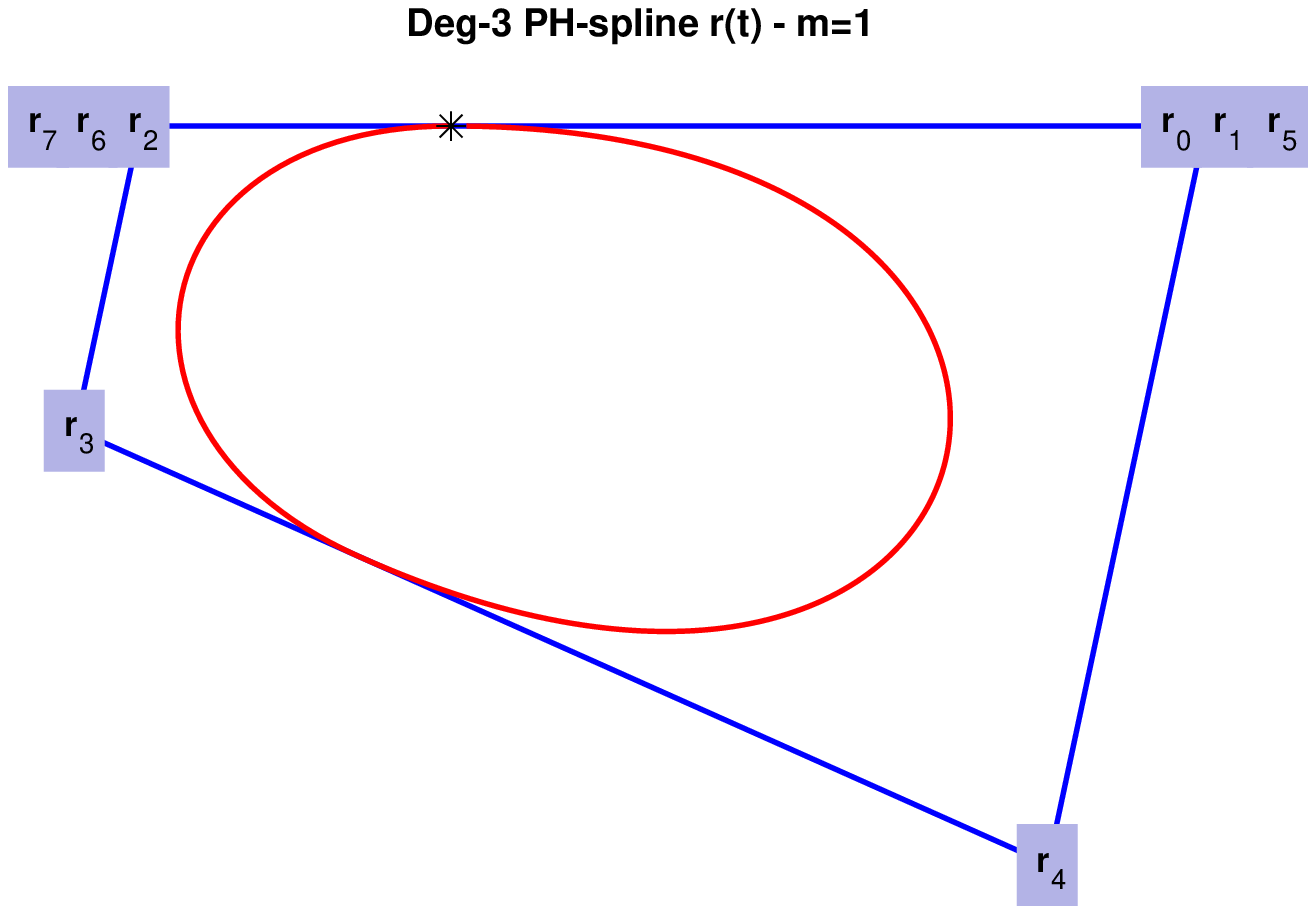}\hfill
\includegraphics[height=0.25\textwidth,valign=t]{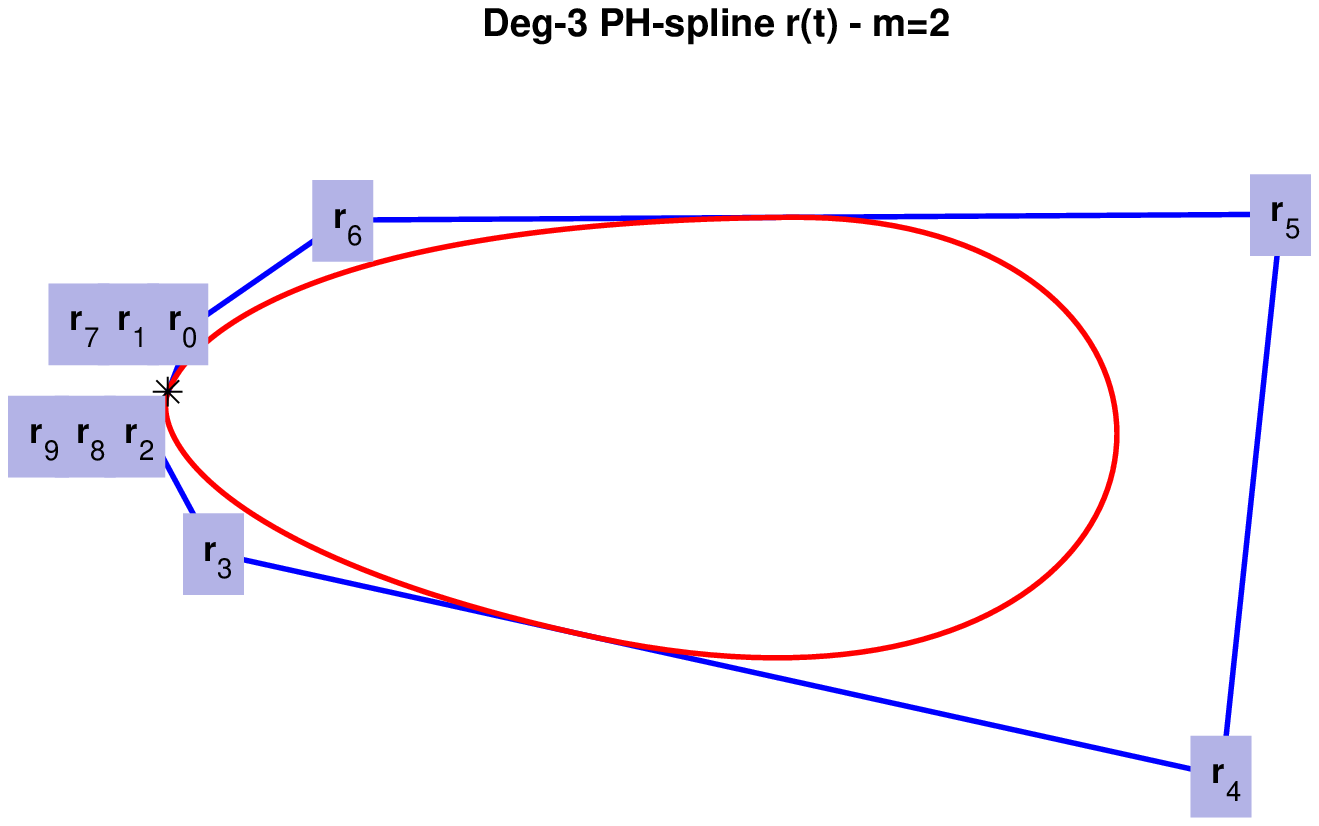}\hfill
\includegraphics[height=0.25\textwidth,valign=t]{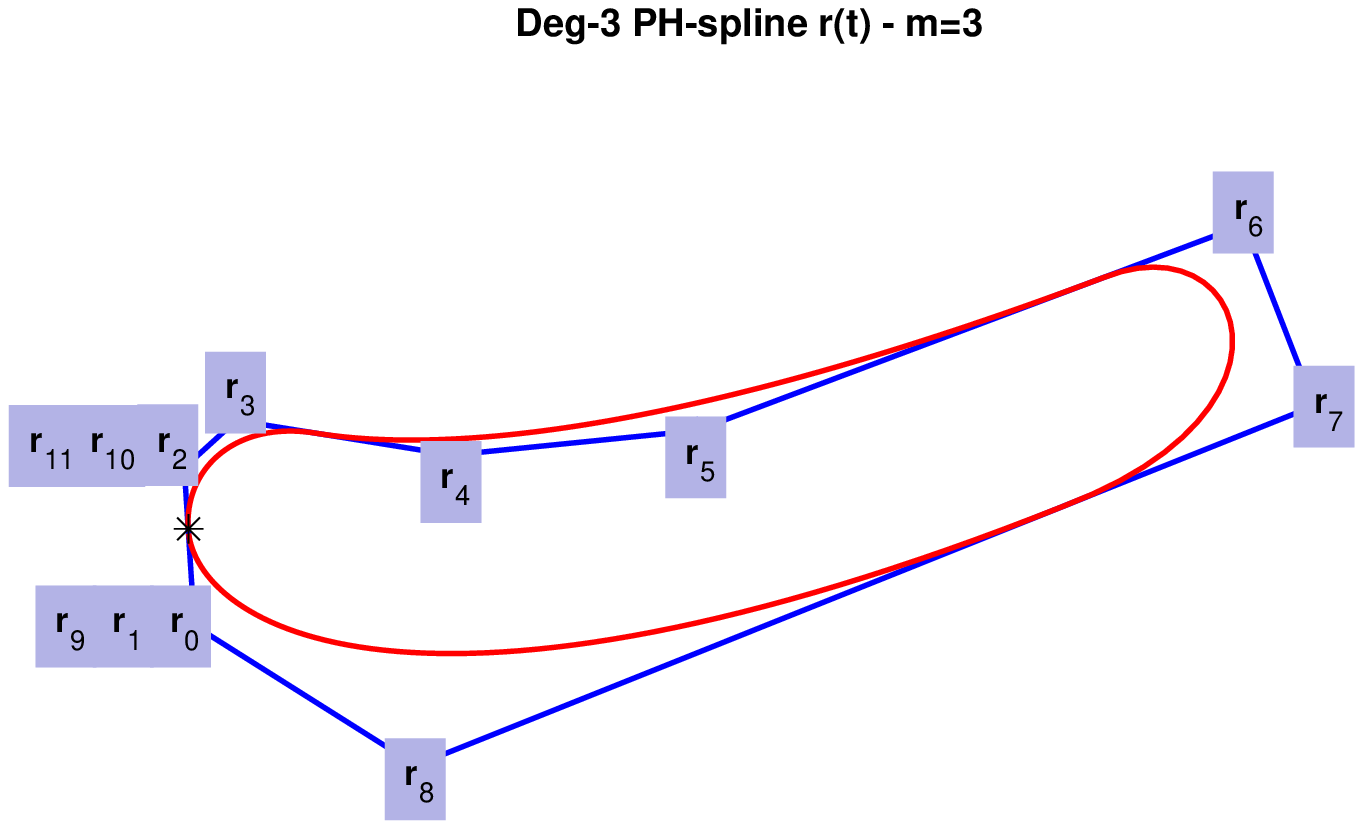}
\caption{Closed cubic PH B-Spline curves with $m=1$ (left), $m=2$ (center), $m=3$ (right) originated from open degree-1 splines $\z(t)$.}
\label{fig:closed_n1_zopen}
\end{figure}

\smallskip
\begin{figure}[h!]
\centering
\resizebox{4.95cm}{!}{\includegraphics{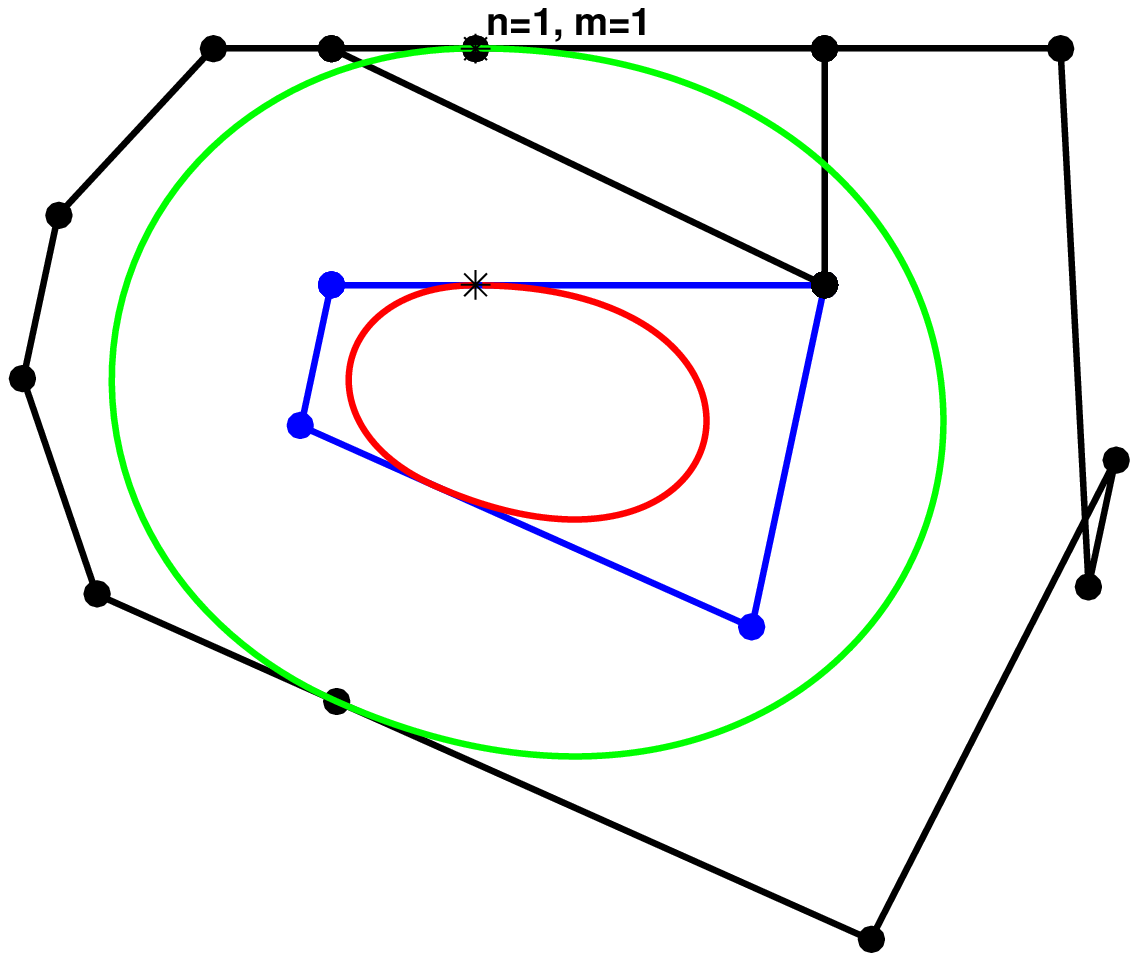}}
\resizebox{4.95cm}{!}{\includegraphics{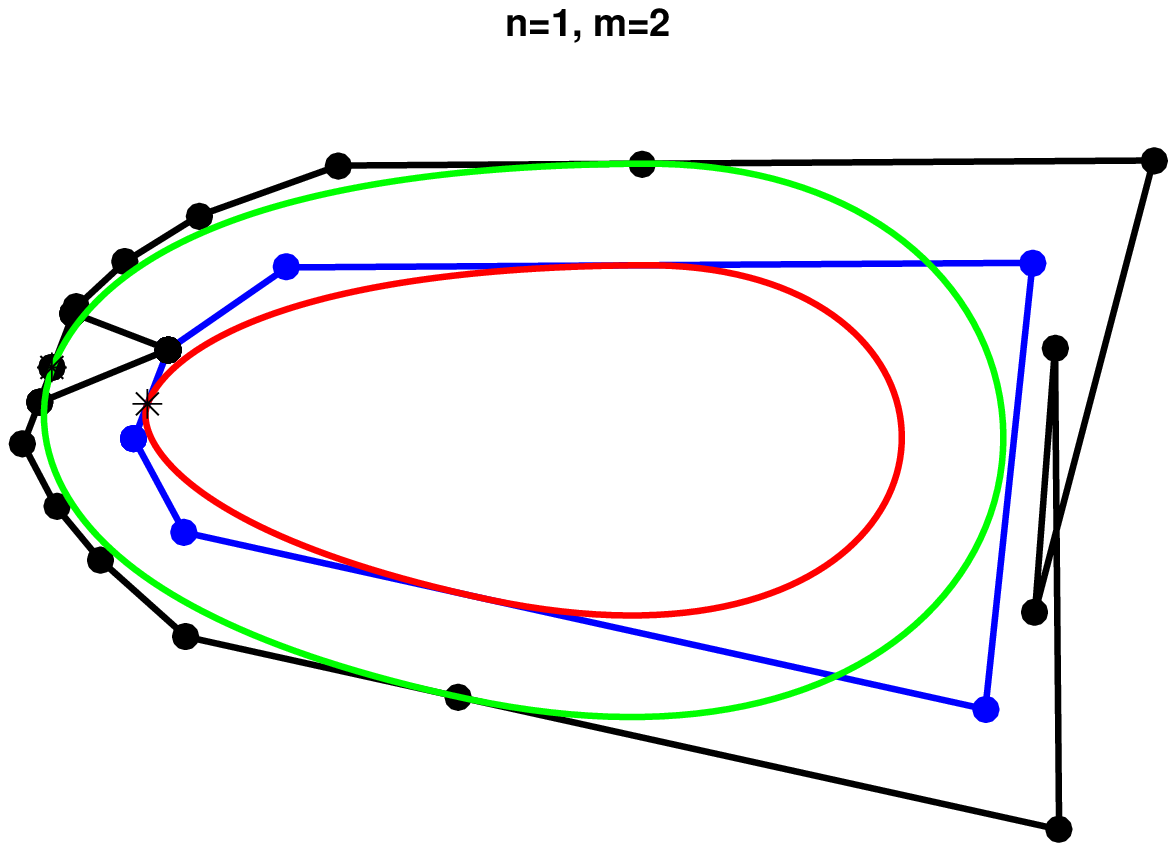}}
\resizebox{4.95cm}{!}{\includegraphics{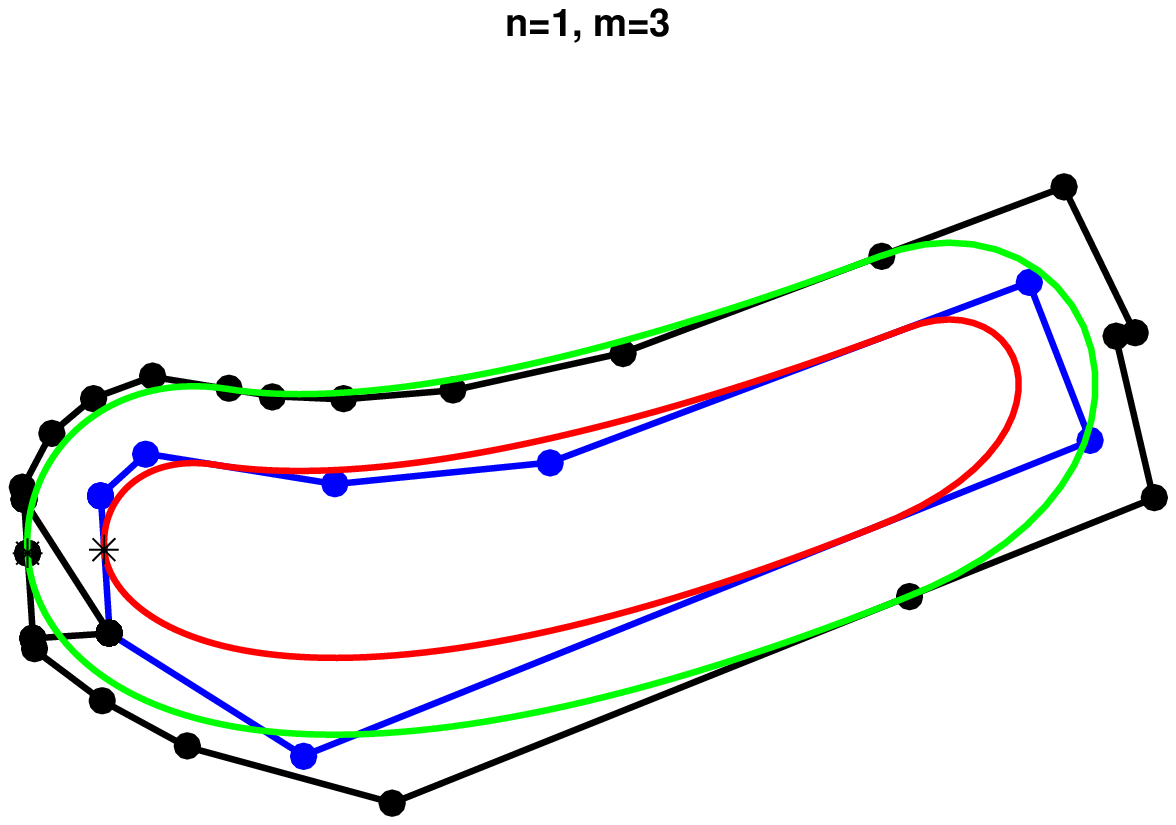}}\vspace{-0.1cm}
\resizebox{4.95cm}{!}{\includegraphics{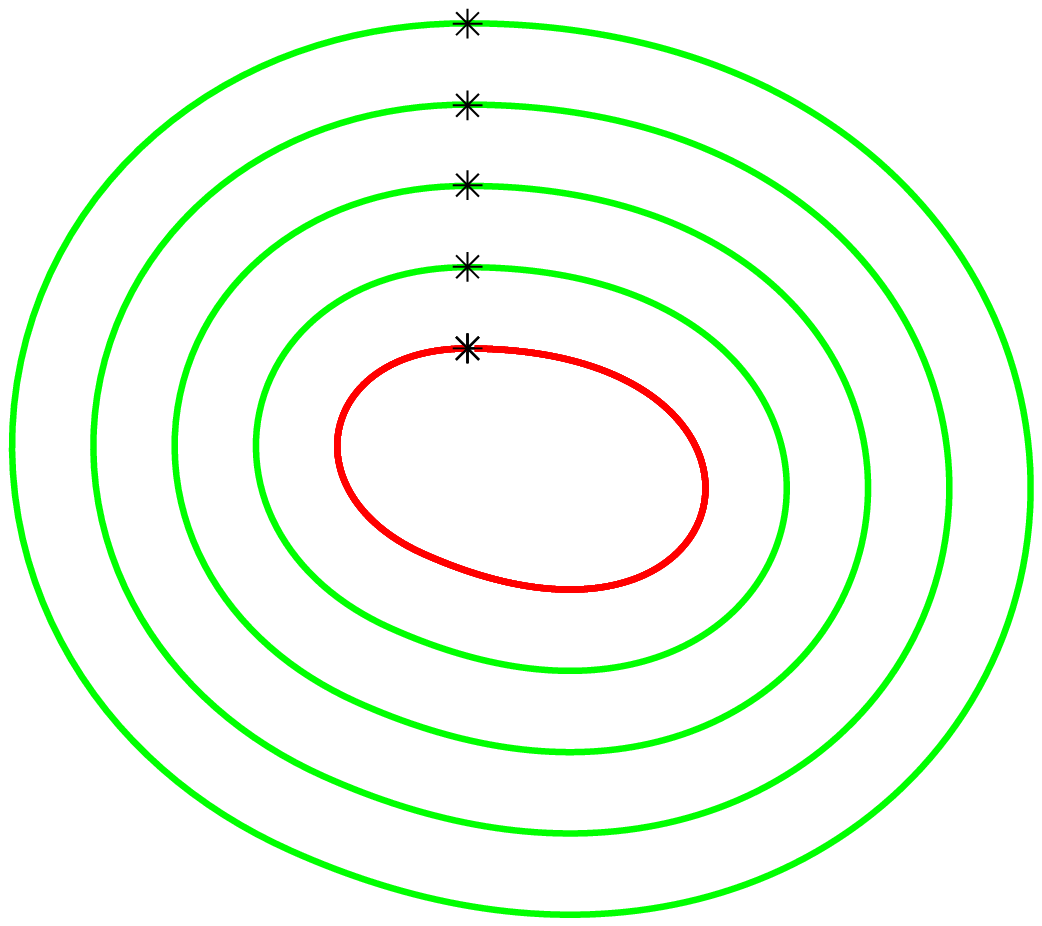}}
\resizebox{4.95cm}{!}{\includegraphics{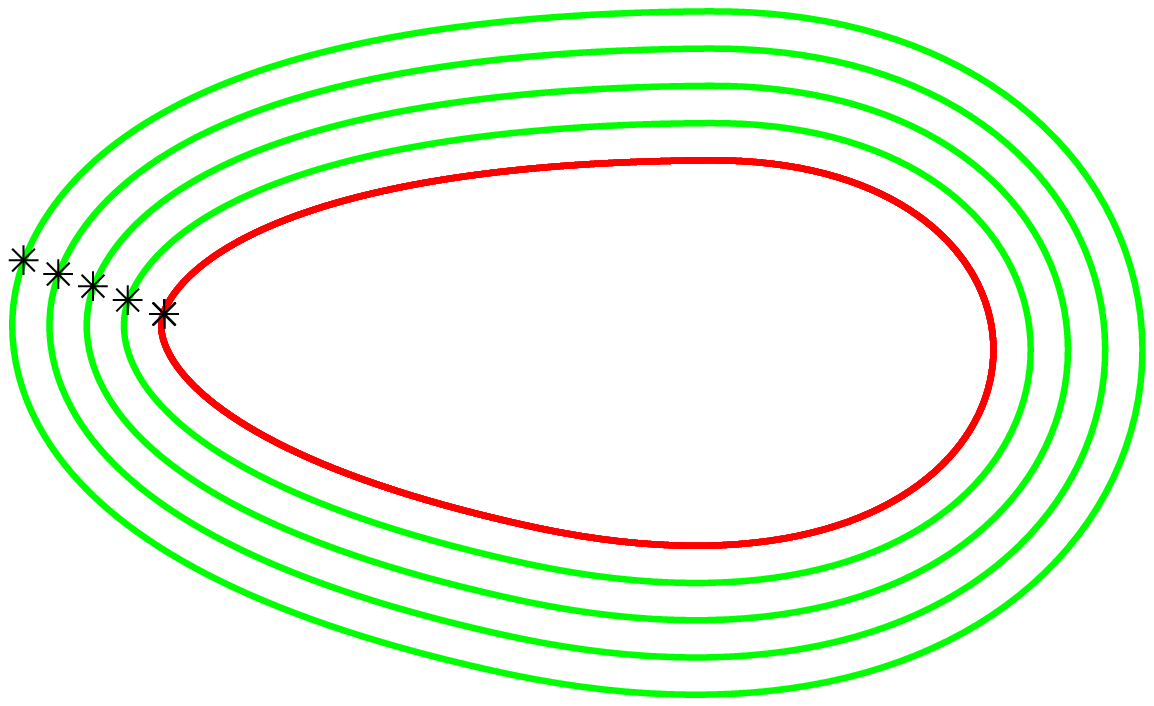}}
\resizebox{4.95cm}{!}{\includegraphics{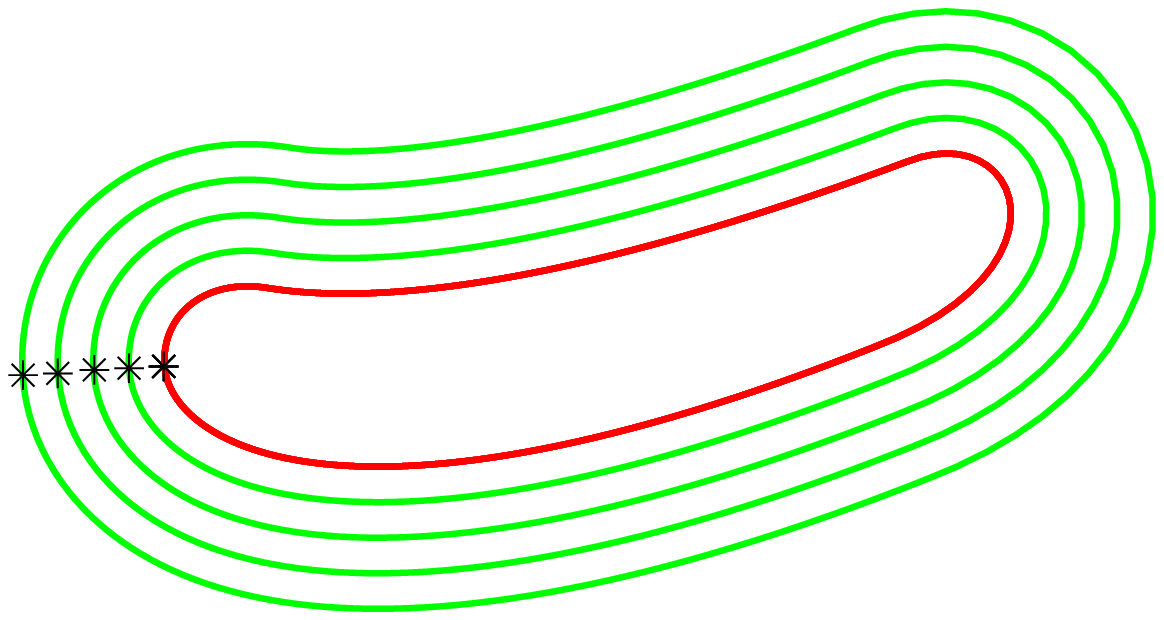}}
\caption{Offsets of closed cubic PH B-Spline curves with: $m=1$ (left), $m=2$ (center), $m=3$ (right).}
\label{fig:offsets_closed_n1}
\end{figure}

\subsection{Closed quintic PH B-Splines ($n=2$)}
\label{sec5n2closed}

Let $m \in \NN$, $m \geq 2$. For a general knot vector
$\bmu =\{ 0=t_0 < t_1 < ... < t_{m+5} \}$
(see Figure \ref{fig_knots_closed_n2} first row),
by applying the above method we construct the knot partitions
$$
\begin{array}{l}
\bnu=\{ \langle 0 \rangle^3 < \langle t_1 \rangle^3 < ... < \langle t_{m+5} \rangle^3 \}, \smallskip \\
\brho =\{ t_{-1} < \langle 0 \rangle^3 < \langle t_1 \rangle^3 < ... < \langle t_{m+5} \rangle^3 <  t_{m+6}  \}, \smallskip \\
\btau =\{ \langle t_{-1} \rangle^5 < \langle 0 \rangle^8 < \langle
t_1 \rangle^8 < ... < \langle t_{m+5} \rangle^8 < \langle t_{m+6}
\rangle^5 \},
\end{array}
$$
illustrated in Figure \ref{fig_knots_closed_n2}. Then, by solving the linear systems \eqref{lgschi} we
calculate the coefficients
$\chi_{k}^{i,j}$, $0\le i,j \le m+2$, $0 \le {k} \le 3m+12$. All of them turn out to be zero with the exception of
$$
\begin{array}{l}
\chi_2^{0,0}= \frac{d_1}{d_1+d_2}, \smallskip\\
\chi_{3k}^{k-2,k-1}= \chi_{3k}^{k-1,k-2} = \frac{1}{6} \, \frac{d_{k+1} \, d_{k+2}}{(d_{k}+d_{k+1})\, (d_{k+1} + d_{k+2})}, \quad k=1,...,m+3\,,  \smallskip\\
\chi_{3k}^{k-2,k}= \chi_{3k}^{k,k-2} = \frac{1}{6} \, \frac{(d_{k+1})^2}{(d_{k}+d_{k+1})(d_{k+1}+d_{k+2})}, \quad k=1,...,m+3\,,  \smallskip\\
\chi_{3k}^{k-1,k-1}=\frac{2}{3} + \frac{1}{3} \, \frac{d_{k} \, d_{k+2}}{(d_{k}+d_{k+1}) \, (d_{k+1} + d_{k+2})}, \quad k=1,...,m+3\,,  \smallskip\\
\chi_{3k}^{k-1,k}= \chi_{3k}^{k,k-1} = \frac{1}{6} \, \frac{d_k \, d_{k+1}}{(d_{k}+d_{k+1}) \, (d_{k+1} + d_{k+2})}, \quad k=1,...,m+3\,,
\end{array}
$$
$$
\begin{array}{l}
\chi_{3k+1}^{k-1,k-1}= \frac{d_{k+2}}{d_{k+1} + d_{k+2}},   \quad k=1,...,m+2\,,  \smallskip\\
\chi_{3k+1}^{k-1,k}= \chi_{3k+1}^{k,k-1} = \frac{1}{2}\,\frac{d_{k+1}}{d_{k+1} + d_{k+2}},   \quad k=1,...,m+2\,,  \smallskip\\
\chi_{3k+2}^{k-1,k}=\chi_{3k+2}^{k,k-1} = \frac{1}{2}\, \frac{d_{k+2}}{d_{k+1}+d_{k+2}},   \quad k=1,...,m+2\,,  \smallskip\\
\chi_{3k+2}^{k,k}= \frac{d_{k+1}}{d_{k+1} +d_{k+2}},   \quad k=1,...,m+2\,,  \smallskip\\
\chi_{3m+10}^{m+2,m+2}= \frac{d_{m+5}}{d_{m+4}+d_{m+5}},
\end{array}
$$
\smallskip
where $d_{k}=t_{k}-t_{k-1}$, $k=1, ..., m+5$.

\begin{figure}[h!]
\centering
\hspace{-0.4cm}\resizebox{9.3cm}{!}{\includegraphics{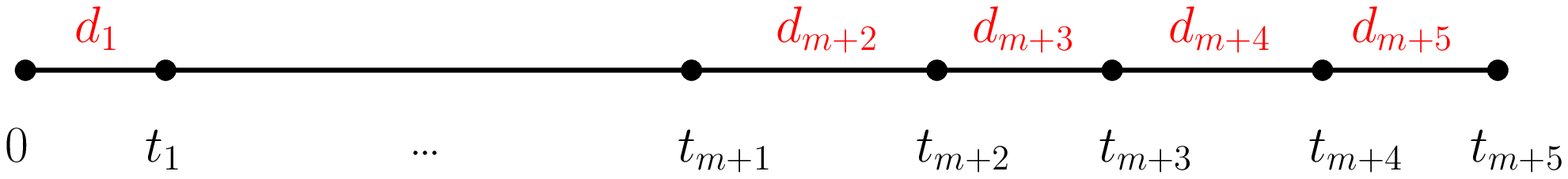}}\\
\hspace{-0.4cm}\resizebox{9.3cm}{!}{\includegraphics{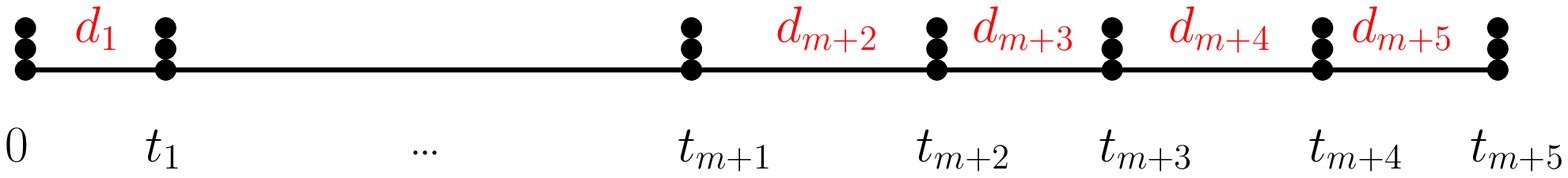}}\\
\resizebox{10.9cm}{!}{\includegraphics{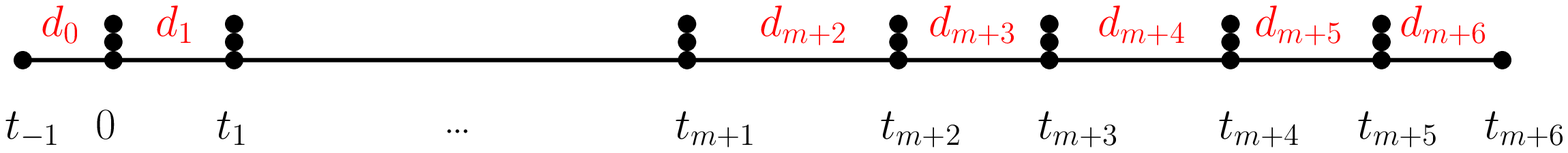}}\\
\resizebox{10.9cm}{!}{\includegraphics{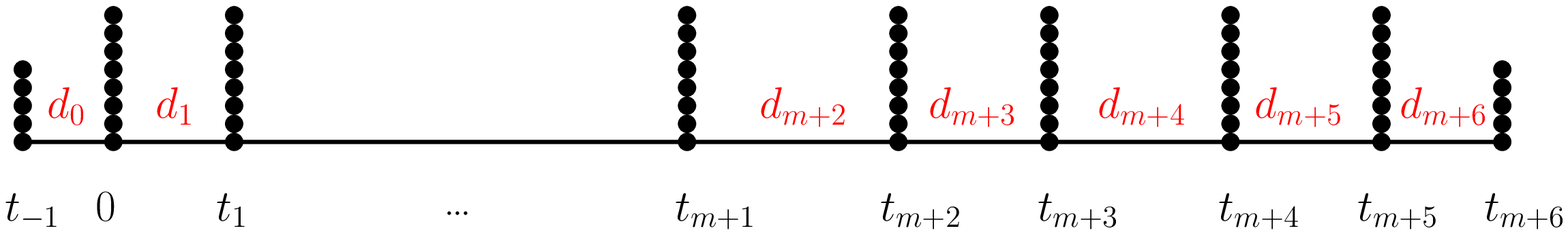}}
\caption{Knot partitions for the closed case $n=2$. From top to bottom: $\bmu$, $\bnu$, $\brho$, $\btau$.}
\label{fig_knots_closed_n2}
\end{figure}

In addition, we compute the coefficients $\zeta_k^{i,j}$, $0 \leq i
\leq 3m+13$, $0 \leq j \leq 3m+12$, $0 \leq k \leq 8m+47$ as the
solutions to the linear systems \eqref{zetas}. All of them
turn out to be zero with the exception of
\begin{equation} \label{zetasn2closed}
\begin{array}{l}
 \zeta_{5}^{0,0} = \frac{d_{0}^2}{6 D_{0}^2}\, , \;
\zeta_{6}^{0,0} = \frac{5 d_{0}}{14 D_{0}}\, , \; \zeta_{6}^{0,1} =
\frac{d_{0}^2}{21 D_{0}^2}\, , \;
\smallskip \\
 \zeta_{7}^{0,0} = \frac{10}{21}\, , \; \zeta_{7}^{0,1}
= \frac{d_{0}^2 d_{2}}{126 D_{0}^2 D_{1}} + \frac{10 d_{0}}{63
D_{0}}\, ,\; \zeta_{7}^{0,2} = \frac{d_{0}^2 d_{1}}{126 D_{0}^2
D_{1}}\, ,\;
\smallskip \\
 \zeta_{8}^{0,1} = \frac{5 d_{0} d_{2}}{126 D_{0}
D_{1}} + \frac{20}{63}\, ,\; \zeta_{8}^{0,2} = \frac{5 d_{0}
d_{1}}{126 D_{0} D_{1}}\, ,\; \zeta_{8}^{1,0} = \frac{10}{21}\, , \;
\smallskip \\
 \zeta_{9}^{0,1} = \frac{5 d_{2}}{42 D_{1}}\, , \;
\zeta_{9}^{0,2} = \frac{5 d_{1}}{42 D_{1}}\, , \; \zeta_{9}^{1,0} =
\frac{5 d_{2}}{14 D_{1}}\, , \; \zeta_{9}^{1,1} = \frac{10}{21}\, ,
\; \zeta_{9}^{2,0} = \frac{5 d_{1}}{14 D_{1}}\, , \;
\smallskip \\
 \zeta_{10}^{1,0} = \frac{d_{2}^2}{6 D_{1}^2}\, , \;
\zeta_{10}^{1,1} = \frac{15 d_{2}}{18 D_{1}}\, , \; \zeta_{10}^{1,2}
= \frac{5 d_{1}}{18 D_{1}}\, , \;
\smallskip \\
 \zeta_{10}^{2,0} = \frac{d_{1}d_{2}}{3 D_{1}^2}\, , \;
\zeta_{10}^{2,1} = \frac{5 d_{1}}{9 D_{1}}\, , \; \zeta_{10}^{3,0} =
\frac{d_{1}^2}{6 D_{1}^2}\, , \;
\smallskip \\
\zeta_{8k+11}^{3k+1,3k+1} = \frac{d_{k+2}^2}{D_{k+1}^2}\, , \;
\zeta_{8k+11}^{3k+1,3k+2} = \frac{5 d_{k+1}d_{k+2}}{9 D_{k+1}^2}\,
,\; \zeta_{8k+11}^{3k+2,3k+1} = \frac{13 d_{k+1}d_{k+2}}{9
D_{k+1}^2}\, , \; \zeta_{8k+11}^{3k+2,3k+2} = \frac{5 d_{k+1}^2}{9
D_{k+1}^2}\, , \; \zeta_{8k+11}^{3k+3,3k+1} = \frac{4 d_{k+1}^2}{9
D_{k+1}^2}\, , \;
k=0,\ldots,m+3\, , \smallskip \\
\zeta_{8k+12}^{3k+1,3k+2} = \frac{4 d_{k+2}^2}{9 D_{k+1}^2}\, , \;
\zeta_{8k+12}^{3k+2,3k+1} = \frac{5 d_{k+2}^2}{9 D_{k+1}^2}\, , \;
\zeta_{8k+12}^{3k+2,3k+2} = \frac{13 d_{k+1}d_{k+2}}{9 D_{k+1}^2}\,
, \; \zeta_{8k+12}^{3k+3,3k+1} = \frac{5 d_{k+1}d_{k+2}}{9
D_{k+1}^2}\, , \; \zeta_{8k+12}^{3k+3,3k+2} =
\frac{d_{k+1}^2}{D_{k+1}^2}\, , \;
k=0,\ldots,m+3\, , \smallskip \\
\zeta_{8k+13}^{3k+1,3k+3} = \frac{d_{k+2}^2}{6 D_{k+1}^2}\, , \;
\zeta_{8k+13}^{3k+2,3k+2} = \frac{5 d_{k+2}}{9 D_{k+1}}\, , \;
\zeta_{8k+13}^{3k+2,3k+3} = \frac{d_{k+1}d_{k+2}}{3 D_{k+1}^2}\, ,\;
\zeta_{8k+13}^{3k+3,3k+1} = \frac{5 d_{k+2}}{18 D_{k+1}}\, , \;
\smallskip \\
\zeta_{8k+13}^{3k+3,3k+2} = \frac{15 d_{k+1}}{18 D_{k+1}}\, , \;
\zeta_{8k+13}^{3k+3,3k+3} = \frac{d_{k+1}^2}{6 D_{k+1}^2}\, , \;
k=0,\ldots,m+2\, , \smallskip \\
\zeta_{8k+14}^{3k+1,3k+4} = \frac{d_{k+2}^2}{21 D_{k+1}^2}\, , \;
\zeta_{8k+14}^{3k+2,3k+3} = \frac{5 d_{k+2}}{14 D_{k+1}}\, , \;
\zeta_{8k+14}^{3k+2,3k+4} = \frac{2 d_{k+1}d_{k+2}}{21 D_{k+1}^2}\,
, \; \zeta_{8k+14}^{3k+3,3k+2} = \frac{10}{21}\, , \;
\zeta_{8k+14}^{3k+3,3k+3} = \frac{5 d_{k+1}}{14 D_{k+1}}\, , \;
\smallskip \\
\zeta_{8k+14}^{3k+3,3k+4} = \frac{d_{k+1}^2}{21 D_{k+1}^2}\, , \;
\zeta_{8k+14}^{3k+4,3k+1} = \frac{5 d_{k+2}}{42 D_{k+1}}\, , \;
\zeta_{8k+14}^{3k+4,3k+2} = \frac{5 d_{k+1}}{42 D_{k+1}}\, , \;
k=0,\ldots,m+2\, , \smallskip  \\
\zeta_{8k+15}^{3k+1,3k+4} = \frac{d_{k+2}^2 d_{k+3}}{126 D_{k+1}^2
D_{k+2}}\, ,\; \zeta_{8k+15}^{3k+1,3k+5} = \frac{d_{k+2}^3}{126
D_{k+1}^2 D_{k+2}}\, ,\; \zeta_{8k+15}^{3k+2,3k+4} = \frac{d_{k+1}
d_{k+2} d_{k+3}}{63 D_{k+1}^2 D_{k+2}} + \frac{10 d_{k+2}}{63
D_{k+1}}\, ,\; \zeta_{8k+15}^{3k+2,3k+5} = \frac{d_{k+1}
d_{k+2}^2}{63 D_{k+1}^2 D_{k+2}}\, ,\;
\smallskip \\
\zeta_{8k+15}^{3k+3,3k+3} = \frac{10}{21}\, , \;
\zeta_{8k+15}^{3k+3,3k+4} = \frac{d_{k+1}^2 d_{k+3}}{126 D_{k+1}^2
D_{k+2}} + \frac{10 d_{k+1}}{63 D_{k+1}}\, ,\;
\zeta_{8k+15}^{3k+3,3k+5} = \frac{d_{k+1}^2 d_{k+2}}{126 D_{k+1}^2
D_{k+2}}\, ,\; \zeta_{8k+15}^{3k+4,3k+1} = \frac{5 d_{k+2}
d_{k+3}}{126 D_{k+1} D_{k+2}}\, ,\;
\smallskip \\
\zeta_{8k+15}^{3k+4,3k+2} = \frac{5 d_{k+1} d_{k+3}}{126 D_{k+1}
D_{k+2}} + \frac{20}{63}\, ,\; \zeta_{8k+15}^{3k+5,3k+1} = \frac{5
d_{k+2}^2}{126 D_{k+1} D_{k+2}}\, ,\; \zeta_{8k+15}^{3k+5,3k+2} =
\frac{5 d_{k+1} d_{k+2}}{126 D_{k+1} D_{k+2}}\, ,\;
k=0,\ldots,m+2\, , \smallskip \\
\end{array}
\end{equation}

\begin{equation*}
\begin{array}{l}
\zeta_{8k+16}^{3k+2,3k+4} = \frac{5 d_{k+2} d_{k+3}}{126 D_{k+1}
D_{k+2}}\, ,\; \zeta_{8k+16}^{3k+2,3k+5} = \frac{5 d_{k+2}^2}{126
D_{k+1} D_{k+2}}\, ,\; \zeta_{8k+16}^{3k+3,3k+4} = \frac{5 d_{k+1}
d_{k+3}}{126 D_{k+1} D_{k+2}} + \frac{20}{63}\, ,\;
\zeta_{8k+16}^{3k+3,3k+5} = \frac{5 d_{k+1} d_{k+2}}{126 D_{k+1}
D_{k+2}}\, ,\;
\smallskip \\
\zeta_{8k+16}^{3k+4,3k+1} = \frac{d_{k+2} d_{k+3}^2}{126 D_{k+1}
D_{k+2}^2} \, ,\; \zeta_{8k+16}^{3k+4,3k+2} = \frac{d_{k+1}
d_{k+3}^2}{126 D_{k+1} D_{k+2}^2} + \frac{10 d_{k+3}}{63 D_{k+2}}\,
,\; \zeta_{8k+16}^{3k+4,3k+3} = \frac{10}{21}\, , \;
\smallskip \\
\zeta_{8k+16}^{3k+5,3k+1} = \frac{d_{k+2}^2 d_{k+3}}{63 D_{k+1}
D_{k+2}^2} \, ,\; \zeta_{8k+16}^{3k+5,3k+2} = \frac{d_{k+1} d_{k+2}
d_{k+3}}{63 D_{k+1} D_{k+2}^2} + \frac{10 d_{k+2}}{63 D_{k+2}}\, ,\;
\zeta_{8k+16}^{3k+6,3k+1} = \frac{d_{k+2}^3}{126 D_{k+1} D_{k+2}^2}
\, ,\; \zeta_{8k+16}^{3k+6,3k+2} = \frac{d_{k+1} d_{k+2}^2}{126
D_{k+1} D_{k+2}^2} \, ,\;
k=0,\ldots,m+2\, , \smallskip \\
\zeta_{8k+17}^{3k+3,3k+4} = \frac{5 d_{k+3}}{42 D_{k+2}}\, , \;
\zeta_{8k+17}^{3k+3,3k+5} = \frac{5 d_{k+2}}{42 D_{k+2}}\, , \;
\zeta_{8k+17}^{3k+4,3k+2} = \frac{d_{k+3}^2}{21 D_{k+2}^2}\, , \;
\zeta_{8k+17}^{3k+4,3k+3} = \frac{5 d_{k+3}}{14 D_{k+2}}\, , \;
\zeta_{8k+17}^{3k+4,3k+4} = \frac{10}{21}\, , \;
\smallskip \\
\zeta_{8k+17}^{3k+5,3k+2} = \frac{2 d_{k+2}d_{k+3}}{21 D_{k+2}^2}\,
, \; \zeta_{8k+17}^{3k+5,3k+3} = \frac{5 d_{k+2}}{14 D_{k+2}}\, , \;
\zeta_{8k+17}^{3k+6,3k+2} = \frac{d_{k+2}^2}{21 D_{k+2}^2}\, , \;
k=0,\ldots,m+2\, , \smallskip \\
\zeta_{8k+18}^{3k+4,3k+3} = \frac{d_{k+3}^2}{6 D_{k+2}^2}\, , \;
\zeta_{8k+18}^{3k+4,3k+4} = \frac{15 d_{k+3}}{18 D_{k+2}}\, , \;
\zeta_{8k+18}^{3k+4,3k+5} = \frac{5 d_{k+2}}{18 D_{k+2}}\, , \;
\zeta_{8k+18}^{3k+5,3k+3} = \frac{d_{k+2}d_{k+3}}{3 D_{k+2}^2}\, ,
\;
\smallskip \\
\zeta_{8k+18}^{3k+5,3k+4} = \frac{5 d_{k+2}}{9 D_{k+2}}\, , \;
\zeta_{8k+18}^{3k+6,3k+3} = \frac{d_{k+2}^2}{6 D_{k+2}^2}\, , \;
k=0,\ldots,m+2\, , \smallskip \\
 \zeta_{8m+37}^{3m+10,3m+12} = \frac{d_{m+5}^2}{6
D_{m+4}^2}\, , \; \zeta_{8m+37}^{3m+11,3m+11} = \frac{5 d_{m+5}}{9
D_{m+4}}\, , \; \zeta_{8m+37}^{3m+11,3m+12} =
\frac{d_{m+4}d_{m+5}}{3 D_{m+4}^2}\, ,\;
\smallskip \\
 \zeta_{8m+37}^{3m+12,3m+10} = \frac{5 d_{m+5}}{18
D_{m+4}}\, , \; \zeta_{8m+37}^{3m+12,3m+11} = \frac{15 d_{m+4}}{18
D_{m+4}}\, , \; \zeta_{8m+37}^{3m+12,3m+12} = \frac{d_{m+4}^2}{6
D_{m+4}^2}\, , \;
\smallskip \\
 \zeta_{8m+38}^{3m+11,3m+12} = \frac{5 d_{m+5}}{14
D_{m+4}}\, , \; \zeta_{8m+38}^{3m+12,3m+11} = \frac{10}{21}\, , \;
\zeta_{8m+38}^{3m+12,3m+12} = \frac{5 d_{m+4}}{14 D_{m+4}}\, , \;
\zeta_{8m+38}^{3m+13,3m+10} = \frac{5 d_{m+5}}{42 D_{m+4}}\, , \;
\zeta_{8m+38}^{3m+13,3m+11} = \frac{5 d_{m+4}}{42 D_{m+4}}\, , \;
\smallskip \\
 \zeta_{8m+39}^{3m+12,3m+12} = \frac{10}{21}\, , \;
\zeta_{8m+39}^{3m+13,3m+10} = \frac{5 d_{m+5} d_{m+6}}{126 D_{m+4}
D_{m+5}}\, ,\; \zeta_{8m+39}^{3m+13,3m+11} = \frac{5 d_{m+4}
d_{m+6}}{126 D_{m+4} D_{m+5}} + \frac{20}{63}\, ,\;
\smallskip \\
 \zeta_{8m+40}^{3m+13,3m+10} = \frac{d_{m+5}
d_{m+6}^2}{126 D_{m+4} D_{m+5}^2} \, ,\; \zeta_{8m+40}^{3m+13,3m+11}
= \frac{d_{m+4} d_{m+6}^2}{126 D_{m+4} D_{m+5}^2} + \frac{10
d_{m+6}}{63 D_{m+5}}\, ,\; \zeta_{8m+40}^{3m+13,3m+12} =
\frac{10}{21}\, , \;
\smallskip \\
 \zeta_{8m+41}^{3m+13,3m+11} = \frac{d_{m+6}^2}{21
D_{m+5}^2}\, , \; \zeta_{8m+41}^{3m+13,3m+12} = \frac{5 d_{m+6}}{14
D_{m+5}}\, , \; \zeta_{8m+42}^{3m+13,3m+12} = \frac{d_{m+6}^2}{6
D_{m+5}^2}\, , \;
\end{array}
\end{equation*}
where $D_k:=d_k+d_{k+1}$, $k=0,...,m+5$.

By means of the computed coefficients $\{\chi_{k}^{i,j}\}^{0\le i,j \le m+2}_{0 \le
{k} \le 3m+12}$ we can thus shortly write the control points of $\r'(t)$
as
$$
\begin{array}{l}
\p_0=\p_1=0, \smallskip \\
\p_2= \frac{d_1}{d_1+d_2} \, \z_0^2,\smallskip\\
\p_{3k}=\frac{2}{3} \z_{k-1}^2 + \frac{1}{3} \, \left(\frac{d_{k+1} \z_{k-2}+d_{k} \z_{k-1}}{d_{k}+d_{k+1}}\right) \, \left( \frac{d_{k+2} \z_{k-1} + d_{k+1} \z_{k}}{d_{k+1}+d_{k+2}} \right),
\quad k=1,...,m+3\,,  \smallskip\\
\p_{3k+1}= \z_{k-1} \, \frac{d_{k+2} \z_{k-1}+d_{k+1} \z_{k}}{d_{k+1}+d_{k+2}},   \quad k=1,...,m+2\,,  \smallskip\\
\p_{3k+2}= \z_{k} \, \frac{d_{k+2} \z_{k-1} +d_{k+1} \z_{k}}{d_{k+1}+d_{k+2}},   \quad k=1,...,m+2\,,  \smallskip\\
\p_{3m+10}= \frac{d_{m+5}}{d_{m+4}+d_{m+5}} \, \z_{m+2}^2, \smallskip\\
\p_{3m+11}=\p_{3m+12}=0,
\end{array}
$$
and the coefficients of the
parametric speed $\sigma(t)$ as
$$
\begin{array}{l}
\sigma_{0}=\sigma_1=0, \smallskip \\
\sigma_2=\frac{d_1}{d_1+d_2} \,  \z_0 \bar{\z}_0,  \smallskip  \\
\sigma_{3k}=
\frac{1}{6} \, \frac{d_{k+1} \, d_{k+2}}{(d_{k}+d_{k+1})\, (d_{k+1}+d_{k+2})}
(\z_{k-2} \bar{\z}_{k-1} + \z_{k-1} \bar{\z}_{k-2})+
\frac{1}{6} \, \frac{(d_{k+1})^2}{(d_{k}+d_{k+1})(d_{k+1}+d_{k+2})}
(\z_{k-2} \bar{\z}_{k} + \z_{k} \bar{\z}_{k-2}) \smallskip  \\
\quad +\left( \frac{2}{3} + \frac{1}{3} \, \frac{d_{k} \, d_{k+2}}{(d_{k}+d_{k+1})\, (d_{k+1}+d_{k+2})} \right)
\z_{k-1} \bar{\z}_{k-1} +
\frac{1}{6} \, \frac{d_{k} \, d_{k+1}}{(d_{k}+d_{k+1})\, (d_{k+1}+d_{k+2})}
(\z_{k-1} \bar{\z}_{k} + \z_{k} \bar{\z}_{k-1}),
\quad  k=1,...,m+3\,,  \smallskip \\
\sigma_{3k+1}= \frac{d_{k+2}}{d_{k+1}+d_{k+2}} \z_{k-1} \bar{\z}_{k-1} +
\frac{1}{2}\,\frac{d_{k+1}}{d_{k+1}+d_{k+2}} (\z_{k-1} \bar{\z}_{k} + \z_{k} \bar{\z}_{k-1}), \quad  k=1,...,m+2\,,  \smallskip\\
\sigma_{3k+2}= \frac{1}{2}\, \frac{d_{k+2}}{d_{k+1}+d_{k+2}} (\z_{k-1} \bar{\z}_{k} + \z_{k} \bar{\z}_{k-1}) +
\frac{d_{k+1}}{d_{k+1}+d_{k+2}} \z_{k} \bar{\z}_{k}, \quad  k=1,...,m+2\,,  \smallskip \\
\sigma_{3m+10}= \frac{d_{m+5}}{d_{m+4}+d_{m+5}} \z_{m+2} \bar{\z}_{m+2}, \smallskip   \\
\sigma_{3m+11}=\sigma_{3m+12}=0.
\end{array}
$$

\begin{figure}[h!]
\centering
\includegraphics[height=0.26\textwidth,valign=t]{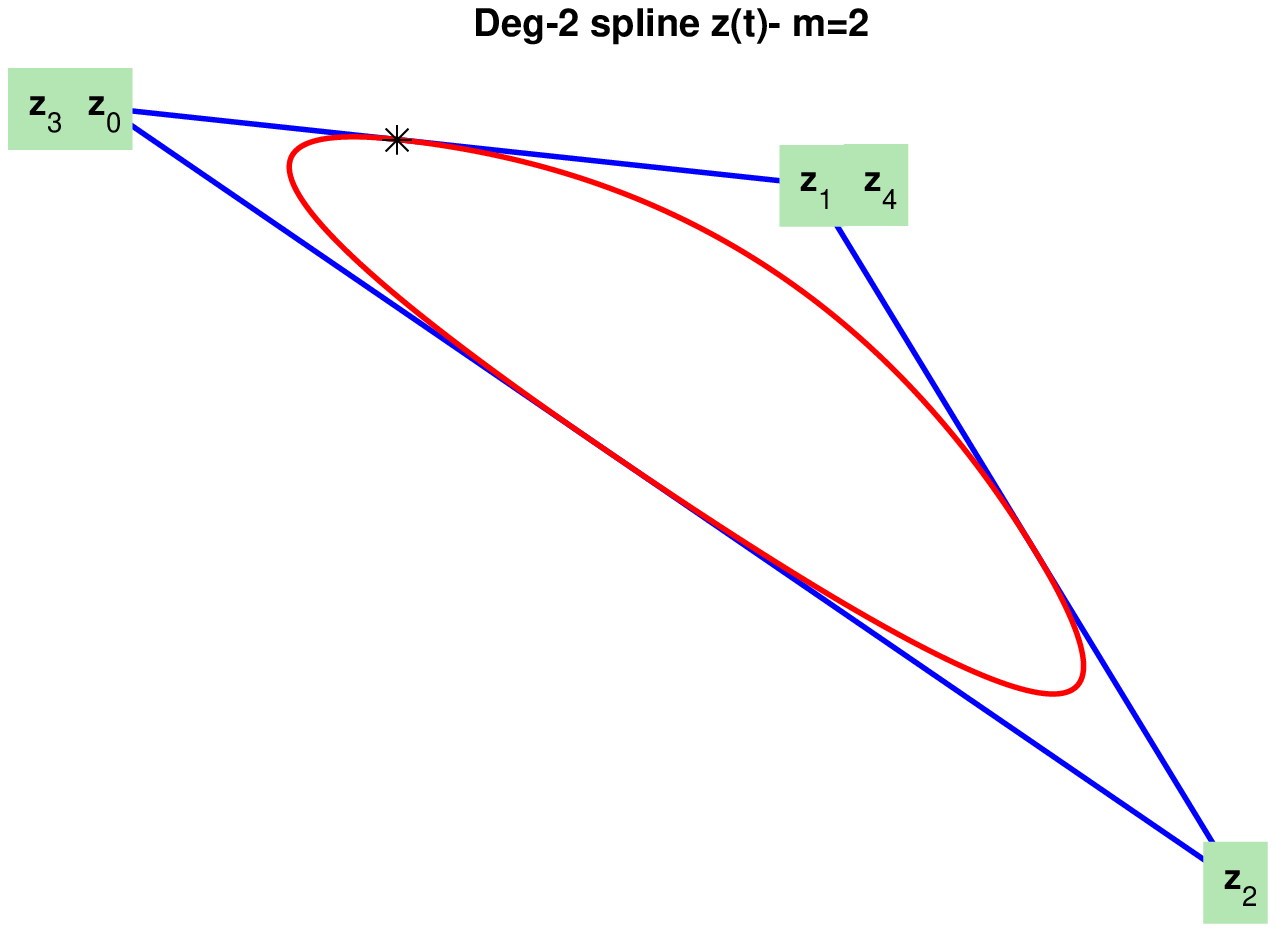}\hspace{-0.4cm}
\includegraphics[height=0.26\textwidth,valign=t]{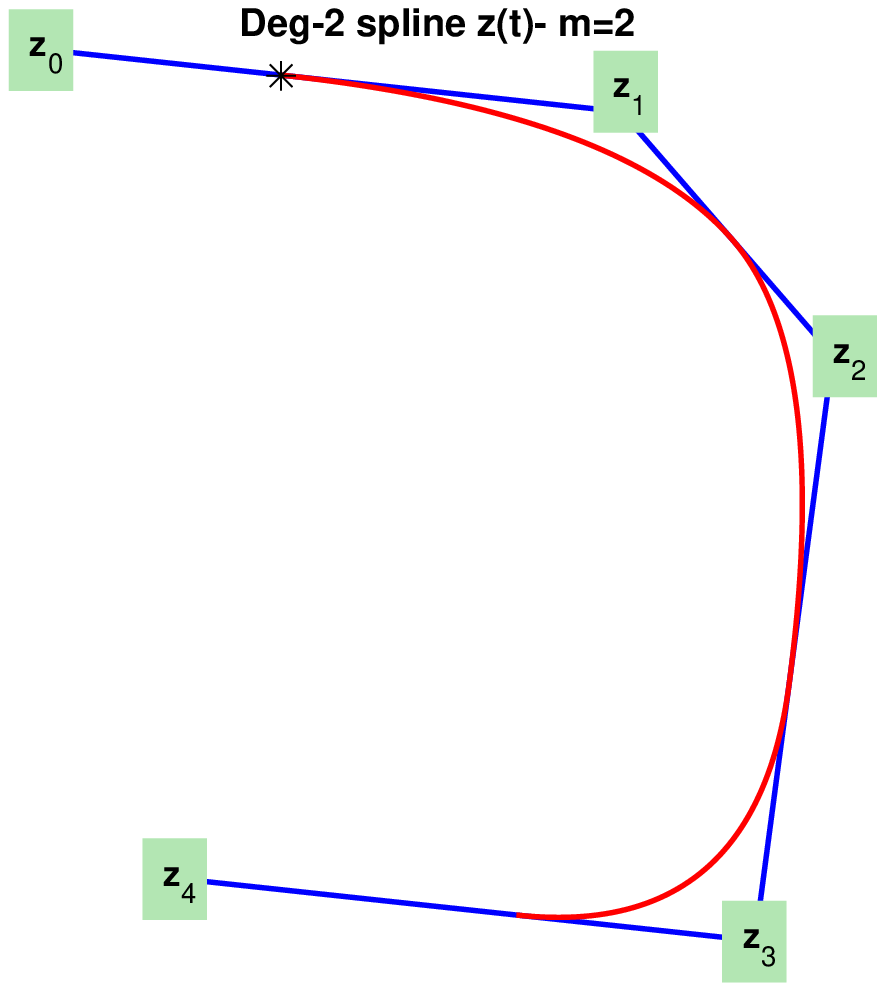} \hspace{-0.75cm}
\includegraphics[height=0.26\textwidth,valign=t]{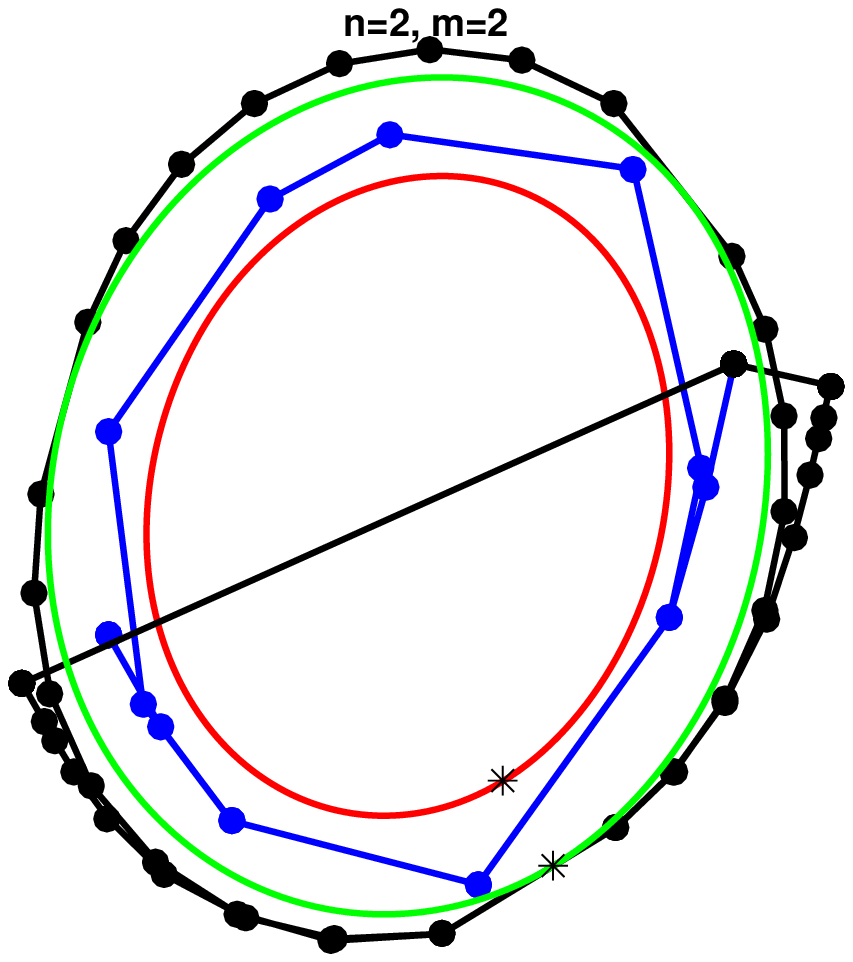}\vspace{0.1cm}
\includegraphics[height=0.26\textwidth,valign=t]{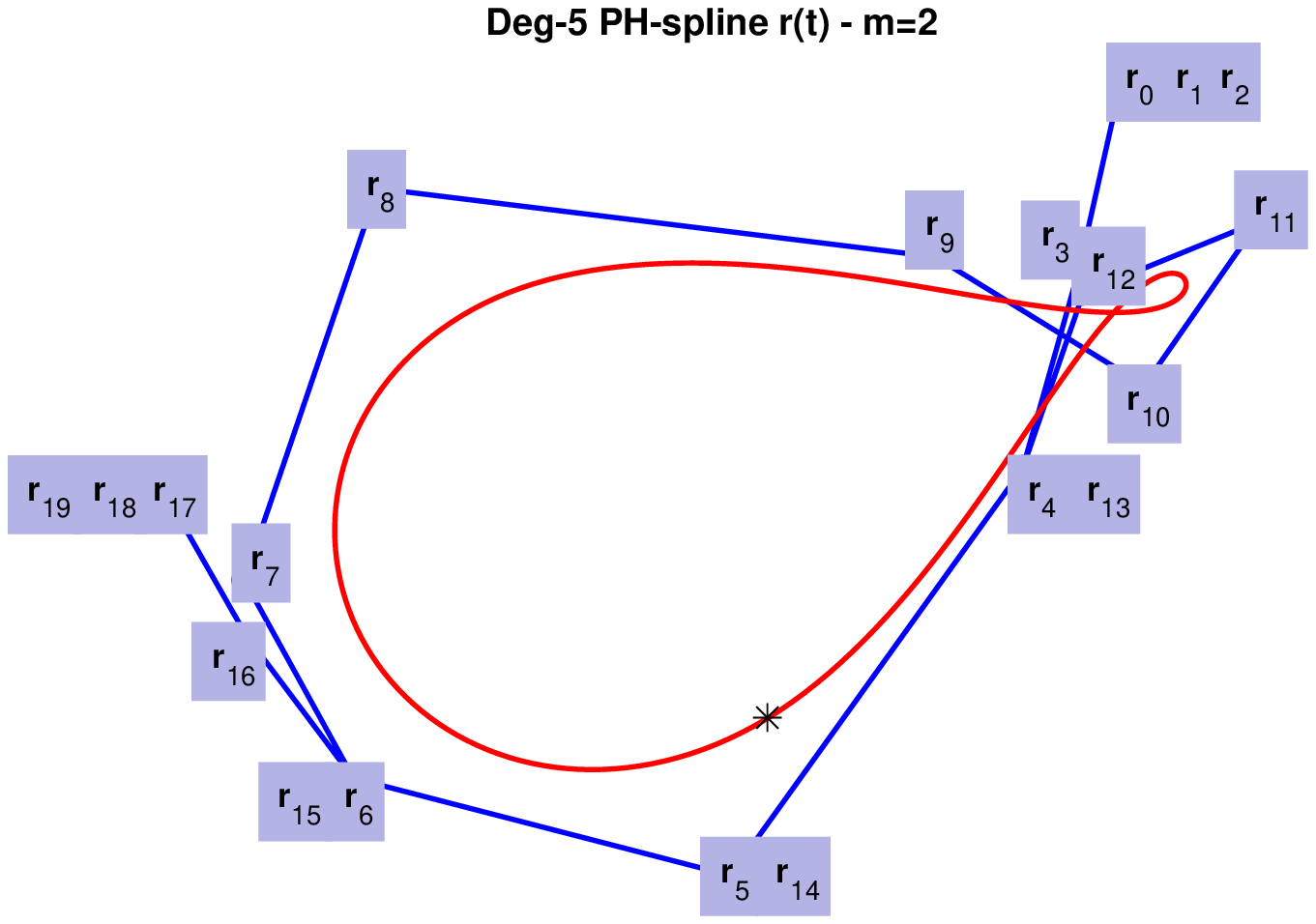}\hspace{-0.4cm}
\includegraphics[height=0.26\textwidth,valign=t]{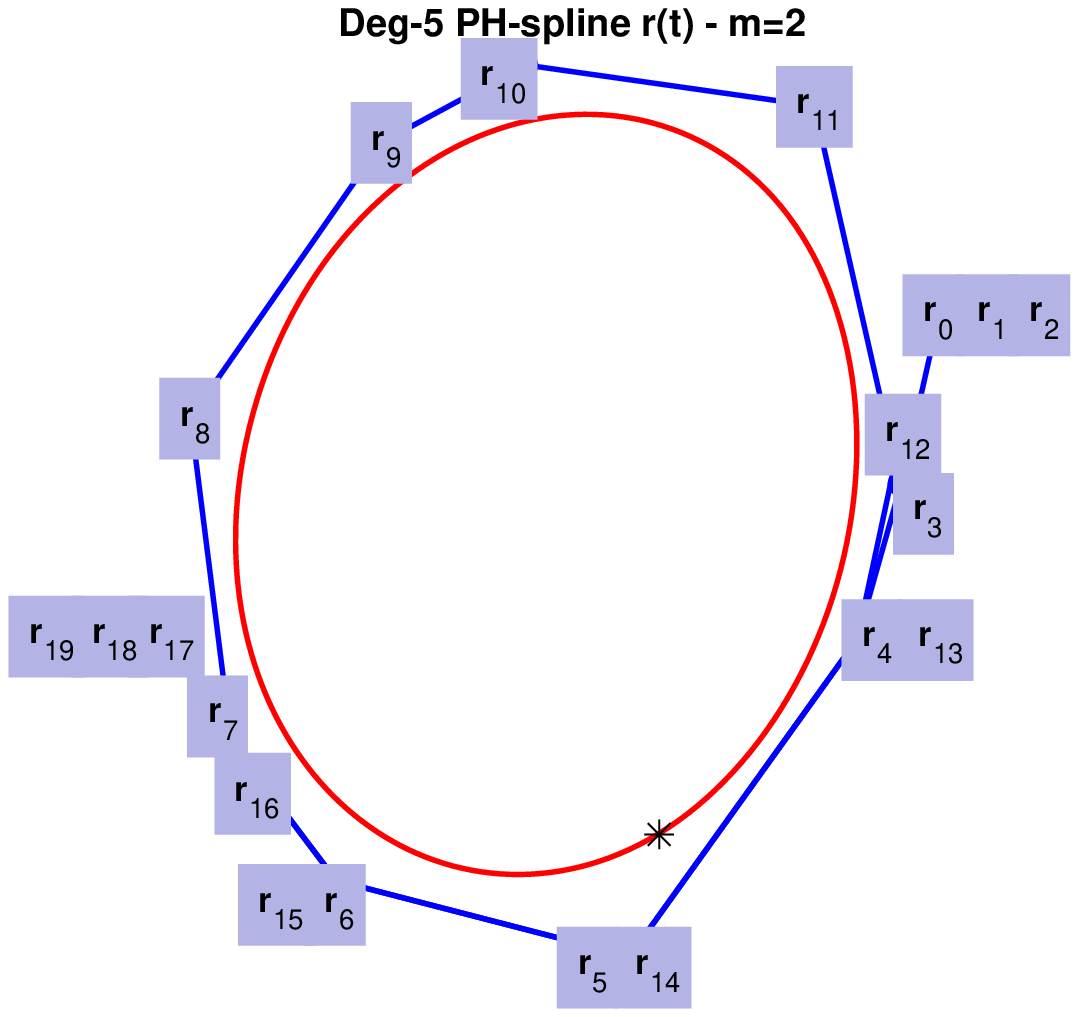}\hspace{-0.75cm}
\includegraphics[height=0.26\textwidth,valign=t]{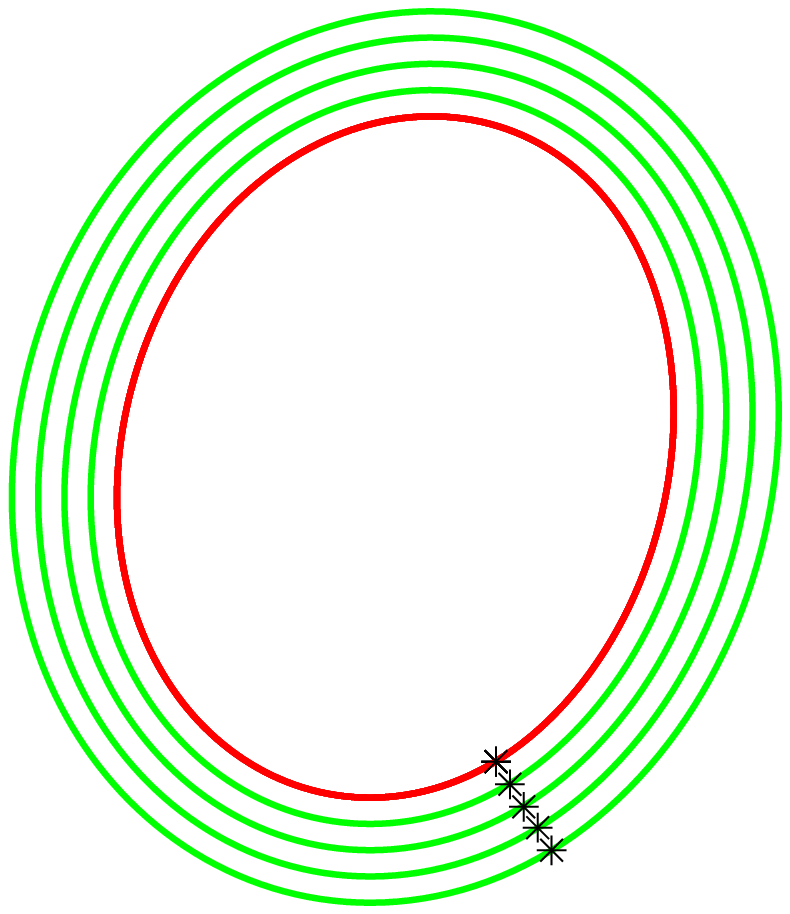}
\caption{Closed quintic PH B-Spline curve with $m=2$ originated from
a closed/open degree-2 spline $\z(t)$ (first/second column) and
offsets of the non self-intersecting curve with and without control
polygon (third column).} \label{fig:closed_n2_m2}
\end{figure}

Thus, according to (\ref{rt}), the closed quintic PH B-Spline curve defined
over the knot partition $\brho$ is given by
$$
\r(t) = \sum_{i=0}^{3m+13} \r_i N_{i,\brho}^{5}(t)\,, \quad t \in [t_2, t_{m+3}] \quad (t_0=0)
$$
with control points
$$
\begin{array}{l}
\r_2=\r_1=\r_0, \vspace{0.05cm}\\
\r_{3}=\r_2 +\frac{d_1}{5} \, \z_0^2, \vspace{0.05cm} \\
\r_{3i+1}=\r_{3i} + \frac{d_{i+1}}{5} \, \bigg( \frac23 \z_{i-1}^2 + \frac13 \, \left( \frac{d_{i+1} \z_{i-2}+d_i \z_{i-1}}{d_i+d_{i+1}} \right ) \,
\left( \frac{d_{i+2} \z_{i-1}+d_{i+1} \z_{i}}{d_{i+1}+d_{i+2}} \right ) \bigg),  \quad i=1,...,m+3\,,  \vspace{0.05cm} \\
\r_{3i+2}=\r_{3i+1} + \frac{\z_{i-1}}{5} \, \big( d_{i+2} \z_{i-1} + d_{i+1} \z_i  \big), \quad i=1,...,m+2\,,  \vspace{0.05cm} \\
\r_{3i+3}=\r_{3i+2} + \frac{\z_i}{5} \, \big( d_{i+2} \z_{i-1} + d_{i+1} \z_i \big), \quad i=1,...,m+2\,,  \vspace{0.05cm} \\
\r_{3m+11}=\r_{3m+10} + \frac{d_{m+5}}{5} \, \z_{m+2}^2, \vspace{0.05cm} \\
\r_{3m+13}=\r_{3m+12}=\r_{3m+11}\,,
\end{array}
$$
and arbitrary $\r_0$. Note that, due to condition
\eqref{condclosed3}, $d_{m+3}=d_2$, $d_{m+4}=d_{3}$. Moreover,
$\z_{m}$, $\z_{m+1}$, $\z_{m+2}$ must be suitably fixed in order to
satisfy condition \eqref{condclosed2}. According to (\ref{Lclosed1})
the total arc length of the closed PH B-Spline curve of degree $5$
is given by
$$
L=(l_{3m+7} - l_4) N_{4,\brho}^5(t_2) + (l_{3m+8} - l_5)
N_{5,\brho}^5(t_2) + (l_{3m+9} - l_6) N_{6,\brho}^5(t_2)\, ,
$$
where $N_{4,\brho}^5(t_2)+N_{5,\brho}^5(t_2)+N_{6,\brho}^5(t_2)=1$
and
$$
\begin{array}{l}
l_2=l_1=l_0=0, \\
l_{3}=l_2 +\frac{d_1}{5} \, \z_0 \bar{\z}_0,  \\
l_{3i+1}=l_{3i} + \frac{2 d_{i+1}}{15}  \z_{i-1} \bar{\z}_{i-1} + \frac{d_{i+1}}{15 (d_i+d_{i+1}) (d_{i+1}+d_{i+2})} \,
\Big( d_i d_{i+2} \z_{i-1} \bar{\z}_{i-1} + \, \frac{d_{i+1}^2}{2} \big( \z_{i-2} \bar{\z}_i + \z_i \bar{\z}_{i-2}   \big) \\
 \hspace{0.8cm} + \frac{d_{i+1} d_{i+2}}{2} \big( \z_{i-2} \bar{\z}_{i-1} + \z_{i-1} \bar{\z}_{i-2}  \big)
+ \frac{d_i d_{i+1}}{2} \big( \z_{i-1} \bar{\z}_i + \z_i \bar{\z}_{i-1} \big) \Big),  \quad i=1,...,m+3\,,  \\
l_{3i+2}=l_{3i+1} + \frac{1}{5} \, \Big( d_{i+2} \z_{i-1} \bar{\z}_{i-1} + \frac{d_{i+1}}{2} \big( \z_{i-1} \bar{\z}_i + \z_i \bar{\z}_{i-1} \Big), \quad i=1,...,m+2\,,   \\
l_{3i+3}=l_{3i+2} + \frac{1}{5} \, \Big( d_{i+1} \z_i \bar{\z}_i + \frac{d_{i+2}}{2} \big( \z_{i-1} \bar{\z}_i + \z_i \bar{\z}_{i-1} \big) \Big), \quad i=1,...,m+2\,,  \\
l_{3m+11}=l_{3m+10} + \frac{d_{m+5}}{5} \, \z_{m+2} \bar{\z}_{m+2},  \\
l_{3m+13}=l_{3m+12}=l_{3m+11}.
\end{array}
$$
Over the knot partition $\btau$ the offset curve $\r_h(t)$
has the rational B-Spline form \be \label{offsetn2closed} \r_h(t) =
\frac{\sum_{k=0}^{8m+47} \q_k N_{k,\btau}^{9}(t)}{\sum_{k=0}^{8m+47}
\gamma_k N_{k,\btau}^{9}(t)}\,, \quad t \in [t_2, t_{m+3}]\ee where,
by exploiting the explicit expressions of the coefficients
$\{\zeta_k^{i,j}\}^{0 \leq i \leq 3m+13, \, 0 \leq j \leq 3m+12}_{0
\leq k \leq 8m+47}$ from (\ref{zetasn2closed}), weights and control
points can easily be obtained; their explicit formulae are reported
in the Appendix.

Some examples of closed PH B-Spline curves of degree $5$ are shown
together with offsets in Figure \ref{fig:closed_n2_m2}.

%********************************************************************************
%********************************************************************************
%********************************************************************************

\section{A practical interpolation problem}
\label{sec6}

We are concerned with $C^2$ continuous quintic PH B-Spline curves defined by setting $n=2$ and $m=3$.
According to Corollary \ref{clampedcase}, starting from
$\mathbf{z}(t)=\sum_{i=0}^{3}\mathbf{z}_{i}N_{i,\bmu}^{2}(t)$ over
the partition $\bmu=\{0,0,0,a,1,1,1\}$, we obtain the PH B-Spline curve \be
\label{phquintic}
\mathbf{r}(t)=\sum_{i=0}^{8}\mathbf{r}_{i}N_{i,\brho}^{5}(t)\, ,
\;\; t\in [0,1]\, , \ee over the knot vector
$\brho=\{0,0,0,0,0,0,a,a,a,1,1,1,1,1,1\}$ from (\ref{partitionsn2clamped}).

We will now use these curves in order to solve the following
practical interpolation problem, several variants of which have been
of interest to the PH community up to date, see, e.g.,
\cite{juettler2001, kongal, waltonmeek, phtrig13, jaklical2,
jaklical}.
\begin{prb} \label{hermiteprb1} Given arbitrary points $\p^*_0$, $\p^*_1$, tangents $\d_0$,
$\d_1$ and curvature values $\kappa_0$, $\kappa_1$ to be
interpolated by a PH quintic B-Spline curve from (\ref{phquintic})
such that
\[
\r(0) = \p^*_0\, ,\; \; \r(1) = \p^*_1\,  ,\; \; \r'(0) = \d_0\, ,\;
\; \r'(1) = \d_1\,  ,\; \; \kappa(0) = \kappa_0\,  ,\; \; \kappa(1)
= \kappa_1,
\]
we look for the control points $\r_0, \r_1, \ldots, \r_8$ of the
curve (\ref{phquintic}).
\end{prb}

According to (\ref{rs_clampedn2}) this means that we have to
determine $\z_0, \ldots, \z_3$. The positional interpolation
constraints clearly imply $\r_0 = \p^*_0$ and $\r_8 = \p^*_1$.
According to (\ref{xyprimecomplex}), (\ref{zsquare}), (\ref{pks})
the tangential interpolation constraints yield the following
conditions. \be \label{tangconst} \r'(0) = \z^2(0) = \p_0 = \z_0^2 =
\d_0\, ,\; \; \r'(1) = \z^2(1) = \p_7 = \z_3^2 = \d_1. \ee
By applying de Moivre's theorem to the two equations in \eqref{tangconst} we obtain the
following solutions for $\z_0$ and $\z_3$: \ba \label{solz0z3} \z_0
& = &\, \pm |\d_0|^\frac{1}{2} \exp\left(\ti
\frac{\omega_0}{2}\right) = \pm
|\d_0|^\frac{1}{2}\left(\cos\left(\frac{\omega_0}{2}\right) + \ti
\sin\left(\frac{\omega_0}{2}\right)\right)  , \nonumber\\
\z_3 & = & \, \pm |\d_1|^\frac{1}{2} \exp\left(\ti
\frac{\omega_1}{2}\right) = \pm
|\d_1|^\frac{1}{2}\left(\cos\left(\frac{\omega_1}{2}\right) + \ti
\sin\left(\frac{\omega_1}{2}\right)\right)\, ,  \ea where
$\omega_k=\arg(\d_k)$ for $k=0,1$.

Note that we can limit ourselves to considering the combinations of
$\z_0$, $\z_3$ with signs $+,+$ and $+,-$, since the others will
give rise to the same solutions. Next, recalling the general formula
for the curvature of a PH curve (see, e.g.,
\cite{phtrig13})
\begin{equation}\label{curvature_ATPH}
\kappa(t)=2 \, \frac{{\rm Im} \Big(\overline{\z}(t) \, \z'(t)
\Big)}{\vert \z(t) \vert^4},
\end{equation}
we can express the curvature constraints at $t=0$ and $t=1$ as
\be\label{curvconst}
\kappa(0) = \frac{4}{a} \cdot
\frac{\Im(\bar{\z}_0 \z_1)}{\vert \z_0 \vert^4} = \frac{4}{a} \cdot
\frac{u_0 v_1 - u_1 v_0}{(u_0^2+v_0^2)^2} = \kappa_0\,, \qquad
\kappa(1) = \frac{4}{1-a} \cdot \frac{\Im(\bar{\z}_3
\z_2)}{\vert \z_3 \vert^4} = \frac{4}{1-a} \cdot \frac{u_2 v_3 - u_3
v_2}{(u_3^2+v_3^2)^2} = \kappa_1\,.
\ee
If $u_0\neq0$ and
$u_3\neq0$, we can derive $v_1$ and $v_2$ from
(\ref{curvconst}), obtaining \be \label{vs} v_1 = \frac{1}{u_0}
(u_1 v_0 +\frac{a}{4} \kappa_0 (u_0^2+v_0^2)^2)\, , \; \; v_2 =
\frac{1}{u_3} (u_2 v_3 -\frac{(1-a)}{4} \kappa_1 (u_3^2+v_3^2)^2)\,.
\ee If \be \label{u0u3_0} u_0=0\,, \;\, \mbox{respectively,}\;\;
u_3=0\, , \ee i.e., if $\d_0$, respectively, $\d_1$ has the form
$(-const., 0)$, we can not proceed as in (\ref{vs}). In this case,
by exploiting the affine invariance of B-Spline curves, we propose
to rotate the initial data in order to avoid the case
(\ref{u0u3_0}). We can thus always assume $u_0\neq0$ and $u_3\neq0$.
To determine the remaining unknowns $u_1$ and $u_2$ in (\ref{vs}) we
consider the equation \be \label{compateq} \r_7 -\r_1 =
\sum_{i=1}^{6} \Delta \r_i\, , \; \; \mbox{with}\; \; \; \Delta \r_i
= \r_{i+1} - \r_i. \ee
Since \be \r'(0) = \frac{5}{a} (\r_1 - \r_0) =
\d_0\, , \; \; \r'(1) = \frac{5}{1-a} (\r_8 - \r_7) = \d_1 \ee the
left hand side of equation (\ref{compateq}) is completely determined
by the given Hermite data as: \be \label{leftside} \r_7 - \r_1 =
\r_8 -\r_0 - \frac{(1-a)}{5} \,\d_1 - \frac{a}{5}\, \d_0 = \p^*_1
-\p^*_0 - \frac{(1-a)}{5}\, \d_1 - \frac{a}{5}\, \d_0. \ee
By (\ref{rs_clampedn2}) the right hand side of (\ref{compateq}) reads
\be \label{rightside} \frac{1}{5} (1 - \frac{a}{3}) \z_1^2 +
\frac{1}{15} (2 + a) \z_2^2 + \frac{1}{5} \z_1 \z_2 + \frac{1}{15}
(1-a)^2 \z_1 \z_3 + \frac{1}{15} a (4-a) \z_0 \z_1 + \frac{1}{15} (3
- 2 a -a^2) \z_2 \z_3 + \frac{1}{15} a^2 \z_0 \z_2. \ee
By separating the real and imaginary parts of this equation we obtain the
following system of two real quadratic equations in the two real
unknowns $u_1$ and $u_2$
\be \label{conics} (1,u_1, u_2) A \left(\begin{array}{c} 1 \\
u_1 \\ u_2 \end{array} \right) = 0\, , \; \; (1,u_1, u_2) B
\left(\begin{array}{c} 1 \\ u_1 \\ u_2
\end{array} \right) = 0\, , \ee
with $3\times 3$ real symmetric matrices $A = (a_{i,j})_{0 \le i,j
\le 2}$ and $B=(b_{i,j})_{0 \le i,j \le 2}$, whose coefficients are
reported in the Appendix.
The two equations in \eqref{conics} represent two conic sections in the real
Euclidean plane. Finding the solutions $u_1$ and $u_2$ of these
equations is equivalent to determining the intersection points of
the two conic sections. In order to obtain them we consider the
pencil of conic sections defined by the given conics as \be
\label{pencil} (1,u_1, u_2) (A - \lambda B) \left(\begin{array}{c} 1 \\ u_1 \\
u_2
\end{array} \right) = 0\, , \; \; \lambda \in \R. \ee
Among the one--parameter set of conics there are three degenerate
conics given by the $\lambda$-values obtained as solutions of the
cubic equation in $\lambda$:
\be \label{det}
\det(A - \lambda B) = -\lambda^3 \det(B)+\lambda^2 (\det(B_3)+\det(B_2)+\det(B_1)) -\lambda (\det(A_3)+\det(A_2)+\det(A_1))+\det(A)
= 0\, ,
\ee
where the matrices $A_k$ for $k=1,2,3$ are obtained by
replacing the $k$-th column of $A$ by the $k$-th column of $B$, and
the matrices  $B_k$ for $k=1,2,3$ are obtained by replacing the
$k$-th column of $B$ by the $k$-th column of $A$. These degenerate
conics are pairs of intersecting lines.
The intersection points of the conics of the pencil (\ref{pencil})
are then easily obtained in the following way. For a solution of the
cubic equation (\ref{det}), by inserting the corresponding
$\lambda$-value into the equation (\ref{pencil}) a quadratic
equation easily decomposable into two linear factors is obtained.
These two linear equations in $u_1$ and $u_2$ can be solved for,
e.g., $u_1$ in dependency of $u_2$, which are then inserted into
another conic of the pencil, for example one of the given conics or
another degenerate one, yielding a quadratic equation in $u_2$. We
thus obtain the coordinates of the intersection points of the pencil
conics.
%There are thus four pairs $(\z_1, \z_2)$ with \be
%\label{z1z2} \z_1 = u_1 + \ti v_1\, , \; \; \z_2 = u_2 + \ti v_2\,,
%\ee i.e., four (? da contare bene con $\z_0$ e $\z_3$) distinct
%solutions to the $C^1/G^2$ Hermite interpolation Problem
%\ref{hermiteprb1}.}
The control points of the corresponding PH B-Spline curves are
obtained by inserting the obtained $u_i$, $v_i$, $i=0,\dots,3$ into
the expression for $\r_i$ in (\ref{rs_clampedn2}). Let $\delta^{++}$
respectively $\delta^{+-}$ be the number of real intersection points
of the conics from (\ref{conics}) for the sign choice $++$
respectively $+-$ in (\ref{solz0z3}). Thus, the total number of
solutions of our interpolation problem is
$\delta=\delta^{++}+\delta^{+-}$. Among all solutions, as it is
classical, see,e.g., \cite{farouki95, farouki96}, the "good" curve
is represented by the curve with minimum absolute rotation index and
bending energy.

The following figures show the cubic curve interpolating the assigned values
and derivatives and the different quintic PH B-spline curves
determined as the solution of the Hermite interpolation problem,
where the endpoint curvatures $\kappa_0$ and $\kappa_1$ are sampled
from the cubic curve. The corresponding conics with coefficient
matrices $A$ and $B$ from \eqref{conics} and a degenerate conic of
the pencil corresponding to a pair of lines are also displayed.

\subsection{Existence of solutions}

In the following we discuss the problem of studying the existence of
a solution for Problem \ref{hermiteprb1}. Assigned a set of values,
tangents and curvatures, there are two situations where a solution
may not be found. The first corresponds to the case where either $A$
or $B$ (or both) in \eqref{conics} yield imaginary conics. The
second is the case where the two conics in \eqref{conics} are real,
but have no intersection. The latter situation has never been
encountered in our experiments. Thus we will focus on discussing the
former problematic case. In particular we will consider the
following problem. Suppose that points $\p^*_0$, $\p^*_1$ and
tangents $\d_0$, $\d_1$ are given, how shall we prescribe the
curvatures $\kappa_0$, $\kappa_1$ in order to guarantee that $A$ and
$B$ represent real conics?

\begin{table}[h!]
\centering
\begin{tabular}{|c||c|c|c|}
\hline
 Type of conic & $I_3^{[M]}$ & $I_2^{[M]}$ & $I_1^{[M]}$ \\
\hline \hline Ellipse & $\ne 0$ & $> 0$ & $I_1^{[M]} I_3^{[M]} < 0$\\
\hline Hyperbola & $\ne 0$ & $< 0$ & --\\
\hline Parabola & $\ne 0$ & $= 0$ & -- \\
\hline Non-degenerate imaginary conic & $\ne 0$ & $> 0$ & $I_1^{[M]} I_3^{[M]} > 0$\\
\hline Pair of intersecting real lines & $= 0$ & $< 0$ & -- \\
\hline Pair of intersecting imaginary lines & $= 0$ & $> 0$ & -- \\
\hline Pair of parallel real lines & $= 0$, rank$(I_3^{[M]})=2$ & $= 0$ & $< 0$ \\
\hline Pair of parallel imaginary lines & $= 0$, rank$(I_3^{[M]})=2$ & $= 0$ & $> 0$ \\
\hline Double line (pair of coinciding real lines) & $= 0$, rank$(I_3^{[M]})=1$ & $= 0$ & -- \\
\hline
\end{tabular}
 \caption{Classification of conic sections.} \label{table_class_conic}
\end{table}

For a generic conic $\mathcal{C}_M: \;  \ m_{1,1} u_1^2 + 2 m_{1,2}
u_1 u_2 + m_{2,2} u_2^2 + 2 m_{0,1} u_1 + 2 m_{0,2} u_2 + m_{0,0} =
0$ with symmetric matrix $M:=(m_{i,j})_{0\leq i,j \leq 2},
m_{i,j}=m_{j,i}$, we suppose $m_{0,0} \ge 0$ and we denote by
$\tilde{M}:=(m_{i,j})_{1\leq i,j \leq 2}$ the $2\times 2$ matrix of
the associated quadratic form. We then consider the following
invariants: \be \label{invariantsconic}
I_1^{[M]}:=trace(\tilde{M})\,, \; I_2^{[M]}:=det(\tilde{M})\,, \;
I_3^{[M]}:=det(M)\,. \ee  Then, according to Table
\ref{table_class_conic} $C_M$ represents an imaginary conic in the
following cases:
\begin{enumerate}
\item[P1] \label{p1} $I_3^{[M]}\neq 0$, $I_2^{[M]}>0$, $I_1^{[M]} I_3^{[M]}>0$ ($\mathcal{C}_M$ is a non-degenerate imaginary conic);
\item[P2] \label{p2} $I_3^{[M]} = 0$,   $I_2^{[M]}>0$ ($\mathcal{C}_M$ is a pair of intersecting imaginary lines);
\item[P3] \label{p3} $I_3^{[M]} = 0$, $\mathrm{rank} (I_3^{[M]})=2$,  $I_2^{[M]}=0$, $I_1^{[M]}>0$. ($\mathcal{C}_M$ is a pair of parallel imaginary lines).
\end{enumerate}
When choosing $\kappa_0$, $\kappa_1$ we shall then avoid that
$\mathcal{C}_A$ or $\mathcal{C}_B$ fall into one of the points
above. To this aim, it is useful to observe that for $M \in
\{A,B\}$, $I_1^{[M]}$ and $I_2^{[M]}$ depend on the location of knot
$a$ only, whereas $I_3^{[M]}$ depends on $a$, $\kappa_0$,
$\kappa_1$, see the formulae of the coefficients of $A$ and $B$ in
the Appendix. For a fixed value of $a$, the sign of $I_1^{[M]}$ and
$I_2^{[M]}$ is thus determined, whereas $I_3^{[M]}=0$ is a quadratic
equation in the two unknowns $\kappa_0$, $\kappa_1$. The general
idea would thus be to first choose the value of $a$ and based on the
corresponding signs of $I_1^{[M]}$ and $I_2^{[M]}$, fix $\kappa_0$,
$\kappa_1$ in order to prevent that $M$ falls into one of the items
above. We will illustrate the idea by means of the four test cases
contained in Examples \ref{exm1} to \ref{exm4}. This is not an
exhaustive discussion of all the possible forms of the two conics
$\mathcal{C}_A$ and $\mathcal{C}_B$, but it suffices to illustrate
the general criterion for choosing suitable values of $a$,
$\kappa_0$ and $\kappa_1$. The general procedure is summarized in
Algorithm \ref{tab:algor}.

\begin{algorithm}[h!]
\begin{flushleft}
{\bf Input:} interpolation points $\p^*_0$, $\p^*_1$ and tangents at these points $\d_0$, $\d_1$.\\
\begin{itemize}
\item[] From \eqref{tangconst} and \eqref{solz0z3} determine all possible values $\z_0$ and $\z_3$, i.e. $u_0$, $v_0$, $u_3$, $v_3$,
and for each combination with signs $+,+$ or $+,-$ in
\eqref{solz0z3} perform the following steps:
\begin{enumerate}
\item  compute the matrices $A$ and $B$ of the two conics in \eqref{conics}. These matrices will depend on the parameters $a$, $\kappa_0$ and $\kappa_1$;
\item choose the value of the knot $a\in (0,1)$ and compute $I_1^{[M]}$, $I_2^{[M]}$, for $M=A,B$. Since these quantities depend on $a$ only, their sign is determined;
\item Choose $\kappa_0$ and $\kappa_1$ in such a way that both $\mathcal{C}_A$ and $\mathcal{C}_B$ are real.
In particular, $\mathcal{C}_A$ or $\mathcal{C}_B$ (or both) may
represent an imaginary conic when $I_2^{[A]}>0$ or $I_2^{[B]}>0$ (or
both). $I_3^{[M]}=0$, $M=A,B$ is a quadratic function of $\kappa_0$
and $\kappa_1$ and represents a conic in the Euclidean plane.
Plotting the graph of $I_3^{[M]}=0$, we can choose $\kappa_0$ and
$\kappa_1$ in the region where the relative signs of $I_1^{[M]}$,
$I_2^{[M]}$ and $I_3^{[M]}$ yield a real conic.
   % \item[-] Derive the expression of $v_1$ and $v_2$ from \eqref{vs};
    %\item By inserting the expressions of $v_1$ and $v_2$ in \eqref{rightside}, derive the system \eqref{conics}.
    \item Solve the system \eqref{conics} for $u_1$ and $u_2$.
    \item For each solution of the system \eqref{conics}, compute $v_1$ and $v_2$ from  \eqref{vs}.
    The obtained values $u_i$, $v_i$, $i=0,\dots,3$ completely determine
    a set of complex coefficients $\z_i$, $i=1,\dots,3$ and thus identify one PH B-spline curve.
      \end{enumerate}
    \end{itemize}
{\bf Output:} PH B-spline curves corresponding to the given initial
data.
\end{flushleft}
\caption{Algorithm for the solution of an Hermite interpolation
problem with given interpolation points $\p^*_0$, $\p^*_1$ and
tangents at these points $\d_0$, $\d_1$.} \label{tab:algor}
\end{algorithm}

% CURRENT EXAMPLE 1
\begin{exm}
\label{exm1} Let us consider the initial data $\p^*_0=(1,0)$,
$\p^*_1=(3,0.5)$, $\d_0=(1,-1)$, $\d_1=(0.2,3)$. Using these data,
from \eqref{solz0z3} we can compute $\z_0$ and $\z_3$ (namely $u_0$,
$v_0$, $u_3$, $v_3$).  For any combination of $\z_0$ and $\z_3$ to
be considered, i.e. taking signs $+,+$ and $+,-$ in \eqref{solz0z3},
we should proceed to computing the intersection of the two conics in
equation \eqref{conics}. Each intersection point of $\mathcal{C}_A$
and $\mathcal{C}_B$ will give rise to a solution for $u_1$, $v_1$,
$u_2$, $v_2$.

The first couple of values corresponds to signs $+,+$. In this case,
any choice of $a$ in $(0,1)$ yields $I_1^{[A]}>0$, $I_2^{[A]}<0$,
$I_1^{[B]}>0$, $I_2^{[B]}<0$. Hence, both $\mathcal{C}_A$ and
$\mathcal{C}_B$ are real conics, independently on the choice of $a$,
$\kappa_0$ and $\kappa_1$. In particular, we will set $a=0.5$, and
sample the curvature values from the cubic polynomial interpolating
$\p^*_0$, $\p^*_1$ and $\d_0$, $\d_1$, which results in $\kappa_0 =
3.040559$ and  $\kappa_1 = 1.066953$. The two conics $\mathcal{C}_A$
and $\mathcal{C}_B$ are displayed in the left part of Figure
\ref{fig:exm1}, while the center and right figures show the PH
B-spline curves corresponding to the two intersection points of the
conics.

The second couple of values corresponds to signs $+,-$. In this
setting $I_1^{[M]}$, $I_2^{[M]}$, $M \in \{A,B\}$, have the same
signs as for the previous case. Choosing the same values for $a$,
$\kappa_0$ and $\kappa_1$, $\mathcal{C}_A$ and $\mathcal{C}_B$ turn
out to be the two hyperbolas plotted in the left part of Figure
\ref{fig:exm1_2}, whereas the PH B-spline curves corresponding to
each of their intersection points are shown in the center and right
part of the figure.

Based on the absolute rotation index and bending energy, displayed
in Figures \ref{fig:exm1} and \ref{fig:exm1_2}, the best obtained PH
B-spline curve turns out to be the one in the center of Figure
\ref{fig:exm1}.

%6
\begin{figure}[h!]
\centering
\includegraphics[width=0.3\textwidth,trim = 0 0cm 0 0cm, clip, valign=t]{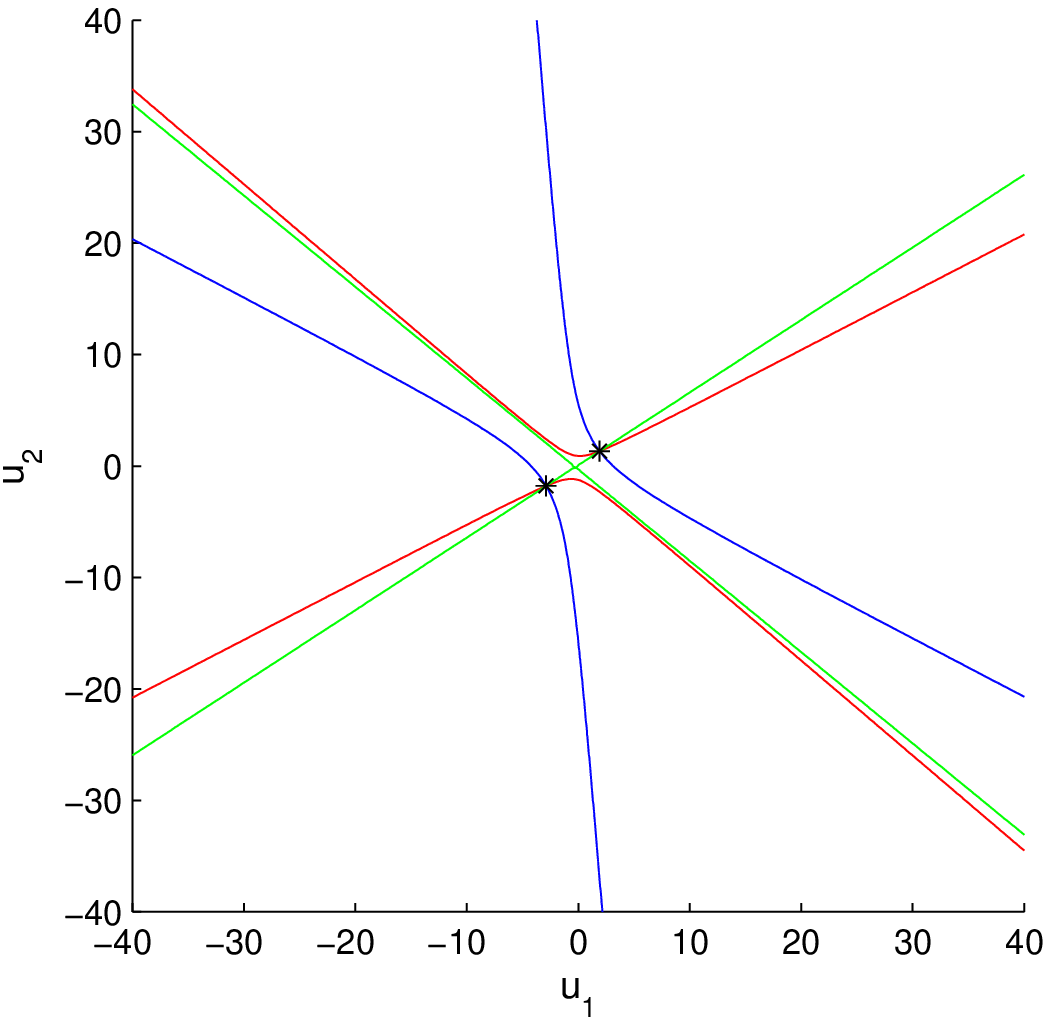}\hfill
\includegraphics[width=0.3\textwidth,valign=t]{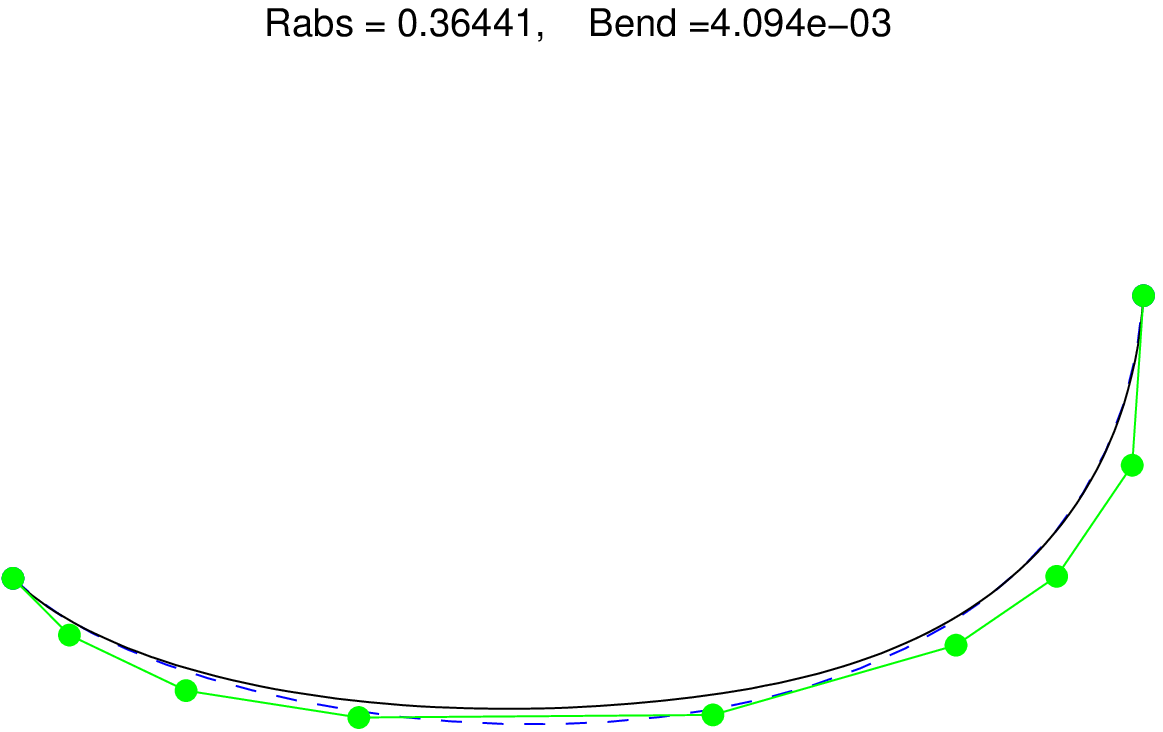}\hfill
\includegraphics[width=0.3\textwidth,valign=t]{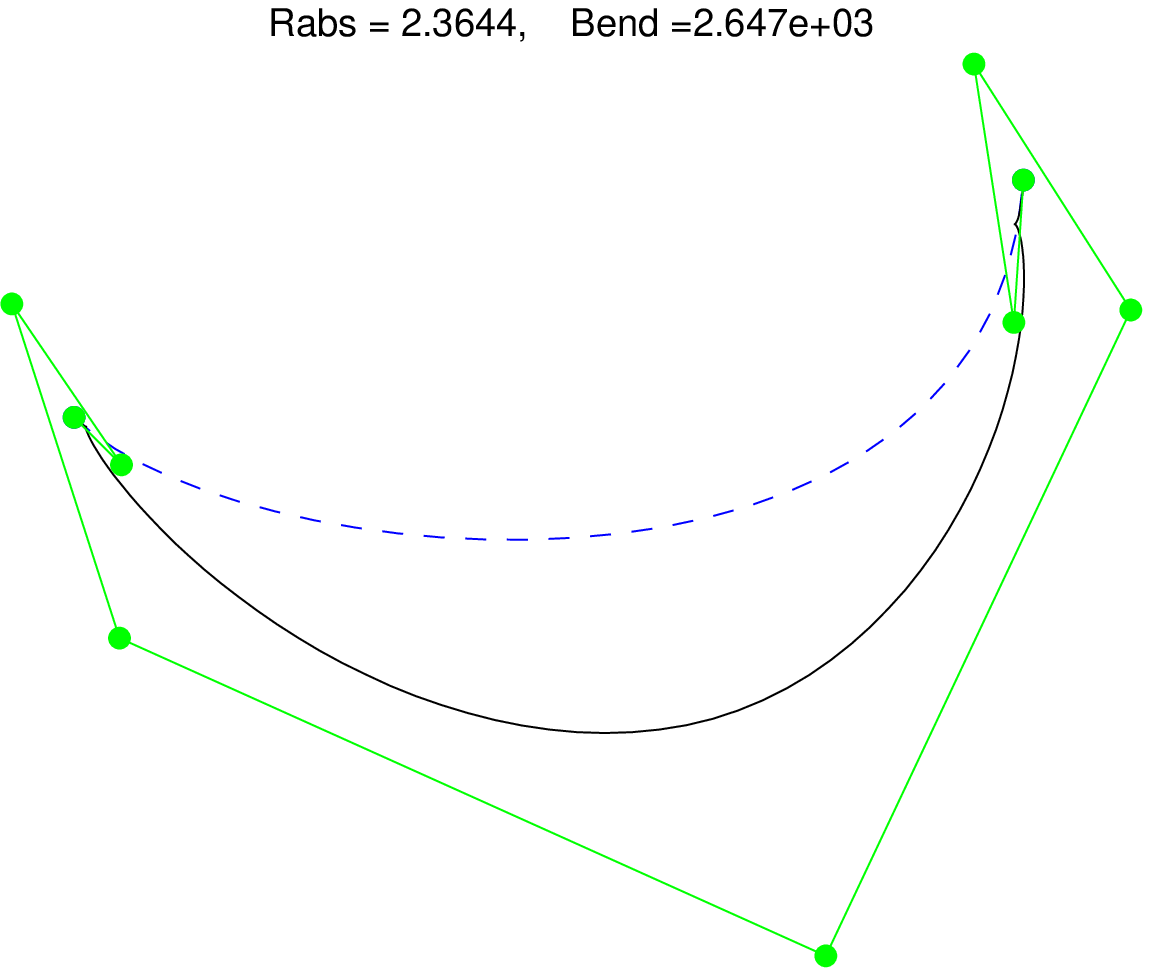}\\
\caption{Illustration of Example \ref{exm1}, case $++$. Left: The
two conics $\mathcal{C}_A$ (blue) and $\mathcal{C}_B$ (red)  and a
degenerate conic (green). Center and right: The PH B-spline curves
(black) corresponding to the two intersection points of the conics
marked with an asterisk. The dashed curve is the cubic polynomial
interpolating $\p^*_0$, $\p^*_1$, $\d_0$, $\d_1$. The values of the
absolute rotation index $Rabs$ and the bending energy $Bend$ of the
PH B-Spline curves are also displayed.} \label{fig:exm1}

\vspace{1cm} \centering
\includegraphics[width=0.3\textwidth,trim = 0 0cm 0 0cm, clip,valign=t]{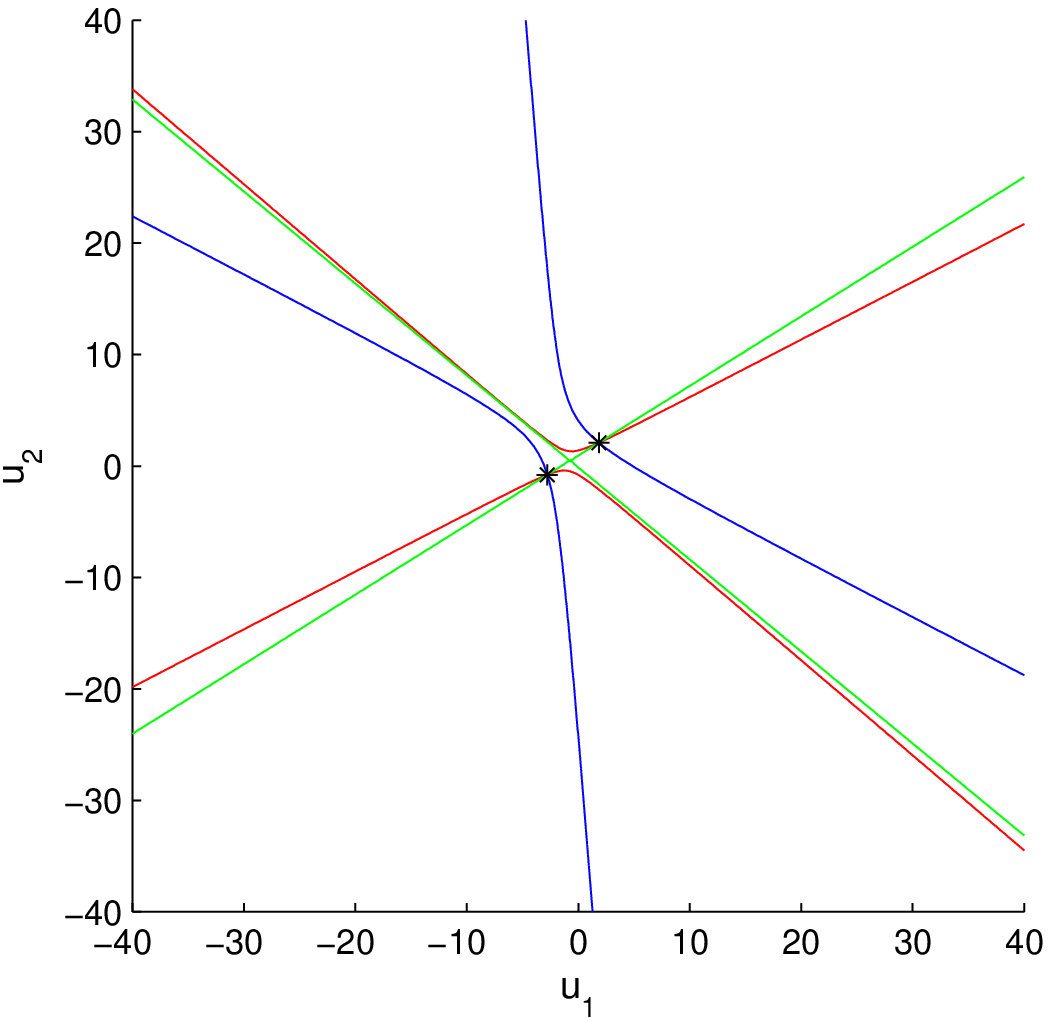}\hfill
\includegraphics[width=0.3\textwidth,valign=t]{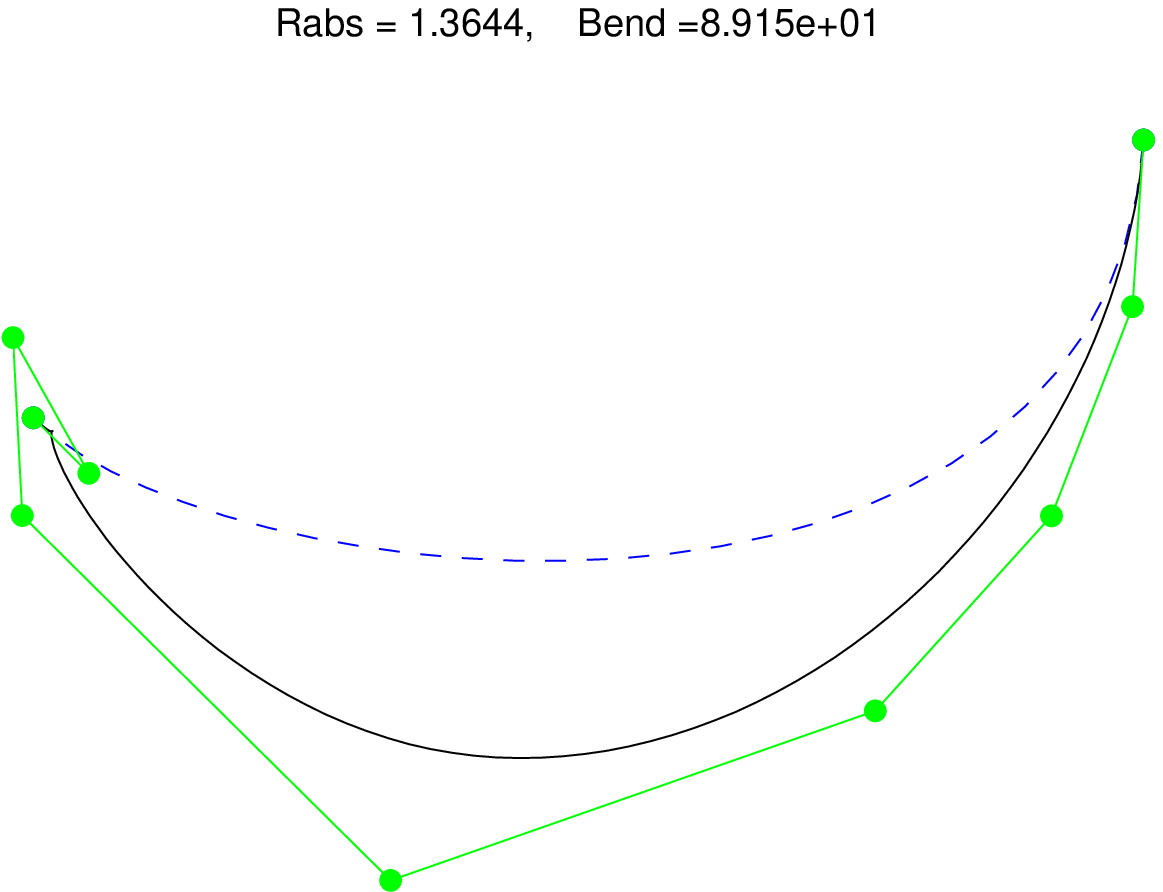}\hfill
\includegraphics[width=0.3\textwidth,valign=t]{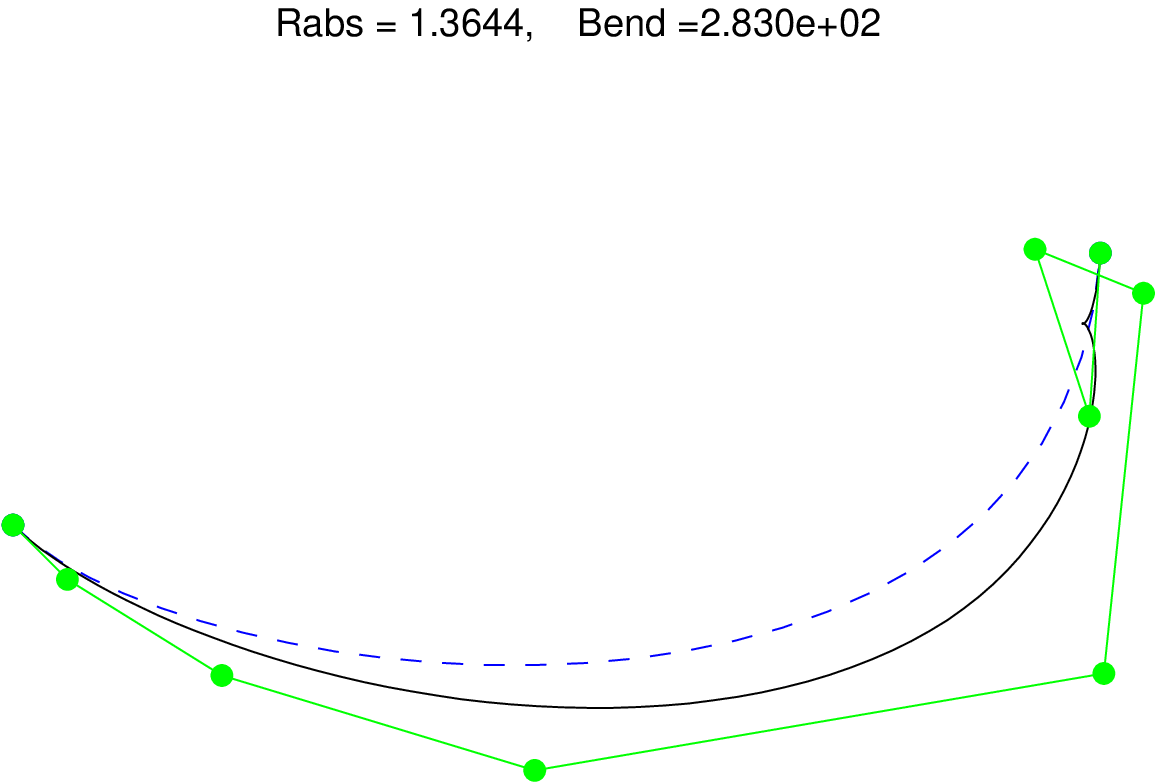}
\caption{Illustration of Example \ref{exm1}, case $+-$. Left: The
two conics $\mathcal{C}_A$ (blue) and $\mathcal{C}_B$ (red)  and a
degenerate conic (green). Center and right: The PH B-spline curves
(black) corresponding to the two intersection points of the conics
marked with an asterisk. The dashed curve is the cubic polynomial
interpolating $\p^*_0$, $\p^*_1$, $\d_0$, $\d_1$. The values of the
absolute rotation index $Rabs$ and the bending energy $Bend$ of the
PH B-Spline curves are also displayed.} \label{fig:exm1_2}
\end{figure}

\end{exm}

% CURRENT EXAMPLE 2
\begin{exm}\label{exm2}
Let us consider the initial data $\p^*_0=(-6,-1)$, $\p^*_1=(1,0)$,
$\d_0=(30,25)$, $\d_1=(25,-30)$. Using these data, from
\eqref{solz0z3} we can compute $\z_0$ and $\z_3$ (namely $u_0$,
$v_0$, $u_3$, $v_3$).  For any combination of $\z_0$ and $\z_3$ to
be considered, i.e. taking signs $+,+$ and $+,-$ in \eqref{solz0z3},
we should proceed to computing the intersection of the two conics in
equation \eqref{conics}.

The first couple of values corresponds to signs $+,+$. In this case,
for any $a\in (0,1)$, $I_1^{[A]}$ and $I_2^{[A]}$ are both positive,
whereas $I_1^{[B]}$ and $I_2^{[B]}$ are both negative. Hence the
choice of $a$ does not influence the signs of these quantities. To
proceed we set $a=0.5$. According to the signs of  $I_1^{[B]}$ and
$I_2^{[B]}$, the conic with matrix $B$ is an hyperbola, and thus it
is a real conic for any choice of $\kappa_0$ and $\kappa_1$.

We shall now investigate whether the conic $\mathcal{C}_A$ may fall
into cases \emph{P1} or \emph{P2} (case \emph{P3} is excluded given
that $I_2^{[A]}>0$.) To this aim, let us consider the conic
$I_3^{[A]}=0$. This is an imaginary conic and in particular
$I_3^{[A]}<0$ for any $\kappa_0$ and $\kappa_1$. The conic
$\mathcal{C}_A$ is thus always a real ellipse.

Since we have no constraint for choosing $\kappa_0$ and $\kappa_1$,
we can set them to the curvatures of the cubic polynomial that
interpolates $\p^*_0$, $\p^*_1$, $\d_0$ and $\d_1$, which yields
$\kappa_0=0.0366$ and $\kappa_1=0.0275$.

Accordingly, the two conics $\mathcal{C}_A$ and $\mathcal{C}_B$ are
represented in Figure \ref{fig:exmp2}, left. The green lines
correspond to a degenerate conic of the pencil of conic sections
\eqref{pencil} and the asterisk markers identify the solutions of
system \eqref{conics}. There are thus two solutions which correspond
to the PH B-spline curves depicted in Figure \ref{fig:exmp2} (solid
line), whereas the dashed line corresponds to the cubic polynomial
from which the curvature values were sampled.

Concerning the second couple of values, corresponding to signs $+,-$
in \eqref{solz0z3}, the signs of $I_1^{[A]}$, $I_2^{[A]}$ and
$I_1^{[B]}$, $I_2^{[B]}$ are the same as for the previous case and
similarly, by setting $a=0.5$, $I_2^{[A]}$ is always negative. We
can thus choose the same curvature values, obtaining the two conics
displayed in Figure \ref{fig:exmp2_2}. This time we have four
intersection points, each of which yields one of the curves in the
figure.

Among all the solutions shown in Figures \ref{fig:exmp2} and Figures
\ref{fig:exmp2_2}, which correspond to the same set of initial data
$\p^*_0$, $\p^*_1$, $\d_0$, $\d_1$, $\kappa_0$, $\kappa_1$, the best
curve can be identified from a study of the absolute rotation index
and bending energy and corresponds to the rightmost curve in Figure
\ref{fig:exmp2_2}.
%3
\begin{figure}[h!]
\centering
\includegraphics[height=0.29\textwidth,trim = 0 0cm 0 0cm, clip,valign=t]{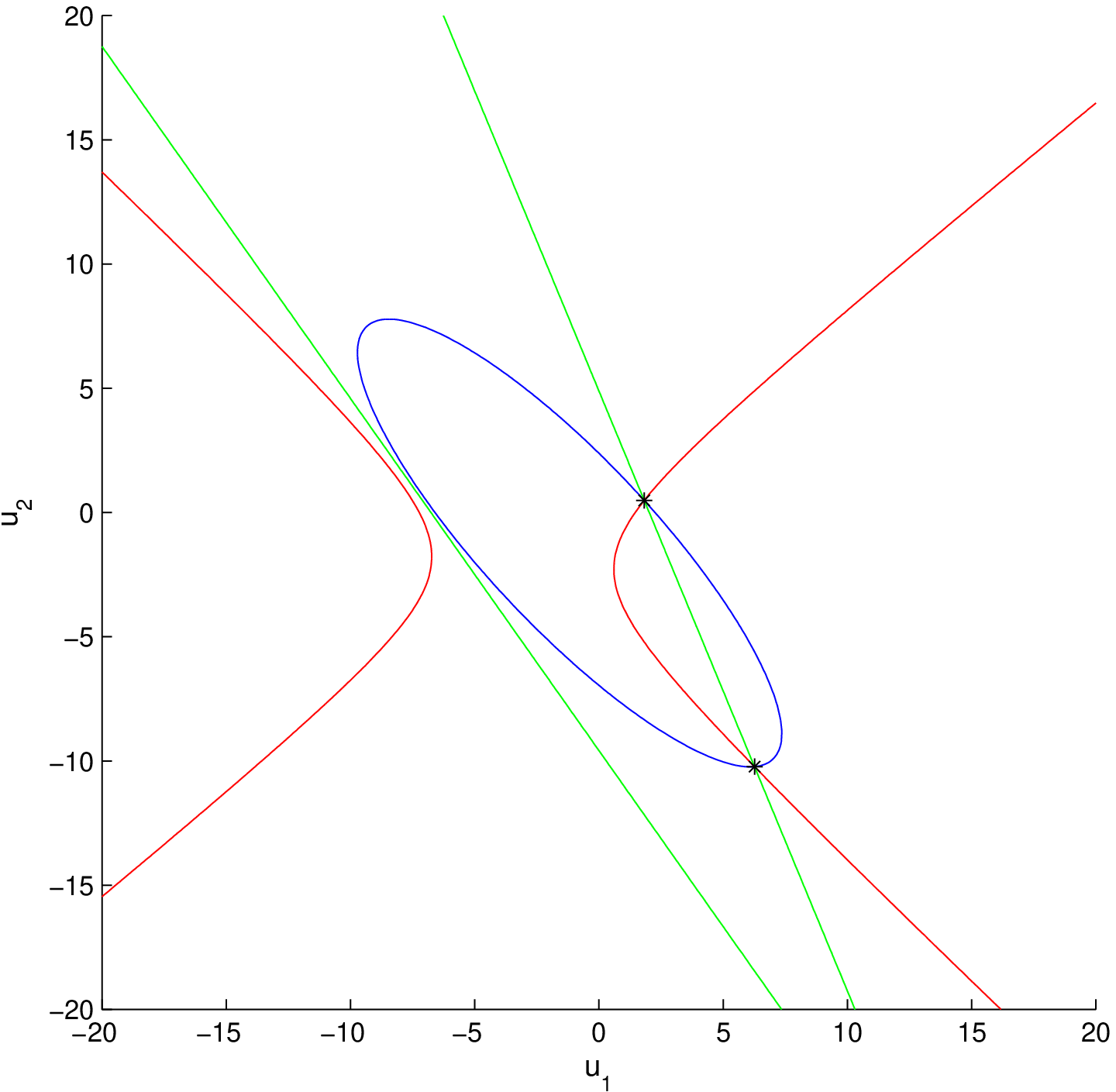}\hfill
\includegraphics[height=0.29\textwidth,valign=t]{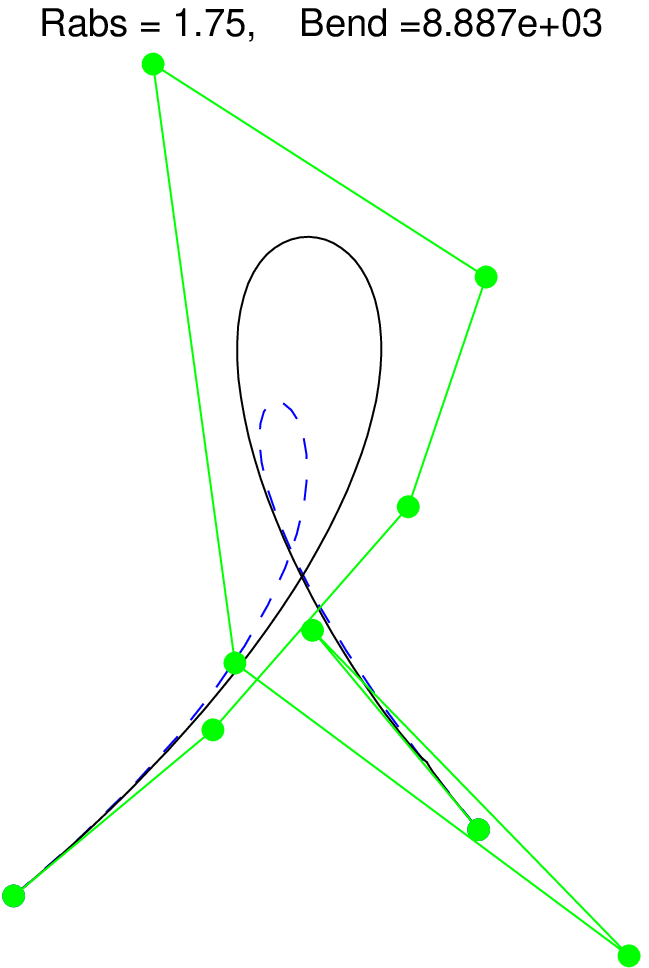}\hfill
\includegraphics[height=0.29\textwidth,valign=t]{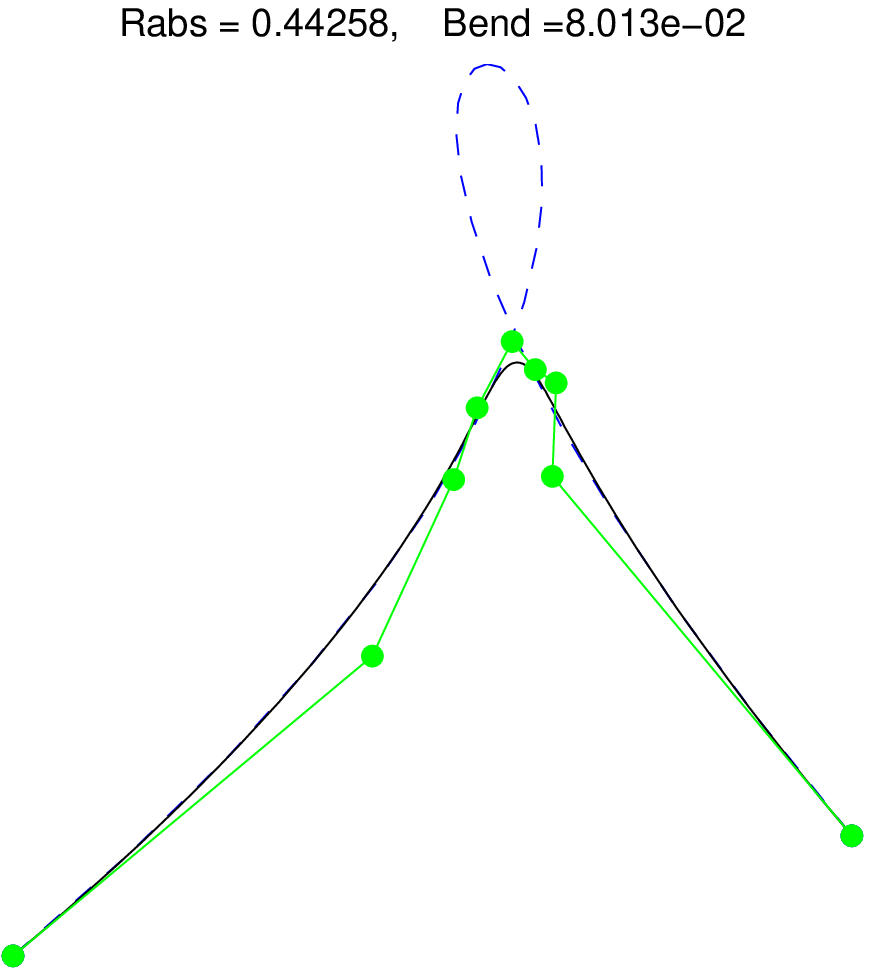}\\
\caption{Illustration of Example \ref{exm2}, case $++$. Left: The
two conics $\mathcal{C}_A$ (blue) and $\mathcal{C}_B$ (red) , a
degenerate conic of the pencil (green); Center and right: PH
B-spline curves (black) corresponding to the two intersection points
and cubic polynomial (dashed, blue) interpolating the given
endpoints and tangents, together with the values of the absolute
rotation index and bending energy of the PH B-Spline curves.}
\label{fig:exmp2}

\vspace{1cm} \centering
\includegraphics[height=0.29\textwidth,trim = 0 0cm 0 0cm, clip,valign=t]{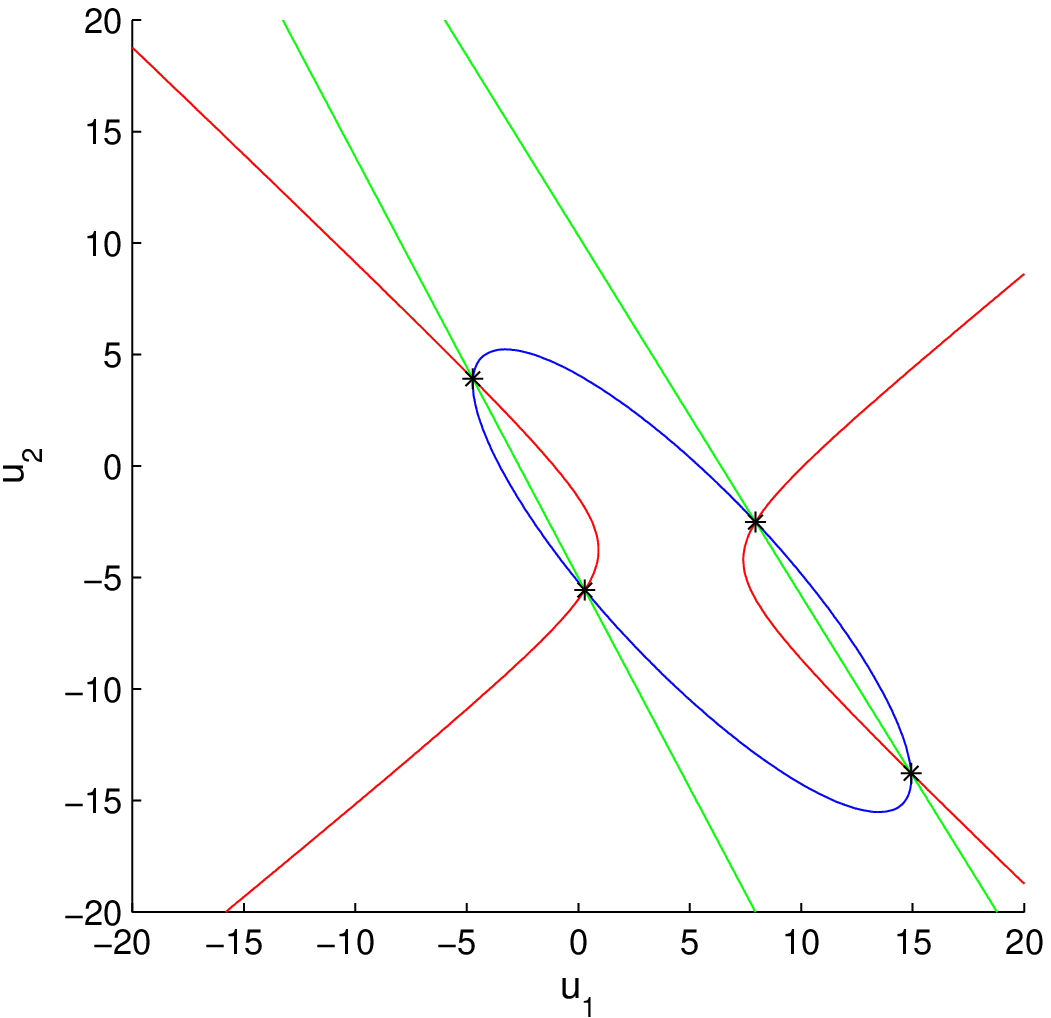}\\[3ex]
\includegraphics[height=0.29\textwidth,valign=t]{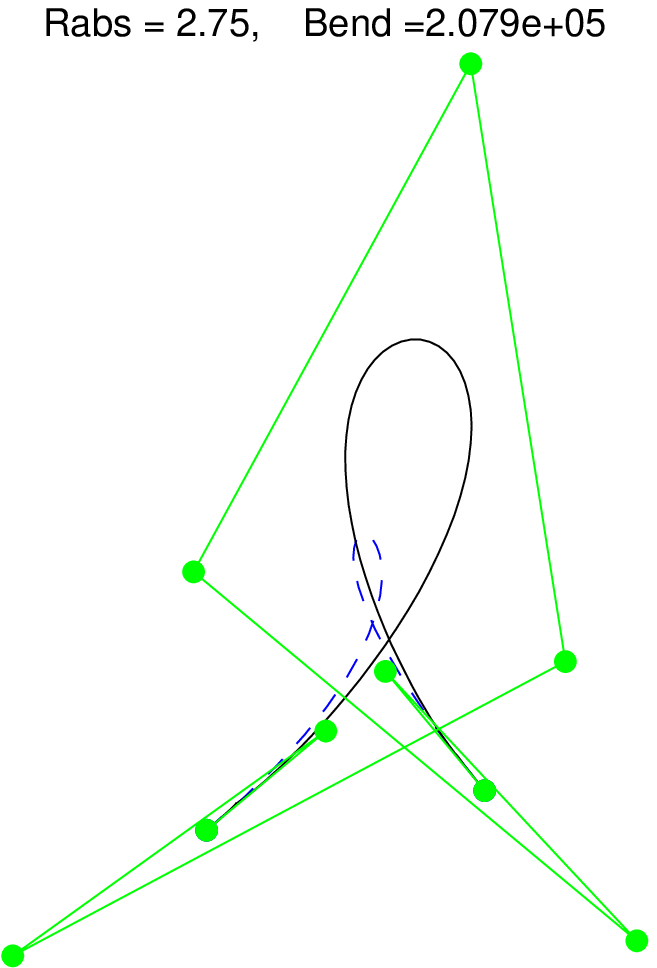}\hfill
\includegraphics[height=0.29\textwidth,valign=t]{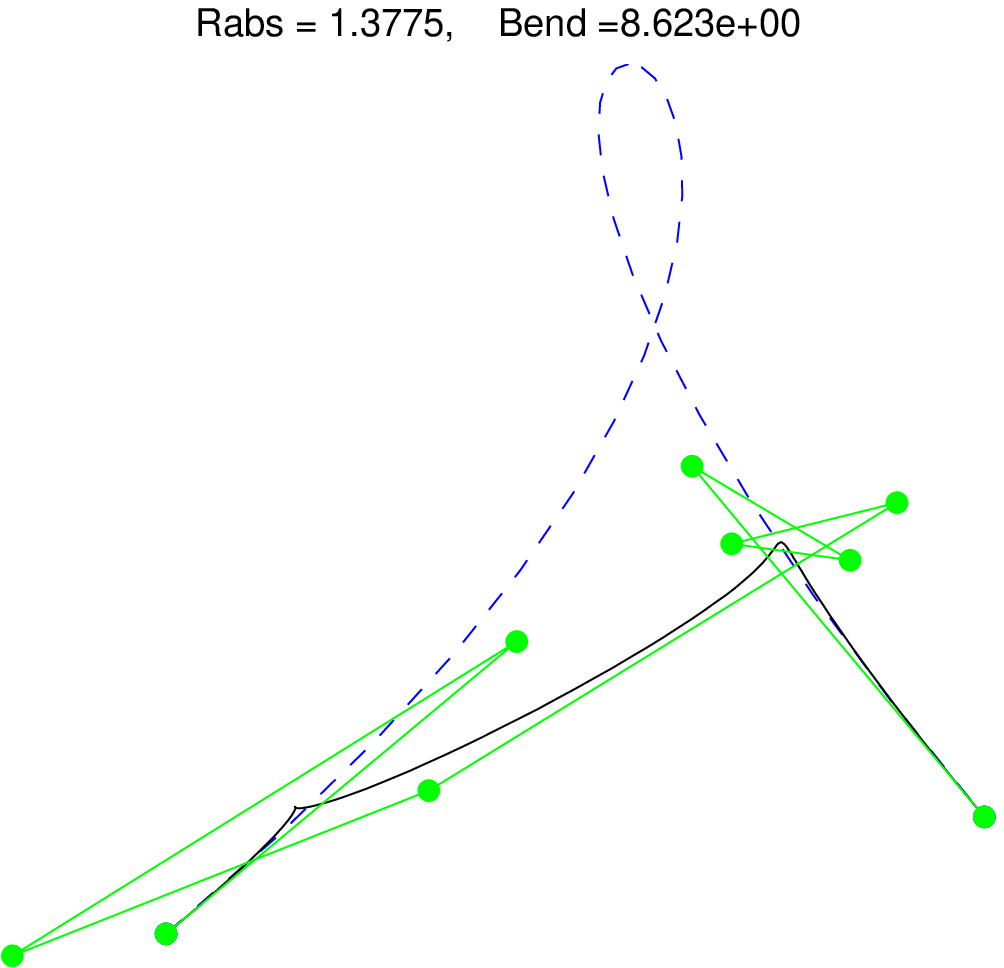}\hfill
\includegraphics[height=0.29\textwidth,valign=t]{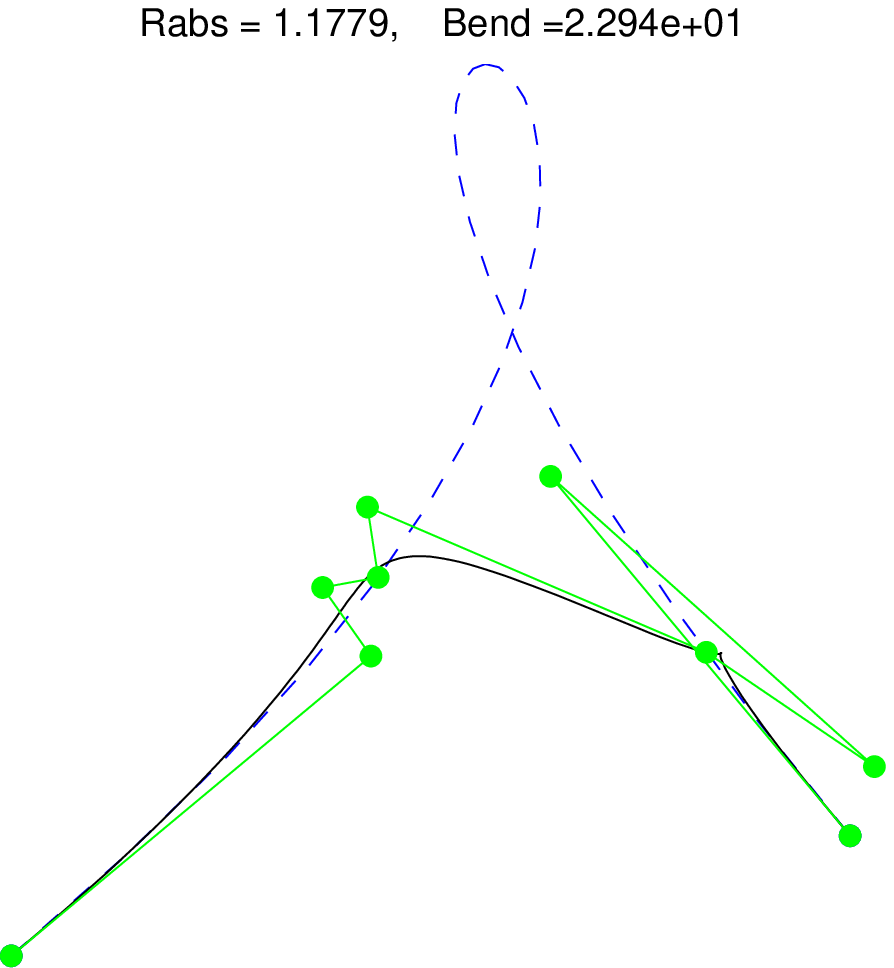}\hfill
\includegraphics[height=0.29\textwidth,valign=t]{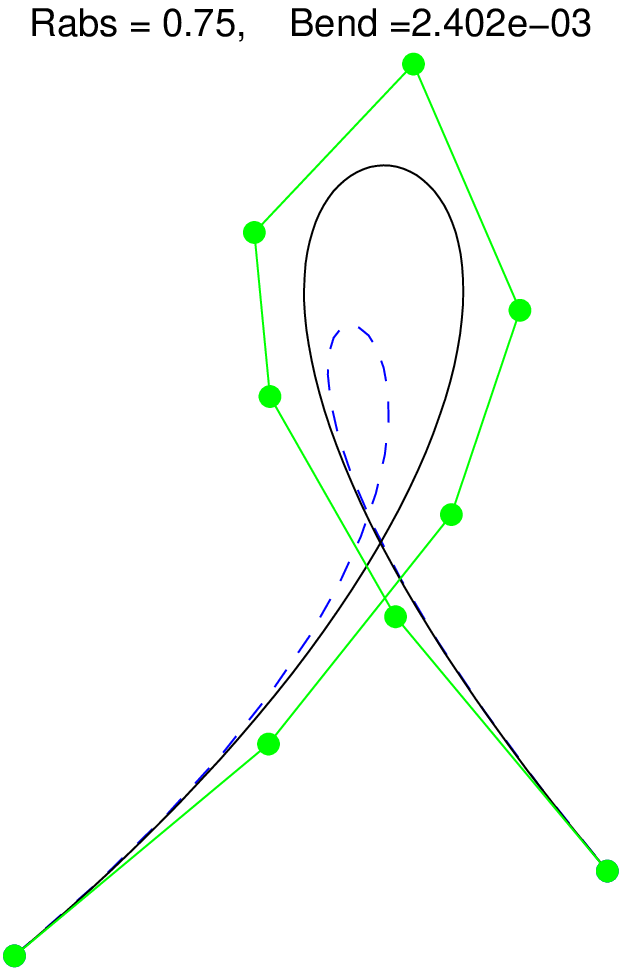}
\caption{Illustration of Example \ref{exm2}, case $+-$. Top: The two
conics $\mathcal{C}_A$ (blue) and $\mathcal{C}_B$ (red), a
degenerate conic of the pencil (green); Bottom: PH B-spline curves
(black) corresponding to the four intersection points and cubic
polynomial (dashed, blue) interpolating the given endpoints and
tangents, together with the values of the absolute rotation index
and bending energy of the PH B-Spline curves.} \label{fig:exmp2_2}
\end{figure}
\end{exm}

% CURRENT EXAMPLE 3
\begin{exm}
\label{exm3} Let us consider the initial data $\p^*_0=(0,0)$,
$\p^*_1=(1,0)$, $\d_0=(-3,1)$, $\d_1=(-3,-1)$. Using these data, we
can compute $u_0$, $v_0$, $u_3$, $v_3$ and for any combination of
these values to be considered, i.e. corresponding to signs $+,+$ and
$+,-$ in \eqref{solz0z3}, we should proceed to computing the
intersections of the two conics in equation \eqref{conics}.

The first couple of values corresponds to signs $+,+$. In this case,
any choice of $a\in(0,1)$ yields $I_1^{[A]}<0$, $I_2^{[A]}>0$ and
$I_2^{[B]}< 0$. The conic $\mathcal{C}_B$ is thus never imaginary
and there remains to determine $\kappa_0$ and $\kappa_1$ such that
$I_3^{[A]}>0$ and thus the conic $\mathcal{C}_A$ is a real ellipse.
On account of the symmetry of the initial data it seems logical to
set $a=0.5$. Accordingly, the implicit curve $I_3^{[A]}=0$ is
plotted in Figure \ref{fig:exm3}, from which we can infer that
$I_3^{[A]}>0$ in the region external to the plot. The curvatures of
the cubic polynomial which interpolates the initial data have both
the value $-0.569210$ and do not belong to the feasible region. We
will thus choose $\kappa_0$ and $\kappa_1$ so as to preserve the
symmetry of the data and be as close as possible to the curvatures
of the cubic polynomial. For example, setting
$\kappa_0=\kappa_1=-2.5$ we get the two conics $\mathcal{C}_A$ and
$\mathcal{C}_B$ displayed in Figure \ref{fig:exm3} (note that $B$ is
precisely a degenerate conic of the pencil.)

The conics have two intersection points each of which corresponds to
one of the two curves plotted in the bottom row of the figure.

We perform a similar study for the second couple of values,
corresponding to signs $+,-$ in \eqref{solz0z3}. For $a=0.5$, the
situation is similar as above, i.e. $I_1^{[A]}<0$, $I_2^{[A]}>0$ and
$I_2^{[B]}< 0$. The conic $\mathcal{C}_B$ is never imaginary, thus
we can limit ourselves to investigating when $A$ is an imaginary
conic. Taking into exam the plot of $I_3^{[A]}=0$ in Figure
\ref{fig:exm3_2}, it can be seen that the previously used values
$\kappa_0=\kappa_1=-2.5$ do not belong to the admissible region,
which is where $I_3^{[A]}>0$. A set of feasible curvature values,
e.g., is $\kappa_0=\kappa_1=-5$, which result in the conics
$\mathcal{C}_A$ and $\mathcal{C}_B$ in Figure \ref{fig:exm3_2}
($\mathcal{C}_B$ is a degenerate conic of the pencil). The PH
B-spline curves corresponding to the two intersection points are
displayed in the same figure. Comparing the obtained curves in
Figures \ref{fig:exm3} and \ref{fig:exm3_2}, we choose the first one
in Figure \ref{fig:exm3} according to its absolute rotation index
and bending energy.

 \begin{figure}[h!]
\centering
\includegraphics[height=0.24\textwidth,trim = 0 0cm 0 0cm, clip,valign=t]{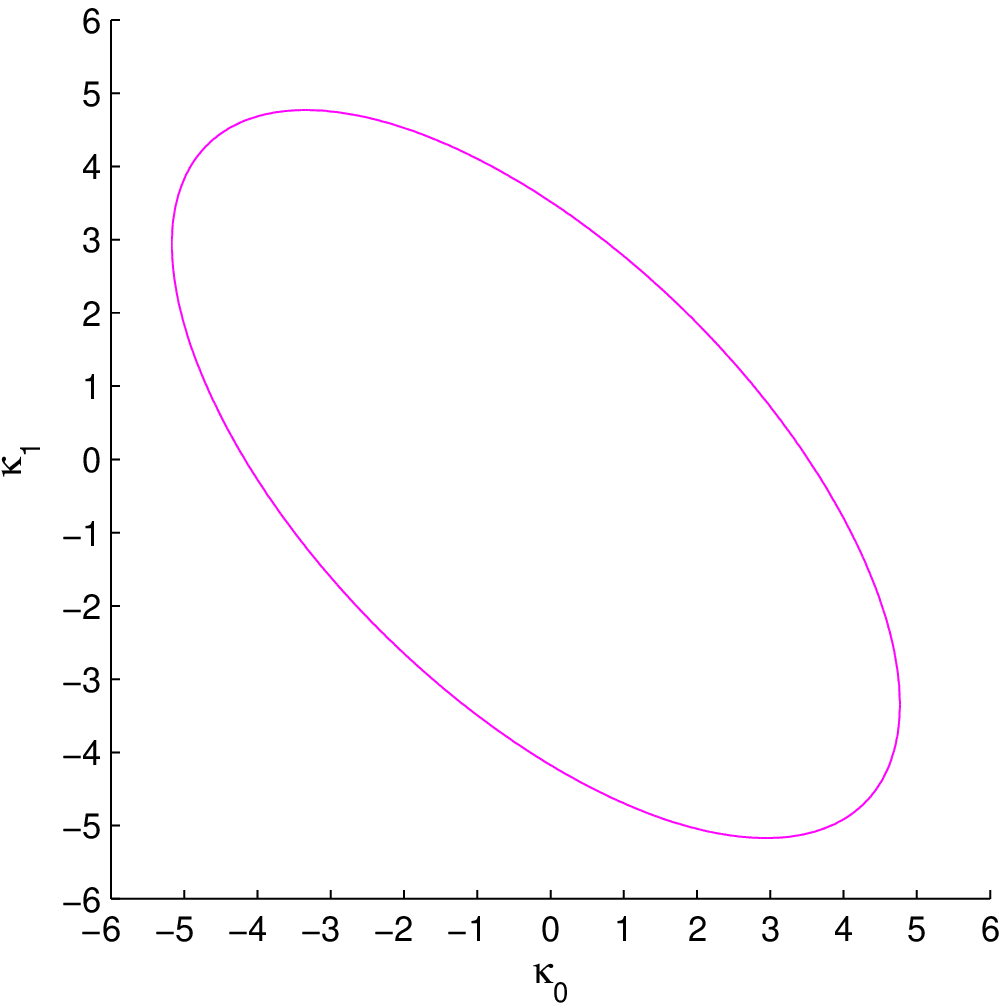}\hfill
\includegraphics[height=0.24\textwidth,trim = 0 0cm 0 0cm, clip,valign=t]{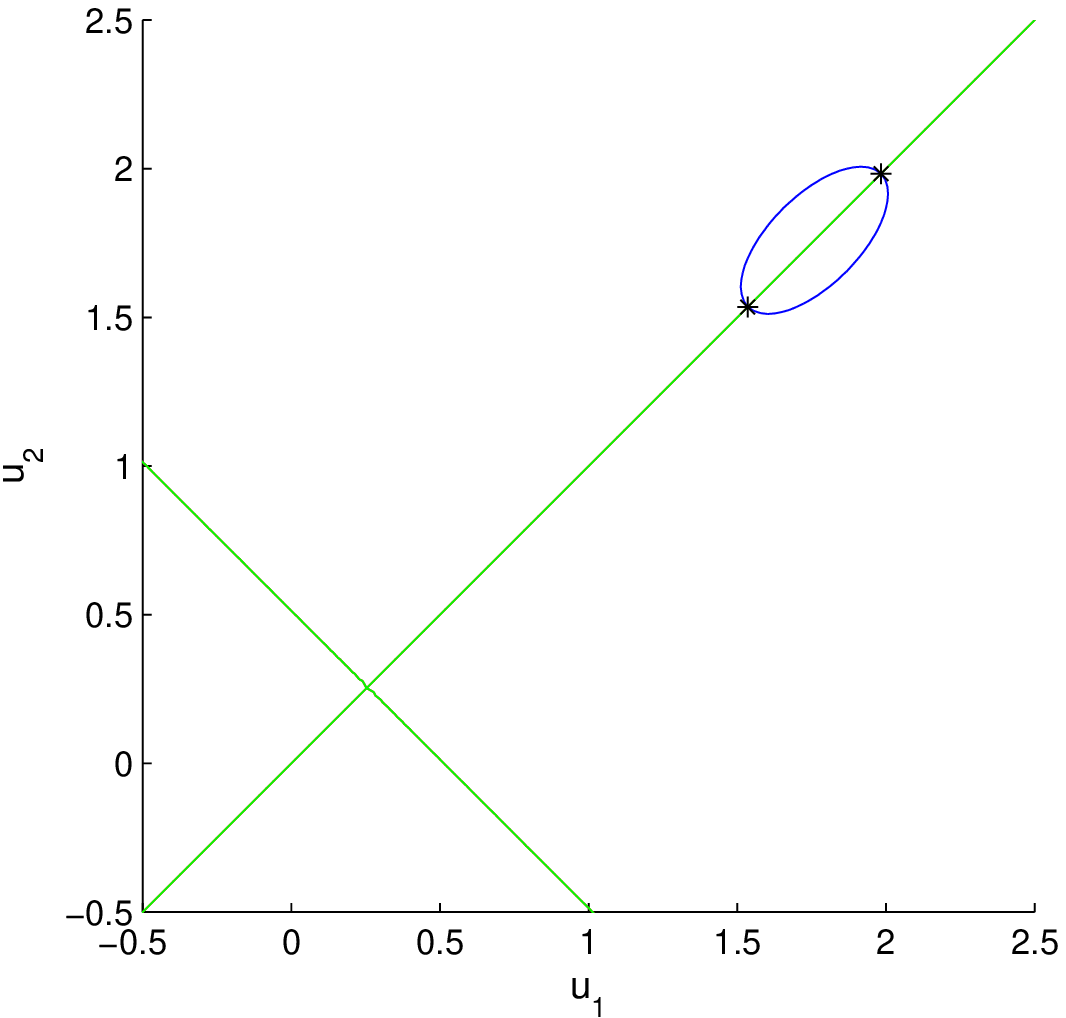}\hfill
\includegraphics[width=0.24\textwidth,valign=t]{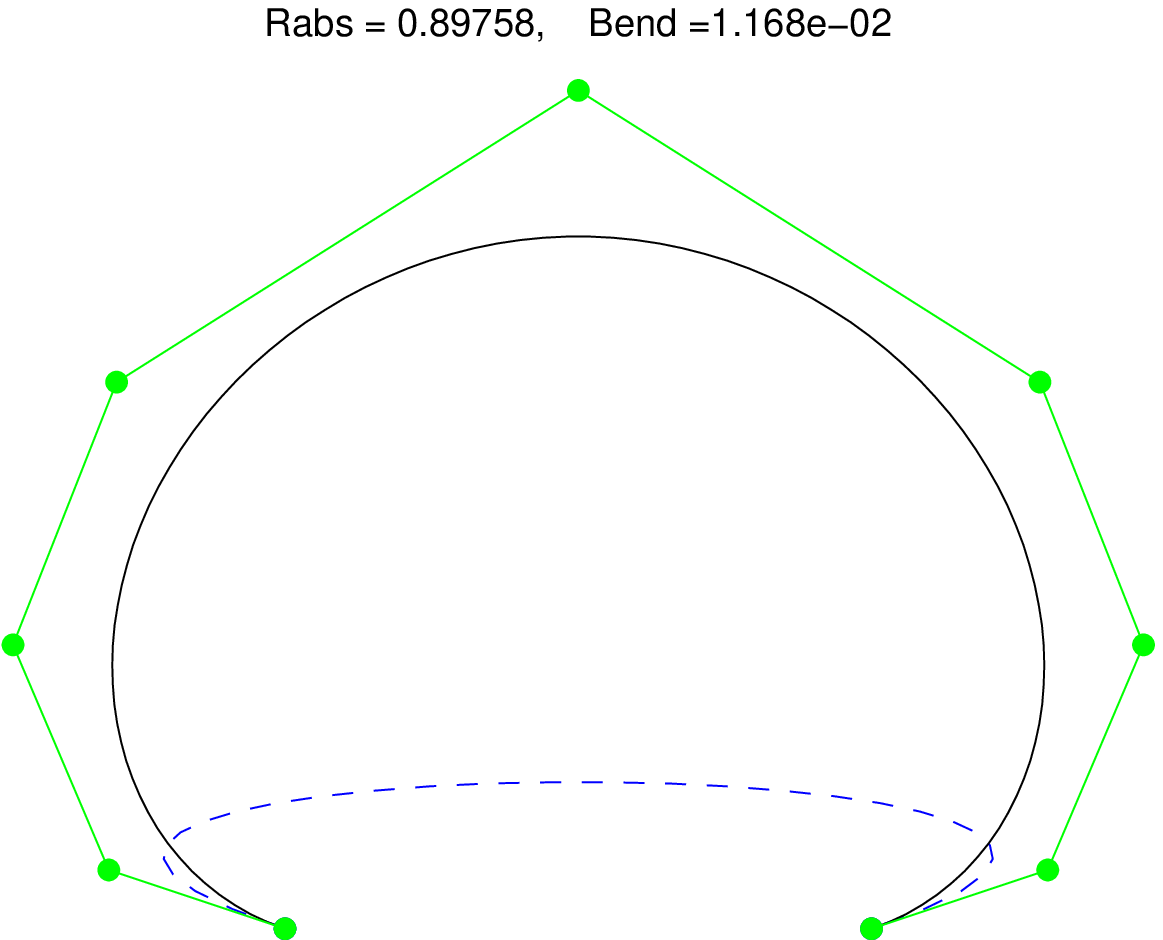}\hfill
\includegraphics[width=0.24\textwidth,valign=t]{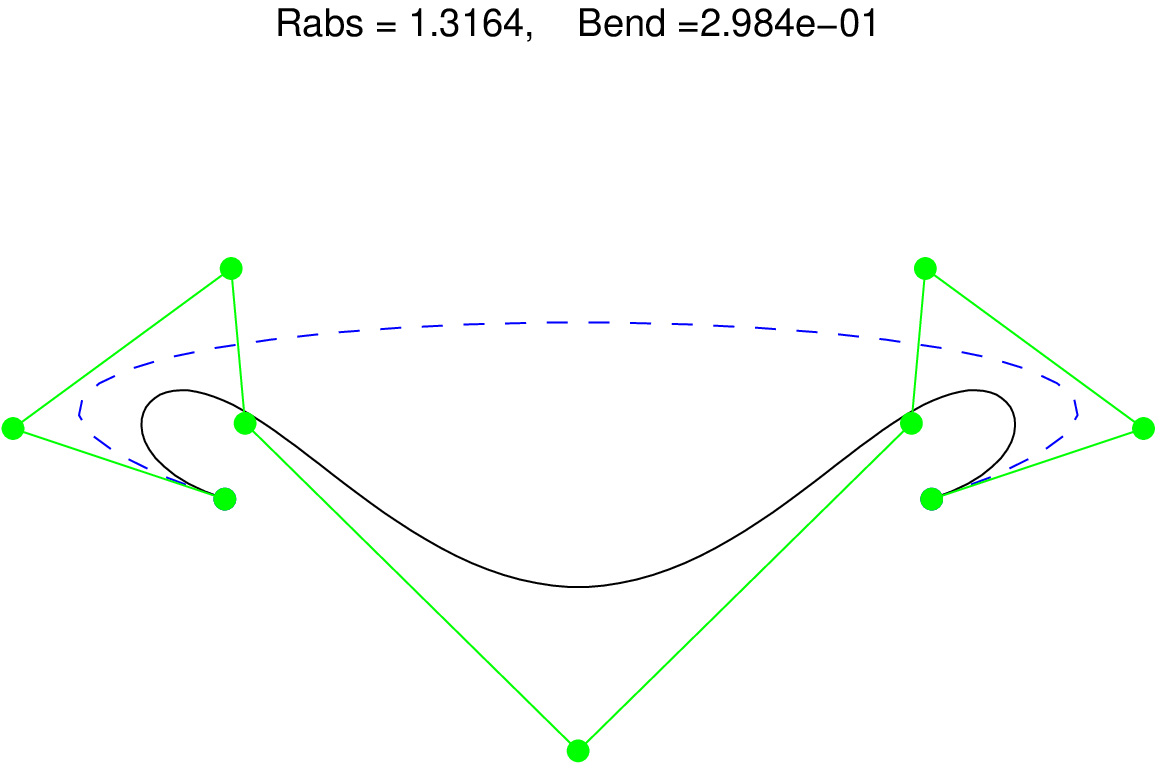}\\
\caption{Illustration of Example \ref{exm3}, case $++$. From left to
right: The conic $I_3^{[A]}=0$; The two conics $\mathcal{C}_A$ and
$\mathcal{C}_B$ (a degenerate conic of the pencil coincides with
$\mathcal{C}_B$); PH B-spline curves (black) corresponding to the
intersection points and cubic polynomial (dashed, blue)
interpolating the given endpoints and tangents, together with the
values of the absolute rotation index and bending energy of the PH
B-Spline curves.} \label{fig:exm3}

\vspace{1cm} \centering
\includegraphics[height=0.24\textwidth,trim = 0 0cm 0 0cm, clip,valign=t]{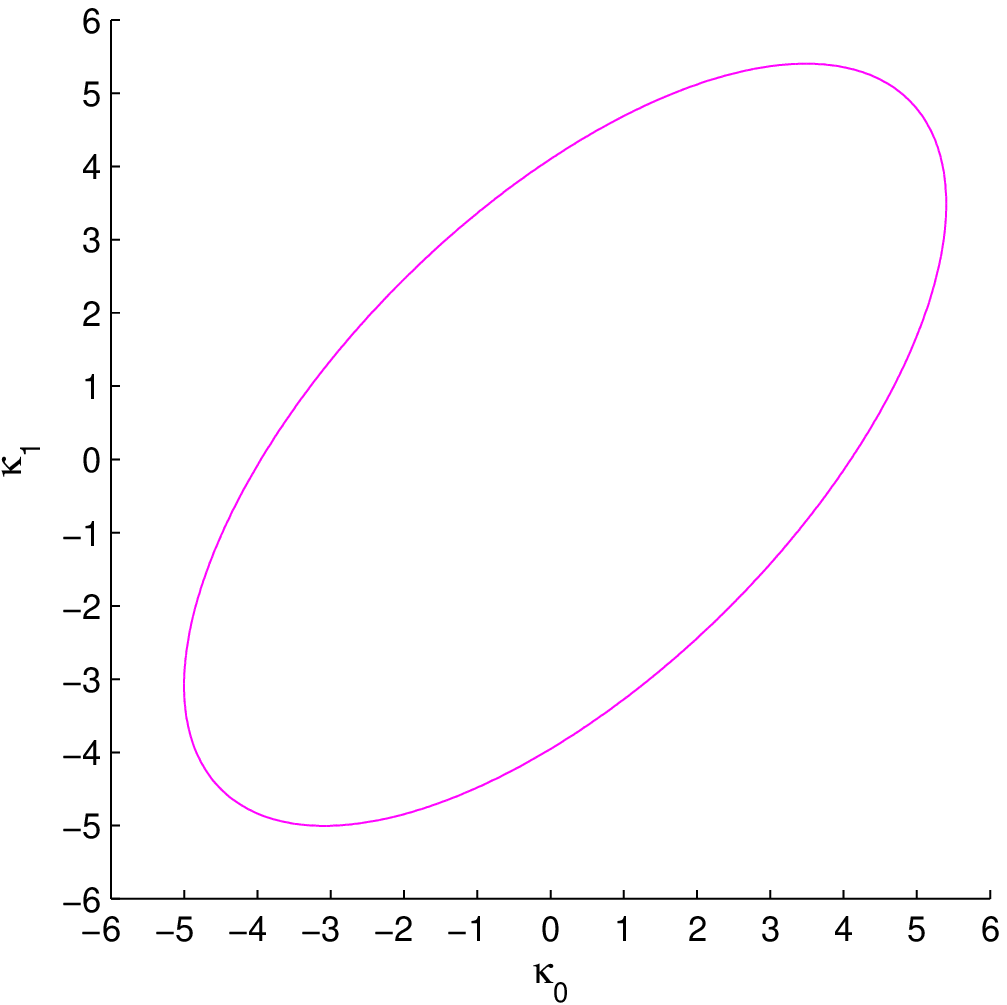}\hfill
\includegraphics[height=0.24\textwidth,trim = 0 0cm 0 0cm, clip,valign=t]{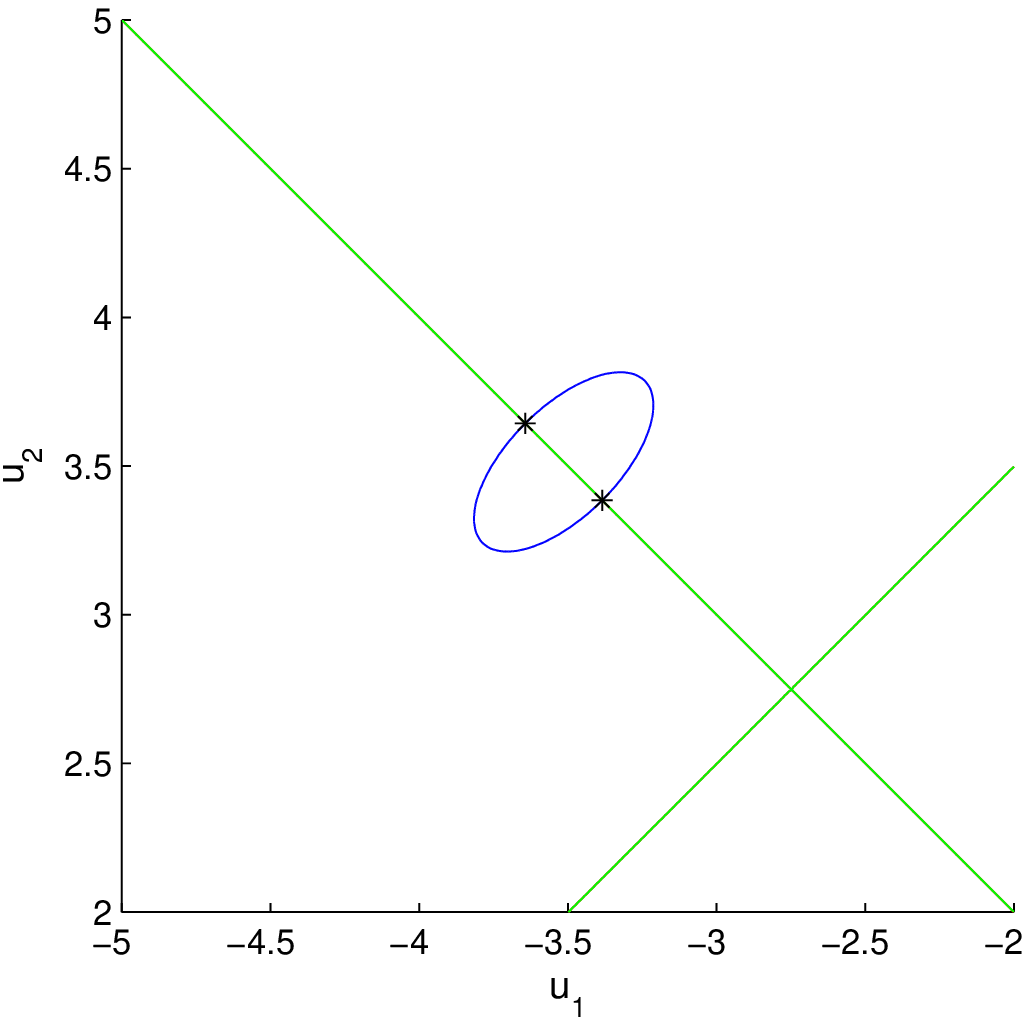}\hfill
\includegraphics[width=0.24\textwidth,valign=t]{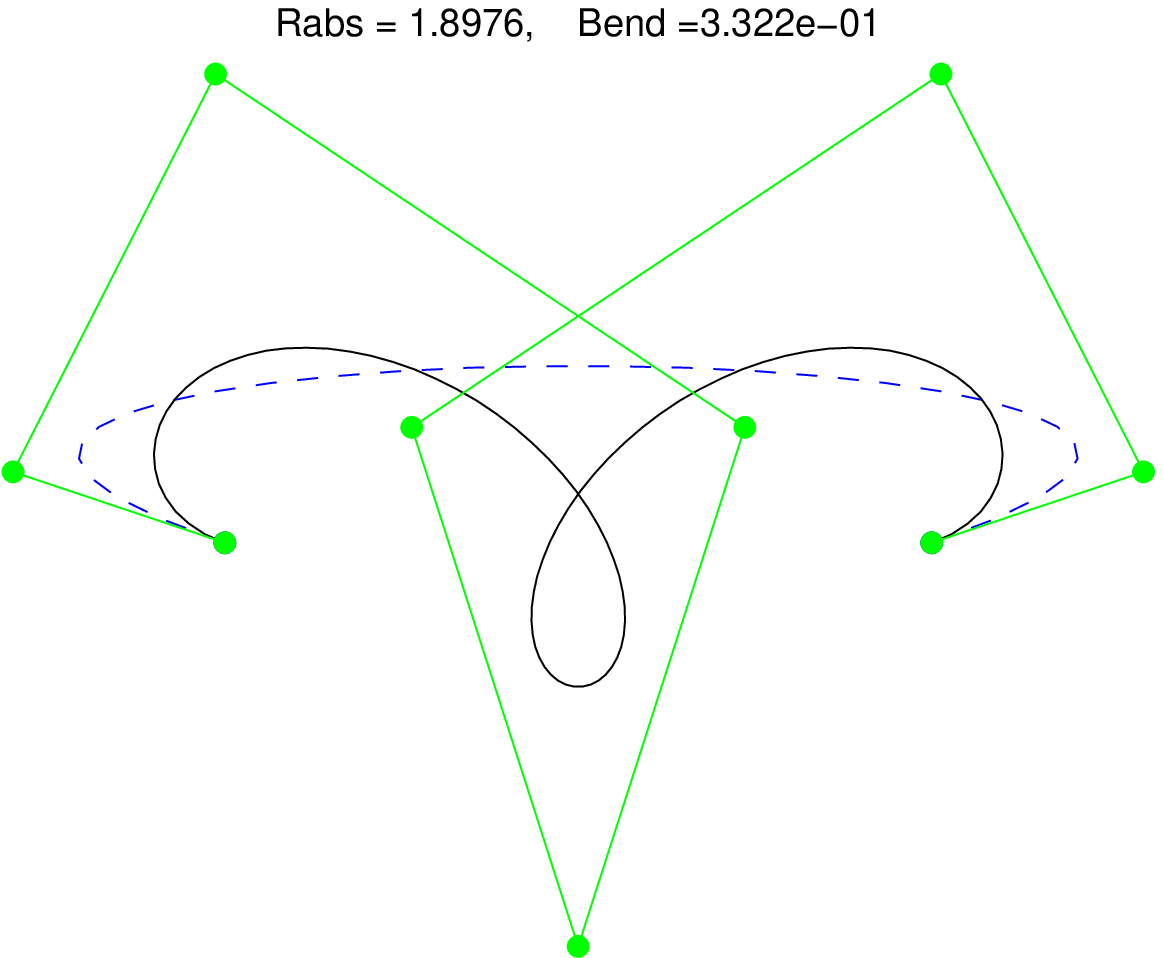}\hfill
\includegraphics[width=0.24\textwidth,valign=t]{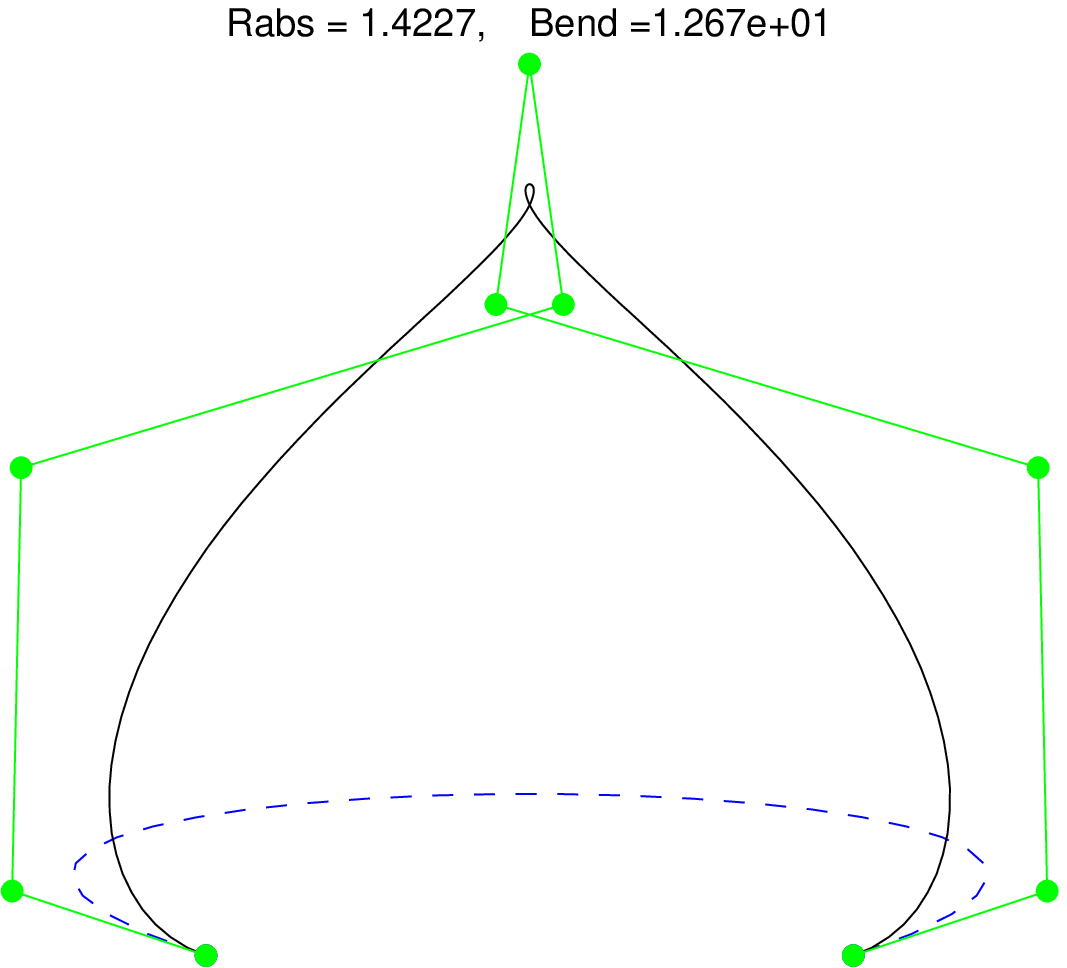}
\caption{Illustration of Example \ref{exm3}, case $+-$. From left to
right: The conic $I_3^{[A]}=0$; The two conics $\mathcal{C}_A$ and
$\mathcal{C}_B$ (a degenerate conic of the pencil coincides with
$\mathcal{C}_B$);  PH B-spline curves (black) corresponding to the
intersection points and cubic polynomial (dashed, blue)
interpolating the given endpoints and tangents, together with the
values of the absolute rotation index and bending energy of the PH
B-Spline curves.} \label{fig:exm3_2}
\end{figure}
\end{exm}

% CURRENT EXAMPLE 4
\begin{exm}
\label{exm4} Let us consider the initial data $\p^*_0=(0,5)$,
$\p^*_1=(-3,4)$, $\d_0=(25,-15)$, $\d_1=(25,-15)$. Using these data,
we can compute $u_0$, $v_0$, $u_3$, $v_3$ and for any combination of
these values to be considered, i.e. corresponding to signs $+,+$ and
$+,-$ in \eqref{solz0z3}, we should proceed to computing the
intersection of the two conics in equation (\ref{conics}).

The first couple of values corresponds to signs $+,+$. In this case,
any choice of $a$ in $(0,1)$ yields $I_1^{[A]}>0$, $I_2^{[A]}>0$,
$I_1^{[B]}<0$, $I_2^{[B]}>0$. Hence we set $a=0.5$ (this is a
natural choice given the symmetry of the data) and accordingly we
study the sign of $I_3^{[A]}$ and $I_3^{[B]}$. The equation
$I_3^{[B]}=0$ describes an imaginary conic and in particular
$I_3^{[B]}>0$ for any $\kappa_0$ and $\kappa_1$. This condition
guarantees that $B$ is a real ellipse. The implicit curve defined by
$I_3^{[A]}=0$ is the ellipse plotted in Figure \ref{fig:exm4}, left.
Considering that $I_1^{[A]}>0$, we shall then choose $\kappa_0$ and
$\kappa_1$ in such a way that $I_3^{[A]}<0$, which corresponds to
all values external to the ellipse displayed in Figure
\ref{fig:exm4}, left.

The curvatures of the cubic polynomial which interpolates the
initial data are equal to $-0.065371$ at $\p^*_0$ and $0.065371$ at
$\p^*_1$ and thus  they do not belong to the admissible region
$I_3^{[A]}<0$. By analogy with the curvatures of the cubic
polynomial, which have the same absolute value and opposite sign, we
will set $\kappa_0=-0.2$ and $\kappa_1=0.2$. This choice results in
the two conics $\mathcal{C}_A$ and $\mathcal{C}_B$ displayed in the
Figure \ref{fig:exm4}. The conics have two intersection points each
of which corresponds to one of the two curves plotted in the figure.

We perform a similar study for the second couple of values,
corresponding to signs $+,-$ in \eqref{solz0z3}. For any $a\in
(0,1)$, the situation is similar as above, i.e. $I_1^{[A]}>0$,
$I_2^{[A]}>0$ and $I_1^{[B]}<0$, $I_2^{[B]}>0$. The conic
$\mathcal{C}_B$ is never imaginary, thus we can limit ourselves to
investigating when $A$ is an imaginary conic. Setting $a=0.5$ and
taking into exam the plot of $I_3^{[A]}=0$ in Figure
\ref{fig:exm4_2}, left, it can be seen that the previously used
values $\kappa_0=-0.2$ and $\kappa_1=0.2$ do not belong to the
admissible region, which is where $I_3^{[A]}<0$. By similarity with
the curvatures of the cubic polynomial, we choose $\kappa_0$ and
$\kappa_1$ having equal absolute value and opposite sign, e.g.,
$\kappa_0=-0.4$ and $\kappa_1=0.4$. This results in the conics
$\mathcal{C}_A$ and $\mathcal{C}_B$ in Figure \ref{fig:exm4_2}. The
PH B-spline curves corresponding to the two intersection points are
depicted in the same figure. Overall, the four curves obtained in
this example have similar absolute rotation number and bending
energy.

\begin{figure}[h!]
\centering
\includegraphics[height=0.24\textwidth,trim = 0 0cm 0 0cm, clip,valign=t]{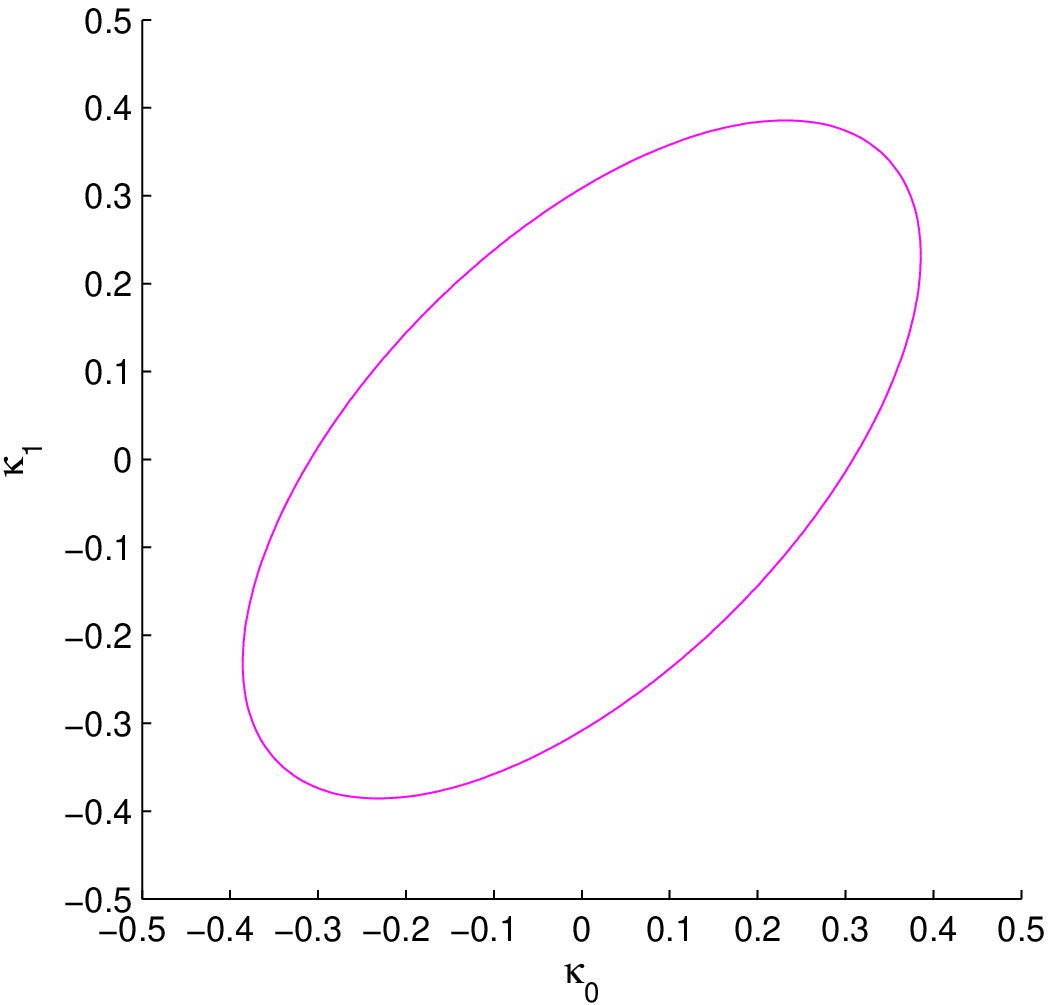}\hfill
\includegraphics[height=0.24\textwidth,trim = 0 0cm 0 0cm, clip,valign=t]{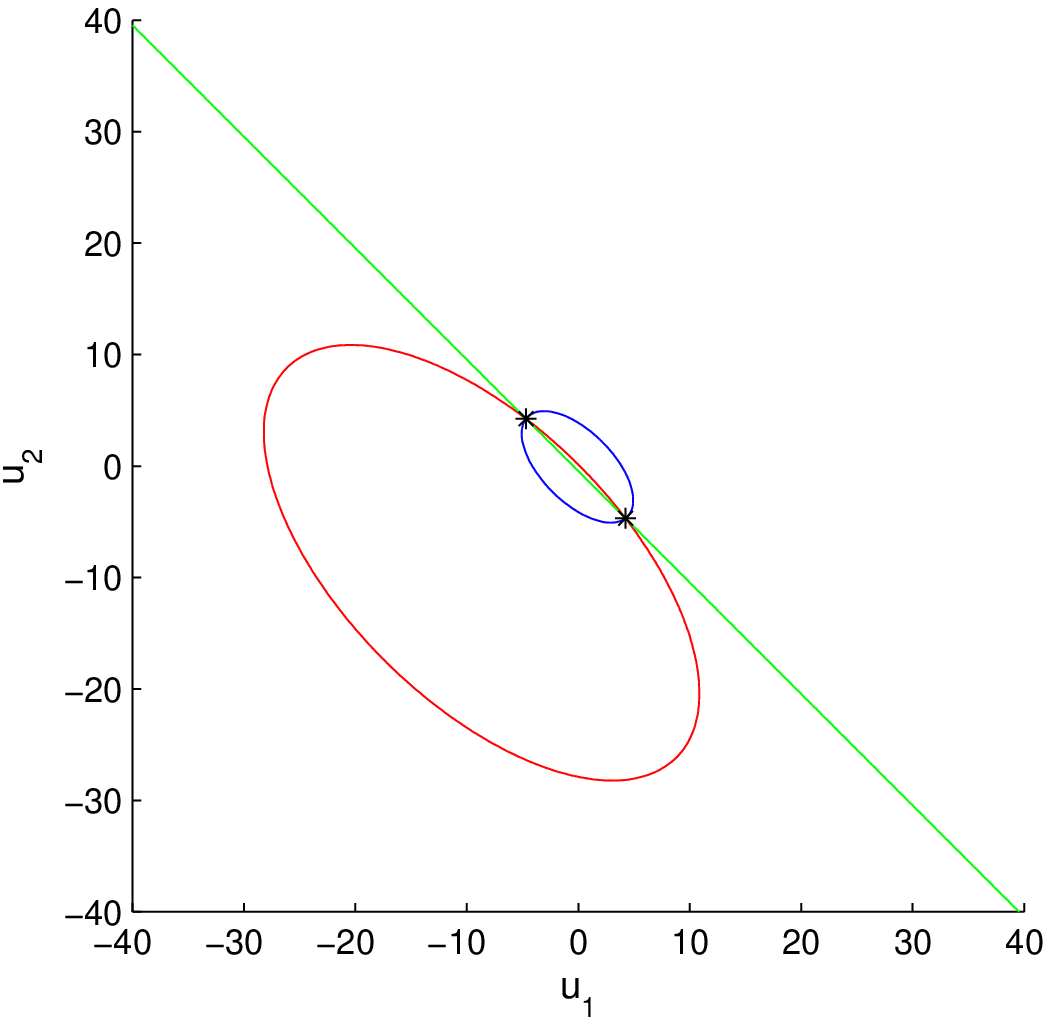}\hfill
\includegraphics[height=0.24\textwidth,valign=t]{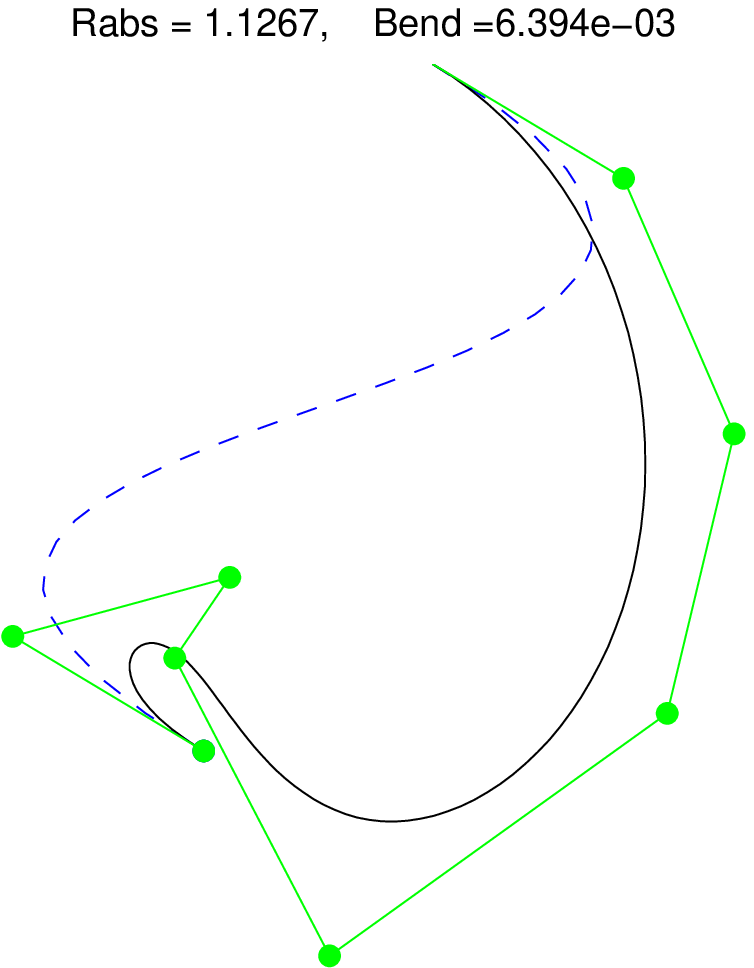}\hfill
\includegraphics[height=0.24\textwidth,valign=t]{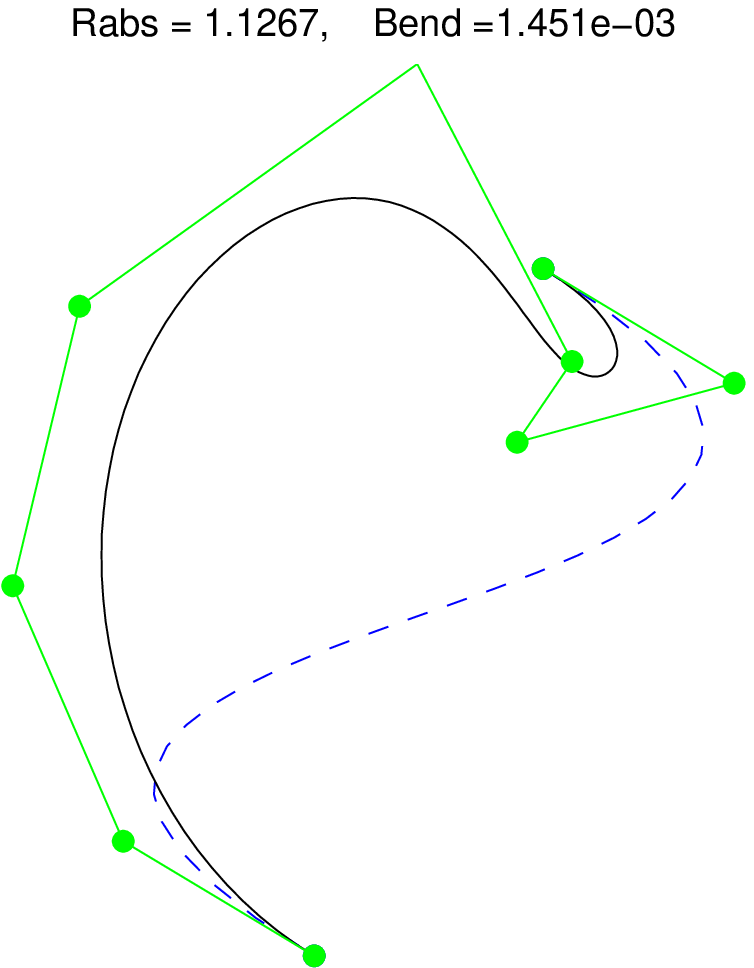}\\
\caption{Illustration of Example \ref{exm4}, case $++$. From left to
right: The conic $I_3^{[A]}=0$; The two conics $\mathcal{C}_A$ and
$\mathcal{C}_B$ (a degenerate conic of the pencil is formed by two
coincident lines); PH B-spline curves (black) corresponding to the
intersection points and cubic polynomial (dashed, blue)
interpolating the given endpoints and tangents, together with the
values of the absolute rotation index and bending energy of the PH
B-Spline curves.} \label{fig:exm4}

\vspace{1cm} \centering
\includegraphics[height=0.24\textwidth,trim = 0 0cm 0 0cm, clip,valign=t]{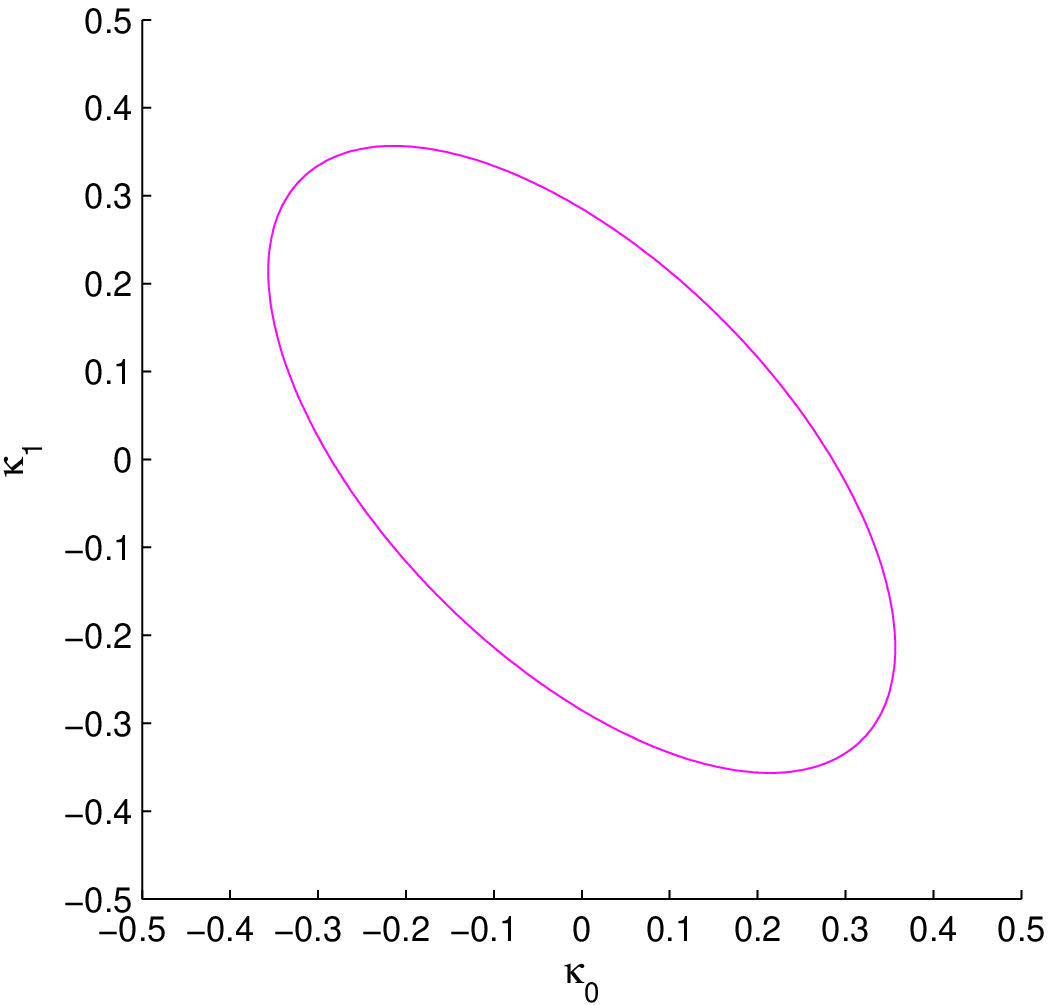}\hfill
\includegraphics[height=0.24\textwidth,trim = 0 0cm 0 0cm, clip,valign=t]{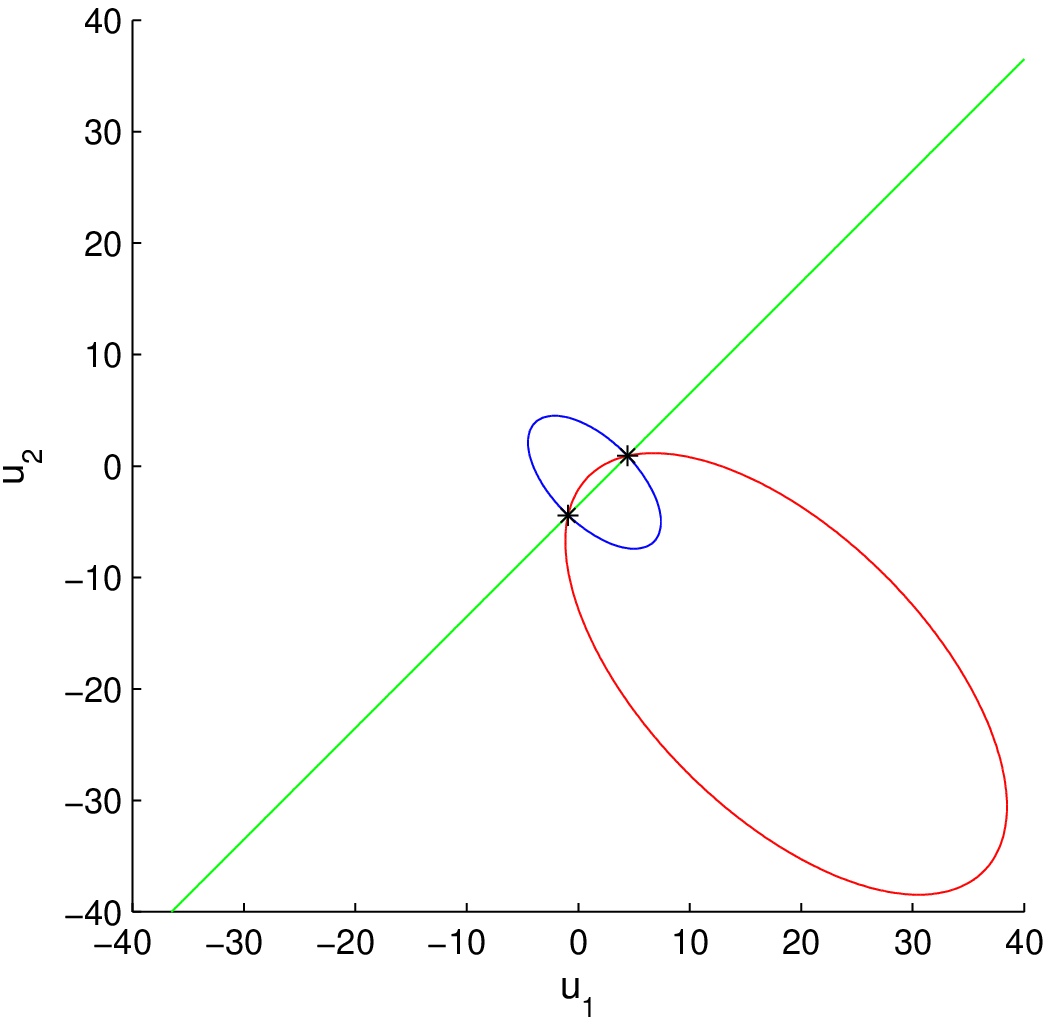}\hfill
\includegraphics[height=0.24\textwidth,valign=t]{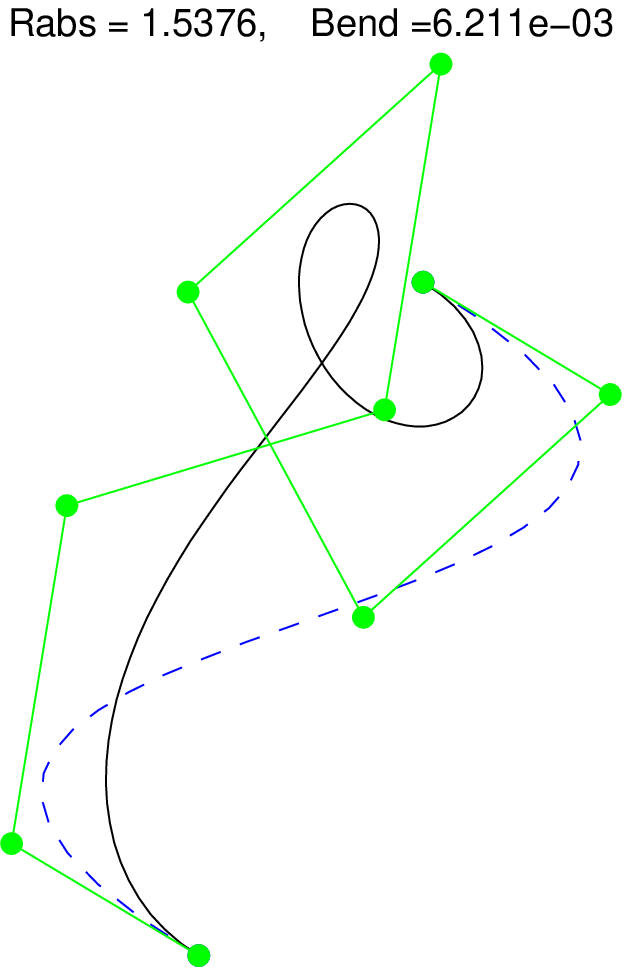}\hfill
\includegraphics[height=0.24\textwidth,valign=t]{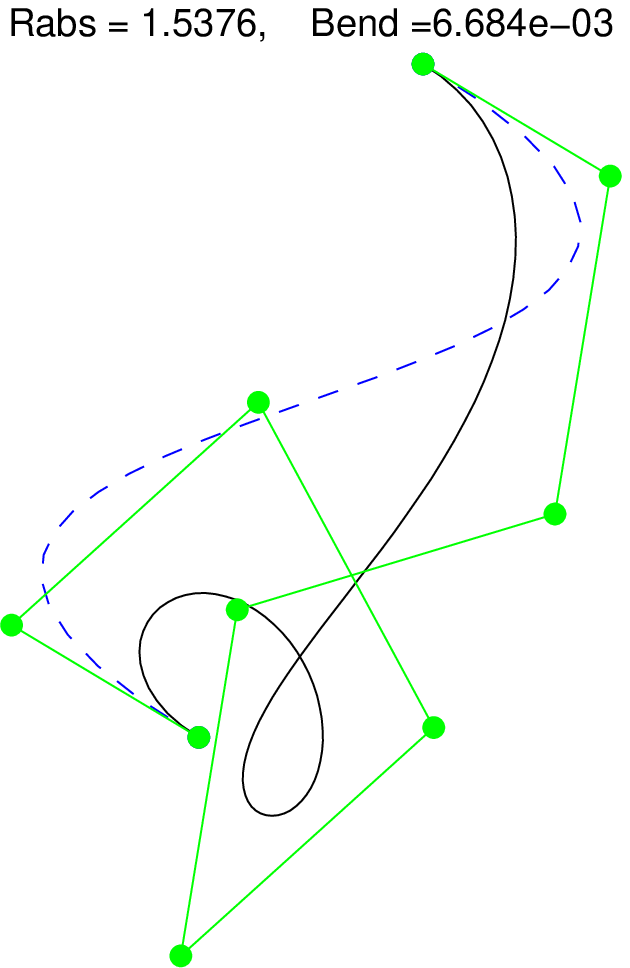}
\caption{Illustration of Example \ref{exm4}, case $+-$. From left to
right: The conic $I_3^{[A]}=0$; The two conics $\mathcal{C}_A$ and
$\mathcal{C}_B$ (a degenerate conic of the pencil is formed by two
coincident lines); PH B-spline curves (black) corresponding to the
intersection points and cubic polynomial (dashed, blue)
interpolating the given endpoints and tangents, together with the
values of the absolute rotation index and bending energy of the PH
B-Spline curves.} \label{fig:exm4_2}
\end{figure}
\end{exm}

\section{Conclusions and future work}
\label{sec7}

We have presented the construction of the very general
class of Pythagorean-Hodograph (PH) B-Spline curves.
 A computational strategy for efficiently calculating their control points as well as their
arc-length and offset curves has been proposed. Moreover, for the
cases of cubic and quintic clamped and closed PH B-Spline curves, an
explicit representation of their control points, their arc-length
and their offsets has been provided. Finally, clamped quintic PH
B-Spline curves have been exploited to solve a second order Hermite
interpolation problem, in order to show an example of practical
application of this new class of curves. This new class of curves
has a great potential for applications in computer-aided design and
manufacturing, robotics, motion control, path planning, computer
graphics, animation, and related fields.\\
The generalization of these planar PH B-Spline
curves to $3$-space is currently under investigation.
In virtue of their high generality and their unique advantages we
hope that these curves will be widely used in practice.
Among other things, in the future we envisage using these curves for
solving different interpolation and approximation problems in the
context of reverse engineering applications.
It might also be interesting to generalize the idea of the paper
\cite{faroukial2015} to understand if a given B-Spline curve is a
{\it PH} B-Spline curve and to recover its related complex pre-image
spline $\z(t)$.

\bibliographystyle{plain}
{\bibliography{biblio_phbsplines}}

\begin{thebibliography}{10}

\bibitem{ph}
G.~Albrecht and R.~T. Farouki.
\newblock Construction of {$C^2$} {Pythagorean}--ho\-do\-graph inter\-polating
  splines by the homotopy method.
\newblock {\em Advances in Computational Mathematics}, 5:417--442, 1996.

\bibitem{cinesi}
X.~Che, G.~Farin, Z.~Gao, and D.~Hansford.
\newblock The product of two {B}--spline functions.
\newblock {\em Advanced Materials Research}, 186:445--448, 2011.

\bibitem{deboor}
C.~de~Boor.
\newblock {\em A Practical Guide to Splines}.
\newblock Springer, Applied Mathematical Sciences Vol. 27, New York, Berlin,
  2001.

\bibitem{farouki94}
R.T. Farouki.
\newblock The conformal map $z \rightarrow z^2$ of the hodograph plane.
\newblock {\em Computer Aided Geometric Design}, 11:363--390, 1994.

\bibitem{farouki96}
R.T. Farouki.
\newblock The elastic bending energy of {Pythagorean-hodograph} curves.
\newblock {\em Computer Aided Geometric Design}, 13:227--241, 1996.

\bibitem{faroukial2015}
R.T. Farouki, C.~Giannelli, and A.~Sestini.
\newblock Identification and ``reverse engineering''of {Pythagorean-hodograph}
  curves.
\newblock {\em Computer Aided Geometric Design}, 34:21--36, 2015.

\bibitem{faroukiphbsplines}
R.T. Farouki, C.~Giannelli, and A.~Sestini.
\newblock Local modification of pythagorean-hodograph quintic spline curves
  using the {B-spline} form.
\newblock {\em Adv. Comput. Math.}, 42(1):199–--225, 2016.

\bibitem{faroukial2001}
R.T. Farouki, B.K. Kuspa, C.~Manni, and A.~Sestini.
\newblock Efficient solution of the complex quadratic tridiagonal system for
  {$C^2$ PH} quintic splines.
\newblock {\em Numerical Algorithms}, 27:35--60, 2001.

\bibitem{farouki95}
R.T. Farouki and C.A. Neff.
\newblock Hermite interpolation by pythagorean hodograph quintics.
\newblock {\em Math. Comp.}, 64(212):1589--1609, 1995.

\bibitem{faroukisakkalis}
R.T. Farouki and T.~Sakkalis.
\newblock Pythagorean hodographs.
\newblock {\em IBM J. Res. Develop.}, 34:736--752, 1990.

\bibitem{gallier}
J.~Gallier.
\newblock {\em Curves and surfaces in geometric modeling}.
\newblock Morgan Kaufmann Publishers, San Francisco, California, 2000.

\bibitem{hoschek}
J.~Hoschek and D.~Lasser.
\newblock {\em Fundamentals of Computer Aided Geometric Design}.
\newblock A K Peters, Wellesley, Massachusetts, 1996.

\bibitem{jaklical2}
G.~Jaklic, J.~Kozak, M.~Krajnc, V.~Vitrih, and E.~Zagar.
\newblock On interpolation by planar cubic {$G^2$} {Pythagorean-hodograph}
  spline curves.
\newblock {\em Mathematics of Computation}, 79(269):305--326, 2010.

\bibitem{jaklical}
G.~Jaklic, J.~Kozak, M.~Krajnc, V.~Vitrih, and E.~Zagar.
\newblock Interpolation by {$G^2$} quintic {Pythagorean-hodograph} curves.
\newblock {\em Numerical Mathematics: Theory, Methods and Applications},
  7(3):374--398, 2014.

\bibitem{juettler2001}
B.~J\"u{}ttler.
\newblock Hermite interpolation by pythagorean hodograph curves of degree
  seven.
\newblock {\em Mathematics of Computation}, 70:1089--1111, 2001.

\bibitem{kongal}
J.H. Kong, S.P. Jeong, and G.I. Kim.
\newblock Hermite interpolation using {PH} curves with undetermined junction
  points.
\newblock {\em Bull. Korean Math. Soc.}, 49:175--195, 2012.

\bibitem{martin}
R.S. Martin and J.H. Wilkinson.
\newblock Symmetric decomposition of positive definite band matrices.
\newblock {\em Numerische Mathematik}, 7:355--361, 1965.

\bibitem{morken}
K.~Morken.
\newblock Some identities for products and degree raising of splines.
\newblock {\em Constructive Approximation}, 7:195--208, 1991.

\bibitem{pelosial2007}
F.~Pelosi, M.L. Sampoli, R.T. Farouki, and C.~Manni.
\newblock A control polygon scheme for design of planar {$C^2$ PH} quintic
  spline curves.
\newblock {\em Computer Aided Geometric Design}, 24:28--52, 2007.

\bibitem{phtrig13}
L.~Romani, L.~Saini, and G.~Albrecht.
\newblock Algebraic-trigonometric {Pythagorean-Hodograph} curves and their use
  for {Hermite} interpolation.
\newblock {\em Advances in Computational Mathematics}, 40(5-6):977--1010, 2014.

\bibitem{schoenberg}
I.J. Schoenberg.
\newblock Contributions to the problem of approximation of equidistant data by
  analytic functions.
\newblock {\em Quart. Appl. Math.}, 4:45–--99 and 112–--141, 1946.

\bibitem{waltonmeek}
D.J. Walton and D.S. Meek.
\newblock {$G^2$} curve design with a pair of {Pythagorean Hodograph} quintic
  spiral segments.
\newblock {\em Computer Aided Geometric Design}, 24:267--285, 2007.

\end{thebibliography}

\newpage
\section{Appendix}
In this section we report some voluminous formulae of the preceding
sections.

\noindent {\bf Section \ref{sec5}}

 The weights and control points of the offset
curves $\r_h(t)$ of the clamped cubic PH B-Spline curves ($n=1$)
from (\ref{offsetn1clamped}) in section \ref{sec5n1clamped} are
given by:

$$
\begin{array}{l}
\gamma_{5k}=\sigma_{2k}, \quad k=0,1,...,m\,, \smallskip \\
\gamma_{5k+1}=\frac{3}{5} \sigma_{2k} + \frac{2}{5} \sigma_{2k+1}, \quad k=0,1,...,m-1\,,  \smallskip \\
\gamma_{5k+2}=\frac{3}{10} \sigma_{2k} +\frac{3}{5} \sigma_{2k+1} +\frac{1}{10} \sigma_{2k+2}, \quad k=0,1,...,m-1\,,  \smallskip \\
\gamma_{5k+3}=\frac{1}{10} \sigma_{2k} +\frac{3}{5} \sigma_{2k+1} +\frac{3}{10} \sigma_{2k+2}, \quad k=0,1,...,m-1\,,  \smallskip \\
\gamma_{5k+4}=\frac{2}{5} \sigma_{2k+1} +\frac{3}{5} \sigma_{2k+2},
\quad k=0,1,...,m-1\,,
\end{array}
$$
and
$$
\begin{array}{l}
\q_{5k}= \Big(\frac{d_{k+1}}{d_k+d_{k+1}} \r_{2k} + \frac{d_{k}}{d_k+d_{k+1}} \r_{2k+1}\Big) \sigma_{2k} -\ti \, h \p_{2k}, \quad k=0,1,...,m\,,  \medskip \medskip \\
\q_{5k+1}=\frac{3}{5} \, \r_{2k+1} \sigma_{2k} +\frac{2}{5} \, \left ( \frac{d_{k+1}}{d_k+d_{k+1}} \, \r_{2k} + \frac{d_k}{d_k + d_{k+1}} \, \r_{2k+1} \right ) \, \sigma_{2k+1}  \\
\quad \quad  -\ti \, h \big(\frac{3}{5} \p_{2k} + \frac{2}{5} \p_{2k+1} \big), \quad k=0,1,...,m-1\,,  \medskip \medskip\\
\q_{5k+2}=\frac{3}{10} \, \r_{2k+2} \sigma_{2k} + \frac{3}{5} \,
\r_{2k+1} \sigma_{2k+1} + \frac{1}{10} \, \left( \frac{d_{k+1}}{d_k
+ d_{k+1}} \, \r_{2k} + \frac{d_k}{d_k + d_{k+1}} \, \r_{2k+1}
\right) \sigma_{2k+2}
\\
\quad \quad -\ti \, h \big(\frac{3}{10} \p_{2k} + \frac{3}{5} \p_{2k+1} + \frac{1}{10} \p_{2k+2} \big), \quad k=0,1,...,m-1\,,  \medskip \medskip\\
\q_{5k+3}= \frac{1}{10} \, \left( \frac{d_{k+2}}{d_{k+1}+d_{k+2}}
\r_{2k+2} + \frac{d_{k+1}}{d_{k+1}+d_{k+2}} \r_{2k+3} \right)
\sigma_{2k}
+ \frac{3}{5} \, \r_{2k+2} \sigma_{2k+1} + \frac{3}{10} \, \r_{2k+1} \sigma_{2k+2} \\
\quad  -\ti \, h \big(\frac{1}{10} \p_{2k}  + \frac{3}{5} \p_{2k+1} + \frac{3}{10} \p_{2k+2} \big), \quad k=0,1,...,m-1\,,  \medskip \medskip\\
\q_{5k+4}=\frac{2}{5} \, \left ( \frac{d_{k+2}}{d_{k+1}+d_{k+2}} \r_{2k+2} + \frac{d_{k+1}}{d_{k+1} + d_{k+2}} \r_{2k+3} \right ) \sigma_{2k+1} + \frac{3}{5} \r_{2k+2} \sigma_{2k+2}  \\
\quad \quad -\ti \, h \big( \frac{2}{5} \p_{2k+1} + \frac{3}{5}
\p_{2k+2} \big), \quad k=0,1,...,m-1\,,
\end{array}
$$
with $d_0:=0$ and $d_{m+1}:=0$.

The weights and control points of the offset curves $\r_h(t)$ of the
clamped quintic PH B-Spline curves ($n=2$) from
(\ref{offsetn2clamped}) in section \ref{sec5n2clamped} are given by:

$$
\begin{array}{l}
\gamma_{8k}= \frac{4}{9} \sigma_{3k} + \frac{5}{9} \frac{\sigma_{3k+1} d_k + \sigma_{3k} d_{k+1}}{d_k+d_{k+1}}, \quad k=0,1,...,m-1\,, \smallskip \\
\gamma_{8k+1}= \frac{5}{9} \frac{\sigma_{3k+1} d_k + \sigma_{3k} d_{k+1}}{d_k+d_{k+1}} +\frac{4}{9} \sigma_{3k+1}, \quad k=0,1,...,m-1\,,  \smallskip \\
\gamma_{8k+2}= \frac{5}{18} \frac{\sigma_{3k+1} d_k + \sigma_{3k} d_{k+1}}{d_k + d_{k+1}} + \frac{5}{9} \sigma_{3k+1} + \frac16 \sigma_{3k+2}, \quad k=0,1,...,m-2\,,  \smallskip \\
\gamma_{8k+3}= \frac{5}{42} \, \frac{\sigma_{3k+1} d_k + \sigma_{3k}
d_{k+1}}{d_k+d_{k+1}} +\frac{10}{21} \sigma_{3k+1}
+ \frac{5}{14} \sigma_{3k+2} + \frac{1}{21} \sigma_{3k+3}, \quad k=0,1,...,m-2\,,  \smallskip \\
\gamma_{8k+4}=\frac{5}{126} \frac{\sigma_{3k+1} d_k + \sigma_{3k}
d_{k+1}}{d_k + d_{k+1}} + \frac{20}{63} \sigma_{3k+1} +
\frac{10}{21} \sigma_{3k+2}
+\frac{10}{63} \sigma_{3k+3} +\frac{1}{126} \frac{\sigma_{3k+4} d_{k+1} + \sigma_{3k+3} d_{k+2}}{d_{k+1}+d_{k+2}}, \quad k=0,1,...,m-2\,, \smallskip \\
\gamma_{8k+5}= \frac{1}{126} \frac{\sigma_{3k+1} d_k+ \sigma_{3k}
d_{k+1}}{d_k+d_{k+1}} + \frac{10}{63} \sigma_{3k+1} + \frac{10}{21}
\, \sigma_{3k+2}
+\frac{20}{63} \sigma_{3k+3} + \frac{5}{126} \frac{\sigma_{3k+4} d_{k+1} + \sigma_{3k+3} d_{k+2}}{d_{k+1} + d_{k+2}}, \quad k=0,1,...,m-2\,, \smallskip \\
\gamma_{8k+6}=\frac{1}{21} \sigma_{3k+1} + \frac{5}{14} \sigma_{3k+2} +\frac{10}{21} \sigma_{3k+3} + \frac{5}{42} \, \frac{\sigma_{3k+4} d_{k+1} + \sigma_{3k+3} d_{k+2}}{d_{k+1}+d_{k+2}}, \quad k=0,1,...,m-2\,, \smallskip \\
\gamma_{8k+7}=\frac{1}{6} \sigma_{3k+2} + \frac{5}{9} \sigma_{3k+3}
+ \frac{5}{18} \, \frac{\sigma_{3k+4} d_{k+1} + \sigma_{3k+3}
d_{k+2}}{d_{k+1}+d_{k+2}}, \quad k=0,1,...,m-2\,,
\end{array}
$$
and
$$
\begin{array}{lll}
\q_{8k}&=& \frac{4}{9} \, \left( \frac{\r_{3k+2} d_k^2+2 \r_{3k+1}
d_k d_{k+1}+ \r_{3k} d_{k+1}^2}{(d_k+d_{k+1})^2} \right) \,
\sigma_{3k} + \frac{5}{9} \, \left( \frac{\r_{3k+1} d_k+ \r_{3k}
d_{k+1}}{d_k+d_{k+1}} \right) \, \left( \frac{\sigma_{3k+1} d_k+
\sigma_{3k} d_{k+1}}{d_k+d_{k+1}} \right)
\smallskip \\
&-&\ti \, h \, \Big( \frac{4}{9} \p_{3k} + \frac{5}{9} \, \left(
\frac{\p_{3k+1} d_k+ \p_{3k} d_{k+1}}{d_k+d_{k+1}} \right)
\Big), \quad k=0,1,...,m-1\,, \smallskip \\
%......
\q_{8k+1}&=&\frac{5}{9} \, \left( \frac{\r_{3k+2} d_k + \r_{3k+1}
d_{k+1}}{d_k+d_{k+1}} \right) \, \left( \frac{\sigma_{3k+1} d_k+
\sigma_{3k} d_{k+1}}{d_k+d_{k+1}} \right)
+\frac{4}{9} \, \left( \frac{\r_{3k+2} d_k^2 +2 \r_{3k+1} d_k d_{k+1}+ \r_{3k} d_{k+1}^2}{(d_k+d_{k+1})^2} \right) \, \sigma_{3k+1} \smallskip \\
&-&\ti \, h \, \Big( \frac{4}{9} \p_{3k+1} + \frac{5}{9} \, \left( \frac{\p_{3k+1} d_k+\p_{3k} d_{k+1}}{d_k+d_{k+1}} \right) \Big), \quad k=0,1,...,m-1\,,  \smallskip \\
%......
\end{array}
$$
$$
\begin{array}{lll}
\q_{8k+2} &=& \frac{5}{18} \, \r_{3k+2} \, \left( \frac{
\sigma_{3k+1} d_k + \sigma_{3k} d_{k+1}}{d_k+d_{k+1}} \right)
+\frac{5}{9} \, \left( \frac{\r_{3k+2} d_k + \r_{3k+1} d_{k+1}}{d_k+d_{k+1}} \right) \, \sigma_{3k+1} \smallskip \\
&+& \frac{1}{6} \, \left( \frac{\r_{3k+2} d_k^2 + 2\r_{3k+1} d_k d_{k+1} + \r_{3k} d_{k+1}^2}{(d_k+d_{k+1})^2} \right) \, \sigma_{3k+2} \smallskip \\
&-& \ti \, h \, \Big( \frac16 \p_{3k+2} + \frac{5}{9} \p_{3k+1}
+\frac{5}{18} \, \left(\frac{\p_{3k+1} d_k + \p_{3k} d_{k+1}}{d_k+d_{k+1}} \right) \Big), \quad k=0,1,...,m-2\,,  \smallskip \\
%......
\q_{8k+3}&=&\frac{5}{42} \, \r_{3k+3} \, \left(\frac{\sigma_{3k+1}
d_k + \sigma_{3k} d_{k+1}}{d_k+d_{k+1}}\right) +\frac{10}{21} \,
\r_{3k+2} \sigma_{3k+1}
+\frac{5}{14} \, \left( \frac{\r_{3k+2} d_k+ \r_{3k+1} d_{k+1}}{d_k+d_{k+1}}\right) \, \sigma_{3k+2} \smallskip \\
&+& \frac{1}{21} \, \left( \frac{\r_{3k+2} d_k^2+2 \r_{3k+1} d_k
d_{k+1}+ \r_{3k} d_{k+1}^2}{(d_k+d_{k+1})^2} \right) \,
\sigma_{3k+3}
\smallskip \\
&-& \ti h \Big(\frac{1}{21} \p_{3k+3} + \frac{5}{14} \p_{3k+2}
+\frac{10}{21} \p_{3k+1}
+\frac{5}{42} \, \left( \frac{\p_{3k+1} d_k+ \p_{3k} d_{k+1}}{d_k+d_{k+1}} \right) \Big), \quad k=0,1,...,m-2\,,  \smallskip \\
%......
\q_{8k+4}&=& \frac{5}{126} \, \left( \frac{\r_{3k+4}
d_{k+1}+\r_{3k+3} d_{k+2}}{d_{k+1}+d_{k+2}} \right)
\, \left( \frac{\sigma_{3k+1} d_k+\sigma_{3k} d_{k+1}}{d_k+d_{k+1}} \right) \smallskip \\
&+& \frac{20}{63} \r_{3k+3} \sigma_{3k+1} + \frac{10}{21} \r_{3k+2}
\sigma_{3k+2} +
\frac{10}{63} \, \left( \frac{\r_{3k+2} d_k+ \r_{3k+1} d_{k+1}}{d_k+d_{k+1}} \right) \, \sigma_{3k+3} \smallskip \\
&+& \frac{1}{126} \, \left( \frac{\r_{3k+2} d_k^2 +2 \r_{3k+1} d_k
d_{k+1} +\r_{3k} d_{k+1}^2}{(d_k+d_{k+1})^2} \right )
\, \left( \frac{\sigma_{3k+4} d_{k+1}+\sigma_{3k+3} d_{k+2}}{d_{k+1}+d_{k+2}} \right) \smallskip \\
&-& \ti h \Big( \frac{1}{126} \, \left( \frac{\p_{3k+4} d_{k+1}+
\p_{3k+3} d_{k+2}}{d_{k+1}+d_{k+2}} \right) +\frac{10}{63} \p_{3k+3}
+ \frac{10}{21} \p_{3k+2} + \frac{20}{63} \p_{3k+1}
+\frac{5}{126} \, \left( \frac{\p_{3k+1} d_k+\p_{3k} d_{k+1}}{d_k+d_{k+1}} \right)  \Big),\smallskip \\
&& k=0,1,...,m-2\,, \smallskip \\
%......
\q_{8k+5}&=& \frac{1}{126} \, \left( \frac{\r_{3k+5} d_{k+1}^2+2
\r_{3k+4} d_{k+1} d_{k+2}+\r_{3k+3} d_{k+2}^2}{(d_{k+1}+d_{k+2})^2}
\right)
\, \left( \frac{\sigma_{3k+1} d_k+ \sigma_{3k} d_{k+1}}{d_k+d_{k+1}} \right) \smallskip \\
&+& \frac{10}{63}  \, \left( \frac{\r_{3k+4} d_{k+1}+ \r_{3k+3}
d_{k+2}}{d_{k+1}+d_{k+2}} \right) \, \sigma_{3k+1}
+ \frac{10}{21} \r_{3k+3} \sigma_{3k+2} + \frac{20}{63} \r_{3k+2} \sigma_{3k+3} \smallskip \\
&+& \frac{5}{126} \, \left( \frac{\r_{3k+2} d_{k}+\r_{3k+1}
d_{k+1}}{d_k+d_{k+1}} \right)
\, \left( \frac{\sigma_{3k+4} d_{k+1}+ \sigma_{3k+3} d_{k+2}}{d_{k+1}+d_{k+2}} \right) \smallskip \\
&-&\ti h \Big( \frac{5}{126} \, \left( \frac{\p_{3k+4}
d_{k+1}+\p_{3k+3} d_{k+2}}{d_{k+1}+d_{k+2}} \right) + \frac{20}{63}
\p_{3k+3} + \frac{10}{21} \p_{3k+2} + \frac{10}{63} \p_{3k+1} +
\frac{1}{126} \, \left(\frac{\p_{3k+1} d_{k} + \p_{3k}
d_{k+1}}{d_k+d_{k+1}} \right)
 \Big), \smallskip\\
&& k=0,1,...,m-2\,,
\end{array}
$$
$$
\begin{array}{lll}
%......
\q_{8k+6}&=&\frac{1}{21} \, \left(\frac{\r_{3k+5} d_{k+1}^2+2
\r_{3k+4} d_{k+1} d_{k+2}+ \r_{3k+3}
d_{k+2}^2}{(d_{k+1}+d_{k+2})^2}\right) \, \sigma_{3k+1}
+ \frac{5}{14} \, \left( \frac{\r_{3k+4} d_{k+1} + \r_{3k+3} d_{k+2}}{d_{k+1}+d_{k+2}} \right) \,  \sigma_{3k+2} \smallskip \\
&+&\frac{10}{21} \r_{3k+3} \sigma_{3k+3} + \frac{5}{42} \, \r_{3k+2}
\, \left( \frac{\sigma_{3k+4} d_{k+1} + \sigma_{3k+3}
d_{k+2}}{d_{k+1}+d_{k+2}} \right)
\smallskip \\
&-& \ti h \Big( \frac{5}{42} \, \left( \frac{\p_{3k+4} d_{k+1}+
\p_{3k+3} d_{k+2}}{d_{k+1}+d_{k+2}} \right)
+ \frac{10}{21} \p_{3k+3} + \frac{5}{14} \p_{3k+2} + \frac{1}{21} \p_{3k+1} \Big), \quad k=0,1,...,m-2\,, \smallskip \\
%......
\q_{8k+7}&=& \frac{1}{6} \, \left( \frac{\r_{3k+5} d_{k+1}^2 +2
\r_{3k+4} d_{k+1} d_{k+2}+ \r_{3k+3} d_{k+2}^2}{(d_{k+1}+d_{k+2})^2}
\right) \,  \sigma_{3k+2}
+ \frac{5}{9} \, \left( \frac{ \r_{3k+4} d_{k+1} + \r_{3k+3} d_{k+2}}{d_{k+1}+d_{k+2}}\right) \, \sigma_{3k+3} \smallskip \\
&+& \frac{5}{18} \, \r_{3k+3} \, \left( \frac{\sigma_{3k+4} d_{k+1}
+ \sigma_{3k+3} d_{k+2}}{d_{k+1}+d_{k+2}} \right) -\ti h
\Big(\frac{5}{18} \, \left( \frac{\p_{3k+4} d_{k+1}+\p_{3k+3}
d_{k+2}}{d_{k+1}+d_{k+2}} \right)
+ \frac{5}{9} \p_{3k+3} + \frac{1}{6} \p_{3k+2} \Big), \smallskip \\
&& k=0,1,...,m-2\,.
\end{array}
$$

The weights and control points of the offset curves $\r_h(t)$ of the
closed cubic PH B-Spline curves ($n=1$) from (\ref{offsetn1closed})
in section \ref{sec5n1closed} are given by:

$$
\begin{array}{l}
\gamma_0=\gamma_1=\gamma_2=\gamma_3=0, \smallskip\\
\gamma_4=\frac{d_0}{10(d_0 + d_1)} \, \sigma_1, \smallskip\\
\gamma_5= \frac{3}{10} \, \sigma_1, \smallskip\\
\gamma_6= \frac{3}{5} \, \sigma_1, \smallskip\\
\gamma_{5k+7}=\sigma_{2k+1}, \quad k=0,1,...,m+1\,, \smallskip\\
\gamma_{5k+8}= \frac{3}{5} \, \sigma_{2k+1} + \frac{2}{5} \, \sigma_{2k+2}, \quad k=0,1,...,m\,, \smallskip\\
\gamma_{5k+9}= \frac{3}{10} \, \sigma_{2k+1} + \frac{3}{5} \, \sigma_{2k+2} + \frac{1}{10} \, \sigma_{2k+3}, \quad k=0,1,...,m\,, \smallskip\\
\gamma_{5k+10}=\frac{1}{10} \, \sigma_{2k+1} + \frac{3}{5} \sigma_{2k+2} + \frac{3}{10} \, \sigma_{2k+3}, \quad k=0,1,...,m\,, \smallskip\\
\gamma_{5k+11}=\frac{2}{5} \, \sigma_{2k+2} + \frac{3}{5} \, \sigma_{2k+3}, \quad k=0,1,...,m\,, \smallskip\\
\gamma_{5m+13}= \frac{3}{5} \, \sigma_{2m+3}, \smallskip\\
\gamma_{5m+14}= \frac{3}{10} \, \sigma_{2m+3}, \smallskip\\
\end{array}
$$
$$
\begin{array}{l}
\gamma_{5m+15}= \frac{d_{m+4}}{10 (d_{m+3}+d_{m+4})} \, \sigma_{2m+3},  \smallskip\\
\gamma_{5m+16}=\gamma_{5m+17}=\gamma_{5m+18}=\gamma_{5m+19}=0\,,
\end{array}
$$
and
$$
\begin{array}{l}
\q_0=\q_1=\q_2=\q_3=0, \smallskip \\
\q_4= \frac{d_0}{10(d_0 + d_1)} \, ( \r_0 \sigma_1 - \ti h \p_1), \smallskip  \\
\q_5= \frac{3}{10}  \, (\r_0 \sigma_1   - \ti h \p_1), \smallskip \\
\q_6= \frac{3}{5} \, (\r_1 \sigma_1   - \ti h \p_1), \smallskip \\
\q_{5k+7}= \Big( \frac{d_{k+2}}{d_{k+1}+d_{k+2}} \r_{2k+1}  +
\frac{d_{k+1}}{d_{k+1}+d_{k+2}} \r_{2k+2} \Big)
\sigma_{2k+1} - \ti h \p_{2k+1}, \quad k=0,1,...,m+1 \smallskip \\
\q_{5k+8}= \frac{3}{5} \, \r_{2k+2} \sigma_{2k+1}  + \frac{2}{5} \,
\Big( \frac{d_{k+2}}{d_{k+1}+d_{k+2}} \r_{2k+1}
+ \frac{d_{k+1}}{d_{k+1}+d_{k+2}} \r_{2k+2} \Big) \sigma_{2k+2} \smallskip \\
\qquad - \ti h \big( \frac{3}{5} \p_{2k+1} + \frac{2}{5} \p_{2k+2} \big), \quad k=0,1,...,m\,,  \smallskip \\
\q_{5k+9}= \frac{3}{10} \r_{2k+3}  \sigma_{2k+1} + \frac{3}{5}
\r_{2k+2} \sigma_{2k+2} + \frac{1}{10} \, \Big(
\frac{d_{k+2}}{d_{k+1}+d_{k+2}} \r_{2k+1} +
\frac{d_{k+1}}{d_{k+1}+d_{k+2}} \r_{2k+2} \Big)
\sigma_{2k+3} \smallskip \\
\qquad - \ti h \big(\frac{3}{10} \p_{2k+1} + \frac{3}{5} \p_{2k+2} + \frac{1}{10} \p_{2k+3} \big), \quad k=0,1,...,m\,, \smallskip \\
\q_{5k+10}= \frac{1}{10} \, \Big( \frac{d_{k+3}}{d_{k+2}+d_{k+3}}
\r_{2k+3} + \frac{d_{k+2}}{d_{k+2}+d_{k+3}} \r_{2k+4} \Big) \,
\sigma_{2k+1}
+ \frac{3}{5} \r_{2k+3} \, \sigma_{2k+2} +\frac{3}{10} \r_{2k+2} \, \sigma_{2k+3} \smallskip \\
\qquad  -\ti h \big( \frac{1}{10} \p_{2k+1}  + \frac{3}{5} \p_{2k+2} +  \frac{3}{10} \p_{2k+3} \big), \quad k=0,1,...,m\,,  \smallskip \\
\q_{5k+11}= \frac{2}{5} \, \Big( \frac{d_{k+3}}{d_{k+2}+d_{k+3}} \r_{2k+3} + \frac{d_{k+2}}{d_{k+2}+d_{k+3}} \r_{2k+4} \Big) \, \sigma_{2k+2} + \frac{3}{5} \r_{2k+3} \, \sigma_{2k+3} \smallskip\\
\qquad  -\ti h \big( \frac{2}{5} \p_{2k+2} + \frac{3}{5} \p_{2k+3} \big), \quad k=0,1,...,m\,, \smallskip \\
\q_{5m+13}= \frac{3}{5} \, (\r_{2m+4} \sigma_{2m+3} - \ti h \p_{2m+3}), \smallskip \\
\q_{5m+14}= \frac{3}{10}  \, (\r_{2m+5} \sigma_{2m+3} - \ti h \p_{2m+3}), \smallskip \\
\q_{5m+15}= \frac{d_{m+4}}{10 (d_{m+3}+d_{m+4})} \, (\r_{2m+5} \sigma_{2m+3}- \ti h \p_{2m+3} ), \smallskip \\
\q_{5m+16}=\q_{5m+17}=\q_{5m+18}=\q_{5m+19}=0.
\end{array}
$$

\begin{rmk}
Since $\sigma_0=\sigma_{2m+4}=0$ as well as $\p_0=\p_{2m+4}=0$, the
values of  $\zeta_{3}^{0,0}$, $\zeta_{4}^{0,0}$, $\zeta_{5}^{1,0}$,
$\zeta_{6}^{1,0}$, $\zeta_{6}^{2,0}$ and
$\zeta_{5m+13}^{2m+3,2m+4}$, $\zeta_{5m+13}^{2m+4,2m+4}$,
$\zeta_{5m+14}^{2m+4,2m+4}$, $\zeta_{5m+15}^{2m+5,2m+4}$,
$\zeta_{5m+16}^{2m+5,2m+4}$ are indeed not used to compute
$\gamma_k$ and $\q_k$, $k=0,...,5m+19$.
\end{rmk}

The weights and control points of the offset curves $\r_h(t)$ of the
closed quintic PH B-Spline curves ($n=2$) from
(\ref{offsetn2closed}) in section \ref{sec5n2closed} are given by:

$$
\begin{array}{l}
\gamma_0=\gamma_1=\gamma_2=\gamma_3=\gamma_4=\gamma_5=\gamma_6=0,\smallskip\\
\gamma_7=\frac{1}{126} \, \left(\frac{d_{0}}{d_{0}+d_{1}} \right )^2
\, \left( \frac{d_{1}}{d_{1}+d_{2}} \right) \, \sigma_{2}, \smallskip \\
\gamma_8=\frac{5}{126} \, \left( \frac{d_{0}}{d_{0}+d_{1}} \right)
\, \left( \frac{d_{1}}{d_{1}+d_{2}} \right) \, \sigma_{2}, \smallskip\\
\gamma_9=\frac{5}{42} \, \left( \frac{d_{1}}{d_{1}+d_{2}}\right) \, \sigma_{2}, \smallskip\\
\gamma_{10}=\frac{5}{18} \, \left(\frac{d_{1}}{d_{1}+d_{2}} \right) \, \sigma_{2}, \smallskip\\
\gamma_{8k+11}=\frac{4}{9} \sigma_{3k+1} + \frac{5}{9} \frac{\sigma_{3k+2} d_{k+1} + \sigma_{3k+1} d_{k+2}}{d_{k+1}+d_{k+2}}, \quad k=0,...,m+3\,, \smallskip\\
\gamma_{8k+12}=\frac{5}{9} \frac{\sigma_{3k+2} d_{k+1} + \sigma_{3k+1} d_{k+2}}{d_{k+1}+d_{k+2}} +\frac{4}{9} \sigma_{3k+2}, \quad k=0,...,m+3\,, \smallskip\\
\gamma_{8k+13}=\frac{5}{18} \frac{\sigma_{3k+2} d_{k+1} + \sigma_{3k+1} d_{k+2}}{d_{k+1} + d_{k+2}} + \frac{5}{9} \sigma_{3k+2} + \frac16 \sigma_{3k+3}, \quad k=0,...,m+2\,, \smallskip\\
\gamma_{8k+14}=\frac{5}{42} \, \frac{\sigma_{3k+2} d_{k+1} +
\sigma_{3k+1} d_{k+2}}{d_{k+1}+d_{k+2}} +\frac{10}{21} \sigma_{3k+2}
+ \frac{5}{14} \sigma_{3k+3} + \frac{1}{21} \sigma_{3k+4}, \quad k=0,...,m+2\,, \smallskip\\
\gamma_{8k+15}=\frac{5}{126} \frac{\sigma_{3k+2} d_{k+1} +
\sigma_{3k+1} d_{k+2}}{d_{k+1} + d_{k+2}} + \frac{20}{63}
\sigma_{3k+2} +
\frac{10}{21} \sigma_{3k+3} +\frac{10}{63} \sigma_{3k+4} +\frac{1}{126} \frac{\sigma_{3k+5} d_{k+2} + \sigma_{3k+4} d_{k+3}}{d_{k+2}+d_{k+3}}, \quad k=0,...,m+2\,, \smallskip\\
\gamma_{8k+16}=\frac{1}{126} \frac{\sigma_{3k+2} d_{k+1}+
\sigma_{3k+1} d_{k+2}}{d_{k+1}+d_{k+2}} + \frac{10}{63}
\sigma_{3k+2} + \frac{10}{21} \, \sigma_{3k+3}
+\frac{20}{63} \sigma_{3k+4} + \frac{5}{126} \frac{\sigma_{3k+5} d_{k+2} + \sigma_{3k+4} d_{k+3}}{d_{k+2} + d_{k+3}}, \quad k=0,...,m+2\,, \smallskip\\
\gamma_{8k+17}=\frac{1}{21} \sigma_{3k+2} + \frac{5}{14} \sigma_{3k+3} +\frac{10}{21} \sigma_{3k+4} + \frac{5}{42} \, \frac{\sigma_{3k+5} d_{k+2} + \sigma_{3k+4} d_{k+3}}{d_{k+2}+d_{k+3}}, \quad k=0,...,m+2\,, \smallskip\\
\gamma_{8k+18}=\frac{1}{6} \sigma_{3k+3} + \frac{5}{9} \sigma_{3k+4} + \frac{5}{18} \, \frac{\sigma_{3k+5} d_{k+2} + \sigma_{3k+4} d_{k+3}}{d_{k+2}+d_{k+3}}, \quad k=0,...,m+2\,, \smallskip\\
\end{array}
$$
$$
\begin{array}{l}
\gamma_{8m+37}=\frac{5}{18} \, \left(\frac{d_{m+5}}{d_{m+4} + d_{m+5}}\right) \, \sigma_{3m+10}, \smallskip\\
\gamma_{8m+38}=\frac{5}{42} \, \left( \frac{d_{m+5}}{d_{m+4}+d_{m+5}}\right) \, \sigma_{3m+10}, \smallskip\\
\gamma_{8m+39}=\frac{5}{126} \, \left(
\frac{d_{m+6}}{d_{m+5}+d_{m+6}} \right)
\, \left( \frac{d_{m+5}}{d_{m+4}+d_{m+5}} \right) \, \sigma_{3m+10},\smallskip\\
\gamma_{8m+40}=\frac{1}{126} \, \left(
\frac{d_{m+6}}{d_{m+5}+d_{m+6}} \right )^2
\, \left( \frac{d_{m+5}}{d_{m+4}+d_{m+5}} \right) \, \sigma_{3m+10}, \smallskip\\
\gamma_{8m+41}=\gamma_{8m+42}=\gamma_{8m+43}=\gamma_{8m+44}=\gamma_{8m+45}=\gamma_{8m+46}=\gamma_{8m+47}=0
\end{array}
$$
and
$$
\begin{array}{lll}
\q_0&=&\q_1 \quad = \quad \q_2 \quad = \quad \q_3 \quad = \quad \q_4 \quad = \quad \q_5 \quad = \quad \q_6 \quad = \quad 0,\smallskip\\
\q_7&=& \frac{1}{126} \, \left( \frac{d_{0}}{d_{0}+d_{1}} \right )^2
\, \left( \frac{d_{1}}{d_{1}+d_{2}} \right) \, \Big(\r_{0} \sigma_{2} - \ti h \p_{2} \Big), \smallskip\\
\q_8&=& \frac{5}{126} \, \left( \frac{d_{0}}{d_{0}+d_{1}} \right) \,
\left( \frac{d_{1}}{d_{1}+d_{2}} \right)
\, \Big( \r_{0} \sigma_{2} -\ti h \p_{2} \Big), \smallskip\\
\q_9&=& \frac{5}{42} \, \left( \frac{ d_{1} }{d_{1}+d_{2}} \right) \, \Big( \r_{0} \sigma_{2} - \ti h \p_{2} \Big), \smallskip\\
\q_{10}&=&\frac{5}{18} \, \left( \frac{d_{1}}{d_{1}+d_{2}} \right) \, \Big( \r_{1} \sigma_{2} - \ti h \p_{2}  \Big), \smallskip\\
\q_{8k+11}&=&\frac{4}{9} \, \left( \frac{\r_{3k+3} d_{k+1}^2+2
\r_{3k+2} d_{k+1} d_{k+2}+ \r_{3k+1} d_{k+2}^2}{(d_{k+1}+d_{k+2})^2}
\right) \, \sigma_{3k+1} + \frac{5}{9} \, \left( \frac{\r_{3k+2}
d_{k+1} + \r_{3k+1} d_{k+2}}{d_{k+1} + d_{k+2}} \right) \, \left(
\frac{\sigma_{3k+2} d_{k+1} + \sigma_{3k+1} d_{k+2}}{d_{k+1} +
d_{k+2}} \right)
\smallskip \\
&-&\ti \, h \, \Big( \frac{4}{9} \p_{3k+1} + \frac{5}{9} \, \left(
\frac{\p_{3k+2} d_{k+1}+ \p_{3k+1} d_{k+2}}{d_{k+1}+d_{k+2}} \right)
\Big), \quad k=0,...,m+3\,, \smallskip\\
\q_{8k+12}&=&\frac{5}{9} \, \left( \frac{\r_{3k+3} d_{k+1} +
\r_{3k+2} d_{k+2}}{d_{k+1}+d_{k+2}} \right) \, \left(
\frac{\sigma_{3k+2} d_{k+1} + \sigma_{3k+1}
d_{k+2}}{d_{k+1}+d_{k+2}} \right)
+\frac{4}{9} \, \left( \frac{\r_{3k+3} d_{k+1}^2 +2 \r_{3k+2} d_{k+1} d_{k+2}+ \r_{3k+1} d_{k+2}^2}{(d_{k+1}+d_{k+2})^2} \right) \, \sigma_{3k+2} \smallskip \\
&-&\ti \, h \, \Big( \frac{4}{9} \p_{3k+2} + \frac{5}{9} \, \left( \frac{\p_{3k+2} d_{k+1}+\p_{3k+1} d_{k+2}}{d_{k+1}+d_{k+2}} \right) \Big), \quad k=0,...,m+3\,, \smallskip\\
\q_{8k+13}&=&\frac{5}{18} \, \r_{3k+3} \, \left( \frac{
\sigma_{3k+2} d_{k+1} + \sigma_{3k+1} d_{k+2}}{d_{k+1}+d_{k+2}}
\right)
+\frac{5}{9} \, \left( \frac{\r_{3k+3} d_{k+1} + \r_{3k+2} d_{k+2}}{d_{k+1}+d_{k+2}} \right) \, \sigma_{3k+2} \smallskip \\
&+& \frac{1}{6} \, \left( \frac{\r_{3k+3} d_{k+1}^2 + 2\r_{3k+2} d_{k+1} d_{k+2} + \r_{3k+1} d_{k+2}^2}{(d_{k+1}+d_{k+2})^2} \right) \, \sigma_{3k+3} \smallskip \\
&-& \ti \, h \, \Big( \frac16 \p_{3k+3} + \frac{5}{9} \p_{3k+2}
+\frac{5}{18} \, \left(\frac{\p_{3k+2} d_{k+1} + \p_{3k+1} d_{k+2}}{d_{k+1}+d_{k+2}} \right) \Big), \quad k=0,...,m+2\,, \smallskip \\
\q_{8k+14}&=&\frac{5}{42} \, \r_{3k+4} \, \left(\frac{\sigma_{3k+2}
d_{k+1} + \sigma_{3k+1} d_{k+2}}{d_{k+1}+d_{k+2}}\right)
+\frac{10}{21} \, \r_{3k+3} \sigma_{3k+2}
+\frac{5}{14} \, \left( \frac{\r_{3k+3} d_{k+1} + \r_{3k+2} d_{k+2}}{d_{k+1}+d_{k+2}}\right) \, \sigma_{3k+3} \smallskip \\
&+& \frac{1}{21} \, \left( \frac{\r_{3k+3} d_{k+1}^2+2 \r_{3k+2}
d_{k+1} d_{k+2}+ \r_{3k+1} d_{k+2}^2}{(d_{k+1}+d_{k+2})^2} \right)
\, \sigma_{3k+4}
\smallskip \\
&-& \ti h \Big(\frac{1}{21} \p_{3k+4} + \frac{5}{14} \p_{3k+3}
+\frac{10}{21} \p_{3k+2}
+\frac{5}{42} \, \left( \frac{\p_{3k+2} d_{k+1} + \p_{3k+1} d_{k+2}}{d_{k+1}+d_{k+2}} \right) \Big), \quad k=0,...,m+2\,, \smallskip\\
\q_{8k+15}&=&\frac{5}{126} \, \left( \frac{\r_{3k+5}
d_{k+2}+\r_{3k+4} d_{k+3}}{d_{k+2}+d_{k+3}} \right)
\, \left( \frac{\sigma_{3k+2} d_{k+1}+\sigma_{3k+1} d_{k+2}}{d_{k+1}+d_{k+2}} \right) \smallskip \\
&+& \frac{20}{63} \r_{3k+4} \sigma_{3k+2} + \frac{10}{21} \r_{3k+3}
\sigma_{3k+3} +
\frac{10}{63} \, \left( \frac{\r_{3k+3} d_{k+1}+ \r_{3k+2} d_{k+2}}{d_{k+1}+d_{k+2}} \right) \, \sigma_{3k+4} \smallskip \\
&+& \frac{1}{126} \, \left( \frac{\r_{3k+3} d_{k+1}^2 +2 \r_{3k+2}
d_{k+1} d_{k+2} +\r_{3k+1} d_{k+2}^2}{(d_{k+1}+d_{k+2})^2} \right )
\, \left( \frac{\sigma_{3k+5} d_{k+2}+\sigma_{3k+4} d_{k+3}}{d_{k+2}+d_{k+3}} \right) \smallskip \\
&-& \ti h \Big( \frac{1}{126} \, \left( \frac{\p_{3k+5} d_{k+2}+
\p_{3k+4} d_{k+3}}{d_{k+2}+d_{k+3}} \right) +\frac{10}{63} \p_{3k+4}
+ \frac{10}{21} \p_{3k+3} + \frac{20}{63} \p_{3k+2}
+\frac{5}{126} \, \left( \frac{\p_{3k+2} d_{k+1} + \p_{3k+1} d_{k+2}}{d_{k+1}+d_{k+2}} \right)  \Big), \smallskip \\
&& k=0,...,m+2\,, \smallskip \\
\q_{8k+16}&=&\frac{1}{126} \, \left( \frac{\r_{3k+6} d_{k+2}^2+2
\r_{3k+5} d_{k+2} d_{k+3}+\r_{3k+4} d_{k+3}^2}{(d_{k+2}+d_{k+3})^2}
\right)
\, \left( \frac{\sigma_{3k+2} d_{k+1} + \sigma_{3k+1} d_{k+2}}{d_{k+1}+d_{k+2}} \right) \smallskip \\
&+& \frac{10}{63}  \, \left( \frac{\r_{3k+5} d_{k+2}+ \r_{3k+4}
d_{k+3}}{d_{k+2}+d_{k+3}} \right) \, \sigma_{3k+2}
+ \frac{10}{21} \r_{3k+4} \sigma_{3k+3} + \frac{20}{63} \r_{3k+3} \sigma_{3k+4} \smallskip \\
&+& \frac{5}{126} \, \left( \frac{\r_{3k+3} d_{k+1}+\r_{3k+2}
d_{k+2}}{d_{k+1}+d_{k+2}} \right)
\, \left( \frac{\sigma_{3k+5} d_{k+2}+ \sigma_{3k+4} d_{k+3}}{d_{k+2}+d_{k+3}} \right) \smallskip \\
&-&\ti h \Big( \frac{5}{126} \, \left( \frac{\p_{3k+5}
d_{k+2}+\p_{3k+4} d_{k+3}}{d_{k+2}+d_{k+3}} \right) + \frac{20}{63}
\p_{3k+4} + \frac{10}{21} \p_{3k+3} + \frac{10}{63} \p_{3k+2} +
\frac{1}{126} \, \left(\frac{\p_{3k+2} d_{k+1} + \p_{3k+1}
d_{k+2}}{d_{k+1}+d_{k+2}} \right)
 \Big), \smallskip \\
&& k=0,...,m+2\,, \smallskip \\
\q_{8k+17}&=&\frac{1}{21} \, \left(\frac{\r_{3k+6} d_{k+2}^2+2
\r_{3k+5} d_{k+2} d_{k+3}+ \r_{3k+4}
d_{k+3}^2}{(d_{k+2}+d_{k+3})^2}\right) \, \sigma_{3k+2}
+ \frac{5}{14} \, \left( \frac{\r_{3k+5} d_{k+2} + \r_{3k+4} d_{k+3}}{d_{k+2}+d_{k+3}} \right) \,  \sigma_{3k+3} \smallskip \\
&+&\frac{10}{21} \r_{3k+4} \sigma_{3k+4} + \frac{5}{42} \, \r_{3k+3}
\, \left( \frac{\sigma_{3k+5} d_{k+2} + \sigma_{3k+4}
d_{k+3}}{d_{k+2}+d_{k+3}} \right)
\smallskip \\
&-& \ti h \Big( \frac{5}{42} \, \left( \frac{\p_{3k+5} d_{k+2}+
\p_{3k+4} d_{k+3}}{d_{k+2}+d_{k+3}} \right)
+ \frac{10}{21} \p_{3k+4} + \frac{5}{14} \p_{3k+3} + \frac{1}{21} \p_{3k+2} \Big), \quad k=0,...,m+2\,, \smallskip\\
\end{array}
$$
$$
\begin{array}{lll}
\q_{8k+18}&=&\frac{1}{6} \, \left( \frac{\r_{3k+6} d_{k+2}^2 +2
\r_{3k+5} d_{k+2} d_{k+3}+ \r_{3k+4} d_{k+3}^2}{(d_{k+2}+d_{k+3})^2}
\right) \,  \sigma_{3k+3}
+ \frac{5}{9} \, \left( \frac{ \r_{3k+5} d_{k+2} + \r_{3k+4} d_{k+3}}{d_{k+2}+d_{k+3}}\right) \, \sigma_{3k+4} \smallskip \\
&+& \frac{5}{18} \, \r_{3k+4} \, \left( \frac{\sigma_{3k+5} d_{k+2}
+ \sigma_{3k+4} d_{k+3}}{d_{k+2}+d_{k+3}} \right) -\ti h
\Big(\frac{5}{18} \, \left( \frac{\p_{3k+5} d_{k+2}+\p_{3k+4}
d_{k+3}}{d_{k+2}+d_{k+3}} \right)
+ \frac{5}{9} \p_{3k+4} + \frac{1}{6} \p_{3k+3} \Big), \smallskip \\
&& k=0,...,m+2\,, \smallskip \\
\q_{8m+37}&=&\frac{5}{18} \, \left( \frac{ d_{m+5}}{d_{m+4}+d_{m+5}}
\right)
\, \Big( \r_{3m+12} \sigma_{3m+10} - \ti \, h \p_{3m+10} \Big), \smallskip\\
\q_{8m+38}&=& \frac{5}{42} \,
\left(\frac{d_{m+5}}{d_{m+4}+d_{m+5}}\right) \,
\Big( \r_{3m+13}  \sigma_{3m+10} - \ti h  \p_{3m+10} \Big), \smallskip\\
\q_{8m+39}&=& \frac{5}{126} \, \left(
\frac{d_{m+6}}{d_{m+5}+d_{m+6}} \right) \, \left(
\frac{d_{m+5}}{d_{m+4}+d_{m+5}} \right) \Big( \r_{3m+13}
\sigma_{3m+10}
- \ti h \p_{3m+10} \Big), \smallskip\\
\q_{8m+40}&=& \frac{1}{126} \, \left(
\frac{d_{m+6}}{d_{m+5}+d_{m+6}} \right )^2 \, \left(
\frac{d_{m+5}}{d_{m+4}+d_{m+5}} \right) \,
\Big( \r_{3m+13} \sigma_{3m+10} - \ti h \p_{3m+10} \Big),  \smallskip\\
\q_{8m+41}&=& \q_{8m+42} \quad = \quad \q_{8m+43} \quad = \quad
\q_{8m+44} \quad = \quad \q_{8m+45} \quad = \quad \q_{8m+46} \quad =
\quad \q_{8m+47}=0.
\end{array}
$$

\begin{rmk}
Since $\sigma_0=\sigma_1=\sigma_{3m+11}=\sigma_{3m+12}=0$ as well as
$\p_0=\p_1=\p_{3m+11}=\p_{3m+12}=0$, the values of
$\zeta_{5}^{0,0}$, $\zeta_{6}^{0,0}$, $\zeta_{6}^{0,1}$,
$\zeta_{7}^{0,0}$, $\zeta_{7}^{0,1}$, $\zeta_{8}^{0,1}$,
$\zeta_{8}^{1,0}$, $\zeta_{9}^{0,1}$, $\zeta_{9}^{1,0}$,
$\zeta_{9}^{1,1}$, $\zeta_{9}^{2,0}$, $\zeta_{10}^{1,0}$,
$\zeta_{10}^{1,1}$, $\zeta_{10}^{2,0}$, $\zeta_{10}^{2,1}$,
$\zeta_{10}^{3,0}$ and $\zeta_{8m+37}^{3m+10,3m+12}$,
$\zeta_{8m+37}^{3m+11,3m+11}$, $\zeta_{8m+37}^{3m+11,3m+12}$,
$\zeta_{8m+37}^{3m+12,3m+11}$, $\zeta_{8m+37}^{3m+12,3m+12}$,
$\zeta_{8m+38}^{3m+11,3m+12}$, $\zeta_{8m+38}^{3m+12,3m+11}$,
$\zeta_{8m+38}^{3m+12,3m+12}$, $\zeta_{8m+38}^{3m+13,3m+11}$,
$\zeta_{8m+39}^{3m+12,3m+12}$, $\zeta_{8m+39}^{3m+13,3m+11}$,
$\zeta_{8m+40}^{3m+13,3m+11}$, $\zeta_{8m+40}^{3m+13,3m+12}$,
$\zeta_{8m+41}^{3m+13,3m+11}$, $\zeta_{8m+41}^{3m+13,3m+12}$,
$\zeta_{8m+42}^{3m+13,3m+12}$ are indeed not used to compute
$\gamma_k$ and $\q_k$, $k=0,...,8m+47$.
\end{rmk}

\bigskip
\noindent {\bf Section \ref{sec6}}

 In section \ref{sec6} the coefficients of
the symmetric matrices $A=\left(a_{i,j}\right)_{i,j=0,1,2}$ and
$B=\left(b_{i,j}\right)_{i,j=0,1,2}$ in equation \eqref{conics} read
as follows:

$$
\begin{array}{lll}
a_{0,0} &=& \p^*_{0,x} - \p^*_{1,x} + \frac{1}{5} (a
\d_{0,x}+(1-a)\d_{1,x}) - \frac{1}{240 u_0^2} \kappa_0^2  a^2 (3-a)
(u_0^2+v_0^2)^4 - \frac{1}{240u_3^2}
\kappa_1^2(1-a)^2(2+a)(u_3^2+v_3^2)^4 \smallskip\\
&& + \frac{1} {80u_0u_3} \kappa_0\kappa_1
a(1-a)(u_0^2u_3^2+v_0^2v_3^2+u_3^2v_0^2+u_0^2v_3^2)^2
 - \frac{1}{60u_0} \kappa_0 a
(u_0^2+v_0^2)^2 (a(4-a)v_0+(1-a)^2v_3) \smallskip\\
&& + \frac{1}{60u_3} \kappa_1 (1-a)
(u_3^2+v_3^2)^2(a^2v_0+(1-a)(3+a)v_3),
\end{array}
$$

$$
\begin{array}{lll}
a_{0,1} &=& \frac{1}{10} \left[ \frac{1}{3 u_0}
(a(4-a)(u_0^2-v_0^2)+(1-a)^2(u_0u_3-v_0v_3)) - \frac{1}{6 u_0^2}
\kappa_0 v_0 a(3-a)(u_0^2+v_0^2)^2 + \frac{1}{4 u_0u_3} \kappa_1
v_0(1-a)(u_3^2+v_3^2)^2 \right],
\end{array}
$$

$$
\begin{array}{lll}
a_{0,2} &=& \frac{1}{10} \left[  \frac{1}{6 u_3^2} \kappa_1 v_3
(a+2)(1-a)(u_3^2+v_3^2)^2 - \frac{1}{4 u_0 u_3} \kappa_0 a v_3
(u_0^2+v_0^2)^2
 + \frac{1}{3 u_3} (a^2 (u_0u_3-v_0 v_3)+(a+3)(1-a)(u_3^2-v_3^2))
\right],
\end{array}
$$

\[
a_{1,1} = \frac{(u_0^2-v_0^2)(3-a)}{15u_0^2}\, , \;\; a_{1,2} =
\frac{u_0 u_3-v_0 v_3}{10 u_0 u_3}\,, \;\; a_{2,2} = \frac{(u_3^2-
v_3^2)(2+a)}{15 u_3^2},
\]

$$
\begin{array}{lll}
b_{0,0}& =& \p^*_{0,y} - \p^*_{1,y} + \frac{1}{5} (a
\d_{0,y}+(1-a)\d_{1,y}) + \frac{1}{60 u_0}  \kappa_0 a
(u_0^2+v_0^2)^2\left[a(4-a)u_0+(1-a)^2u_3\right]  \smallskip\\
&& - \frac{1}{60u_3} \kappa_1 (1-a)
(u_3^2+v_3^2)^2\left[u_3(1-a)(3+a)+u_0a^2\right]\,,
\end{array}
$$

$$
\begin{array}{lll}
b_{0,1} &= & \frac12\left[ \frac{1}{30 u_0}  \kappa_0 a
(3-a)(u_0^2+v_0^2)^2+ \frac{1}{15 u_0} (1-a)^2 (u_0v_3+u_3v_0) +
\frac{1}{15} 2av_0(4-a) - \frac{1}{20u_3}
\kappa_1(1-a)(u_3^2+v_3^2)^2 \right]\,,
\end{array}
$$

$$
\begin{array}{lll}
b_{0,2} &=& \frac{1}{20} \left[ \frac{1}{3 u_3} (2
a^2(u_3v_0+u_0v_3)+4 u_3
v_3(3+a)(1-a)-\kappa_1(2+a)(1-a)(u_3^2+v_3^2)^2)
 + \frac{1}{2u_0}  \kappa_0 a (u_0^2+v_0^2)^2 \right],
\end{array}
$$

\[
b_{1,1} = \frac{2 v_0 (3-a)}{15 u_0}\, , \;\; b_{1,2} =
\frac{v_0u_3+u_0v_3}{10 u_0u_3} \, , \;\; b_{2,2} = \frac{2
v_3(2+a)}{15 u_3}\,.
\]

\end{document}